\documentclass[reqno]{amsart}

\usepackage{amsfonts, upgreek}
\usepackage{amssymb, amsopn, latexsym,amsmath,graphics,verbatim,amsthm,enumerate,amscd,tabularx}
\usepackage[all,cmtip]{xy}
\usepackage{mathtools, graphicx}
\usepackage{fancyhdr}
\usepackage[active]{srcltx}
\usepackage[british]{babel}
\usepackage{todonotes}
\usepackage{amsmath,amssymb,amsthm,mathrsfs, url}
\usepackage[T1]{fontenc}
\usepackage[flushleft]{paralist}
\usepackage{hyperref}
\usepackage{ae}
\usepackage{xcolor, dsfont}
\usepackage{aecompl}
\usepackage{scalerel}[2014/03/10]
\usepackage{stackengine}
\usepackage{tikz-cd}
\tikzcdset{scale cd/.style={every label/.append style={scale=#1},
		cells={nodes={scale=#1}}}}

\theoremstyle{plain}
\newtheorem*{theorem*}{Theorem}
\newtheorem*{conjecture*}{Conjecture}

\newtheorem*{conjectureA*}{Conjecture A*}

\newtheorem{thm}{Theorem}[section]
\newtheorem{theorem}[thm]{Theorem}

\newtheorem{corollary}[thm]{Corollary}
\newtheorem{lemma}[thm]{Lemma}
\newtheorem*{lemma*}{Lemma}
\newtheorem{conjecture}[thm]{Conjecture}
\newtheorem{proposition}[thm]{Proposition}
\theoremstyle{definition}
\newtheorem{definition}[thm]{Definition}
\theoremstyle{remark}
\newtheorem{remark}[thm]{Remark}

\input cyracc.def \font\tencyr=wncyr10 \def\russe{\tencyr\cyracc}
\def\Sha{\text{\russe{Sh}}}

\newdir{ (}{{}*!/-5pt/@^{(}}

\SelectTips{eu}{10}
\UseTips
\entrymodifiers={+!!<0pt,\fontdimen22\textfont2>}

\DeclareMathOperator{\Gal}{Gal}
\DeclareMathOperator{\rank}{rank}
\DeclareMathOperator{\corank}{corank}

\DeclareMathOperator{\ord}{ord}
\DeclareMathOperator{\Selm}{Sel}

\DeclareMathOperator{\img}{img}
\DeclareMathOperator{\coker}{coker}
\DeclareMathOperator{\Hom}{Hom}

\DeclareMathOperator{\res}{res}

\DeclareMathOperator{\coind}{Coind}

\newcommand{\Q}{{\mathbb{Q}}}
\newcommand{\Z}{{\mathbb{Z}}}
\newcommand{\N}{{\mathbb{N}}}

\newcommand{\K}{\mathbb{K}}
\newcommand{\Fp}{{\mathbb{F}_p}}

\newcommand{\Zp}{{\mathbb{Z}_p}}

\newcommand{\Qp}{{\mathbb{Q}_p}}
\newcommand{\Ql}{{\mathbb{Q}_l}}
\newcommand{\ilim}{\mathop{\varprojlim}\limits}
\newcommand{\dlim}{\mathop{\varinjlim}\limits}
\newcommand{\Sel}{{\Selm_p}}
\newcommand{\Selinf}{{\Selm_{p^{\infty}}}}
\newcommand{\Selinfg}{{\Selm^{Gr}_{p^{\infty}}}}

\newcommand{\cN}{\mathcal{N}}

\newcommand{\cO}{\mathcal{O}}

\newcommand{\cy}[1]{\mathbb{Z}/#1\mathbb{Z}}
\newcommand{\fp}{\mathfrak{p}}

\newcommand{\fG}{\mathfrak{G}}
\newcommand{\fH}{\mathfrak{H}}

\newcommand{\fq}{\mathfrak{q}}
\newcommand{\KK}{\mathds{K}}

\newcommand{\rg}{\textup{rank}}

\newcommand{\Selpm}{\Selm^{\bullet\star}_{p^\infty}}
\newcommand{\Selgpm}{\Selm^{Gr, \bullet\star}_{p^\infty}}

\newcommand{\Selpmsingle}{\Selm^{\pm}_{p^\infty}}

\newcommand{\overbar}[1]{\mkern 1.5mu\overline{\mkern-1.5mu#1\mkern-1.5mu}\mkern 1.5mu}

\newcommand{\joinrelshort}{\mathrel{\mkern-9mu}}
\newcommand{\shortlongrightarrow}{\relbar\joinrelshort\rightarrow}
\newcommand{\isomarrow}{\mathrel{\mathop{\setbox0\hbox{$\mathsurround0pt
        \shortlongrightarrow$}\ht0=0.7\ht0\box0}\limits
        ^{\sim\mkern2mu}}}

\newcommand\reallywidehat[1]{\arraycolsep=0pt\relax%
\begin{array}{c}
\stretchto{
  \scaleto{
    \scalerel*[\widthof{\ensuremath{#1}}]{\kern-.5pt\bigwedge\kern-.5pt}
    {\rule[-\textheight/2]{1ex}{\textheight}} %WIDTH-LIMITED BIG WEDGE
  }{\textheight} %
}{0.5ex}\\           % THIS SQUEEZES THE WEDGE TO 0.5ex HEIGHT
#1\\                 % THIS STACKS THE WEDGE ATOP THE ARGUMENT
\rule{-1ex}{0ex}
\end{array}
}

\newenvironment{outerlist}[1][\enskip\textbullet]%
        {\begin{itemize}[#1]}{\end{itemize}%
         \vspace{-0.1\baselineskip}}

\def \br #1 {{\color{blue} #1 }}
\def \brr #1 {{\color{red} #1}}
\def \brrr #1 {{\color{purple} #1 }}

\begin{document}

\title[The $\mathfrak{M}_H(G)$-property for Selmer groups at supersingular reduction]{On the $\mathfrak{M}_H(G)$-property for Selmer groups at supersingular reduction}

\date{\today}

\author{S\"{o}ren Kleine, Ahmed Matar, and Sujatha Ramdorai}

\begin{abstract}
 Let $E$ be an elliptic curve defined over $\Q$ which has good supersingular reduction at the odd prime $p$. We study the variation of Iwasawa invariants and the $\mathfrak{M}_H(G)$-property for signed Selmer groups over $\Z_p$-extensions of an imaginary quadratic number field $K$ that lie inside the $\Z_p^2$-extension $\KK_\infty$ of $K$ and are not necessarily cyclotomic. We prove several equivalent criteria for the validity of the $\mathfrak{M}_H(G)$-property which involve the growth of $\mu$-invariants of the signed Selmer groups over intermediate shifted $\Z_p$-extensions in $\KK_\infty$, and the boundedness of $\lambda$-invariants as one runs over $\Z_p$-extensions of $K$ inside $\KK_\infty$. We give examples where the $\mathfrak{M}_H(G)$-property holds, and also examples where we can prove that it does not hold.

 It is striking that although the case of supersingular reduction is much more difficult than the case of ordinary reduction, we get finer results here; moreover, we are able to derive analogous criteria for the validity of the  $\mathfrak{M}_H(G)$-property of the classical Selmer group, as well as the fine Selmer group. Many of the properties that we investigate have not been studied before in this non-torsion setting.

 Further, we study various implications between the $\mathfrak{M}_H(G)$-properties for Selmer groups, signed Selmer groups and fine Selmer groups. We apply our results to a conjecture of Mazur, and prove implications between the $\mathfrak{M}_H(G)$-property and Conjectures~A and B of Coates and Sujatha.
\end{abstract}

\maketitle

\tableofcontents

\part{Introduction and overview of results}
Before giving a detailed description of the contents of this paper we provide a short overview.
\section{Introduction}
This short introduction contains a brief description of the results of the paper. Full details and precise statements are given in the next section. Let $p$ be a rational prime, and let $E$ be an elliptic curve defined over $\Q$ with good ordinary reduction at $p$. Let $K$ be a number field with at least one complex prime which we assume to be totally imaginary if ${p=2}$. Denote by $\KK_{\infty}$ a $\mathbb{Z}^2_p$-extension of $K$. In our earlier paper \cite{MHG}, we defined a subset of $\Zp$-subextensions of $\KK_{\infty}/K$ which is denoted here again by $\mathcal{H}$. Let $X(E/\KK_{\infty})$ be the Pontryagin dual of the classical $p$-primary Selmer group of $E/\KK_{\infty}$. Let ${L_{\infty} \in \mathcal{H}}$ and ${H=\Gal(\KK_{\infty}/L_{\infty})}$. In loc. cit. we discussed the property when $X(E/\KK_\infty)/X(E/\KK_{\infty})[p^{\infty}]$ is finitely generated over $\Zp[[H]]$, which was called the $\mathfrak{M}_H(G)$-property, and proved five equivalent conditions for this property to hold.

In this paper, we handle the supersingular case: assume that ${p \geq 5}$, $E$ has good supersingular reduction at $p$ and $K$ is an imaginary quadratic field in which $p$ splits. We study the $\mathfrak{M}_H(G)$-property even in this case, where, instead of working with the classical $p$-primary Selmer group we first work with the signed Selmer groups $\Selpm(E/\KK_{\infty})$, where ${\bullet, \star \in \{+,-\}}$. In this setup, we prove similar equivalences to the $\mathfrak{M}_H(G)$-property as we did in loc. cit. In addition, we prove four entirely new equivalences. We use the equivalent criteria for deriving concrete examples where the $\mathfrak{M}_H(G)$-property holds, and we also find examples where it fails. The latter class of examples is closely connected with $\Z_p$-extensions of $K$ in which some primes of bad reduction split completely. In order to be able to handle these $\Z_p$-extensions, we have to generalize the theory accordingly.

We also study the $\mathfrak{M}_H(G)$-property for the Selmer group $\Selinf(E/\KK_{\infty})$, whose Pontryagin dual we denote by $X(E/\KK_{\infty})$. Since $X(E/\KK_{\infty})$ is not a torsion $\Lambda_2$-module, we need to consider its $\Lambda_2$-torsion submodule when studying the $\mathfrak{M}_H(G)$-property. Nevertheless, we can prove many equivalences also in this non-torsion setting. To the best of our knowledge, we are the first to study all of these properties of Selmer groups in the supersingular case. This requires some new tools such as a control theorem between the $\Lambda_2$-torsion submodule of $X(E/\KK_{\infty})$ and the $\Lambda$-torsion submodule of $X(E/L_{\infty})$ where $L_{\infty}/K$ is a $\Zp$-extension. The novelty in some of our proofs lies in the use of suitable  techniques from commutative and homological algebra. Moreover, we also study fine Selmer groups, which encode information related to the weak Leopoldt conjecture. Similar criteria equivalent to the $\mathfrak{M}_H(G)$-property for the fine Selmer group are derived. We will show among other things that the $\mathfrak{M}_H(G)$-property for the signed Selmer group implies the $\mathfrak{M}_H(G)$-property for the Selmer group, and that the latter is equivalent to the $\mathfrak{M}_H(G)$-property for the fine Selmer group. Moreover, we study relationships between Conjectures A and B of the Coates-Sujatha paper \cite{CS1} and the $\mathfrak{M}_H(G)$-property.

In \cite{MHG} we also discussed a conjecture of Mazur which discusses the number of $\Zp$-extensions of $K$ where the rank of $E$ does not stay bounded. In addition, we establish an interesting relationship between this conjecture and the $\mathfrak{M}_H(G)$-property. We then used this relationship to construct examples where the conjecture holds. In the present article we study a similar relationship in the supersingular case. Constructing examples in this case is much more challenging since our formula involves a sum of $\lambda$-invariants involving four pairs of signs: + +, -- --, + --, -- +. Our attempts to construct examples are presented in the last section.

Although we restricted ourselves to the case of elliptic curves, we are convinced that much of the theory developed in the current paper can be generalized to abelian varieties or more general $p$-adic Galois representations. In the next section we state our precise results in full detail.

\section{Basic definitions and statement of results}\label{section:mainintroduction}

This long section is divided into five subsections. After fixing notation that will be used throughout the paper, we describe in detail the results that will be proved.

\subsection{Definitions and important notation}\label{section:definitions}
Let ${p \geq 5}$ be a rational prime, and let $E$ be an elliptic curve defined over $\Q$ with good supersingular reduction at $p$. Let $K$ be an imaginary quadratic field. Assume that $p$ splits in $K/\Q$. Denote by $\KK_{\infty}$ the $\mathbb{Z}^2_p$-extension of $K$. Let $S$ be a finite set of nonarchimedean primes of $K$ containing all the primes dividing $p$ and all the primes where $E$ has bad reduction. Also let $S_p$ be the set of primes of $K$ above $p$. We let $K_S$ be the maximal extension of $K$ unramified outside $S$. Suppose now that $L$ is a field with ${K \subseteq L \subseteq K_S}$. We let ${G_S(L)=\Gal(K_S/L)}$. For any ${v \in S}$ we define ${J_v(E/L)=\dlim \bigoplus_{w|v}H^1(F_w,E)[p^{\infty}]}$, where the direct limit runs over finite extensions $F$ of $K$ contained in $L$ and where $w$ denotes any prime above $v$ in $F$, respectively. The  $p^{\infty}$-Selmer group of $E/L$ is defined as

$$\displaystyle 0 \longrightarrow \Selinf(E/L) \longrightarrow H^1(G_S(L), E[p^{\infty}]) \longrightarrow \bigoplus_{v \in S} J_v(E/L).$$

We denote the Pontryagin dual of $\Selinf(E/L)$ by $X(E/L)$.

Also for any ${v \in S}$ we define ${M_v(E/L)=\dlim \bigoplus_{w|v}H^1(F_w,E[p^{\infty}])}$ where the direct limit runs over the finite extensions $F$ of $K$ contained in $L$. The $p^\infty$-fine Selmer group of $E$ over $L$ is defined via the exact sequence

$$\displaystyle 0 \longrightarrow R_{p^{\infty}}(E/L) \longrightarrow H^1(G_S(L), E[p^{\infty}]) \longrightarrow \bigoplus_{v \in S} M_v(E/L).$$

We denote the Pontryagin dual of $R_{p^\infty}(E/L)$ by $Y(E/L)$.
It is well-known that, as opposed to the ordinary case, the Selmer group over the cyclotomic $\Z_p$-extension ${L = K_{cyc}}$ is not cotorsion over the Iwasawa algebra ${\Lambda = \Lambda_1 = \Z_p[[T]]}$ (see \cite[Proposition~5.3]{IP}). On the contrary, $Y(E/K_{cyc})$ is known to be $\Lambda$-torsion (this can be derived from Proposition~\ref{Hcyc_prop}, for example); in fact, this is an equivalent formulation of the \emph{weak Leopoldt conjecture} for $E$ over $K_{cyc}$, see \cite[Lemma~7.1]{Lim17}.

In \cite{Kob}, Kobayashi constructed \emph{signed Selmer groups} $\Selpmsingle(E/\Q_{cyc})$ which he proved to be cotorsion over the Iwasawa algebra $\Zp[[\Gal(\Q_{cyc}/\Q)]]$ (here, $\Q_{cyc}$ is the cyclotomic $\Zp$-extension of $\Q$). Kobayashi's construction was generalized by Kim \cite{KimSignedSelmer} where he introduced multi-signed Selmer groups. Following Kobayashi and Kim, we make the following definitions.

Let $K'$ be a finite extension of $K$, and suppose that $L_{\infty}/K'$ denotes a $\Z_p$-extension with intermediate layers $L_n$, ${n \in \N}$. Let $v$ be a prime of $K'$ above $p$. Assume that $v$ is totally ramified in $L_\infty/K'$ (we will always assume that this crucial hypothesis is satisfied, and we will enlarge the base field if necessary in order to achieve that assumption). For ${n \in \N}$ let $L_{n,v}$ denote the completion of $L_n$ at some prime above $v$, respectively.

For any ${n \in \N}$, we define the following subgroups of $\hat{E}(L_{n,v})$ where $\hat{E}$ is the formal group of $E/\Qp$:
\begin{equation}\label{plusdef}
\begin{tikzcd}[scale cd=0.88]
\hat{E}^+(L_{n,v}) = \{ P \in \hat{E}(L_{n,v}) \mid \textup{Tr}_{n/m+1}(P) \in \hat{E}(L_{m,v}) \textup{ for each even $m$, $0 \le m \le n-1$}\},
\end{tikzcd}
\end{equation}
\begin{equation}\label{minusdef}
\begin{tikzcd}[scale cd=0.88]
\hat{E}^-(L_{n,v}) = \{ P \in \hat{E}(L_{n,v}) \mid \textup{Tr}_{n/m+1}(P) \in \hat{E}(L_{m,v}) \textup{ for each odd $m$, $0 \le m \le n-1$}\}.
\end{tikzcd}
\end{equation}
In other words, in the definitions of the groups $\hat{E}^+$ and $\hat{E}^-$ different sets of indices (even $m$ vs. odd $m$) matter.

We let
\begin{align} \label{eq:def2-3} \hat{E}^\pm (L_{\infty,v}) = \bigcup_n \hat{E}^\pm(L_{n,v}).
	\end{align}

Fix a prime $\fp$ of $K$ above $p$ and let $\bar{\fp}$ be the other prime above $p$ (recall that $p$ splits in $K/\Q$). Let $L^{(\fp)}_{\infty}/K$ and $L^{(\bar{\fp})}_{\infty}/K$ be the $\Zp$-extensions in which $\fp$ (respectively $\bar{\fp}$) is unramified.

Now let $L_{\infty}/K$ be a $\Zp$-extension. If $L_{\infty} \neq L^{(\fp)}_{\infty}$, we define $i(L_{\infty}, \fp)$ to be $m$ where $L_m$ is the $m$-th layer of the $\Zp$-extension $L_{\infty}/K$ such that $L_m=L_{\infty} \cap L^{(\fp)}_{\infty}$. Note that the index $i(L_\infty, \fp)$ is designed such that $L_{\infty}/L_{i(L_{\infty}, \fp)}$ is totally ramified at all primes above $\fp$. If ${L_{\infty} \neq L^{(\bar{\fp})}_{\infty}}$, we similarly define $i(L_{\infty}, \bar{\fp})$ by intersecting with $L^{(\bar{\fp})}_{\infty}$.

If $L_{\infty} \neq L^{(\fp)}_{\infty}$ and $v$ is a prime of $L_{\infty}$ above $\fp$, we define $\hat{E}^{\pm}(L_{\infty,v})$ as in \eqref{plusdef}, \eqref{minusdef} and \eqref{eq:def2-3} with respect to the $\Zp$-extension $L_{\infty}/L_{i(L_{\infty}, \fp)}$ (the Galois group of this extension is ${\Gal(L_\infty/K)^{p^{i(L_\infty, \fp)}} \cong \Z_p}$). If $L_{\infty}=L^{(\fp)}_{\infty}$ and $v$ is a prime of $L_{\infty}$ above $\fp$, we define $\hat{E}^{\pm}(L_{\infty,v})$ to be $\hat{E}(L_{\infty,v})$. Although this last definition in the case of $L^{(\fp)}_{\infty}$ might seem strange, it is actually natural if one takes into account the definition of $\hat{E}^{\pm}(\KK_{\infty,v})$ in Section~\ref{section:control} (cf. the proof of Lemma~\ref{containment_lemma}). If $v$ is a prime of $L_{\infty}$ above $\bar{\fp}$, we make similar definitions for $\hat{E}^{\pm}(L_{\infty,v})$.

Let $\bullet, \star \in \{+,-\}$. We now define (multi-)signed Selmer groups of $E$ over $L_{\infty}$ via the exact sequence

\begin{equation}\label{Selpm_def}
\displaystyle 0 \longrightarrow \Selpm(E/L_{\infty}) \longrightarrow \Selinf(E/L_{\infty}) \longrightarrow J_\fp^\bullet(E/L_\infty) \times J_{\bar{\fp}}^\star(E/L_\infty),
\end{equation}
where
\begin{equation}\label{J_v_def}
	J_\fp^\bullet(E/L_\infty) \times J_{\bar{\fp}}^\star(E/L_\infty) = \prod_{v \mid \fp} \frac{H^1(L_{\infty,v}, E[p^\infty])}{\hat{E}^{\bullet}(L_{\infty,v}) \otimes \Q_p/\Z_p} \times \prod_{v \mid \bar{\fp}} \frac{H^1(L_{\infty,v}, E[p^\infty])}{\hat{E}^{\star}(L_{\infty,v}) \otimes \Q_p/\Z_p}.
\end{equation}

We also define in Section~\ref{section:control} signed Selmer groups $\Selpm(E/\KK_{\infty})$ over the $\Z_p^2$-extension $\KK_{\infty}$ of $K$. When $L_{\infty}/K$ is totally ramified at all primes of $K$ above $p$, the signed Selmer group $\Selpm(E/L_{\infty})$ is defined in the standard way as in \cite{Kob}, \cite{IP} and \cite{Hamidi}. However, when $L_{\infty}/K$ is not totally ramified at primes of $K$ above $p$, the definition of $\Selpm(E/L_{\infty})$ may seem somewhat strange to the reader. We have made this definition in order to ensure that we have a control theorem (Theorem~\ref{SelmerControlThm_prop}). The paper \cite{GHKL} defines signed Selmer groups using ray class fields following \cite{KimSignedSelmer}. It is not clear to us how our definition compares to the one in \cite{GHKL}.

Now let $G:=\Gal(\KK_{\infty}/K)$, and recall that ${G \cong \Z_p^2}$. We denote the Iwasawa algebra ${\Lambda(G)=\Zp[[G]]}$ by $\Lambda_2$. If $\sigma$ and $\tau$ are topological generators of $G$, then ${\Lambda_2 \cong \Zp[[T_1, T_2]]}$ via the map sending $\sigma-1$ to $T_1$ and $\tau-1$ to $T_2$. Moreover, if ${\Gamma \cong \Z_p}$ denotes a quotient of $G$, then we will abbreviate $\Lambda(\Gamma)$ to $\Lambda$ if the group $\Gamma$ is clear from the context.

Let $\mathcal{E}$ be the set of all $\Zp$-extensions of $K$. Then $\mathcal{E}$ is in bijection with the elements of
$$\mathbb{P}^1(\Zp)=\{(a,b)\in \mathbb{Z}^2_p \, | \, p \,\,\text{does not divide both} \,a \,\text{and}\, b\}/\sim,$$
where $(a_1,b_1)\sim(a_2,b_2)$ if there exists $t \in \mathbb{Z}^{\times}_p$ with $a_1=ta_2$ and $b_1=tb_2$. In the bijection an element $[(a,b)] \in \mathbb{P}^1(\Zp)$ maps to the $\Zp$-extension being the fixed field of $\KK_{\infty}$ of the closed subgroup generated by $\sigma^a\tau^b$.

In \cite{Gb_II} Greenberg introduced the following natural topology on $\mathcal{E}$ (which we call Greenberg's topology). For $L_{\infty} \in \mathcal{E}$ and $n$ a positive integer we define
$$\mathcal{E}(L_{\infty},n):=\{L'_{\infty} \in \mathcal{E} \, | \, [L'_{\infty} \cap L_{\infty} :K] \geq p^n\}.$$
This means that $\mathcal{E}(L_{\infty},n)$ consists of all $\Zp$-extensions of $K$ which coincide with $L_{\infty}$ at least up to the $n$-th layer. Taking $\mathcal{E}(L_{\infty},n)$ as a base of neighborhoods of $L_{\infty}$ yields a topology on $\mathcal{E}$, and in fact $\mathcal{E}$ is compact with regard to this topology. The bijection ${\mathbb{P}^1(\Zp) \leftrightarrow \mathcal{E}}$ becomes a homeomorphism with regard to Greenberg's topology.

For $\bullet, \star \in \{+,-\}$ we let $X^{\bullet\star}(E/\KK_{\infty})$ be the Pontryagin dual of $\Selpm(E/\KK_{\infty})$. Also if ${L_{\infty} \in \mathcal{E}}$, then we let $X^{\bullet\star}(E/L_{\infty})$ be the Pontryagin dual of $\Selpm(E/L_{\infty})$, and we denote the Iwasawa algebra $\Zp[[\Gal(L_{\infty}/K)]]$ by $\Lambda(\Gal(L_{\infty}/K))$ or when it is clear from context simply by $\Lambda$.

For any subgroup $H$ of $G$ we denote its fixed field by $\KK_{\infty}^H$. We now define for each choice of the signs ${\bullet, \star \in \{+, -\}}$ a certain subset $\mathcal{H}^{\bullet\star}$ of subgroups $H$ of $G$ which are isomorphic to $\Z_p$. In view of the bijection ${\mathbb{P}^1(\Z_p) \leftrightarrow \mathcal{E}}$, the subset  $\mathcal{H}^{\bullet\star}$ will correspond to a subset of $\Z_p$-extensions of $K$ which have certain properties.
\begin{definition} \label{def:mathcal{H}}
For ${\bullet, \star \in \{+,-\}}$ let $\mathcal{H}^{\bullet\star}$ be the set of all subgroups ${H = \overbar{\langle \sigma^a\tau^b \rangle}}$ of ${G = \Gal(\KK_{\infty}/K)}$ which are topologically generated by $\sigma^a\tau^b$ for some ${[(a,b)] \in \mathbb{P}^1(\Zp)}$ such that
\begin{enumerate}[(a)]
\item No prime in $S$ splits completely in $\KK_{\infty}^H/K$,
\item Both primes of $K$ above $p$ ramify in $\KK_{\infty}^H/K$,
\item $X^{\bullet\star}(E/\KK_{\infty}^H)$ is a torsion $\Lambda_{G/H}$-module, where ${\Lambda_{G/H} = \Z_p[[G/H]]}$.
\end{enumerate}
By abuse of notation, we say that a $\Z_p$-extension ${L_{\infty} \in \mathcal{E}}$ of $K$ is contained in $\mathcal{H}^{\bullet\star}$ iff ${\Gal(\KK_{\infty}/L_{\infty}) \in \mathcal{H}^{\bullet\star}}$.
\end{definition}
We will show in Lemma~\ref{Xinf_torsion_lemma} that $X^{\bullet\star}(E/\K_\infty)$ is a torsion $\Lambda_2$-module if ${\mathcal{H}^{\bullet\star} \ne \emptyset}$.
Recalling that $K_{cyc}$ denotes the cyclotomic $\Z_p$-extension of $K$, we have that ${H_{cyc}:=\Gal(\KK_{\infty}/K_{cyc})}$ is contained in both $\mathcal{H}^{++}$ and $\mathcal{H}^{--}$ (see Proposition~\ref{Hcyc_prop}). Therefore $\mathcal{H}^{++}$ and $\mathcal{H}^{--}$ are non-empty. For ${\bullet \ne \star}$ the fact that ${\mathcal{H}^{\bullet\star} \ne \emptyset}$ is not known, and we will impose this  condition as an assumption.

If $\mathcal{H}^{\bullet\star}$ is non-empty, then for all but finitely many $L_{\infty} \in \mathcal{E}$ we have $L_{\infty} \in \mathcal{H}^{\bullet\star}$ (see Proposition \ref{Zpextensions_prop}). We note that the sets $\mathcal{H}^{++}$ and $\mathcal{H}^{--}$ are  infinite in our setting (see Proposition~\ref{Zpextensions_prop} and Proposition~\ref{Hcyc_prop}).
In the first part of this article we study the variation of Iwasawa invariants of signed Selmer groups as one runs over the elements in $\mathcal{H}^{\bullet\star}$ (see Theorem~\ref{thm:local-maximality}). Right now we just mention one corollary of our results.

Recall that $\Q_{cyc}$ denotes the cyclotomic $\Zp$-extension of $\Q$ and that $X^{\pm}(E/\Q_{cyc})$ is the Pontryagin dual of the plus/minus Selmer group over $\Q_{cyc}$, $\Selm^{\pm}(E/\Q_{cyc})$, defined as in Kobayashi's paper \cite{Kob}. Then \cite[Conjecture 6.3]{Pollack} predicts that the $\mu$-invariant of $X^{\pm}(E/\Q_{cyc})$ is zero.  Using the isomorphisms \eqref{quad_twist_isom+}  and \eqref{quad_twist_isom-} in Section~\ref{section:auxiliary}, this leads to the conjecture that in the case of doubly-signed Selmer groups where both signs are equal, $X^{++}(E/K_{cyc})$ and $X^{--}(E/K_{cyc})$ should have $\mu$-invariant zero. By combining the plus/minus analytic $\mu$-invariant data in \cite{LMFDB} together with \cite[Theorem 4.1]{Kob} which proves one divisibility of the main conjecture, we get many examples where ${\mu(X^{++}(E/K_{cyc}))=0}$ and ${\mu(X^{--}(E/K_{cyc}))=0}$. It will follow from Theorem~\ref{thm:local-maximality} that the vanishing of $\mu(X^{\bullet\bullet}(E/K_{cyc}))$ for ${\bullet \in \{+,-\}}$ implies that we actually have ${\mu = 0}$ in an entire neighborhood of $K_{cyc}$ with respect to Greenberg's topology, i.e. there exists some ${n \in \N}$ such that ${\mu(X^{\bullet\bullet}(E/L_\infty)) = 0}$ for each ${L_\infty \in \mathcal{E}(K_{cyc},n)}$ (see Corollary~\ref{cor:mu=0}). In other words, having $\mu$-invariant zero is an open condition with respect to Greenberg's topology.

Recall that a torsion $\Lambda$-module has $\mu$-invariant zero if and only if it is finitely generated as a $\Z_p$-module. A vast generalisation of this property for Iwasawa modules over larger (not even necessarily commutative) Iwasawa algebras is what we called \emph{the $\mathfrak{M}_H(G)$-property} in \cite{MHG} (see \cite{CS2} and Definition~\ref{def:MHG} below).

In the second part of the paper we address this property for signed Selmer groups and in particular prove that the validity of the $\mathfrak{M}_H(G)$-property is also an open condition with respect to Greenberg's topology. Note that this property generalises a classical problem which in the ordinary setting is known as the $\mathfrak{M}_H(G)$-conjecture to non-cyclotomic $\Z_p$-extensions. We have to introduce some more notation.

Let $\bullet, \star \in \{+,-\}$ and let ${\KK_\infty^H \in \mathcal{E}}$ be such that both primes of $K$ above $p$ ramify in $\KK_\infty^H$. We define $H_n:=H^{p^n}$. For every $n$, we fix a finite extension $K_{H,n}/K$ of degree $p^n$ such that $K_{\infty}^{H_n}=K_{H,n}K_{\infty}^H$. Then $K_{\infty}^{H_n}/K_{H,n}$ is a $\Zp$-extension. For any $n \geq 0$, let $L^{(\fp)}_n$ (resp. $L^{(\bar{\fp})}_n$) be the subextensions of $L^{(\fp)}_{\infty}/K$ (resp. $L^{(\bar{\fp})}_{\infty}/K$) of degree $p^n$ over $K$.  We now make the following definition for any $n$: If $v$ is a prime of $\KK_{\infty}^{H_n}$ above $\fp$ (resp. $\bar{\fp}$), we define $\hat{E}^{\pm}((\KK_{\infty}^{H_n})_v)$ as in \eqref{plusdef}, \eqref{minusdef} and \eqref{eq:def2-3} with respect to the $\Zp$-extension $\KK_{\infty}^{H_n}/L^{(\fp)}_{i(\KK_{\infty}^H, \fp)+n}$ (resp. $\KK_{\infty}^{H_n}/L^{(\bar{\fp})}_{i(\KK_{\infty}^H, \bar{\fp})+n}$). 

With the above definitions, we define $\Selpm(E/\KK_{\infty}^{H_n})$ as we did in \eqref{Selpm_def}. Note that when $n=0$ this definition of $\Selpm(E/\KK_{\infty}^H)$ is the same as the definition \eqref{Selpm_def}.

We define $\Lambda_{H, n}:=\Zp[[\Gal(\KK_{\infty}^{H_n}/K_{H,n})]]$. Note that for any $n$ $\Selpm(E/\KK_{\infty}^{H_n})$ is a $\Gal(\KK_{\infty}^{H_n}/K)$-module and hence by restriction it is a $\Gal(\KK_{\infty}^{H_n}/K_{H,n})$-module which makes its Pontryagain dual $X^{\bullet\star}(E/\KK_{\infty}^{H_n})$ a finitely generated $\Lambda_{H, n}$-module. For every $n$, we let ${G_{H,n}=\Gal(\KK_{\infty}/K_{H,n})}$ and we write $\mu_{G_{H,n}/H_n}(X^{\bullet\star}(E/\KK_{\infty}^{H_n}))$ for the $\mu$-invariant of $X^{\bullet\star}(E/\KK_{\infty}^{H_n})$ as a $\Lambda_{H, n} = \Z_p[[G_{H,n}/H_n]]$-module. When we are working with a fixed ${H \in \mathcal{H}^{\bullet\star}}$, then, to ease notation, we will sometimes drop the subscript $H$ from all the symbols in this paragraph.

We will sometimes abbreviate $\KK_{\infty}^{H_n}$ to $F^n_{\infty}$ and write ${F^0_{\infty} = F_{\infty}}$ in this article. The Galois groups ${\Gal(F_{\infty}/K) \cong \Gal(F^n_{\infty}/K_{H,n})}$
are isomorphic to $\Z_p$ and will be abbreviated to $\Gamma$ in the following diagram of fields:
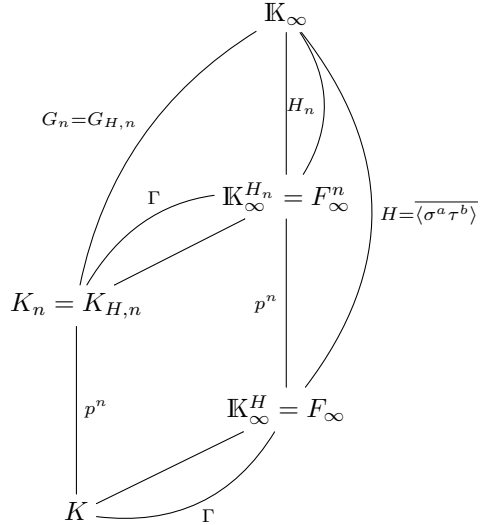
\[ \xymatrix{& \KK_{\infty} \ar@{-}@/_1.4pc/[dddl]_{G_n = G_{H,n}} \ar@{-}[dd] \ar@{-}@/^1.2pc/[dd]_{H_n} \ar@{-}@/^2.7pc/[dddd]^{H = \overline{\langle \sigma^a \tau^b\rangle}}\\ & \\ & \KK_{\infty}^{H_n} = F^n_{\infty} \ar@{-}@/_1.4pc/[dl]_\Gamma\ar@{-}[dd]_{p^n} \ar@{-}[dl] \\ K_n = K_{H,n} & \\ & \KK_{\infty}^H = F_{\infty} \ar@{-}[dl] \ar@{-}@/^1.4pc/[dl]^\Gamma  \\ K \ar@{-}[uu]_{p^n}  & }\]

In \cite{MHG} the authors studied the \emph{$\mathfrak{M}_H(G)$-property} of Selmer groups in the good ordinary setting. The main purpose of the present article is to derive an analogue of \cite[Theorem~1.3]{MHG} for signed Selmer groups, Selmer groups and fine Selmer groups in the supersingular case.

\begin{definition} \label{def:free-and-torsion}
If $R$ is a integral domain and $M$ a finitely generated $R$-module, let
\begin{outerlist}
\item $T_R(M)$ be the torsion $R$-submodule of $M$,
\item $F_R(M)=M/T_R(M)$.
\end{outerlist}
We then define $M_f:=T_R(M)/M[p^{\infty}]$.
\end{definition}

When $M$ is itself a torsion $R$-module we have $M_f=M/M[p^{\infty}]$. This will be the case when we consider signed Selmer groups and fine Selmer groups. However the Pontryagain dual of the Selmer group over $\KK_{\infty}$ has positive $\Lambda_2$-rank and hence in the case of the Selmer group we need to consider its $\Lambda_2$-torsion submodule.

\begin{definition} \label{def:MHG}
Let $\bullet, \star \in \{+,-\}$ and $H \in \mathcal{H}^{\bullet\star}$. We say that $X^{\bullet\star}(E/\KK_{\infty})$ \emph{satisfies the $\mathfrak{M}_H(G)$-property} if $X^{\bullet\star}(E/\KK_{\infty})_f$ is finitely generated over ${\Lambda(H):=\Zp[[H]]}$.
We say that $X(E/\KK_{\infty})$  \emph{satisfies the $\mathfrak{M}_H(G)$-property} if $X(E/\KK_{\infty})_f$ is finitely generated over ${\Lambda(H):=\Zp[[H]]}$. We say that $Y(E/\KK_{\infty})$  \emph{satisfies the $\mathfrak{M}_H(G)$-property} if $Y(E/\KK_{\infty})_f$ is finitely generated over ${\Lambda(H):=\Zp[[H]]}$.
\end{definition}
	
For $H=H_{cyc}$ it is conjectured that ${X(E/\KK_{\infty})_f = X(E/\KK_{\infty})/X(E/\KK_{\infty})[p^\infty]}$ is finitely generated over $\Z_p[[H]]$ if $p$ is odd and $E$ has good ordinary reduction at $p$ (see \cite{CS2}); this is known as the \emph{$\mathfrak{M}_H(G)$-conjecture}.

An important special case is the following: If $X^{\bullet\star}(E/\KK_{\infty}^H)$, ${H \in \mathcal{H}^{\bullet\star}}$, is $\Lambda$-torsion and has $\mu$-invariant zero, then this $\Lambda$-module is finitely generated over $\Z_p$. It will follow from our control theorem (Theorem~\ref{SelmerControlThm_prop}) that $X^{\bullet\star}(E/\KK_{\infty})_H$ is also finitely generated over $\Z_p$, and therefore $X^{\bullet\star}(E/\KK_{\infty})$ satisfies the $\mathfrak{M}_H(G)$-property. As we have pointed out above, it is conjectured that both $\mu(X^{++}(E/\KK_{\infty}^H)) = 0$ and $\mu(X^{--}(E/\KK_{\infty}^H)) = 0$ if ${H = H_{cyc}}$ fixes the cyclotomic $\Z_p$-extension ${\KK_\infty^{H_{cyc}} = K_{cyc}}$ of $K$. 

\subsection{Main results}
We first formulate our main result for signed Selmer groups. Let ${\bullet, \star \in \{+,-\}}$. If $X^{\bullet\star}(E/\KK_{\infty})$ is $\Lambda_2$-torsion, let ${f_{\infty}^{\bullet\star} \in \Lambda_2}$ be the characteristic power series of $X^{\bullet\star}(E/\KK_{\infty})$, and write ${f^{\bullet\star}_{\infty}=p^{m^{\bullet\star}}g^{\bullet\star}_{\infty}}$ where ${m^{\bullet\star}=\mu_G(X^{\bullet\star}(E/\KK_{\infty}))}$, so that ${p \nmid g^{\bullet\star}_{\infty}}$. 
Also if ${H \in \mathcal{H}^{\bullet\star}}$, we let $\lambda_H^{\bullet\star}$ be the $\lambda$-invariant of $X^{\bullet\star}(E/\KK_{\infty}^H)$. The first main result of this article is the following
\begin{theorem}\label{main_theorem}
With notation as above, let $\bullet, \star \in \{+,-\}$. We assume that $\mathcal{H}^{\bullet\star}$ is non-empty and we choose some ${H=\overbar{\langle \sigma^a\tau^b \rangle}\in \mathcal{H}^{\bullet\star}}$.

Then $X^{\bullet\star}(E/\KK_{\infty})$ is $\Lambda_2$-torsion and the following are equivalent:
\begin{enumerate}[(a)]
\item $X^{\bullet\star}(E/\KK_{\infty})_f$ is finitely generated over $\Lambda(H)$.
\item For any $m>0$ such that $p^m$ annihilates $X^{\bullet\star}(E/\KK_{\infty})[p^{\infty}]$ and $X^{\bullet\star}(E/\KK_{\infty}^H)[p^{\infty}]$ we have that  $H_1(H,X^{\bullet\star}(E/\KK_{\infty})/p^m)$ is finite.
\item $H_1(H,X^{\bullet\star}(E/\KK_{\infty})_f/p)$ is finite.
\item $\mu_G(X^{\bullet\star}(E/\KK_{\infty}))=\mu_{G/H}(X^{\bullet\star}(E/\KK_{\infty}^H))$.
\item For all $n$, $X^{\bullet\star}(E/\KK_{\infty}^{H_n})$ is a torsion $\Lambda_{H,n}$-module and \[ \mu_{G_n/H_n}(X^{\bullet\star}(E/\KK_{\infty}^{H_n})) = p^n\mu_G(X^{\bullet\star}(E/\KK_{\infty})).\]
\item For all $n$, $X^{\bullet\star}(E/\KK_{\infty}^{H_n})$ is a torsion $\Lambda_{H,n}$-module. Furthermore, if $X^{\bullet\star}(E/\KK_{\infty})[p^{\infty}]$ is pseudo-isomorphic to $\bigoplus_{i=1}^s \Lambda_2/p^{m_i^{\bullet\star}}$, then for all $n$ $X^{\bullet\star}(E/\KK_{\infty}^{H_n})[p^{\infty}]$ is pseudo-isomorphic to $\bigoplus_{i=1}^s  \left(\Lambda_{H,n}/p^{m_i^{\bullet\star}}\right)^{p^n}$.
\item The image of $g^{\bullet\star}_{\infty}$ in $\Lambda_2/p\Lambda_2$ is not divisible by the coset of ${(1+T_1)^a(1+T_2)^b-1}$ (here ${T_1 = \sigma - 1}$ and ${T_2 = \tau - 1}$).
\item $\lambda(X^{\bullet\star}(E/L_{\infty}))$ is bounded as $L_{\infty}$ varies through the elements in a neighborhood of $\KK_{\infty}^H$.
\item We have an injective $\Lambda(H)$-homomorphism
$$X^{\bullet\star}(E/\KK_{\infty})_f \hookrightarrow \Lambda(H)^{\lambda_H^{\bullet\star}}$$ with finite cokernel and a $\Lambda_2$-homomorphism
$$0 \longrightarrow X^{\bullet\star}(E/\KK_{\infty}) \longrightarrow \bigoplus_{i=1}^{s} \Lambda_2/{(f_i^{\bullet\star})}^{n_i^{\bullet\star}} \oplus \bigoplus_{j=1}^{t} \Lambda_2/p^{m_j^{\bullet\star}} \longrightarrow B^{\bullet\star} \longrightarrow 0$$
where $s \leq \lambda_H^{\bullet\star}$, $B^{\bullet\star}$ is a pseudo-null $\Lambda_2$-module, ${f_i^{\bullet\star} \in \Lambda_2 \setminus \Lambda(H)}$ are irreducible power series and ${\mu_G(X^{\bullet\star}(E/\KK_{\infty}))=\sum_{j=1}^t m_j^{\bullet\star}}$.
\end{enumerate}
Moreover, if ${\star = \bullet}$, i.e. either ${H \in \mathcal{H}^{++}}$ or ${H \in \mathcal{H}^{--}}$, and if $i(\KK_{\infty}^H,\fp)=i(\KK_{\infty}^H, \bar{\fp})$, then the above statements are also equivalent to the following assertion:
\begin{enumerate}[(j)]
	\item $\lambda(X^{\bullet\star}(E/L_{\infty}))$ is constant as $L_{\infty}$ varies through the elements in a neighborhood of $\KK_{\infty}^H$.
\end{enumerate}
\end{theorem}
The above theorem is the supersingular analogue of \cite[Theorem 1.3]{MHG}, but with four extra equivalences $(b)$, $(c)$, $(f)$ and $(j)$. The additional hypothesis ${i(\KK_{\infty}^H,\fp)=i(\KK_{\infty}^H, \bar{\fp})}$ is needed to prove a control theorem that is used in the proof of statement~$(j)$. As far as we know there does not exist a control theorem for mixed signed Selmer groups in the literature. Assuming that both signs are the same, we prove in Section~\ref{section:fukuda} a control theorem under the assumption that ${i(\KK_{\infty}^H,\fp)=i(\KK_{\infty}^H, \bar{\fp})}$. This theorem generalizes the control theorem of Iovita and Pollack \cite[Theorem 6.8]{IP} which assumes that both primes of $K$ above $p$ are totally ramified in the $\Z_p$-extension $\KK_\infty^H/K$.

As an easy consequence of part $(g)$ above, the following result shows that under the assumption of Theorem~\ref{main_theorem} the equivalent conditions hold for all but finitely many ${L_\infty \in \mathcal{E}}$:
\begin{proposition}\label{almost_all_prop}
Assume that $\mathcal{H}^{\bullet\star} \neq \emptyset$. Then for all but finitely many ${L_\infty \in \mathcal{E}}$, with ${L_\infty = \KK_\infty^H}$, we have that  $X^{\bullet\star}(E/\KK_{\infty})_f$ is a finitely generated $\Lambda(H)$-module.
\end{proposition}

We also prove theorems similar to Theorem~\ref{main_theorem} but for the fine Selmer group as well as the Selmer group. First we deal with the fine Selmer group. We shall show that $Y(E/\KK_{\infty})$ a finitely generated $\Lambda_2$-torsion module (Lemma~\ref{Yinf_torsion_lemma}). Let ${\tilde{f}_{\infty} \in \Lambda_2}$ be the characteristic power series of $Y(E/\KK_{\infty})$.
We write
\begin{align} \label{eq:tildef_infty} \tilde{f}_{\infty}=p^{\tilde{\mu}}\tilde{g}_{\infty}
\end{align}
where ${\tilde{\mu}=\mu_G(Y(E/\KK_{\infty}))}$, so that ${p \nmid \tilde{g}_{\infty}}$.

\begin{theorem}\label{main_theorem2}
Let ${\bullet, \star \in \{+,-\}}$ and ${H=\overbar{\langle \sigma^a\tau^b \rangle}\in \mathcal{H}^{\bullet\star}}$. The following statements are equivalent:
\begin{enumerate}[(a)]
\item $Y(E/\KK_{\infty})_f$ is finitely generated over $\Lambda(H)$.
\item The image of $\tilde{g}_{\infty}$ in $\Lambda_2/p$ is not divisible by the coset of ${(1+T_1)^a(1+T_2)^b-1}$.
\item $\lambda_{G/H}(Y(E/L_\infty))$ is bounded on a neighbourhood $U$ of $\KK_\infty^H$.
\item \resizebox{.9\hsize}{!}{$\mu_G(Y(E/\KK_\infty)) = \mu_{G/H}(Y(E/\KK_\infty^H)) - \mu_{G/H}(T_{\Lambda}([F_{\Lambda_2}(X(E/\KK_{\infty}))]_H))$}.
\item For all $n \geq 0$, $Y(E/\KK_\infty^{H_n})$ is a torsion $\Lambda_{H,n}$-module and \\
\resizebox{.9\hsize}{!}{$p^n\mu_G(Y(E/\KK_\infty)) = \mu_{G_n/H_n}(Y(E/\KK_\infty^{H_n})) - \mu_{G_n/H_n}(T_{\Lambda_{H,n}}([F_{\Lambda_2}(X(E/\KK_{\infty}))]_{H_n}))$.}
\end{enumerate}
Moreover, consider the statements \begin{enumerate}
\item[(f)] $\mu_G(Y(E/\KK_\infty)) = \mu_{G/H}(Y(E/\KK_\infty^H))$.
\item[(g)] $\mu_{G/H}(Y(E/L_\infty))$ is constant on a neighbourhood $U$ of $\KK_\infty^H$.
\end{enumerate}
Then $(f) \Longrightarrow (g) \Longrightarrow [(a) -(e)]$.
\end{theorem}
Note that the Pontryagin dual $X(E/\KK_\infty)$ of the Selmer group appears in the error terms in statements~$(d)$ and $(e)$.

Now we state our theorem for the Selmer group. Recall that it is well-known that $X(E/K_{cyc})$ is not $\Lambda$-torsion in the supersingular setting. Let ${f_{\infty} \in \Lambda_2}$ be the characteristic power series of $T_{\Lambda_2}(X(E/\KK_{\infty}))$. 
We write
\begin{align} \label{eq:f_infty} f_{\infty}=p^{\mu}g_{\infty}
\end{align}
where ${\mu=\mu_G(T_{\Lambda_2}(X(E/\KK_{\infty})))}$, so that ${p \nmid g_{\infty}}$. 

\begin{theorem}\label{main_theorem3}
Let ${\bullet, \star \in \{+,-\}}$ and ${H=\overbar{\langle \sigma^a\tau^b \rangle}\in \mathcal{H}^{\bullet\star}}$. Consider the following statements:
\begin{enumerate}[(a)]
\item $X(E/\KK_{\infty})_f$ is finitely generated over $\Lambda(H)$.
\item For any $m>0$ such that $p^m$ annihilates $X(E/\KK_{\infty})[p^{\infty}]$ and $X(E/\KK_{\infty}^H)[p^{\infty}]$ we have that $H_1(H,T_{\Lambda_2}(X(E/\KK_{\infty}))/p^m)$ is finite.
\item \resizebox{.9\hsize}{!}{$\mu_G(T_{\Lambda_2}(X(E/\KK_\infty))) = \mu_{G/H}(T_{\Lambda}(X(E/\KK_\infty^H))) - \mu_{G/H}(T_{\Lambda}([F_{\Lambda_2}(X(E/\KK_{\infty}))]_H))$}.
\item For all $n \geq 0$ we have\\
 \resizebox{.9\hsize}{!}{$p^n\mu_G(T_{\Lambda_2}(X(E/\KK_\infty))) = \mu_{G_n/H_n}(T_{\Lambda_{H,n}}(X(E/\KK_\infty^{H_n}))) - \mu_{G_n/H_n}(T_{\Lambda_{H,n}}([F_{\Lambda_2}(X(E/\KK_{\infty}))]_{H_n}))$.}
\item $\lambda_{G/H}(T_{\Lambda}(X(E/L_\infty)))$ is bounded on a neighbourhood $U$ of $\KK_\infty^H$.
\item The image of $g_{\infty}$ in $\Lambda_2/p$ is not divisible by the coset of ${(1+T_1)^a(1+T_2)^b-1}$.
\item $\mu_G(T_{\Lambda_2}(X(E/\KK_\infty))) = \mu_{G/H}(T_{\Lambda}(X(E/\KK_\infty^H)))$.
\item $\mu_{G/H}(T_{\Lambda}(X(E/L_\infty)))$ is constant on a neighbourhood $U$ of $\KK_\infty^H$.
\end{enumerate}

Then the following assertions hold.
\begin{enumerate}[(1)]
\item $(a)-(f)$ are all equivalent.
\item We have the following implications: ${(g) \Longrightarrow (h) \Longrightarrow [(a)-(f)]}$.
\item For all but finitely many ${H \in \mathcal{H}^{\bullet\star}}$ we have that ${\mu_{G_n/H_n}(T_{\Lambda_{H,n}}([F_{\Lambda_2}(X(E/\KK_{\infty}))]_{H_n})) = 0}$ for every ${n \ge 0}$.
\item There exists $C \geq 0$ such that for any $n \geq 0$ and any $H \in \mathcal{H}^{\bullet\star}$ we have that  \[ \mu_{G_n/H_n}(T_{\Lambda_{H,n}}([F_{\Lambda_2}(X(E/\KK_{\infty}))]_{H_n})) \leq C\] and
\[ \lambda_{G_n/H_n}(T_{\Lambda_{H,n}}([F_{\Lambda_2}(X(E/\KK_{\infty}))]_{H_n})) \leq C. \]
\end{enumerate}
\end{theorem}

The term $\mu_{G_n/H_n}(T_{\Lambda_{H,n}}([F_{\Lambda_2}(X(E/\KK_{\infty}))]_{H_n}))$ appears in Theorems~\ref{main_theorem2} and \ref{main_theorem3}. It turns out that for all but finitely many $H$, this term is zero for every $n$, see Proposition~\ref{mu_vanishing_prop}. Moreover, we will show in Theorem~\ref{mu-invariants_theorem} that this term may be replaced by the $\mu$-invariants of other objects, for example $\mu_{G_n/H_n}(H_1(H_n, Y(E/\KK_{\infty})))$.

The relationship between the $\mathfrak{M}_H(G)$-properties for the signed Selmer group, the Selmer group and the fine Selmer group is given in the following theorem.

\begin{theorem}\label{equivalence_theorem}
Let ${\bullet, \star \in \{+,-\}}$ and ${H=\overbar{\langle \sigma^a\tau^b \rangle}\in \mathcal{H}^{\bullet\star}}$. Consider the following statements:
\begin{enumerate}[(a)]
\item $X^{\bullet\star}(E/\KK_{\infty})_f$ is finitely generated over $\Lambda(H)$.
\item $T_{\Lambda_2}(X(E/\KK_{\infty}))_f$ is finitely generated over $\Lambda(H)$.
\item $Y(E/\KK_{\infty})_f$ is finitely generated over $\Lambda(H)$.
\end{enumerate}
We have the following implications ${(a) \Longrightarrow (b) \Longleftrightarrow (c)}$.
\end{theorem}

We see from the above theorem that the $\mathfrak{M}_H(G)$-property for the signed Selmer group (for any choice of signs $\bullet$ and $\star$) implies that the $\mathfrak{M}_H(G)$-property also holds for the Selmer group and also for the fine Selmer group. A criterion for the converse implication to hold will be proved in  Section~\ref{section:main-theorem2} (see Proposition~\ref{prop:MH(G)impliesMH(G)+-}). Theorem~\ref{equivalence_theorem} will be key to proving Theorem~\ref{main_theorem2} (specifically ${(a) \Longrightarrow (e)}$). Therefore we will prove Theorem~\ref{equivalence_theorem} before Theorem~\ref{main_theorem2}.

Using the equivalences from Theorems~\ref{main_theorem} and \ref{main_theorem3}, we can not only provide classes of examples where the $\mathfrak{M}_H(G)$-property holds true, but we can also construct examples where it fails. This is the main objective of the final Section~\ref{section:counterexample}.

\subsection{Relations between $X(E/\KK_{\infty})$ and $Y(E/\KK_{\infty})$}

Let ${H \in \mathcal{H}^{\bullet\star}}$ and assume that $X^{\bullet\star}(E/\KK_\infty)_f$ is finitely generated over $\Lambda(H)$. Using results of \cite{Matar_Torsion} we will show in Proposition~\ref{fineSelmerisom_prop} that for all ${n \geq 0}$ we have that $Y(E/\KK_{\infty}^{H_n})$ is $\Lambda_{H,n}$-torsion, ${\mu_{G_n/H_n}(Y(E/\KK_{\infty}^{H_n}))=\mu_{G_n/H_n}(T_{\Lambda_{H,n}}(X(E/\KK_{\infty}^{H_n}))}$ and ${\lambda_{G_n/H_n}(Y(E/\KK_{\infty}^{H_n}))=\lambda_{G_n/H_n}(T_{\Lambda_{H,n}}(X(E/\KK_{\infty}^{H_n}))}$. By utilizing these relations together with the asymptotic formulas for the growth of Iwasawa invariants in Section~\ref{section:monsky} we prove the following two results.

\begin{theorem}\label{2dimWingberg_theorem1}
We have $\mu_G(T_{\Lambda_2}(X(E/\KK_\infty))) = \mu_G(Y(E/\KK_\infty))$.
\end{theorem}

\begin{theorem}\label{2dimWingberg_theorem2}
Let ${\bullet, \star \in \{+,-\}}$ and ${H=\overbar{\langle \sigma^a\tau^b \rangle}\in \mathcal{H}^{\bullet\star}}$. Suppose that $Y(E/\KK_\infty)_f$ is finitely generated over $\Lambda(H)$. Then the following two statements hold.
\begin{compactenum}		
\item Write ${\Upsilon = (1+T_1)^a(1+T_2)^b-1}$, and let $\bar{\Upsilon}$ be its image in $\Lambda_2/p$. Then \[ v_{\bar{\Upsilon}}(\overline{g_\infty}) = v_{\bar{\Upsilon}}(\overline{\tilde{g}_\infty}), \]
where $v_{\bar{\Upsilon}}$ denotes the valuation with respect to the height one prime ideal $(\bar{\Upsilon})$. 
	
		In particular, as also shown in Theorem~\ref{equivalence_theorem}, $X(E/\KK_\infty)_f$ is finitely generated as a $\Lambda(H)$-module if and only if $Y(E/\KK_\infty)_f$ is finitely generated as a $\Lambda(H)$-module.
\item If $X(E/\KK_\infty)_f$ and $Y(E/\KK_\infty)_f$ are finitely generated as $\Lambda(H)$-modules, then
		\[ \rg_{\Lambda(H)}(X(E/\KK_\infty)_f) = \rg_{\Lambda(H)}(Y(E/\KK_\infty)_f). \]
	\end{compactenum}
\end{theorem}

\subsection{Results related to Conjectures~A and B} \label{section:introConjecturesAandB}
We now discuss an interesting relationship between the $\mathfrak{M}_H(G)$-property for the Selmer group and Conjecture~A and Conjecture~B in the Coates-Sujatha paper \cite{CS1}.
Now let us start with \cite[Conjecture~A]{CS1}.

\begin{conjecture} \label{conjectureA}
	$Y(E/K_{cyc})$ is a finitely generated $\Z_p$-module. In other words, it is a finitely generated torsion $\Lambda$-module, and ${\mu(Y(E/K_{cyc})) = 0}$.
\end{conjecture}
In this context, we can easily prove the following
\begin{proposition}\label{propositionA}
	Assume that Conjecture~\ref{conjectureA} is true. Then both $X(E/\KK_\infty)$ and $Y(E/\KK_{\infty})$ satisfy the $\mathfrak{M}_H(G)$-property with ${H = H_{cyc}}$ (the subgroup fixing ${K_{cyc} \subseteq \KK_\infty}$). Furthermore, we have ${\mu_G(T_{\Lambda_2}(X(E/\KK_\infty)))=0}$.
\end{proposition}
The following statement is a slightly stronger version of Conjecture~B in the Coates-Sujatha paper \cite[Conjecture~B]{CS1} in our setting.

\begin{conjecture}\label{conjectureB}
$Y(E/\KK_{\infty})$ is a pseudo-null $\Lambda_2$-module.
\end{conjecture}

We shall show that
\begin{proposition}\label{propositionB}
Assume that Conjecture~\ref{conjectureB} is true. Then we have that both $X(E/\KK_\infty)$ and $Y(E/\KK_{\infty})$ satisfy the $\mathfrak{M}_H(G)$-property for any ${H \in \mathcal{H}^{\bullet\star}}$, with ${\bullet,\star \in \{+,-\}}$. Furthermore, we have ${\mu_G(T_{\Lambda_2}(X(E/\KK_\infty)))=0}$.
\end{proposition}

\subsection{On the non-triviality of torsion submodules}
The question of whether $X(E/\KK_{\infty})_f$ satisfies the $\mathfrak{M}_H(G)$-property is nontrivial only if ${T_{\Lambda_2}(X(E/\KK_{\infty})) \neq 0}$ and hence a natural question arises:

\textbf{Question:} Are there are any examples where ${T_{\Lambda_2}(X(E/\KK_{\infty})) \neq 0}$?

We have not been able to answer this question. On the other hand the next proposition shows that a condition related to Conjecture \ref{conjectureA} guarantees that ${T_{\Lambda_2}(X(E/\KK_{\infty}))=0}$.

\begin{proposition}\label{propositionC}
If $Y(E/K_{cyc})$ is finite, then ${T_{\Lambda_2}(X(E/\KK_{\infty}))=0}$.
\end{proposition}

In Theorem~\ref{main_theorem3}(c) we see the term $\mu_{G/H}(T_{\Lambda}([F_{\Lambda_2}(X(E/\KK_{\infty}))]_H))$. This leads us to the following question.

\textbf{Question:} Can the $\mu_{G/H}$- and $\lambda_{G/H}$-invariants of $T_{\Lambda}([F_{\Lambda_2}(X(E/\KK_{\infty}))]_H)$ be nonzero for some ${H \in \mathcal{H}^{\bullet\star}}$?

In Section~\ref{section:control2} we present two generic examples $M_1$ and $M_2$ of finitely generated torsion-free $\Lambda_2$-modules such that ${\lambda(T_{\Lambda}(M_1/T_1))=1}$ and ${\mu(T_{\Lambda}(M_2/T_1))=1}$. However, it is unclear whether such examples can occur in the case of the Selmer group. In our attempt to answer the above question we present a scenario in which we are able to show that $X(E/\KK_{\infty})$ is a free $\Lambda_2$-module (hence $\mu_{G/H}$- and $\lambda_{G/H}$-invariants of $T_{\Lambda}([F_{\Lambda_2}(X(E/\KK_{\infty}))]_H)$ are zero for all ${H \in \mathcal{H}^{\bullet\star}}$).

The next result also provides a setting where the torsion submodule of $X(E/\KK_\infty)$ is trivial. Let $N$ be the conductor of $E$. For our next theorem we assume the Heegner hypothesis: Every prime dividing $N$ splits in $K/\Q$. If $K[1]$ denotes the Hilbert class field of $K$, then a choice of an ideal $\cN$ of $\cO_K$ such that ${\cO_K/\cN \cong \cy{N}}$ and a modular parametrization ${X_0(N) \to E}$ allow us to define a Heegner point ${y_1 \in E(K[1])}$. Let $y_K$ be the trace of this point down to $K$. For any prime $v$ of $\Q$ let $c_v$ be the Tamagawa number at $v$. Our result is

\begin{theorem}\label{freestructure_theorem}
Assume the following
\begin{enumerate}[(a)]
\item ${y_K \notin pE(K_v)}$ for some prime $v$ of $K$ above $p$.
\item $p \nmid \prod_{v |N} c_v$.
\item If $E$ has multiplicative reduction at a prime $\ell$ of $\Q$ dividing $N$, then ${\ell \not\equiv 1 \mod p}$.
\end{enumerate}
Then we have
\begin{enumerate}
\item ${X(E/\KK_{\infty}) \cong (\Lambda_2)^2}$.
\item For any ${\bullet, \star \in \{+,-\}}$ and any ${H \in \mathcal{H}^{\bullet\star}}$ we have that $Y(E/\KK_{\infty}^H)$ is finite. In particular, Conjecture~\ref{conjectureA} is true.
\item Conjecture~\ref{conjectureB} is true.
\end{enumerate}
\end{theorem}

\subsection{Mazur's growth number conjecture}
Finally, we mention an application of Theorem~\ref{main_theorem} to the number of ${L_{\infty} \in \mathcal{E}}$ where the rank of $E(L_n)$ stays bounded as $L_n$ runs over the layers of $L_\infty$. As a first observation, if $\mathcal{H}^{+-}$ and $\mathcal{H}^{-+}$ are nonempty, it follows from Proposition~\ref{Hcyc_prop}, Proposition~\ref{Torsion_prop} and Lemma~\ref{rankbound_lemma} that the rank of $E$ stays bounded in all but finitely many $\Zp$-extensions ${L \in \mathcal{E}}$ (see \cite[Lemma~2.3]{MHG}). When $E$ has good ordinary reduction at $p$, the precise number of $\Zp$-extensions where the rank of $E$ stays bounded is predicted by Mazur's growth number conjecture (\cite[Section~18]{MazurMCA}). This conjecture has been extended to the supersingular case in \cite[Conjecture~2.2]{LeiSprung}:

\begin{conjecture}\label{MazurConjecture}
Let $E/\Q$ be an elliptic curve with good reduction at $p$. The Mordell-Weil rank of $E$ stays bounded along any $\Z_p$-extension of the imaginary quadratic field  $K$, unless the extension is anticyclotomic and the root number of $E/K$ is $-1$.
\end{conjecture}

In our paper \cite{MHG} which considered the case when $E$ has ordinary reduction at $p$, an interesting connection was proven between the $\mathfrak{M}_H(G)$-property and Conjecture~\ref{MazurConjecture}. Here a similar connection is shown in the supersingular case. For any ${\bullet, \star \in \{+, -\}}$ let $\Omega^{\bullet\star}$ be the set of all ${H \in \mathcal{H}^{\bullet\star}}$ such that $X^{\bullet\star}(E/\KK_{\infty})_f$ is a finitely generated $\Lambda(H)$-module. We now define ${\Omega:= \Omega^{++} \times \Omega^{--} \times \Omega^{+-} \times \Omega^{-+}}$. For any tuple ${(A, B, C, D) \in \Omega}$ it makes sense to define $\lambda$-invariants ${\lambda_A^{++} = \lambda_{G/A}(X^{++}(E/\KK_\infty^A))}$, ${\lambda_B^{--} = \lambda_{G/B}(X^{--}(E/\K_\infty^B))}$, etc. Let $N$ be the conductor of $E$ and $h_K$ the class number of $K$. In relation to Conjecture~\ref{MazurConjecture}, we will show
\begin{theorem}\label{rankbound_theorem}
Let $t$ be the number of $\Zp$-extensions of $K$, where the rank of $E$ does not stay bounded. Then $$t \le \min\{\lambda^{++}_A + \lambda^{--}_B + \lambda^{+-}_C + \lambda^{-+}_D  \mid  (A,B,C,D) \in \Omega\}. $$
If every prime dividing $N$ splits in $K/\Q$, $p \nmid h_K$ and $\Gal(\Q(E[p])/\Q)=GL_2(\Fp)$, then
$$ 1 \le t \le \min\{\lambda^{++}_A + \lambda^{--}_B + \lambda^{+-}_C + \lambda^{-+}_D \, | \, (A,B,C,D) \in \Omega\}-1. $$
\end{theorem}
The condition $p \nmid h_K$ is imposed in the above theorem in order to ensure that every prime of $K$ above $p$ is totally ramified the anticyclotomic $\Zp$-extension $K_{ac}/K$.

In \cite{MHG} we proved a similar theorem in the ordinary case and used this theorem to give examples verifying Conjecture~\ref{MazurConjecture}. In the supersingular case considered here, we can easily construct examples verifying Conjecture~\ref{MazurConjecture} when the root number is 1 using Theorem~\ref{conj_rankzero_theorem}. However, when the root number is -1 we are unable to use the above theorem to give examples that verify Conjecture~\ref{MazurConjecture}. Constructing examples in the supersingular case is more difficult because the bound $t$ in the theorem above involves four $\lambda$-invariants rather than just one in the ordinary case. In our attempt to verify Conjecture~\ref{MazurConjecture}, we will construct examples in Section~\ref{section:appendix} when the root number is $-1$ such that $t \leq 3$. The paper~\cite{GHKL} also studies Conjecture~\ref{MazurConjecture} via a different approach than the one we use (that paper does not study the $\mathfrak{M}_H(G)$-property).

Let us briefly describe the structure of the paper, which consists of three more parts. The second part of the paper  (Sections~\ref{section:control}-\ref{section:monsky}) provides the ingredients that will be used in the third part (Sections~\ref{section:main-theorem}-\ref{section:freeSelmer}) for the proofs of our main results. Finally, we conclude the paper with the announced application towards Mazur's Conjecture~\ref{MazurConjecture} (Sections~\ref{section:bounding_ranks} and \ref{section:appendix}) and we consider $\Z_p$-extensions of $K$ where some prime in $S$ might split completely (Section~\ref{section:counterexample}).

\textbf{Acknowledgements.} The third named author acknowledges support from NSERC Discovery grant 2019-03987. The authors would like to thank B.D.~Kim, Chan-Ho Kim, Antonio Lei and Katharina Müller for helpful discussions on the topics of this article. We are very grateful to Robert Pollack for his help with the computations in Section~\ref{section:examples}.

\part{Key results used for the main theorems}
In the next sections we develop the theory needed for the proofs of our main results. Many of these results are interesting in their own right.
\section{Further definitions and the control theorem for signed Selmer groups} \label{section:control}
In this first auxiliary section we prove the control theorem which will be needed in the sequel.

In the rest of the paper, for ${L_{\infty} \in \mathcal{E}}$, ${L_{\infty} = \KK_{\infty}^H}$, we often abbreviate the Iwasawa algebra ${\Z_p[[\Gal(\KK_{\infty}/L_{\infty})]] =: \Lambda(H)}$ to $\Lambda$. Recall the definition of $\mathcal{H}^{\bullet\star}$ from Definition~\ref{def:mathcal{H}}. We begin this section by defining the plus and minus norm groups over $\KK_{\infty}$ and subsequently the plus and minus Selmer groups over $\KK_{\infty}$.

Now let $k/\Q_p$ be a finite unramified extension. Let $k_{cyc}$ denote the cyclotomic $\Zp$-extension of $k$ with tower fields $k_n$. For any $n \geq 0$ we can define the plus and minus norm groups $E^{\pm}(k_n)$ as in \eqref{plusdef} and \eqref{minusdef}.

Let $k^{ur}$ denote the unramified $\Zp$-extension of $k$ and $k_{cyc}^{ur}$ the compositum of $k_{cyc}$ and $k^{ur}$. We have $\Gal(k_{cyc}^{ur}/k)\cong \Z_p^2$. For any integer $m \geq 0$, let $k^{(m)}$ denote the unique subextension of $k^{ur}$ such that $k^{(m)}/k$ is of degree $p^m$. Similarly for any $n \geq 1$, define $k_n^{ur}$ to be the unramified $\Zp$-extension of $k_n$ and $k_n^{(m)}$ to be the unique subextension such that $k_n^{(m)}/k_n$ has degree $p^m$. Note that $k_{cyc}^{ur}=\bigcup_{n,m} k_n^{(m)}$.

On replacing $k$ by $k^{(m)}$  we can define the plus and minus norm groups $E^{\pm}(k_n^{(m)})$ as in \eqref{plusdef} and \eqref{minusdef}. We then finally define

$$E^{\pm}(k_{cyc}^{ur}) = \bigcup_{n,m} E^{\pm}(k_n^{(m)}). $$
The definition of $E^{\pm}(k_{cyc}^{ur})$ above is taken from \cite{LeiSuj}.

Recall that $p$ is assumed to split in $K/\Q$. Let $w$ be a prime of $\KK_{\infty}$ above $p$ and $v$ the prime of $K$ below $w$ (i.e. either ${v = \fp}$ or ${v = \bar{\fp}}$). Then ${\KK_{\infty,w}=K_{v,cyc}^{ur} = \Q_{p, cyc}^{ur}}$ and so $E^{\pm}(\KK_{\infty,w})$ is defined. For any $\bullet, \star \in \{+,-\}$ the signed Selmer groups of $E/\KK_{\infty}$ are defined as

$$ \begin{tikzcd}[scale cd=0.855] \displaystyle 0 \longrightarrow \Selpm(E/\KK_{\infty}) \longrightarrow \Selm_{p^{\infty}}(E/\KK_{\infty}) \longrightarrow \prod_{w | \fp} \frac{H^1(\KK_{\infty,w}, E[p^{\infty}])}{E^{\bullet}(\KK_{\infty,w})\otimes \Qp/\Zp} \times \prod_{w | \bar{\fp}} \frac{H^1(\KK_{\infty,w}, E[p^{\infty}])}{E^{\star}(\KK_{\infty,w})\otimes \Qp/\Zp} \end{tikzcd} $$

Recall from Section~\ref{section:mainintroduction} that we sometimes abbreviate $\KK_{\infty}^{H_n}$ to $F^n_{\infty}$ and write ${F^0_{\infty} = F_{\infty} = \KK_\infty^H}$.

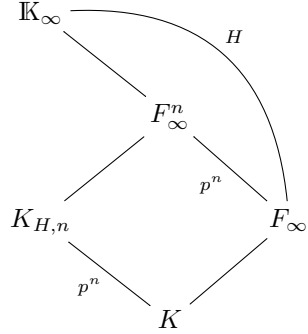
\[ \xymatrix{\KK_{\infty} \ar@{-}[dr] \ar@{-}@/^2.7pc/[ddrr]^{H} & & \\ & F^n_{\infty} \ar@{-}[dr]_{p^n} \ar@{-}[dl] & \\ K_{H,n} \ar@{-}[dr]_{p^n} & & F_{\infty} \ar@{-}[dl] \\ & K & & }\]

Before stating our first control theorem we need the following lemma.

\begin{lemma}\label{containment_lemma}
Let $w$ be a prime of $\KK_{\infty}$ above $p$. Let $\bullet, \star \in \{+,-\}$, $H \in \mathcal{H}^{\bullet\star}$, $c \geq 0$ and $F^c_{\infty}=\KK_{\infty}^{H_c}$. Then $\hat{E}^{\pm}(F^c_{\infty,w})$ is contained in $\hat{E}^{\pm}(\KK_{\infty,w})$.
\end{lemma}
\begin{proof}
Assume that $w$ lies over the prime $\fp$ of $K$. If it lies above $\bar{\fp}$, the proof will be identical. Let $\epsilon \in \{-, +\}$ and let $P \in \hat{E}^{\epsilon}(F^c_{\infty,w})$. We need to show that $P \in \hat{E}^{\epsilon}(\KK_{\infty,w})$. Recall from Section~\ref{section:mainintroduction} that $\hat{E}^{\pm}(F^c_{\infty,w})$ is defined as in \eqref{plusdef}, \eqref{minusdef} and \eqref{eq:def2-3} with respect to the $\Zp$-extension $F^c_{\infty}/L^{(\fp)}_{i(\KK_{\infty}^H, \fp)+c}$. Let $\mathcal{L}_{\infty}$ be the union of the completions at $w$ of all number fields inside $F^c_{\infty}$ and denote the completion of $L^{(\fp)}_{i(\KK_{\infty}^H, \fp)+c}$ at the prime below $w$ by $\mathcal{L}$. Let $\{\mathcal{L}_s\}$ be the intermediate layers of the $\Zp$-extension $\mathcal{L}_{\infty}/\mathcal{L}$ so that ${P \in \hat{E}(\mathcal{L}_n)}$ for some $n$. As explained above, we have $\KK_{\infty,w}=k_{cyc}^{ur}$ where $k= K_v$ ($v$ is the prime of $K$ below $w$). Then ${\hat{E}^{\epsilon}(k_{cyc}^{ur}) = \bigcup_{n',m} \hat{E}^{\epsilon}(k_{n'}^{(m)})}$. Now ${\mathcal{L}_n=k(\alpha)}$ for some $\alpha$, and $\alpha$ is contained in $k_{n'}^{(m)}$ for some ${n' \geq n}$ which implies that we have ${\mathcal{L}_n \subseteq k_{n'}^{(m)}}$.  We need to show that ${P \in \hat{E}^{\epsilon}(k_{n'}^{(m)})}$.

The field $k_{n'}^{(m)}$ is the compositum of $k_{n'}$ and $k^{(m)}$ and these fields are linearly disjoint over $k$. As $\mathcal{L}/k$ is unramified, we have ${\mathcal{L}=k^{(s)}}$ for some $s$. Note that we necessarily have that ${s \leq m}$ because ${\mathcal{L}_n \subseteq k_{n'}^{(m)}}$. We have that $\mathcal{L}_n/k^{(s)}$ is totally ramified and so $\mathcal{L}_n$ and $k^{(m)}$ are linearly disjoint over $k^{(s)}$. Since ${\mathcal{L}_n \subseteq k_{n'}^{(m)}}$ and $k_{n'}^{(m)}/k^{(m)}$ is a cyclic extension, we see that ${\mathcal{L}_tk^{(m)}=k_t^{(m)}}$ for all $t \leq n$. In particular, ${P \in \hat{E}(k_n^{(m)})}$. We need to show that ${P \in \hat{E}^{\epsilon}(k_n^{(m)})}$. Using the previous facts we have for any ${t \leq n}$
$$\text{Tr}_{k_n^{(m)}/k_t^{(m)}}=\text{Tr}_{\mathcal{L}_nk^{(m)}/\mathcal{L}_tk^{(m)}}=\text{Tr}_{\mathcal{L}_n/\mathcal{L}_t}.$$
Therefore it follows from this that ${P \in \hat{E}^{\epsilon}(k_n^{(m)})}$ as desired.
\end{proof}
The above key lemma is needed to prove the next result. Note that the fact that $\mathcal{L}_{\infty}/\mathcal{L}$ is totally ramified and $\mathcal{L}/k$ is unramified was crucial in the proof. This is one of the main reasons why we defined $\hat{E}^{\pm}(F^c_{\infty,w})$ as in \eqref{plusdef}, \eqref{minusdef} and \eqref{eq:def2-3} with respect to the $\Zp$-extension $F^c_{\infty}/L^{(\fp)}_{i(\KK_{\infty}^H, \fp)+c}$.

\begin{lemma}\label{restriction_map_lemma}
Let $\bullet, \star \in \{+,-\}$.
\begin{enumerate}[(a)]
\item If $L_{\infty} \in \mathcal{E}$, then the restriction map induces a map
$$\Selpm(E/L_{\infty}) \longrightarrow \Selpm(E/\KK_{\infty})^{\Gal(\KK_{\infty}/L_{\infty})}. $$
\item If $H \in \mathcal{H}^{\bullet\star}$ and $n \geq 0$, then the restriction map induces a map
$$\Selpm(E/\KK_{\infty}^{H_n}) \longrightarrow \Selpm(E/\KK_{\infty})^{H_n}. $$
\end{enumerate}
\end{lemma}
\begin{proof}
Lemma \ref{containment_lemma} proves $(b)$. Conditions $(a)$ and $(c)$ of Definition~\ref{def:mathcal{H}} were not used in the proof of Lemma~\ref{containment_lemma}, therefore we also get the result of $(a)$ for all $\mathcal{E} \setminus \{L_{\infty}^{(\fp)}, L_{\infty}^{(\bar{\fp})}\}$. Now let $L_{\infty}=L^{(\fp)}_{\infty}$. Recall from Section~\ref{section:mainintroduction} that if $w$ is a prime of $L_{\infty}$ above $\fp$, we defined $\hat{E}^{\pm}(L_{\infty,w})$ to be $\hat{E}(L_{\infty,w})$. Since $\fp$ is unramified in $L_{\infty}/K$, we see from this that $\hat{E}^{\pm}(L_{\infty,v}) \subseteq \hat{E}^{\pm}(\KK_{\infty,w})$. We make a similar observation when $L_{\infty}=L^{(\bar{\fp})}_{\infty}$. This completes the proof of $(a)$.
\end{proof}

The following remark will be used in Section~\ref{section:counterexample}.
\begin{remark} \label{rem:welldefined}
	The fact that no prime in $S$ splits completely in $\KK_\infty^H/K$ was not used in the proof of Lemma~\ref{containment_lemma}. Therefore the second part of Lemma~\ref{restriction_map_lemma} is still valid for $\Zp$-extensions of $K$ in which some prime in $S$ splits completely.
\end{remark}

The desired control theorem is
\begin{theorem}\label{SelmerControlThm_prop}
Let $\bullet, \star \in \{+,-\}$ and $H \in \mathcal{H}^{\bullet\star}$. Let $n \geq 0$ and consider the map $$s^{\bullet\star}_n: \Selpm(E/\KK_{\infty}^{H_n}) \longrightarrow \Selpm(E/\KK_{\infty})^{H_n}$$ from Lemma~\ref{restriction_map_lemma}(b). This map is an isomorphism.
\end{theorem}
The proof of Theorem~\ref{SelmerControlThm_prop} will be broken up into a few results below. First we record the important lemma

\begin{lemma}\label{p-torsion_lemma}
If $w$ is a prime of $K_{\infty}$ above $p$, then we have ${E(\KK_{\infty,w})[p^{\infty}]=0}$. Moreover, ${E(\KK_{\infty})[p^{\infty}]=0}$.
\end{lemma}
\begin{proof}
The proof is identical to \cite[Lemma~2.1]{IP}.
\end{proof}

We need a result originally due to Howson~\cite{HowsonThesis} (actually the following result is proven in a much more general setting in \cite{HowsonThesis}).

If $M$ is a finitely generated $\Lambda(G)$-module, we define its rank as $\dim_{K(G)}(K(G) \otimes_{\Lambda(G)} M)$, where $K(G)$ denotes the field of fractions of $\Lambda(G)$. We define the lower central $p$-series of ${G \cong \Z_p^2}$, $\{G_i\}_{i=1}^{\infty}$ (usually denoted as $\{P_i(G)\}_{i=1}^{\infty}$), by ${G_{i+1}=\overbar{G_i^p}}$ for any ${i \ge 1}$. We have the following result.
\begin{theorem}[Howson]\label{HowsonTheorem}
Let $I_{G_i}$ be the augmentation ideal of $G_i$. Then ${\rank_{\Lambda(G)}(M)=m}$ if and only if $\dim_{\Qp}(\Qp\otimes_{\Zp}M/(I_{G_i}M))=m p^{2i}+O(p^{i})$.
\end{theorem}
\begin{proof}
This is \cite[Theorem~2.22]{HowsonThesis}.
\end{proof}

\begin{proposition}\label{pm_points_structure_prop}
Let $H \in \mathcal{H}^{\bullet\star}$ and ${n \geq 0}$.
\begin{enumerate}[(a)]
\item If $w$ is a prime of $\KK_{\infty}^{H_n}$ above $\fp$, then
$$\Big(\hat{E}^{\pm}((\KK_{\infty}^{H_n})_w)\otimes \Qp/\Zp\Big)^{\vee} \cong \Zp[[\Gal((\KK_{\infty}^{H_n})_w/(L^{(\fp)}_{i(\KK_{\infty}^H,  \fp)+n})_w)]]^{[(L^{(\fp)}_{i(\KK_{\infty}^H, \fp)+n})_w:\Qp]}. $$

\item If $w$ is a prime of $\KK_{\infty}^{H_n}$ above $\bar{\fp}$, then
$$\Big(\hat{E}^{\pm}((\KK_{\infty}^{H_n})_w)\otimes \Qp/\Zp\Big)^{\vee} \cong \Zp[[\Gal((\KK_{\infty}^{H_n})_w/(L^{(\bar{\fp})}_{i(\KK_{\infty}^H, \bar{\fp})+n})_w)]]^{[(L^{(\bar{\fp})}_{i(\KK_{\infty}^H, \bar{\fp})+n})_w:\Qp]}. $$

\end{enumerate}
\end{proposition}
\begin{proof}
Suppose we are in the setup of $(a)$. Recall from Section~\ref{section:mainintroduction} that $\hat{E}^{\pm}((\KK_{\infty}^{H_n})_w)$ is defined as in \eqref{plusdef}, \eqref{minusdef} and \eqref{eq:def2-3} with respect to the $\Zp$-extension $\KK_{\infty}^{H_n}/L^{(\fp)}_{i(\KK_{\infty}^H, \fp)+n}$. Let $\mathcal{L}_{\infty}$ be the union of the completions at $w$ of all number fields inside $\KK_{\infty}^{H_n}$ and denote the completion of $L^{(\fp)}_{i(\KK_{\infty}^H, \fp)+n}$ at the prime below $w$ by $k$. Then $\mathcal{L}_{\infty}/k$ is a totally ramified $\Zp$-extension and $k/\Qp$ is unramified. Moreover $\mathcal{L}_{\infty}/\Qp$ is an abelian extension. This is the setup in \cite{KimParityConj}. Therefore the desired result follows from \cite[Proposition 3.13]{KimParityConj} and the discussion on page 5 of \cite{KimCofree}. The proof of $(b)$ is identical.
\end{proof}
Many of our results depend crucially on the above proposition. As we saw above, the proof depends on applying results of B.D. Kim in \cite{KimParityConj}. One of the main reasons for defining $\hat{E}^{\pm}((\KK_{\infty}^{H_n})_w)$ as in \eqref{plusdef}, \eqref{minusdef} and \eqref{eq:def2-3} with respect to the $\Zp$-extension $\KK_{\infty}^{H_n}/L^{(\fp)}_{i(\KK_{\infty}^H, \fp)+n}$ was to ensure that assumption (A) in Section~3.1 of loc. cit. was satisfied, and this was used in the above proof.

Now we have the following important lemma.

\begin{lemma}\label{freeness_lemma}
Let $w$ be a prime of $\KK_{\infty}$ above $p$ and let $G_w$ be a corresponding decomposition group inside $\Gal(\KK_{\infty}/K)$. Then $(\hat{E}^{\pm}(\KK_{\infty,w})\otimes \Qp/\Zp)^{\vee} \cong \Zp[[G_w]]$.
\end{lemma}
\begin{proof}
Let $H=H_{cyc}$ and let ${H_w \subseteq H}$ denote the decomposition subgroup. Let $\mathbb{Q}_p^{cyc}$ be the cyclotomic $\Zp$-extension of $\Qp$ and $\mathbb{Q}_p^{ur}$ be the unramified $\Zp$-extension of $\Qp$. Then $H_w$ can be identified with $\Gal(\mathbb{Q}_p^{ur}/\Qp)$. Let ${B=\Gal(\mathbb{Q}_p^{cyc}/\mathbb{Q}_p)}$. We have ${G_w=H_w \times B}$. Recall that $H_n$ is defined to be $H^{p^n}$ and let $H_{n,w}$ denote the decomposition subgroup of $w$ inside $H_n$. Clearly we have ${(H_w)^{p^n} \subseteq H_{n,w}}$. To simplify notation, let ${L^n_{\infty,w}=\KK_{\infty,w}^{(H_w)^{p^n}}}$. The proof of \cite[Lemma 5.2]{LeiSuj} shows that for any ${n \geq 0}$ we have
\begin{equation}\label{localisom1}
(\hat{E}^{\pm}(\KK_{\infty,w})\otimes \Qp/\Zp)^{(H_w)^{p^n}}=\hat{E}^{\pm}(L^n_{\infty,w})\otimes \Qp/\Zp.
\end{equation}

Also as in the proof of Proposition \ref{pm_points_structure_prop} we can show that
\begin{equation}\label{localisom2}
(\hat{E}^{\pm}(L^n_{\infty,w})\otimes \Qp/\Zp)^{\vee} \cong \Zp[[B]]^{p^n}.
\end{equation}

Let $Y=(\hat{E}^{\pm}(\KK_{\infty,w})\otimes \Qp/\Zp)^{\vee}$. Then by \eqref{localisom1} and \eqref{localisom2} we have

\begin{equation}\label{coinvariants-isom}
Y_{(G_w)^{p^n}}\cong (Y_{(H_w)^{p^n}})_{B^{p^n}}\cong (\Zp[[B]]^{p^n})_{B^{p^n}}\cong \mathbb{Z}_p^{p^{2n}}.
\end{equation}

By Nakayama's Lemma and the isomorphism~\eqref{coinvariants-isom} for $n=0$ we get that $Y$ is a cyclic $\Zp[[G_w]]$-module. Therefore we have a surjection ${\psi: \Zp[[G_w]] \twoheadrightarrow Y}$. We need to show that $\ker \psi$ is trivial. By \eqref{coinvariants-isom} and Theorem~\ref{HowsonTheorem} we get that ${\rank_{\Zp[[G_w]]}(Y)=1}$. If $\ker \psi$ was nontrivial, then it would have $\Z_p[[G_w]]$-rank one since it is $\Zp[[G_w]]$-torsion-free. This would then imply that $Y$ is $\Zp[[G_w]]$-torsion which is a contradiction. This completes the proof.
\end{proof}

\begin{proposition}\label{invariants_prop}
Let $w$ be a prime of $\KK_{\infty}$ above $p$. Let $c \geq 0$ and $F^c_{\infty}=\KK_{\infty}^{H_c}$ for some ${H \in \mathcal{H}^{\bullet\star}}$. Then we have $$(\hat{E}^{\pm}(\KK_{\infty,w})\otimes \Qp/\Zp)^{H_{c,w}}=\hat{E}^{\pm}(F^c_{\infty,w})\otimes \Qp/\Zp.$$
\end{proposition}
\begin{proof}
Assume that $w$ lies over the prime $\fp$ of $K$. If it lies above $\bar{\fp}$, the proof will be identical. Let $k$ be the completion of $L^{(\fp)}_{i(F_{\infty}, \fp)+c}$ at its prime below $w$. Let ${d=[k: \Qp]}$ and ${\tilde{G}_{c,w}=\Gal(F^c_{\infty,w}/k)}$. By Proposition~\ref{pm_points_structure_prop} we have
\begin{equation}\label{localisom3}
(\hat{E}^{\pm}(F^c_{\infty,w})\otimes \Qp/\Zp )^{\vee} \cong \Zp[[\tilde{G}_{c,w}]]^d.
\end{equation}

Let ${G_w=\Gal(\KK_{\infty,w}/\Qp)}$ and ${G_{c,w}=\Gal(\KK_{\infty,w}/k^{(c)})}$. Since ${[G_w: G_{c,w}]=d}$, it follows from from Lemma~\ref{freeness_lemma} we have that
\begin{equation}\label{localisom4}
(\hat{E}^{\pm}(\KK_{\infty,w})\otimes \Qp/\Zp)^{\vee} \cong \Zp[[G_{c,w}]]^d.
\end{equation}

An identical proof to \cite[Lemma~8.17]{Kob} shows that the Kummer map induces injections

\begin{equation}\label{kummer_map1}
\hat{E}^{\pm}(F^c_{\infty,w})\otimes \Qp/\Zp \hookrightarrow H^1(F^c_{\infty,w}, E[p^{\infty}]),
\end{equation}

\begin{equation}\label{kummer_map2}
\hat{E}^{\pm}(\KK_{\infty,w})\otimes \Qp/\Zp \hookrightarrow H^1(\KK_{\infty,w}, E[p^{\infty}]).
\end{equation}

The kernel of the restriction map ${H^1(F^c_{\infty,w}, E[p^{\infty}]) \to H^1(\KK_{\infty,w}, E[p^{\infty}])^{H_{c,w}}}$ is $H^1(\Gal(\KK_{\infty,w}/F^c_{\infty,w}), E(\KK_{\infty,w})[p^{\infty}])$ which is trivial by Lemma~\ref{p-torsion_lemma}. Therefore taking into account the injections~\eqref{kummer_map1} and \eqref{kummer_map2} the inclusion map ${\hat{E}^{\pm}(F^c_{\infty,w}) \hookrightarrow \hat{E}^{\pm}(\KK_{\infty,w})}$ from Lemma~\ref{containment_lemma} induces an injection
$$\theta: \hat{E}^{\pm}(F^c_{\infty,w})\otimes \Qp/\Zp \hookrightarrow (\hat{E}^{\pm}(\KK_{\infty,w})\otimes \Qp/\Zp)^{H_{c,w}}. $$

Let ${Y:=\left((\hat{E}^{\pm}(\KK_{\infty,w})\otimes \Qp/\Zp)^{H_{c,w}}\right)^{\vee}}$ and ${X:=(\hat{E}^{\pm}(F^c_{\infty,w})\otimes \Qp/\Zp)^{\vee}}$. From \eqref{localisom3} and \eqref{localisom4} both $Y$ and $X$ are isomorphic to ${\Zp[[\tilde{G}_{c,w}]]^d}$.

Dualizing the map $\theta$ we have a surjection ${\hat{\theta}: Y \twoheadrightarrow X}$. If ${\ker \hat{\theta}}$ were not trivial, then as ${Y \cong \Zp[[\tilde{G}_{c,w}]]^d}$ we would have that ${\rank_{\Zp[[\tilde{G}_{c,w}]]}(\ker \hat{\theta}) > 0}$. As $\hat{\theta}$ is surjective this would then in turn imply that ${\rank_{\Zp[[\tilde{G}_{c,w}]]}(X) < d}$ which contradicts the fact that ${X \cong \Zp[[\tilde{G}_{c,w}]]^d}$. Therefore $\hat{\theta}$ is an isomorphism. This completes the proof.

\end{proof}

Let ${c \geq 0}$ and ${F^c_{\infty}=\KK_{\infty}^{H_c}}$. We define
$$J_p^{\bullet\star}(E/F^c_{\infty}) = \bigoplus_{w | \fp} \frac{H^1(F^c_{\infty,w}, E[p^{\infty}])}{\hat{E}^{\bullet}(F^c_{\infty,w})\otimes \Qp/\Zp} \times \bigoplus_{w | \bar{\fp}} \frac{H^1(F^c_{\infty,w}, E[p^{\infty}])}{\hat{E}^{\star}(F^c_{\infty,w})\otimes \Qp/\Zp} \; ,$$

$$J_p^{\bullet\star}(E/\KK_{\infty}) = \bigoplus_{w | \fp}\frac{H^1(\KK_{\infty,w}, E[p^{\infty}])}{\hat{E}^{\bullet}(\KK_{\infty,w})\otimes \Qp/\Zp} \times  \bigoplus_{w | \bar{\fp}}\frac{H^1(\KK_{\infty,w}, E[p^{\infty}])}{\hat{E}^{\star}(\KK_{\infty,w})\otimes \Qp/\Zp} \; .$$

For any ${v \in S \setminus S_p}$ we define
$$J_v(E/F^c_{\infty}) = \bigoplus_{w | v} H^1(F^c_{\infty,w}, E)[p^{\infty}] \; ,$$

$$J_v(E/\KK_{\infty}) = \bigoplus_{w | v}H^1(\KK_{\infty,w}, E)[p^{\infty}]. $$

Now consider the commutative diagram
\begin{equation*}
\begin{tikzcd}[column sep = small, scale cd=0.83]
0 \arrow[r] & \Selpm(E/\KK_{\infty})^{H_c} \arrow[r] & H^1(G_S(\KK_{\infty}), E[p^{\infty}])^{H_c}  \arrow[r] & J_p^{\bullet\star}(E/\KK_{\infty})^{H_c} \times \bigoplus_{v \in S \setminus S_p} J_v(E/\KK_{\infty})^{H_c}\\
0 \arrow[r] & \Selpm(E/F^c_{\infty}) \arrow[r] \arrow[u, "s_c^{\bullet\star}"] & H^1(G_S(F^c_{\infty}), E[p^{\infty}]) \arrow[u, "g_c"] \arrow[r] & J_p^{\bullet\star}(E/F^c_{\infty}) \times \bigoplus_{v \in S \setminus S_p} J_v(E/F^c_{\infty}) \arrow[u, "h_c^{\bullet\star}"]
\end{tikzcd}
\end{equation*}
From the snake lemma we see from the above commutative diagram that $\ker s_c^{\bullet\star}$ and $\coker s_c^{\bullet\star}$ will be trivial if we can show that the groups $\ker g_c$, $\coker g_c$ and $\ker h_c^{\bullet\star}$ are trivial. Let $h'_c$ the restriction of $h_c^{\bullet\star}$ to ${\bigoplus_{v \in S \setminus S_p} J_v(E/F^c_{\infty})}$. Taking into account Lemma~\ref{p-torsion_lemma}, the proof of \cite[Proposition~2.8]{MHG} shows that $\ker g_c$, $\coker g_c$ and $\ker h'_c$ are trivial. Therefore Theorem~\ref{SelmerControlThm_prop} follows from the following

\begin{proposition}\label{localinjection_prop}
Let $w$ be a prime of $\KK_{\infty}$ above $p$. The map
\[ \gamma_c^{\pm}: \frac{H^1(F^c_{\infty,w}, E[p^{\infty}])}{E^{\pm}(F^c_{\infty,w})\otimes \Qp/\Zp} \longrightarrow  \left(\frac{H^1(\KK_{\infty,w}, E[p^{\infty}])}{E^{\pm}(\KK_{\infty,w})\otimes \Qp/\Zp}\right)^{H_{c,w}}\]
is injective.
\end{proposition}
\begin{proof}
Consider the following commutative diagram
\begin{equation*}
\begin{tikzcd}[column sep = small, scale cd=0.84]
0 \arrow[r] & (E^{\pm}(\KK_{\infty,w})\otimes \Qp/\Zp)^{H_{c,w}} \arrow[r] & H^1(\KK_{\infty,w}, E[p^{\infty}])^{H_{c,w}}  \arrow[r] & \left(\frac{H^1(\KK_{\infty,w}, E[p^{\infty}])}{\hat{E}^{\pm}(\KK_{\infty,w})\otimes \Qp/\Zp}\right)^{H_{c,w}} &\\
0 \arrow[r] & E^{\pm}(F^c_{\infty,w})\otimes \Qp/\Zp \arrow[r] \arrow[u, "\alpha_c^{\pm}"] &H^1(F^c_{\infty,w}, E[p^{\infty}]) \arrow[u, "\beta_c"] \arrow[r] & \frac{H^1(F^c_{\infty,w}, E[p^{\infty}])}{\hat{E}^{\pm}(F^c_{\infty,w})\otimes \Qp/\Zp} \arrow[u, "\gamma_c^{\pm}"] \arrow[r] &0
\end{tikzcd}
\end{equation*}
By Lemma~\ref{p-torsion_lemma} we have that ${E(\KK_{\infty,w})[p^{\infty}]=0}$. Therefore ${\ker \beta_c = H^1(\Gal(\KK_{\infty,w}/F^c_{\infty,w}), E(\KK_{\infty,w})[p^{\infty}])=0}$. By Proposition~\ref{invariants_prop} $\alpha_c^{\pm}$ is an isomorphism. The desired result follows from the snake lemma.
\end{proof}

We end this section with a control theorem for the fine Selmer group.
\begin{proposition}\label{fineSelmercontrol_prop1}
Let ${\bullet, \star \in \{+,-\}}$ and ${H \in \mathcal{H}^{\bullet\star}}$. For any ${n \geq 0}$ the map induced by restriction
$$t_n: R_{p^{\infty}}(E/\KK_{\infty}^{H_n}) \longrightarrow R_{p^{\infty}}(E/\KK_{\infty})^{H_n}$$
is an isomorphism.
\end{proposition}
\begin{proof}
The proof is very similar to the proof of \cite[Proposition~2.8]{MHG}. In the proof of loc. cit. we need to replace $J_v(E/\KK_{\infty})$ (resp. $J_v(E/F_n)$) with $M_v(E/\KK_{\infty})$ (resp. $M_v(E/F_n)$) where the latter two groups are defined right after Theorem~\ref{main_theorem3} in Section~\ref{section:mainintroduction}. Using similar notation to the proof of loc. cit. it will suffice to show that each of the groups below is trivial:

\begin{enumerate}[(i)]
\item $A_n:=H^1(H_n, E(\KK_{\infty})[p^{\infty}])$,
\item $B_n:=H^2(H_n, E(\KK_{\infty})[p^{\infty}])$,
\item $C_n:=\prod_{w \in S_n \setminus S_{p,n}} H^1(H_{n,w}, E(\KK_{\infty,w})[p^{\infty}])$,
\item $D_n:=\prod_{w \in S_{p,n}} H^1(H_{n,w}, E(\KK_{\infty,w})[p^{\infty}])$.
\end{enumerate}

Here $S_{p,n}$ denotes the set of primes of the field $F_n$ above $p$. The groups $A_n, B_n$ and $D_n$ are trivial by Lemma~\ref{p-torsion_lemma}. Now we deal with $C_n$: let ${w \in S_n \setminus S_{p,n}}$, and let $v$ be the prime of $K_{H,n}$ below $w$. As primes in $S$ do not split completely in $\KK_{\infty}^H/K$ (because ${H \in \mathcal{H}^{\bullet\star}}$), $F_{n,w}/(K_{H,n})_{v}$ is a $\Zp$-extension. This together with the fact that $(K_{H,n})_{v}$ has no $\mathbb{Z}_p^2$-extension implies that ${H_{n,w}=0}$. Therefore ${C_n=0}$ for all $n$. This completes the proof.
\end{proof}

\section{Results on $\Selpm(E/L)$ for $L=\KK_{\infty}$ and $L \in \mathcal{E}$} \label{section:auxiliary}
In this section, we study properties of the sets $\mathcal{H}^{\bullet\star}$ for ${\bullet,\star \in \{+,-\}}$.
The following proposition shows that both $\mathcal{H}^{++}$ and $\mathcal{H}^{--}$ are not empty. Recall from Section~\ref{section:mainintroduction} that we write $\Lambda_2$ for $\Lambda(G)$; moreover, for any quotient ${\Gamma \cong \Z_p}$ of $G$ we abbreviate $\Lambda(\Gamma)$ to $\Lambda$.

\begin{proposition}\label{Hcyc_prop}
The subgroup $H_{cyc}$ of $G$ fixing the cyclotomic $\Z_p$-extension of $K$ belong to both $\mathcal{H}^{++}$ and $\mathcal{H}^{--}$.
\end{proposition}
\begin{proof}
Let $\Q_{cyc}$ be the cyclotomic $\Zp$-extension of $\Q$. Let $\Selpmsingle(E/K_{cyc})$ be the plus and minus Selmer groups over $\Q_{cyc}$ as in Kobayashi's paper \cite{Kob}. Let $X^{\pm}(E/\Q_{cyc})$ denote the Pontryagin dual of $\Selpmsingle(E/\Q_{cyc})$. Since $p$ is odd, the extensions $\Q_{cyc}/\Q$ and $K/\Q$ are disjoint. Therefore we may identify the Galois groups ${\Gal(\Q_{cyc}/\Q)=\Gal(K_{cyc}/K)}$. Let ${\Gamma_{cyc}=\Gal(K_{cyc}/K)}$. Let $E^{(K)}$ be the quadratic twist of $E$ by $K$. Just as in \cite[Lemma~9.2]{MHG} we can show that we have an isomorphisms of $\Gamma_{cyc}$-modules

\begin{equation}\label{quad_twist_isom+}
\Selm^{++}_{p^\infty}(E/K_{cyc})\cong \Selm^+_{p^\infty}(E/\Q_{cyc}) \times \Selm^+_{p^\infty}(E^{(K)}/\Q_{cyc}),
\end{equation}

\begin{equation}\label{quad_twist_isom-}
\Selm^{--}_{p^\infty}(E/K_{cyc})\cong \Selm^-_{p^\infty}(E/\Q_{cyc}) \times \Selm^-_{p^\infty}(E^{(K)}/\Q_{cyc}).
\end{equation}\\

From \cite[Theorem~7.3(ii)]{Kob} we have that $X^{\pm}(E/\Q_{cyc})$ and $X^{\pm}(E^{(K)}/\Q_{cyc})$ are $\Lambda$-torsion. Therefore from the above isomorphisms it follows that $X^{++}(E/K_{cyc})$ and $X^{--}(E/K_{cyc})$ are both $\Lambda$-torsion.

It is well-known that no prime of $K$ splits completely in $K_{cyc}$. Moreover, since $p$ is odd, the primes above $p$ are totally ramified in $K_{cyc}/K$ (the prime $p$ is totally ramified in the cyclotomic $\Z_p$-extension $\Q_{cyc}$ of $\Q$). This completes the proof.
\end{proof}

\begin{lemma}\label{Xinf_torsion_lemma}
Let ${\bullet, \star \in \{+,-\}}$. If ${\mathcal{H}^{\bullet\star} \neq \emptyset}$, then $X^{\bullet\star}(E/\KK_{\infty})$ is a torsion $\Lambda_2$-module.
\end{lemma}
\begin{proof}
	Choose some ${H \in \mathcal{H}^{\bullet\star}}$ and recall that $X^{\bullet\star}(E/\KK_{\infty}^H)$ is a torsion $\Lambda$-module. In view of Theorem \ref{SelmerControlThm_prop}, it follows that $X^{\bullet\star}(E/\KK_{\infty})_H$ is $\Lambda$-torsion. It is well-known that this implies the intended result. 
\end{proof}

\begin{lemma}\label{Yinf_torsion_lemma}
$Y(E/\KK_{\infty})$ is a torsion $\Lambda_2$-module.
\end{lemma}
\begin{proof}
By Proposition~\ref{Hcyc_prop} and Lemma~\ref{Xinf_torsion_lemma} we have that $X^{++}(E/\KK_{\infty})$ is a torsion $\Lambda_2$-module. As $R_{p^{\infty}}(E/\KK_{\infty})$ is contained in $\Selm^{++}_{p^{\infty}}(E/\KK_{\infty})$, the dual $X^{++}(E/\KK_{\infty})$ surjects onto $Y(E/\KK_{\infty})^{\vee}$. Therefore $Y(E/\KK_{\infty})$ is $\Lambda_2$-torsion.
\end{proof}

\begin{proposition}\label{Torsion_prop}
Let ${\bullet, \star \in \{+,-\}}$. Suppose that ${\mathcal{H}^{\bullet\star} \neq \emptyset}$. Then $X^{\bullet\star}(E/L_{\infty})$ is a torsion $\Lambda$-module for all but finitely many ${L_{\infty} \in \mathcal{E}}$.
\end{proposition}
\begin{proof}
Consider the maps
$$s^{\bullet\star}: \Selpm(E/L_{\infty}) \longrightarrow \Selpm(E/\KK_{\infty})^{\Gal(\KK_{\infty}/L_{\infty})}$$
from Lemma~\ref{restriction_map_lemma}$(a)$. The same argument as in the proof of \cite[Proposition~2.3]{MHG} shows that $\ker s^{\bullet\star}$ is cofinitely generated over $\Zp$.  Therefore by considering the duals of the maps $s^{\bullet\star}$, we see that $X^{\bullet\star}(E/L_{\infty})$ will be a torsion $\Lambda$-module if this is the case for $X^{\bullet\star}(E/\KK_{\infty})_{\Gal(\KK_{\infty}/L_{\infty})}$. We now proceed as in the proof of \cite[Proposition~2.3]{MHG} to get the desired result.
\end{proof}

\begin{proposition}\label{Zpextensions_prop}
Let $\bullet, \star \in \{+,-\}$. Suppose that $\mathcal{H}^{\bullet\star}$ is non-empty. Then for all but finitely many ${L_{\infty} \in \mathcal{E}}$ we have
\begin{enumerate}[(a)]
\item No prime in $S$ splits completely in $L_{\infty}/K$.
\item Every prime of $K$ above $p$ ramifies in $L_{\infty}/K$.
\item $X^{\bullet\star}(E/L_{\infty})$ is a torsion $\Lambda$-module.
\end{enumerate}
\end{proposition}
\begin{proof}
Taking Proposition~\ref{Torsion_prop} into account, this can be proven in the same way as \cite[Prop.~2.5]{MHG}.
\end{proof}

\section{The control theorem for the Selmer group and results on the Iwasawa invariants of $T_{\Lambda_{H,n}}([F_{\Lambda_2}(X(E/\KK_{\infty}))]_{H_n})$}\label{section:control2}
In this section, we prove some key results that will allow us to study the $\mathfrak{M}_H(G)$-property for the Selmer group. First we need an auxiliary lemma.
Recall from Section~\ref{section:mainintroduction} (see~\eqref{eq:f_infty}) that $f_\infty$ denotes the characteristic power series of $T_{\Lambda_2}(X(E/\KK_\infty))$. Let ${H=\overbar{\langle \sigma^a\tau^b \rangle}\subseteq \Gal(\KK_\infty/K)}$. Fix an integer ${n \geq 0}$ and write ${\Upsilon_n=(1+T_1)^{p^na}(1+T_2)^{p^nb}-1}$. Note that $\Upsilon_n + 1$ generates ${H_n = H^{p^n}}$ and therefore ${\Z_p[[\Upsilon_n]] = \Lambda(\Gal(\KK_\infty/F_\infty^n)) = \Lambda(H_n)}$. We also recall that ${\Lambda_{H,n} = \Lambda(G_{H,n}/H_n)}$ where ${G_{H,n} = \Gal(\KK_\infty/K_{H,n})}$ is a subgroup of $G$ of finite index.

\begin{lemma}\label{rankequal_lemma}
With the above notation the following statements are equivalent.
\begin{enumerate}[(a)]
\item $p^n\rank_{\Lambda_2}(X(E/\KK_{\infty}))=\rank_{\Lambda_{H,n}}(X(E/\KK_{\infty})_{H_n})$,
\item $(T_{\Lambda_2}(X(E/\KK_{\infty})))_{H_n}$ is $\Lambda_{H,n}$-torsion.
\item $\Upsilon_n$ is relatively prime to $f_{\infty}$.
\end{enumerate}
\end{lemma}
\begin{proof}
First we show that $(b)$ and $(c)$ are equivalent. When ${n=0}$ this equivalence can be proven in an identical way to the proof of \cite[Lemma~2.2]{MHG}. The equivalence in the general case of ${n \geq 0}$ can be proven by similar arguments. However one needs to be careful when ${n>0}$ since here $\Lambda_{H,n}$ is not $\Lambda_2/\Upsilon_n$. To simplify notation, we denote $T_{\Lambda_2}(X(E/\KK_{\infty}))$ by $Z$. Define ${\Lambda'_{H,n}:=\Lambda_2/\Upsilon_n}$. This completed group ring can be identified with $\Lambda(G/H_n)$. Thus we see that ${\Lambda_{H,n} = \Lambda(G_{H,n}/H_n)}$ is a subring of $\Lambda'_{H,n}$. Since ${[G:G_{H,n}]=p^n}$, the ring $\Lambda'_{H,n}$ is a free $\Lambda_{H,n}$-module of rank $p^n$.

Let $\text{Ann}_{\Lambda'_{H,n}}(Z_{H_n})$  (resp. $\text{Ann}_{\Lambda_{H,n}}(Z_{H_n})$) be the annihilator ideal of $Z_{H_n}$ inside $\Lambda'_{H,n}$ (resp. $\Lambda_{H,n})$. Since as noted above $\Lambda'_{H,n}$ is a free $\Lambda_{H,n}$-module of rank $p^n$, it follows that $\Lambda'_{H,n}/\text{Ann}_{\Lambda'_{H,n}}(Z_{H_n})$ is finitely generated over its subring $\Lambda_{H,n}/\text{Ann}_{\Lambda_{H,n}}(Z_{H_n})$ and hence by \cite[Proposition~5.1]{AM} $\Lambda'_{H,n}/\text{Ann}_{\Lambda'_{H,n}}(Z_{H_n})$ is integral over  $\Lambda_{H,n}/\text{Ann}_{\Lambda_{H,n}}(Z_{H_n})$. This implies by \cite[Theorem~2.2.5]{HS} that we have equality of Krull dimensions: $$\dim(\Lambda_{H,n}/\text{Ann}_{\Lambda_{H,n}}(Z_{H_n}))=\dim(\Lambda'_{H,n}/\text{Ann}_{\Lambda'_{H,n}}(Z_{H_n})). $$
We know that $Z_{H_n}$ is a torsion $\Lambda_{H_n}$-module if and only if ${\dim(\Lambda_{H,n}/\text{Ann}_{\Lambda_{H,n}}(Z_{H_n})) \leq 1}$. Thus we conclude that $Z_{H_n}$ is a torsion $\Lambda_{H_n}$-module if and only if ${\dim(\Lambda'_{H,n}/\text{Ann}_{\Lambda'_{H,n}}(Z_{H_n})) \leq 1}$.

Now we proceed as in the proof of \cite[Lemma~2.2]{MHG}. Taking into account \cite[Lemma~2.1]{MHG}, we have by \cite[Chapt.~VII, \S 4.4 Theorem~5]{Bourbaki}, that there exist irreducible power series ${f_j \in \Zp[[T_1, T_2]]}$, integers $m_i, n_j$ and an exact sequence

$$0 \to  W \to Z \to B \to 0, $$\\
where ${W=\bigoplus_{i=1}^s \Lambda_2/p^{m_i} \oplus \bigoplus_{j=1}^t \Lambda_2/f_j^{n_j}}$ and $B$ is a pseudo-null $\Lambda_2$-module. From this exact sequence, we get another exact sequence

\begin{equation}\label{structure_seq1}
B^{\Upsilon_n=0} \to  W/\Upsilon_n \to Z/\Upsilon_n \to B/\Upsilon_n \to 0.
\end{equation}\\
Since $B$ is a pseudo-null $\Lambda_2$-module, it has Krull dimension at most one as a $\Lambda'_{H,n}$-module. So the $\Lambda'_{H,n}$-modules $B/\Upsilon_n$ and $B^{\Upsilon_n=0}$ also have Krull dimension at most one.

Therefore, it follows from the sequence~\eqref{structure_seq1} that the $\Lambda'_{H,n}$-module ${Z_{H_n}=Z/\Upsilon_n}$ has Krull dimension at most one if and only if the $\Lambda'_{H,n}$-module $W/\Upsilon_n$ has Krull dimension at most one. Hence by our earlier observation, in order to prove that $(b)$ and $(c)$ are equivalent, we must show that the Krull dimension of $W/\Upsilon_n$ as a $\Lambda'_{H,n}$-module is at most one if and only if $f_{\infty}$ and $\Upsilon_n$ are relatively prime. We now use the fact that in a UFD $R$ two nonzero elements ${x,y \in R}$ form a regular sequence if and only if they are relatively prime. Since $p$ does not divide $\Upsilon_n$, it follows from \cite[Theorem~17.4]{Matsumura} that if ${D=\Lambda_2/p^{m_i}}$, then $D/\Upsilon_n$ has Krull dimension one as a $\Lambda'_{H,n}$-module. So we see that the Krull dimension of $W/\Upsilon_n$ as a $\Lambda'_{H,n}$-module is at most one if and only if for all $j$ the $\Lambda'_{H,n}$-module $\Lambda_2/\langle f_j^{n_j}, \Upsilon_n \rangle$ has Krull dimension at most one and this is the case, if and only if, for all $j$, $f_j$ and $\Upsilon_n$ are relatively prime (see loc. cit.). Therefore we see that $(b)$ and $(c)$ are equivalent.

Recall from the introduction that ${G_{H,n}=\Gal(\KK_{\infty}/K_{H,n})}$. Since ${[G:G_{H,n}]=p^n}$, the Iwasawa algebra $\Lambda_2$ is a free $\Zp[[G_{H,n}]]$-module of rank $p^n$. We now replace $G$ by $G_{H,n}$ so that $\Lambda_2$ is replaced by $\Zp[[G_{H,n}]]$. We also replace $H$ by $H_n$. Then by \cite[Corollary~1.5]{Howson} $(a)$ becomes ${\rank_{\Lambda_2}(X(E/\KK_{\infty}))=\rank_{\Lambda_2/\Upsilon_n}(X(E/\KK_{\infty})_H)}$ and $(b)$ is that $(T_{\Lambda_2}(X(E/\KK_{\infty})))_H$ is $\Lambda_2/\Upsilon_n$-torsion. We now prove that these new statements $(a)$ and $(b)$ are equivalent.

Directly from Definition~\ref{def:free-and-torsion} we get the exact sequence
$$0 \to T_{\Lambda_2}(X(E/\KK_{\infty})) \to X(E/\KK_{\infty}) \to F_{\Lambda_2}(X(E/\KK_{\infty})) \to 0. $$
Since $F_{\Lambda_2}(X(E/\KK_{\infty}))$ is $\Lambda_2$-torsion-free, it follows that ${H_1(H, F_{\Lambda_2}(X(E/\KK_{\infty}))) = 0}$. Therefore the above exact sequence induces an exact sequence
$$0 \to (T_{\Lambda_2}(X(E/\KK_{\infty})))_H \to X(E/\KK_{\infty})_H \to (F_{\Lambda_2}(X(E/\KK_{\infty})))_H \to 0. $$
Noting that ${\rank_{\Lambda_2}(X(E/\KK_{\infty}))=\rank_{\Lambda_2}(F_{\Lambda_2}(X(E/\KK_{\infty})))}$, we see from this exact sequence that to show the equivalence of $(a)$ and $(b)$ it will suffice to show that ${\rank_{\Lambda_2/\Upsilon_n}(F_{\Lambda_2}(X(E/\KK_{\infty}))_H)=\rank_{\Lambda_2}(F_{\Lambda_2}(X(E/\KK_{\infty})))}$. This will follow from the next
\begin{lemma}
	With notation as above, we have
	$$\rank_{\Lambda_2/\Upsilon_n}(F_{\Lambda_2}(X(E/\KK_{\infty}))_H)=\rank_{\Lambda_2}(F_{\Lambda_2}(X(E/\KK_{\infty}))).$$
\end{lemma}
\begin{proof}
	It follows from \cite[Theorem~1.1]{Howson} that $\rank_{\Lambda_2}(F_{\Lambda_2}(X(E/\KK_{\infty})))$ is given by the alternating sum
	\[\begin{tikzcd}[scale cd=0.8]\rg_{\Zp}(H_2(G,F_{\Lambda_2}(X(E/\KK_{\infty})))) - \rg_{\Zp}(H_1(G, F_{\Lambda_2}(X(E/\KK_{\infty})))) + \rg_{\Zp}(H_0(G,F_{\Lambda_2}(X(E/\KK_{\infty})))). \end{tikzcd} \]
	Since ${H_1(H, F_{\Lambda_2}(X(E/\KK_{\infty}))) = 0}$ and as ${cd_p(H)=cd_p(G/H)=1}$, the Hochschild-Serre spectral sequence implies that the first term in the alternating sum vanishes; moreover, it also follows from these ingredients that
	\[ H_1(G, F_{\Lambda_2}(X(E/\KK_{\infty}))) = H_1(G/H, F_{\Lambda_2}(X(E/\KK_{\infty}))_H)\]
	and
	\[ H_0(G, F_{\Lambda_2}(X(E/\KK_{\infty}))) = H_0(G/H, F_{\Lambda_2}(X(E/\KK_{\infty}))_H).\]
	Applying Howson's formula \cite[Theorem~1.1]{Howson} again (this time for $G/H$ instead of $G$) proves the lemma.
\end{proof}
This also concludes the proof of Lemma~\ref{rankequal_lemma}.
\end{proof}

We now prove the following control theorem:

\begin{proposition}\label{SelmerControlThm_prop2}
Let ${\bullet, \star \in \{+,-\}}$ and ${H \in \mathcal{H}^{\bullet\star}}$ and let ${n \geq 0}$. Then the natural map (induced by restriction):
$$s_n: \Selinf(E/\KK_{\infty}^{H_n}) \to \Selinf(E/\KK_{\infty})^{H_n}$$ is an isomorphism.
\end{proposition}
\begin{proof}
For simplicity, let ${F_{\infty}^n=\KK_{\infty}^{H_n}}$. First we note that for any prime $v$ of $F_{\infty}^n$ above $p$ we have ${H^1(F^n_{\infty,v}, E)[p^{\infty}]=0}$. This follows from the fact that $E$ has good supersingular reduction at $p$ and every prime of $K$ above $p$ ramifies in $\KK_{\infty}^H/K$ (see \cite[pg.~70]{Gb_LNM}). Similarly for any prime $w$ of $\KK_{\infty}$ above $p$ we have ${H^1(\KK_{\infty,w}, E)[p^{\infty}]=0}$. Let $S_{tame, \infty}$ (resp. $S_{tame,n}$) be the set of primes of $\KK_{\infty}$ (resp. $F_{\infty}^n$) above a prime ${v \in S}$ where ${v \nmid p}$. Then we have a commutative diagram with vertical maps induced by restriction
\begin{equation*}
\begin{tikzcd}[column sep = small, scale cd=0.85]
0 \arrow[r] & \Selinf(E/\KK_{\infty})^{H_n} \arrow[r] &  H^1(G_S(\KK_{\infty}), E[p^{\infty}])^{H_n}  \arrow[r] & \Big(\underset{w \in S_{tame, \infty}}{\prod} H^1(\KK_{\infty,w}, E(\KK_{\infty,w})[p^{\infty}])\Big)^{H_n}\\
0 \arrow[r] & \Selinf(E/F_{\infty}^n) \ar[r] \arrow[u, "s_n"] &H^1(G_S(F_{\infty}^n), E[p^{\infty}]) \arrow[u, "g_n"] \arrow[r, "\theta_n"] & \underset{v \in S_{tame, n}}{\prod} H^1(F^n_{\infty,v}, E(F_{\infty,v}^n)[p^{\infty}]) \arrow[u, "h_n"]
\end{tikzcd}
\end{equation*}
By the snake lemma we have an exact sequence
$$0 \to \ker s_n \to \ker g_n \to \ker h_n \cap \img \theta_n \to \coker s_n \to \coker g_n.$$

From this exact sequence we see that to show that $s_n$ is an isomorphism we only need to show that $g_n$ is an isomorphism and that $h_n$ is injective.

We have
$$\ker {g_n} =H^1(H_n, E(\KK_{\infty})[p^{\infty}])$$
and we have an injection
$$\coker g_n \hookrightarrow H^2(H_n, E(\KK_{\infty})[p^{\infty}]).$$
Therefore we see from Lemma~\ref{p-torsion_lemma} that $g_n$ is an isomorphism. As for the map $h_n$ it follows from Shapiro's lemma that $$\ker h_n=\prod_{w \in S_{tame,n}} H^1(H_{n,w}, E(\KK_{\infty,w})[p^{\infty}])$$ where in the product above we have also written $w$ for a fixed prime of $\KK_{\infty}$ above $w$ and $H_{n,w}$ is the decomposition group. As primes in $S$ do not split completely in $F_{\infty}^n/K$ by assumption, $F_{\infty,w}^n/K_{n,w}$ is a $\Zp$-extension for any ${w \in S_{tame,n}}$. Since $K_{n,w}$ has no $\mathbb{Z}_p^2$-extension we may conclude that ${\KK_{\infty,w} = F_{\infty,w}^n}$. Therefore the decomposition subgroup $H_{n,w}$ of ${H_n = \Gal(\KK_\infty/F_{\infty}^n)}$ is trivial, and ${\ker h_n=0}$ as required. This completes the proof.
\end{proof}

\begin{lemma}\label{WL_lemma}
Let ${\bullet, \star \in \{+,-\}}$ and ${H \in \mathcal{H}^{\bullet\star}}$. If $Y(E/\KK_{\infty})_f$ is finitely generated over $\Lambda(H)$, then \begin{enumerate}
\item $Y(E/\KK_{\infty}^{H_n})$ is a torsion $\Lambda_{H,n}$-module, and
\item ${H^2(G_S(\KK_{\infty}^{H_n}), E[p^{\infty}])=0}$
\end{enumerate} for all ${n \geq 0}$.
\end{lemma}
\begin{proof}
Suppose that $Y(E/\KK_{\infty})_f$ is finitely generated over $\Lambda(H)$. Consider the following commutative diagram with exact rows
\begin{equation*}
\xymatrix {
Y(E/\KK_{\infty})_{H_n}  \ar[d]^{\hat{t}_n} \ar[r] & (Y(E/\KK_{\infty})_f)_{H_n}  \ar[d]^{\theta_n} \ar[r] &0\\
Y(E/\KK_{\infty}^{H_n}) \ar[r] & Y(E/\KK_{\infty}^{H_n})_f \ar[r] &0}
\end{equation*}
The map $\hat{t}_n$ is the dual of the map $t_n$ in Proposition~\ref{fineSelmercontrol_prop1} and $\theta_n$ is induced by $\hat{t}_n$. By Proposition~\ref{fineSelmercontrol_prop1}, $\hat{t}_n$ is surjective and hence so is $\theta_n$. Since $Y(E/\KK_{\infty})_f$ is finitely generated over $\Lambda(H)$ and $H_n$ has finite index in $H$, the quotient $(Y(E/\KK_{\infty})_f)_{H_n}$ is finitely generated over $\Zp$. Thus we get that $Y(E/\KK_{\infty}^{H_n})_f$ is finitely generated over $\Zp$. This implies that $Y(E/\KK_{\infty}^{H_n})$ is a torsion $\Lambda_{H,n}$-module. This latter implies by \cite[Theorem~1.1]{Matar_Torsion} that ${H^2(G_S(\KK_{\infty}^{H_n}), E[p^{\infty}])=0}$ as desired.
\end{proof}

\begin{proposition}\label{Selrank_prop}
The following assertions hold.
\begin{enumerate}[(a)]
\item Let ${\bullet, \star \in \{+,-\}}$ and ${H \in \mathcal{H}^{\bullet\star}}$. We have that ${\rank_{\Lambda_{H,n}}(X(E/\KK_{\infty}^{H_n}))=2p^n}$ for ${n=0}$. Assuming that $Y(E/\KK_{\infty})_f$ is finitely generated over $\Lambda(H)$ this holds for all ${n \geq 0}$.
\item ${\rank_{\Lambda_2}(X(E/\KK_{\infty}))=2}$.
\end{enumerate}
\end{proposition}
\begin{proof}
For simplicity, let ${F^n_{\infty}=\KK_{\infty}^{H_n}}$. Using the same notation as in the proof of Proposition~\ref{SelmerControlThm_prop2}, we have exact sequences

\begin{equation}\label{Seldef1}
0 \to \Selinf(E/F^n_{\infty}) \to H^1(G_S(F^n_{\infty}), E[p^{\infty}]) \to \prod_{w \in S_{tame,n}} H^1(F^n_{\infty,w}, E(F^n_{\infty,w})[p^{\infty}]),
\end{equation}

\begin{equation}\label{Seldef2}
0 \to \Selinf(E/\KK_{\infty})\to  H^1(G_S(\KK_{\infty}), E[p^{\infty}]) \to \prod_{w \in S_{tame, \infty}} H^1(\KK_{\infty,w}, E(\KK_{\infty,w})[p^{\infty}]).
\end{equation}

Now assume that $Y(E/\KK_{\infty})_f$ is finitely generated over $\Lambda(H)$ if ${n >0}$. For any ${w \in S_{tame,n}}$ we have by \cite[Proposition~2]{Gb_IPR} that $H^1(F^n_{\infty,w}, E(F^n_{\infty,w})[p^{\infty}])^{\vee}$ is a finitely generated $\Zp$-module. Therefore as $S_{tame,n}$ is finite (since ${H \in \mathcal{H}^{\bullet\star}}$) it follows that ${(\prod_{w \in S_{tame, n}} H^1(F^n_{\infty,w}, E(F^n_{\infty,w})[p^{\infty}]))^{\vee}}$ is a torsion $\Lambda_{H,n}$-module. If ${n=0}$ we have that ${H^2(G_S(F_{\infty}), E[p^{\infty}])=0}$ by Proposition~\ref{Seln_surjective_prop}(3) and if ${n >0}$ we have that ${H^2(G_S(F^n_{\infty}), E[p^{\infty}])=0}$ by Lemma \ref{WL_lemma}. Therefore from \cite[Proposition~3]{Gb_IPR} we have that ${\rank_{\Lambda_{H,n}}(H^1(G_S(F^n_{\infty}), E[p^{\infty}]))=2p^n}$. Assertion~$(a)$ now follows from the remarks above.

As for $(b)$ we first note that by Lemma~\ref{H2_Kinf_vainishing_lemma} we have that ${H^2(G_S(\KK_{\infty}), E[p^{\infty}])=0}$. Therefore from \cite[Proposition~3.2]{OV} we have that ${\rank_{\Lambda_2}(H^1(G_S(K_{\infty}), E[p^{\infty}]))=2}$.

Let $S_{tame, K_{cyc}}$ be the set of primes of $K_{cyc}$ above a prime ${v \in S}$ where ${v \nmid p}$. Let
$$A:=\prod_{w \in S_{tame, K_{cyc}}} H^1(K_{cyc,w}, E(K_{cyc,w})[p^{\infty}]),$$
$$B:=\prod_{w \in S_{tame, \infty}} H^1(\KK_{\infty,w}, E(\KK_{\infty,w})[p^{\infty}]).$$
Note that every prime of $K_{cyc}$ that does not divide $p$ splits completely in $\KK_{\infty}/K_{cyc}$. Therefore we see that ${B^{H_{cyc}}=A}$, and by \cite[Proposition~2]{Gb_IPR} $A$ is a cofinitely generated $\Zp$-module. Therefore $B^{\vee}$ is a cofinitely generated $\Zp[[H_{cyc}]]$-module and hence is a torsion $\Lambda_2$-module. By the above remarks this implies assertion~$(b)$.
\end{proof}

Putting together Lemma~\ref{rankequal_lemma}, Proposition~\ref{SelmerControlThm_prop2} and Proposition~\ref{Selrank_prop}, we obtain an important corollary.
\begin{corollary} \label{cor:HinH}
	Let ${\bullet, \star \in \{+,-\}}$ and ${H \in \mathcal{H}^{\bullet\star}}$. Then the equivalent conditions from Lemma~\ref{rankequal_lemma} hold for $H$ and ${n=0}$. Assuming that $Y(E/\KK_{\infty})_f$ is finitely generated over $\Lambda(H)$ these equivalent conditions also hold for all ${n \geq 0}$.
\end{corollary}

The main control theorem that we need is the following result. For any ${n \geq 0}$, ${\bullet, \star \in \{+,-\}}$ and ${H \in \mathcal{H}^{\bullet\star}}$, let $\hat{s_n}$ denote the dual of the map $s_n$ in Proposition~\ref{SelmerControlThm_prop2}.

\begin{theorem}\label{TorsionControlThm}
Assume that $Y(E/\KK_{\infty})_f$ is finitely generated over $\Lambda(H)$ if ${n >0}$. Each map $\hat{s_n}$ induces an injection
$$\hat{s_n}': (T_{\Lambda_2}(X(E/\KK_{\infty})))_{H_n} \hookrightarrow T_{\Lambda_{H,n}}(X(E/\KK_{\infty}^{H_n}))$$
such that there exists an $\Lambda_{H,n}$-isomorphism
\[ \phi_{H_n}: T_{\Lambda_{H,n}}([F_{\Lambda_2}(X(E/\KK_{\infty}))]_{H_n})  \isomarrow \coker \hat{s_n}'. \]
\end{theorem}
\begin{proof}
Corollary~\ref{cor:HinH} implies that $(T_{\Lambda_2}(X(E/\K_{\infty})))_{H_n}$ is $\Lambda_{H,n}$-torsion. From this we see that the map $\hat{s_n}$ induces a map ${\hat{s_n}': (T_{\Lambda_2}(X(E/\KK_{\infty})))_{H_n} \to T_{\Lambda}(X(E/\KK_{\infty}^{H_n}))}$. Now consider the commutative diagram

\begin{equation*}
\begin{tikzcd}[column sep = small, scale cd=0.99]
0 \arrow[r] & (T_{\Lambda_2}(X(E/\KK_{\infty})))_{H_n} \arrow[r] \arrow[d, "\hat{s_n}'"] & X(E/\KK_{\infty})_{H_n} \arrow[r] \arrow[d, "\hat{s_n}"] & (F_{\Lambda_2}(X(E/\KK_{\infty})))_{H_n} \arrow[r] \arrow[d, "\hat{s_n}''"] &0\\
0 \arrow[r] & T_{\Lambda_{H,n}}(X(E/\KK_{\infty}^{H_n})) \arrow[r] & X(E/\KK_{\infty}^{H_n}) \arrow[r] & F_{\Lambda_{H,n}}(X(E/\KK_{\infty}^{H_n})) \arrow[r] &0
\end{tikzcd}
\end{equation*}
The top row is exact since $F_{\Lambda_2}(X(E/\KK_{\infty}))$ is $\Lambda_2$-torsion-free. From the snake lemma applied to this diagram we get an exact sequence
\begin{equation}\label{snakelemma_seq2}
0 \to \ker \hat{s_n}' \to \ker \hat{s_n} \to \ker \hat{s_n}'' \to \coker \hat{s_n}' \to \coker \hat{s_n} \to \coker \hat{s_n}'' \to 0.
\end{equation}
By Proposition~\ref{SelmerControlThm_prop2} we have ${\ker \hat{s_n}=\coker \hat{s_n}=0}$. Therefore from the exact sequence (\ref{snakelemma_seq2}) we see that ${\ker \hat{s_n}'=0}$ and that there is an isomorphism ${\phi_H: \ker \hat{s_n}'' \to \coker \hat{s_n}'}$. So to complete the proof it will suffice to show that $\ker \hat{s_n}''=T_{\Lambda_{H,n}}([F_{\Lambda_2}(X(E/\KK_{\infty}))]_{H_n})$.
By Lemma~\ref{rankequal_lemma} we have that $(T_{\Lambda_2}(X(E/\KK_{\infty})))_{H_n}$ is $\Lambda_{H,n}$-torsion and so
$$\rank_{\Lambda_{H,n}}(X(E/\KK_{\infty})_{H_n})=\rank_{\Lambda_{H,n}}([F_{\Lambda_2}(X(E/\KK_{\infty}))]_{H_n}).$$
Also we clearly have
$$\rank_{\Lambda_{H,n}}(X(E/\K_{\infty}^{H_n}))=\rank_{\Lambda_{H,n}}(F_{ \Lambda_{H,n}}(X(E/\KK_{\infty}^{H_n})).$$
Therefore Proposition~\ref{SelmerControlThm_prop2} implies that
\begin{equation}\label{rank_equality_coinvariants}
\rank_{\Lambda_{H,n}}([F_{\Lambda_2}(X(E/\KK_{\infty}))]_{H_n})=\rank_{\Lambda_{H,n}}(F_{ \Lambda_{H,n}}(X(E/\KK_{\infty}^{H_n}))).
\end{equation}
Since $\coker \hat{s_n}=0$, the exact sequence (\ref{snakelemma_seq2}) shows that ${\coker \hat{s_n}''=0}$. Therefore from the equality~\eqref{rank_equality_coinvariants} we see that $$\rank_{\Lambda_{H,n}}([F_{\Lambda_2}(X(E/\KK_{\infty}))]_{H_n})=\rank_{\Lambda_{H,n}}(\img {\hat{s}_n}'').$$
Furthermore $\img {\hat{s}_n}''$ is $\Lambda_{H,n}$-torsion-free. It follows from these two facts that ${\ker {\hat{s}_n}''=T_{\Lambda_{H,n}}([F_{\Lambda_2}(X(E/\KK_{\infty}))]_{H,n})}$. This completes the proof.
\end{proof}

\begin{remark}
In relation to the object $T_{\Lambda_{H,n}}([F_{\Lambda_2}(X(E/\KK_{\infty}))]_{H_n})$ that appears in the statement of Proposition \ref{TorsionControlThm}, one may ask the following question: If ${\Lambda_2=\Zp[[T,U]]}$, ${\Lambda=\Lambda_2/T}$ and $M$ is a finitely generated torsion-free $\Lambda_2$-module, can ${T_{\Lambda}(M/T)}$ be non-trivial? The answer is yes. In fact we will give two examples: one with ${\lambda(T_{\Lambda}(M/T))=1}$ and one with ${\mu(T_{\Lambda}(M/T))=1}$. Let ${M_1 = \langle T,U \rangle}$ be the ideal of $\Lambda_2$ generated by $T$ and $U$. Then $M_1$ is certainly a torsion-free $\Lambda_2$-module. We have an exact sequence
$$0 \to M_1 \to \Lambda_2 \to \Lambda_2/M_1 \to 0. $$
From this exact sequence, we get the following exact sequence:
$$\Lambda_2[T] \to (\Lambda_2/M_1)[T] \to M_1/T \to \Lambda_2/T. $$
We have ${\Lambda_2[T]=0}$, ${(\Lambda_2/M_1)[T]=\Lambda_2/M_1}$ and ${\Lambda_2/T=\Lambda}$. Therefore we have
$$0 \to \Lambda_2/M_1 \to M_1/T \to \Lambda. $$
From this we see that the $\Lambda$-torsion submodule of $M_1/T$ is isomorphic to ${\Lambda_2/M_1 \cong \Zp}$, so ${\lambda(T_{\Lambda}(M_1/T))=1}$. If one instead considers ${M_2=\langle p,T \rangle}$, then the previous arguments show that ${T_{\Lambda}(M_2/T)\cong \Lambda/p}$, so ${\mu(T_{\Lambda}(M_2/T))=1}$.
\end{remark}

\begin{lemma}\label{structureisom_lemma}
There exists a pseudo-null $\Lambda_2$-module $A$ such that for any ${n \geq 0}$, ${[(a,b)] \in \mathbb{P}^1(\Zp)}$, ${H=\overbar{\langle \sigma^a\tau^b \rangle}}$ and ${\Upsilon_n=(1+T_1)^{p^na}(1+T_2)^{p^nb}-1}$ we have a $\Lambda_{H,n}$-isomorphism
$$T_{\Lambda_{H,n}}([F_{\Lambda_2}(X(E/\KK_{\infty}))]/\Upsilon_n) \cong A[\Upsilon_n]. $$
\end{lemma}
\begin{proof}
For simplicity, we denote $F_{\Lambda_2}(X(E/\KK_{\infty}))$ by $\mathfrak{F}$. Since $\mathfrak{F}$ is a torsion-free $\Lambda_2$ module, it follows from  \cite[Prop.~5.1.8]{NSW} that there exist a reflexive $\Lambda_2$-module $M$ and a pseudo-null $\Lambda_2$-module $A$ such that we have an exact sequence
\begin{equation}\label{refmod_seq}
0 \to \mathfrak{F} \to M \to A \to 0.
\end{equation}
This exact sequence induces the exact sequence
\begin{equation}\label{refmod_seq2}
M[\Upsilon_n] \to A[\Upsilon_n] \to \mathfrak{F}/\Upsilon_n \to M/\Upsilon_n.
\end{equation}
We claim that the $\Lambda_{H,n}$-module $M/\Upsilon_n$ is torsion-free. To see this, we note that by \cite[Corollary~3.7]{EG} we have that $M$ is a (second) syzygy. So ${M=\ker(B \to C)}$ where $B$ and $C$ are free $\Lambda_2$-modules. It follows from this that we have an exact sequence
$$0 \to M \to B \to D \to 0$$
where $B$ is a free $\Lambda$-module and $D$ is a torsion-free $\Lambda_2$-module. From this exact sequence we get an exact sequence
$$D[\Upsilon_n] \to M/\Upsilon_n \to B/\Upsilon_n.$$
Since $D$ is torsion-free, ${D[\Upsilon_n]=0}$. Since $B$ is a free $\Lambda_2$-module and $\Lambda_2$ is a free $\Zp[[G_{H,n}]]$-module, $B$ is a free $\Zp[[G_{H,n}]]$-module whence $B/\Upsilon_n$ is a free $\Lambda_{H,n}$-module. Therefore as claimed $M/\Upsilon_n$ is a torsion-free $\Lambda_{H,n}$-module. Also as $M$ is reflexive, it is a torsion-free $\Lambda_2$-module. Therefore ${M[\Upsilon_n]=0}$. It follows from these facts and the exact sequence (\ref{refmod_seq2}) that we have an isomorphism ${T_{\Lambda_{H,n}}(\mathfrak{F}/\Upsilon_n) \cong A[\Upsilon_n]}$.
\end{proof}

The rest of the section is devoted to proving results about Iwasawa invariants of $T_{\Lambda_{H,n}}([F_{\Lambda_2}(X(E/\KK_{\infty}))]_{H_n})$.

\begin{proposition}\label{mu_vanishing_prop}
Let ${\bullet,\star \in\{+,-\}}$. Then for all but finitely many ${H \in \mathcal{H}^{\bullet\star}}$ we have that ${\mu_{G_n/H_n}(T_{\Lambda_{H,n}}([F_{\Lambda_2}(X(E/\KK_{\infty}))]_{H_n})) = 0}$ for every ${n \ge 0}$.
\end{proposition}
\begin{proof}
For simplicity, we denote $F_{\Lambda_2}(X(E/\KK_{\infty}))$ by $\mathfrak{F}$. By Lemma \ref{structureisom_lemma} we have that there exists a pseudo null $\Lambda_2$-module $A$ such that for any ${H=\overbar{\langle \sigma^a\tau^b \rangle}\in \mathcal{H}^{\bullet\star}}$ we have that there exists a $\Lambda_{H,n}$-isomorphism ${T_{\Lambda_{H,n}}(\mathfrak{F}/\Upsilon_n) \cong A[\Upsilon_n]}$, where ${\Upsilon_n=(1+T_1)^{p^na}(1+T_2)^{p^nb}-1}$.

Letting ${N := (A[\Upsilon_n])[p^\infty]}$ for brevity we have that ${\mu_{G_n/H_n}(T_{\Lambda_{H,n}}(\mathfrak{F}/\Upsilon_n)) = \mu_{G_n/H_n}(N)}$.

Since $A$ is pseudo-null, the same holds true for $A[p^\infty]$. Therefore the height of its annihilator ideal ${\text{Ann}(A[p^\infty]) \subseteq \Lambda_2}$ is at least two. If this height equals three, then since $\Lambda_2$ has Krull dimension equal to three we may conclude that ${N \subseteq A[p^\infty]}$ is finite.

So we are left with the case that the height of $\text{Ann}(A[p^\infty])$ is exactly two. In order to handle this case we will use an idea from \cite{local_max}, see the proofs of Lemmas~3.1 and 5.5 in loc.cit.
We call a subgroup ${H \in \mathcal{H}^{\bullet\star}}$ of $G$ \emph{bad} if there exists at least one integer ${n \in \N}$ such that ${\mu_{G_n/H_n}(T_{\Lambda_{H,n}}(\mathfrak{F}/\Upsilon_n)) > 0}$. For any such $H$, let $n$ be the minimal integer such that this inequality holds, and let $I_H$ be the ideal
\[ I_H = \textup{Ann}_{\Lambda_2}(A[p^\infty]) + (\Upsilon_n). \]
We want to prove that there are only finitely many bad choices for $H$. To this end, note that if the height of $I_H$ was equal to three, then $(A[\Upsilon_n])[p^\infty]$ would be finite, and so in that case we would have ${\mu_{G_n/H_n}(T_{\Lambda_{H,n}}(\mathfrak{F}/\Upsilon_n)) = 0}$, contrary to our assumptions. Therefore for any bad $H$, we must have that the height of the ideal $I_H$ is equal to two.

Now let $\fp_H$ be any minimal prime ideal of $I_H$. Then the height of $\fp_H$ is also equal to two. Since $\fp_H$ contains the annihilator ideal of $A[p^\infty]$ and as $A$ is a pseudo-null $\Lambda_2$-module, any such prime ideal $\fp_H$ is also a minimal prime ideal of $\textup{Ann}_{\Lambda_2}(A[p^\infty])$. Note that $\textup{Ann}_{\Lambda_2}(A[p^\infty])$ has only finitely many pairwise distinct minimal prime ideals because $\Lambda_2$ is a Noetherian domain.

Therefore the only possible way to produce an infinite number of bad subgroups $H$ would be to have two (in fact, infinitely many) bad subgroups ${H_1 \ne H_2}$ such that the corresponding minimal prime ideals $\fp_{H_1}$ and $\fp_{H_2}$, chosen as above, are in fact equal. So suppose that ${\fp_{H_1} = \fp_{H_2}}$, and let $n_1$ and $n_2$ be the minimal integers that are used in the definitions of $I_{H_1}$ and $I_{H_2}$. Then
\[ \fp_{H_1} \supseteq \textup{Ann}_{\Lambda_2}(A[p^\infty]) + (\Upsilon_{n_1}) + (\Upsilon'_{n_2}), \]
where we denote by $\Upsilon$ (respectively, $\Upsilon'$) the variable attached to $H_1$ (respectively, $H_2$). The groups $H_1^{p^{n_1}}$ and $H_2^{p^{n_2}}$ are generated topologically by ${\gamma_1 = \Upsilon_{n_1} +1}$ and ${\gamma_2 = \Upsilon_{n_2}' + 1}$. Since ${H_1 \ne H_2}$, we know that the subgroup ${\langle \gamma_1, \gamma_2 \rangle}$ of $G$ has finite index. Therefore the height of $\fp_{H_1}$ is three, contrary to our assumptions.

This shows that ${\fp_{H_1} \ne \fp_{H_2}}$ for any two distinct bad subgroups $H_1$ and $H_2$ of $G$. Since the number of minimal primes of $\textup{Ann}_{\Lambda_2}(A[p^\infty])$ is finite, we can conclude that there are only finitely many bad subgroups $H$.
\end{proof}

\begin{lemma}\label{invaraints bound_lemma}
	Let $\bullet,\star \in\{+,-\}$. Then \begin{compactenum}[(a)]
		\item there exists ${C \geq 0}$ such that for any ${n \geq 0}$ and any ${H \in \mathcal{H}^{\bullet\star}}$ we have ${\lambda_{G_n/H_n}(T_{\Lambda_{H,n}}([F_{\Lambda_2}(X(E/\KK_{\infty}))]_{H_n})) \leq C}$,
		\item there exists ${D \geq 0}$ such that for any ${n \geq 0}$ and any ${H \in \mathcal{H}^{\bullet\star}}$ we have ${\mu_{G_n/H_n}(T_{\Lambda_{H,n}}([F_{\Lambda_2}(X(E/\KK_{\infty}))]_{H_n})) \leq D}$.
	\end{compactenum}
\end{lemma}
\begin{proof}
	For simplicity, we denote $F_{\Lambda_2}(X(E/\KK_{\infty}))$ by $\mathfrak{F}$. By Lemma \ref{structureisom_lemma} we have that there exists a pseudo null $\Lambda_2$-module $A$ such that for any ${H=\overbar{\langle \sigma^a\tau^b \rangle}\in \mathcal{H}^{\bullet\star}}$ we have  that there exists a $\Lambda_{H,n}$-isomorphism ${T_{\Lambda_{H,n}}(\mathfrak{F}/\Upsilon_n) \cong A[\Upsilon_n]}$, where ${\Upsilon_n=(1+T_1)^{p^na}(1+T_2)^{p^nb}-1}$.
	
	Let $A_f=A/A[p^{\infty}]$. Since $A$ is a pseudo-null $\Lambda_2$-module, $A_f$ is also pseudo-null i.e. it has Krull dimension at most one. Since multiplication by $p$ is injective on $A_f$ by definition, it follows that $A_f/p$ is finite by \cite[Corollary~11.9]{AM}, i.e. $A_f$ is a finitely generated free $\Z_p$-module. Now the exact sequence
	$$0 \to A[p^{\infty}] \to A \to A_f \to 0$$
	induces an exact sequence
	$$0 \to (A[p^{\infty}])[\Upsilon_n] \to A[\Upsilon_n] \to A_f[\Upsilon_n]. $$
	Since clearly ${\lambda_{G_n/H_n}((A[p^{\infty}])[\Upsilon_n])=0}$, we have
	$$\lambda_{G_n/H_n}(T_{\Lambda_{H,n}}(\mathfrak{F}/\Upsilon_n))=\lambda_{G_n/H_n}(A[\Upsilon_n]) \leq \lambda_{G_n/H_n}(A_f[\Upsilon_n]) \leq \rank_{\Zp}(A_f). $$
	This implies the statement in part~$(a)$.
	
	Part~$(b)$ of the lemma follows from Proposition~\ref{mu_vanishing_prop}.
\end{proof}

We have shown above that the term $\mu_{G_n/H_n}(T_{\Lambda_{H,n}}([F_{\Lambda_2}(X(E/\KK_{\infty}))]_{H_n}))$, which shows up in Theorems~\ref{main_theorem2} and \ref{main_theorem3}, is usually zero. We now show that this term can be replaced by the $\mu$-invariant of other objects. Note that the proof below depends on the proof of Theorem~\ref{main_theorem2} which will be proven in Sections~\ref{section:main-theorem2a} and \ref{section:main-theorem2}.

\begin{theorem}\label{mu-invariants_theorem}
Let ${\bullet, \star \in \{+,-\}}$ and ${H=\overbar{\langle \sigma^a\tau^b \rangle}\in \mathcal{H}^{\bullet\star}}$. Let $Y(E/\KK_{\infty})^{\circ}$ denote the maximal pseudo-null submodule of $Y(E/\KK_{\infty})$. Then the following assertions hold.
\begin{enumerate}
\item We have ${F_{\Lambda_2}(X(E/\KK_{\infty}))=F_{\Lambda_2}(H^1(G_S(\KK_{\infty}), E[p^{\infty}])^\vee)}$ and hence for any ${n \geq 0}$ \\
\resizebox{.9\hsize}{!}{$\mu_{G_n/H_n}(T_{\Lambda_{H,n}}([F_{\Lambda_2}(X(E/\KK_{\infty}))]_{H_n}))= \mu_{G_n/H_n}(T_{\Lambda_{H,n}}([F_{\Lambda_2}(H_1(G_S(\KK_{\infty}), E[p^{\infty}]))]_{H_n}))$}.
\item Assume that $T_{\Lambda_2}(X(E/\KK_{\infty}))_f$ is finitely generated over $\Lambda(H)$ (or equivalently by Theorem~\ref{equivalence_theorem} that $Y(E/\KK_{\infty})_f$ is finitely generated over $\Lambda(H)$). Then for any ${n \geq 0}$ we have
\begin{align*}
\mu_{G_n/H_n}(H_1(H_n, Y(E/\KK_{\infty})))&=\mu_{G_n/H_n}(H_1(H_n, Y(E/\KK_{\infty})^{\circ}))\\
&=\mu_{G_n/H_n}(H_0(H_n, Y(E/\KK_{\infty})^{\circ}))\\
&=\mu_{G_n/H_n}(T_{\Lambda_{H,n}}([F_{\Lambda_2}(X(E/\KK_{\infty}))]_{H_n})).
\end{align*}
\end{enumerate}
\end{theorem}
\begin{proof}
Using the same notation as in the proof of Proposition~\ref{Selrank_prop} we have an exact sequence
$$0 \to \Selinf(E/\KK_{\infty})\to  H^1(G_S(\KK_{\infty}), E[p^{\infty}]) \to \prod_{w \in S_{tame, \infty}} H^1(\KK_{\infty,w}, E(\KK_{\infty,w})[p^{\infty}]). $$
Dualizing this sequence we get
$$U \to H^1(G_S(\KK_{\infty}), E[p^{\infty}])^\vee \xrightarrow{\phi} X(E/\KK_{\infty}) \to 0, $$
where ${U=\prod_{w \in S_{tame, \infty}} H^1(\KK_{\infty,w}, E(\KK_{\infty,w})[p^{\infty}])^{\vee}}$.

To prove the first part of theorem, namely ${F_{\Lambda_2}(X(E/\KK_{\infty}))=F_{\Lambda_2}({H^1(G_S(\KK_{\infty}), E[p^{\infty}])^\vee})}$, it easy to see that is suffices to prove that the induced map from $\phi$ given by
$$\phi': T_{\Lambda_2}(H^1(G_S(\KK_{\infty}), E[p^{\infty}])^\vee) \to  T_{\Lambda_2}(X(E/\KK_{\infty}))$$
is surjective. To this end, let ${x \in T_{\Lambda_2}(X(E/\KK_{\infty}))}$. Since $\phi$ is surjective, there exists ${y \in {H^1(G_S(\KK_{\infty}), E[p^{\infty}])^\vee}}$ with ${\phi(y)=x}$. Since ${x \in T_{\Lambda_2}(X(E/\KK_{\infty}))}$, there exists a nonzero ${r \in \Lambda_2}$ with ${\phi(ry)=r\phi(y)=rx=0}$. So ${ry \in \ker \phi}$. In the proof of Proposition~\ref{Selrank_prop} we showed that $U$ is $\Lambda_2$-torsion. Therefore ${\ker \phi}$ is $\Lambda_2$-torsion. Thus there exists a nonzero ${s \in \Lambda_2}$ with ${(sr)y=0}$. Therefore ${y \in T_{\Lambda_2}({H^1(G_S(\KK_{\infty}), E[p^{\infty}])^\vee})}$ and whence $\phi'$ is surjective.

Now we prove the second part. Let ${n \geq 0}$. Assume that $T_{\Lambda_2}(X(E/\KK_{\infty}))_f$ is finitely generated over $\Lambda(H)$ (or equivalently by Theorem~\ref{equivalence_theorem} that $Y(E/\KK_{\infty})_f$ is finitely generated over $\Lambda(H)$). First we observe that each of the groups $H_1(H_n, Y(E/\KK_{\infty}))$, $H_1(H_n, Y(E/\KK_{\infty})^{\circ})$, $H_0(H_n, Y(E/\KK_{\infty})^{\circ})$ is a finitely generated $\Lambda_{H,n}$-torsion module. To simplify notation, we will denote $Y(E/\KK_{\infty})$ simply by $Y_{\infty}$ and $Y(E/\KK_{\infty})^{\circ}$ by $Y_{\infty}^{\circ}$.

According to Lemma~\ref{Yinf_torsion_lemma} $Y_{\infty}$ is a finitely generated $\Lambda_2$-torsion module. Therefore  by \cite[Chapt.~VII, \S 4.4 Theorem~5]{Bourbaki}, there exist irreducible power series ${f_j \in \Zp[[T_1, T_2]]}$, integers $m_i, n_j$ and an exact sequence

\begin{equation}\label{fineSelmer_seq1}
0 \to Y_{\infty}^{\circ} \to Y_{\infty} \to W,
\end{equation}
where ${W=\bigoplus_{i=1}^s \Lambda_2/p^{m_i} \oplus \bigoplus_{j=1}^t \Lambda_2/f_j^{n_j}}$. Since ${cd_p(H_n)=1}$, the functor $H_1(H_n, -)$ is left exact. Therefore from the above we get an exact sequence

\begin{equation}\label{fineSelmer_seq2}
0 \to H_1(H_n, Y_{\infty}^{\circ}) \to H_1(H_n, Y_{\infty}) \to H_1(H_n, W).
\end{equation}

Since ${Y_{\infty,f} = (Y_\infty)_f}$ is finitely generated over $\Lambda(H)$, it follows from Lemma~\ref{WL_lemma} that $Y(E/\KK_{\infty}^{H_n})$ is a torsion $\Lambda_{H,n}$-module. This implies by Proposition~\ref{fineSelmercontrol_prop1} that $H_0(H_n, Y_{\infty})$ is a torsion $\Lambda_{H,n}$-module.

Recall that ${H=\overbar{\langle \sigma^a\tau^b \rangle}}$ and ${\Upsilon_n=(1+T_1)^{p^na}(1+T_2)^{p^nb}-1}$.  Since $H_0(H_n, Y_{\infty})$ is a torsion $\Lambda_{H,n}$-module, an identical proof to the equivalence of $(b)$ and $(c)$ in Lemma~\ref{rankequal_lemma} shows that $\Upsilon_n$ is relatively prime to the character power series $\tilde{f}$ of $Y_{\infty}$. Thus we get ${H_1(H_n, W)=0}$. Therefore from the exact sequence~\eqref{fineSelmer_seq2} we get that ${H_1(H_n, Y_{\infty})=H_1(H_n, Y_{\infty}^{\circ})}$.

Since $(Y_{\infty})^{\circ}$ is a pseudo-null $\Lambda_2$-module, it has Krull dimension at most one. Therefore both $H_1(H_n, Y_{\infty}^{\circ})$ and $H_0(H_n, Y_{\infty}^{\circ})$ have Krull dimension at most one as a $\Lambda(G/H_n)$-module. By the same arguments as in the proof of the equivalence of $(b)$ and $(c)$ in Lemma~\ref{rankequal_lemma} we see that $H_1(H_n, Y_{\infty}^{\circ})$ and $H_0(H_n, Y_{\infty}^{\circ})$ also have Krull dimension at most one as a $\Lambda_{H,n}$-module. Thus ${H_1(H_n, Y_{\infty})=H_1(H_n, Y_{\infty}^{\circ})}$ and $H_0(H_n, Y_{\infty})$ are $\Lambda_{H,n}$-torsion. Also by Lemma~\ref{mu_pseudo-null_lemma} below we have that ${\mu_{G_n/H_n}(H_1(H_n, Y_{\infty}^{\circ}))=\mu_{G_n/H_n}(H_0(H_n, Y_{\infty}^{\circ}))}$. Thus we have shown that each of the groups $H_1(H_n, Y(E/\KK_{\infty}))$, $H_1(H_n, Y(E/\KK_{\infty})^{\circ})$, $H_0(H_n, Y(E/\KK_{\infty})^{\circ})$ is a finitely generated $\Lambda_{H,n}$-torsion module and that their $\mu$-invariants are equal. Therefore it only remains to show that ${\mu_{G_n/H_n}(H_1(H_n, Y_{\infty})) = \mu_{G_n/H_n}(T_{\Lambda_{H,n}}([F_{\Lambda_2}(X(E/\KK_{\infty}))]_{H_n}))}$.

Recall from \cite[Corollary~1.7]{Howson} that if $M$ is a finitely generated $\Lambda(G_n)$-module, then
$$p^{\mu_{G_n}(M)}=\prod_{i \geq 0} \# H_i(G_n, M[p^{\infty}])^{(-1)^i} = \chi(G_n, M[p^\infty]). $$\\
We have a similar formula when $M$ is a finitely generated $\Lambda_{H,n}$-module. From the Hochschild-Serre spectral sequence it is easy to prove that
$$\chi(G_n, Y_{\infty}[p^{\infty}])=\prod_{i=0,1}\chi(G_n/H_n, H_i(H_n, Y_{\infty}[p^{\infty}]))^{(-1)^i}. $$\\
From this we get
\begin{equation}\label{mu_invariant_equality1}
\mu_{G_n}(Y_{\infty})=\mu_{G_n/H_n}(H_0(H_n, Y_{\infty}[p^{\infty}]))-\mu_{G_n/H_n}(H_1(H_n, Y_{\infty}[p^{\infty}])).
\end{equation}\\
Consider the exact sequence
\begin{equation}\label{Yinf_exactseq}
0 \to Y_{\infty}[p^{\infty}] \to Y_{\infty} \to Y_{\infty, f} \to 0.
\end{equation}\\
Since ${cd_p(H)=1}$, the functor $H_1(H,-)$ is left exact. Thus from \eqref{Yinf_exactseq} we get a long exact sequence
\begin{align}
0 &\to H_1(H_n, Y_{\infty}[p^{\infty}]) \to H_1(H_n, Y_{\infty}) \to H_1(H_n, Y_{\infty, f}) \nonumber \\
&\to H_0(H_n, Y_{\infty}[p^{\infty}]) \to H_0(H_n, Y_{\infty}) \to H_0(H_n, Y_{\infty, f}) \to 0.
\end{align}

We claim that all the groups in this exact sequence are $\Lambda_{H,n}$-torsion. Clearly $H_0(H_n, Y_{\infty}[p^{\infty}])$ is $\Lambda_{H,n}$-torsion and we showed above that $H_1(H_n, Y_{\infty})$ is $\Lambda_{H,n}$-torsion. Since $Y_{\infty,f}$ is finitely generated over $\Lambda(H)$ and $H_n$ has finite index in $H$, we have that $H_0(H_n, Y_{\infty,f})$ is finitely generated over $\Zp$ and hence is $\Lambda_{H,n}$-torsion. It follows that all groups appearing in the exact sequence above are $\Lambda_{H,n}$-torsion, so the alternating sum of their $\mu$-invariants is zero. Thus noting the equality~\eqref{mu_invariant_equality1} we get that $\mu_{G_n}(Y_{\infty})$ is equal to
\begin{equation}
\begin{tikzcd}[scale cd=0.83]
	\mu_{G_n/H_n}( H_0(H_n, Y_{\infty}))-\mu_{G_n/H_n}(H_1(H_n, Y_{\infty}))+ \mu_{G_n/H_n}(H_1(H_n, Y_{\infty, f}))- \mu_{G_n/H_n}(H_0(H_n, Y_{\infty, f})).
\end{tikzcd}
\end{equation}

As noted above $H_0(H_n, Y_{\infty,f})$ is finitely generated over $\Zp$ and hence has $\mu$-invariant zero. Moreover, as $Y_{\infty,f}$ has no elements of order $p$ and $H_1(H_n, Y_{\infty, f})$ is a submodule of $Y_{\infty,f}$, we have that ${\mu_{G_n/H_n}(H_1(H_n, Y_{\infty, f}))=0}$. Thus we conclude that
\begin{equation}\label{mu_invariant_equality2}
\mu_{G_n}(Y_{\infty})=\mu_{G_n/H_n}(H_0(H_n, Y_{\infty}))-\mu_{G_n/H_n}(H_1(H_n, Y_{\infty})).
\end{equation}
Since $\Lambda(G)$ is a free $\Lambda(G_n)$-module of rank $p^n$, we have ${p^n\mu_G(Y_{\infty})=\mu_{G_n}(Y_{\infty})}$. By Proposition~\ref{fineSelmercontrol_prop1} we have ${\mu_{G_n/H_n}(H_0(H_n, Y_{\infty}))=\mu_{G_n/H_n}(Y(E/\KK_\infty^{H_n}))}$. Thus \eqref{mu_invariant_equality2} becomes
\begin{equation}\label{mu_invariant_equality3}
p^n\mu_G(Y_{\infty})=\mu_{G_n/H_n}(Y(E/\KK_\infty^{H_n}))-\mu_{G_n/H_n}(H_1(H_n, Y_{\infty})).
\end{equation}
Comparing this with Theorem~\ref{main_theorem2}$(e)$ we get ${\mu_{G_n/H_n}(H_1(H_n, Y_{\infty})) = \mu_{G_n/H_n}(T_{\Lambda_{H,n}}([F_{\Lambda_2}(X(E/\KK_{\infty}))]_{H_n}))}$ as desired.
\end{proof}

We end this section with the following lemma which was used in the proof of Theorem~\ref{mu-invariants_theorem}. It will also be used in the proof of Lemma~\ref{mu_inequality_lemma}.

\begin{lemma}\label{mu_pseudo-null_lemma}
Let $A$ be a pseudo-null $\Lambda_2$-module, ${[(a,b)] \in \mathbb{P}^1(\Zp)}$, ${H=\overbar{\langle \sigma^a\tau^b \rangle}}$, ${\Upsilon_n=(1+T_1)^{p^na}(1+T_2)^{p^nb}-1}$ and ${n \geq 0}$. Then ${\mu_{G_n/H_n}(A^{\Upsilon_n=0})=\mu_{G_n/H_n}(A/\Upsilon_n)}$.
\end{lemma}
\begin{proof}
Let ${A'=A/A[p^{\infty}]}$. Since $A$ is a pseudo-null $\Lambda_2$-module, it has Krull dimension at most one. Therefore $A'$ also has dimension at most one. Since multiplication by $p$ is injective on $A'$, it follows from \cite[Corollary~11.9]{AM} that $A'/p$ has dimension zero and hence is finite. Therefore $A'$ is finitely generated over $\Zp$. Now consider the exact sequence
$$0 \to A[p^{\infty}] \to A \to A' \to 0. $$
From this sequence we get the following exact sequence

$$0 \to A[p^{\infty}]^{\Upsilon_n=0} \to A^{\Upsilon_n=0} \to A'^{\Upsilon_n=0} \to A[p^{\infty}]/\Upsilon_n \to A/\Upsilon_n \to A'/\Upsilon_n \to 0. $$
Noting that the alternating sum of the $\mu$-invariants of the terms in this exact sequence of finitely generated torsion $\Lambda(G_n/H_n)$-modules is zero, we get that \begin{eqnarray*} \mu_{G_n/H_n}(A^{\Upsilon_n=0}) & - & \mu_{G_n/H_n}(A/\Upsilon_n) \end{eqnarray*}
equals
\begin{eqnarray*} 
\mu_{G_n/H_n}(A[p^{\infty}]^{\Upsilon_n=0})-\mu_{G_n/H_n}(A[p^{\infty}]/\Upsilon_n) + \mu_{G_n/H_n}(A'^{\Upsilon_n=0})-\mu_{G_n/H_n}(A'/\Upsilon_n).\end{eqnarray*}

As $A'$ is finitely generated over $\Zp$, we may conclude that ${\mu_{G_n/H_n}(A'^{\Upsilon_n=0})=\mu_{G_n/H_n}(A'/\Upsilon_n)=0}$, so we see from the above that it suffices to prove the desired result of the lemma for $A[p^{\infty}]$. Replacing $A$ with $A[p^{\infty}]$, we may assume that $A$ is $p$-primary.

Since $A$ is $p$-primary and finitely generated over the Noetherian ring $\Lambda_2$, it is annihilated by $p^m$ for some $m$. For any $k$, we have an exact sequence

$$0 \to A[p^{k-1}] \to A[p^k] \xrightarrow{\cdot p^{k-1}} p^{k-1}(A[p^k]) \to 0$$\\
Denoting $p^{k-1}(A[p^k])$ by $D$, we get from the above sequence an exact sequence

$$\resizebox{.96\hsize}{!}{$0 \to A[p^{k-1}]^{\Upsilon_n=0} \to A[p^k]^{\Upsilon_n=0} \to D^{\Upsilon_n=0} \to A[p^{k-1}]/\Upsilon_n \to A[p^k]/\Upsilon_n \to D/\Upsilon_n \to 0.$}$$

All of the terms in this exact sequence are finitely generated torsion $\Lambda(G_n/H_n)$-modules, so the alternating sum of their $\mu$-invariants are zero, i.e.

\begin{eqnarray*} \mu_{G_n/H_n}(A[p^k]^{\Upsilon_n=0}) & - & \mu_{G_n/H_n}(A[p^k]/\Upsilon_n) \end{eqnarray*}
is equal to
\begin{eqnarray*} 
\mu_{G_n/H_n}(A[p^{k-1}]^{\Upsilon_n=0})-\mu_{G_n/H_n}(A[p^{k-1}]/\Upsilon_n) + \mu_{G_n/H_n}(D^{\Upsilon_n=0})-\mu_{G_n/H_n}(D/\Upsilon_n).
\end{eqnarray*}
Since $D$ is annihilated by $p$, the above equality together with a devissage argument shows that we may assume that $A$ is annihilated by $p$. We will denote $G_{H,n}$ by $G_n$. Let ${\Omega_2=\Lambda(G)/p}$ and ${\Omega_{2,n}=\Lambda(G_n)/p}$. Since ${G_n \cong H_n \times G_n/H_n}$, we may choose an isomorphism ${\Omega_{2,n} \cong \Fp[[\dot{U}, \dot{T}]]}$ where $\Upsilon_n$ corresponds to $\dot{T}$. Hence, we must show ${\mu_{G_n/H_n}(A^{\dot{T}=0})=\mu_{G_n/H_n}(A/\dot{T})}$. Now $A$ is finitely generated over $\Omega_2$. Moreover, $\Omega_2$ has Krull dimension two, whereas $A$ as a $\Lambda_2$- or $\Omega_2$-module has Krull dimension at most one because it is a pseudo-null $\Lambda_2$-module.

Let $\text{Ann}_{\Omega_2}(A)$ (resp. $\text{Ann}_{\Omega_{2,n}}(A)$) be the annihilator ideal of $A$ inside $\Omega_2$ (resp. $\Omega_{2,n}$). Since $A$ has Krull dimension at most one as a $\Omega_2$-module, we have ${\dim(\Omega_2/\text{Ann}_{\Omega_2}(A))\leq 1}$. Since ${[G:G_n]=p^n}$, the ring $\Omega_2$ is a free $\Omega_{2,n}$-module of rank $p^n$. It follows from this that $\Omega_2/\text{Ann}_{\Omega_2}(A)$ is finitely generated over its subring $\Omega_{2,n}/\text{Ann}_{\Omega_{2,n}}(A)$ and hence $\Omega_2/\text{Ann}_{\Omega_2}(A)$ is integral over  $\Omega_{2,n}/\text{Ann}_{\Omega_{2,n}}(A)$. This implies by \cite[Theorem~2.2.5]{HS} that ${\dim(\Omega_{2,n}/\text{Ann}_{\Omega_{2,n}}(A))=\dim(\Omega_{2,n}/\text{Ann}_{\Omega_{2,n}}(A)) \leq 1}$. Therefore $A$ is also a finitely generated torsion $\Omega_{2,n}$-module.

Taking \cite[Lemma~2.1]{MHG} into account, we have that there exist irreducible power series ${h_i \in \Fp[[\dot{S}, \dot{T}]]}$, integers $d_i$ and an exact sequence

$$0 \to E \to A \to C \to 0$$\\
where ${E=\bigoplus_{i=1}^{w} \Omega_{2,n}/{h_i}^{d_i}}$ and $C$ is finite. From this sequence we get an exact sequence

\begin{equation}
0 \to E^{\dot{T}=0} \to A^{\dot{T}=0} \to C^{\dot{T}=0} \to E/{\dot{T}} \to A/\dot{T} \to C/\dot{T}.
\end{equation}\\
Since $C$ is finite, it follows from the above exact sequence that ${\mu_{G_n/H_n}(A^{\dot{T}=0}) = \mu_{G_n/H_n}(E^{\dot{T}=0})}$ and ${\mu_{G_n/H_n}(A/\dot{T})=\mu_{G_n/H_n}(E/\dot{T})}$. It is easy to see that both ${\mu_{G_n/H_n}(E^{\dot{T}=0})}$ and $\mu_{G_n/H_n}(E/\dot{T})$ are equal to the number of summands of $E$ with ${h_i=u\dot{T}}$ for some unit ${u \in \Omega_{2,n}}$. Therefore ${\mu_{G_n/H_n}(A^{\dot{T}=0})=\mu_{G_n/H_n}(A/\dot{T})}$ as desired. This completes the proof.
\end{proof}

\section{Surjectivity of the global-to-local map for signed Selmer groups} \label{section:pseudo-null}

In this section, we prove the surjectivity of the global-to-local map defining the signed Selmer group over both $\KK_\infty$ and $\KK_\infty^{H}$, ${H \in \mathcal{H}^{\bullet\star}}$. Along the way we collect important auxiliary results on the cohomology of the Selmer groups.

We begin the section by proving an Euler characteristic formula for global Galois cohomology groups. The setup is as follows: Let $p$ be any rational prime, $L$ a number field, and $L_{\infty}/L$ a $\Zp$-extension with tower fields $L_n$. Let $V$ be a representation space over $\Qp$ for $G_L$ (the absolute Galois group of $L$) of dimension $d$, $T$ a $G_L$-invariant lattice, and ${A=V/T\cong(\Qp/\Zp)^d}$. Assume that $T$ is unramified outside a finite number of primes of $L$. Now let $\Sigma$ be a finite set of places of $K$ containing the primes where $T$ is ramified, the primes above $p$, and all the infinite places. Let $L_{\Sigma}$ denote the maximal extension of $L$ unramified outside $\Sigma$. As in Section~\ref{section:mainintroduction}, if $F$ is a field with ${L \subseteq F \subseteq L_{\Sigma}}$, we let ${G_{\Sigma}(F)=\Gal(L_{\Sigma}/F)}$. Finally, let $\Lambda$ be the Iwasawa algebra attached to $L_{\infty}/L$, and for each real place $v$ of $L$ let $d_v^{\pm}$ be the dimension of the $\pm$-eigenspace for complex conjugation above $v$ acting on $V$. The next proposition proves a formula that relates the $\Lambda$-coranks of $H^1(G_{\Sigma}(L_{\infty}), A)$ and $H^2(G_{\Sigma}(L_{\infty}), A)$ and a formula that relates the $\Lambda$-ranks of $\ilim_n H^1(G_{\Sigma}(L_n),T)$ and $\ilim_n H^2(G_{\Sigma}(L_n),T)$, where the inverse limits of the former two groups are defined with respect to corestriction.

\begin{proposition}\label{Euler_char_prop}
Assume that $p$ is odd and ${H^0(G_{\Sigma}(L_{\infty}), A)=0}$. Then
$$\corank_{\Lambda}(H^1(G_{\Sigma}(L_{\infty}), A))-\corank_{\Lambda}(H^2(G_{\Sigma}(L_{\infty}), A))=\delta(L,V), $$
$$\rank_{\Lambda}(\ilim_n H^1(G_{\Sigma}(L_n),T))-\rank_{\Lambda}(\ilim_n H^2(G_{\Sigma}(L_n),T))=\delta(L,V),$$
where ${\delta(L,V)=\underset{v \, \text{complex}}{\Sigma} d + \underset{v\, \text{real}}{\Sigma} d_v^-}$.
\end{proposition}
\begin{proof}
The first formula is proven in \cite[Prop.~3]{Gb_IPR}. To prove the second formula, we therefore only need to show that
\begin{equation}\label{rank_equality}
\corank_{\Lambda}(H^i(G_{\Sigma}(L_{\infty}), A))=\rank_{\Lambda}(\ilim_n H^1(G_{\Sigma}(L_n),T)) \,\, \text{for} \,\, i=1,2.
\end{equation}
To show this, we use Jannsen's spectral sequence \cite{Jannsen}. If $M$ is a finitely generated $\Lambda$-module, we let ${E^i(M)=\text{Ext}_{\Lambda}^i(M, \Lambda)}$. Hence ${E^0(M)=\Hom_{\Lambda}(M, \Lambda):=M^+}$ is just the $\Lambda$-dual of $M$. Note that $M^+$ is a free $\Lambda$-module of the same rank as $M$. By \cite[Corollary~3]{Jannsen} we have an isomorphism
$$\ilim_n H^1(G_{\Sigma}(L_n),T) \cong (H^1(G_{\Sigma}(L_{\infty}), A)^{\vee})^+. $$
From this we get \eqref{rank_equality} for ${i=1}$. By loc. cit. we also have an exact sequence
\begin{align*}
0 \longrightarrow E^1(H^1(G_{\Sigma}(L_{\infty}), A)^{\vee}) \longrightarrow \ilim_n H^2(G_{\Sigma}(L_n),T) &\longrightarrow (H^2(G_{\Sigma}(L_{\infty}), A)^{\vee})^+ \\
\longrightarrow E^2(H^1(G_{\Sigma}(L_{\infty}), A)^{\vee}) \longrightarrow 0.
\end{align*}
By \cite[Prop.~5.5.3]{NSW} the first and last terms in the exact sequence above are $\Lambda$-torsion. Therefore from this we get \eqref{rank_equality} for ${i=2}$.

\end{proof}

If $\fG$ is a profinite group, $\fH$ a closed subgroup of $\fG$ and $A$ a discrete $p$-primary $\fH$-module we write ${\coind_{\fH}^{\fG}(A):=\Hom_{\Zp[[\fH]]}(\Zp[[\fG]], A)}$ for the coinduced $\fG$-module. We will need the following basic lemma.

\begin{lemma}\label{coind_lemma}
Let $\fG$ be a pro-$p$ $p$-adic Lie group which contains no element of order $p$ and let $\fH$ be a closed subgroup of $\fG$. Let $A$ be a discrete $p$-primary $\fH$-module such that $A^{\vee}$ is a finitely generated $\Zp[[\fH]]$-module. Then ${\coind_{\fH}^{\fG}(A)^{\vee}}$ is a finitely generated $\Zp[[\fG]]$-module and  ${\rank_{\Zp[[\fH]]}(A^{\vee})=\rank_{\Zp[[\fG]]}(\coind_{\fH}^{\fG}(A))^{\vee}}$. Moreover, if $A^{\vee}$ is a free $\Zp[[\fH]]$-module, then ${\coind_{\fH}^{\fG}(A)^{\vee}}$ is a free $\Zp[[\fG]]$-module.
\end{lemma}
\begin{proof}
First we show that $\coind_{\fH}^{\fG}(A)^{\vee}$ is a finitely generated $\Zp[[\fG]]$-module. Applying Shapiro's Lemma we have
$$H_0(\fG, \coind_{\fH}^{\fG}(A)^{\vee})\cong H^0(\fG,\coind_{\fH}^{\fG}(A))\cong H^0(\fH, A)\cong H_0(\fH, A^{\vee}).$$
Therefore by Nakayama's Lemma $\coind_{\fH}^{\fG}(A)^{\vee}$ is a finitely generated $\Zp[[\fG]]$-module.
We have by \cite[Theorem~1.1]{Howson} that
$$\rank_{\Zp[[\fG]]}(\coind_{\fH}^{\fG}(A)^{\vee})=\sum_{i \geq 0} (-1)^i\corank_{\Zp}(H^i(\fG, \coind_{\fH}^{\fG}(A))). $$
Applying Shapiro's Lemma the right hand side becomes
$$\sum_{i \geq 0} (-1)^i\corank_{\Zp}(H^i(\fH, A))$$
which is $\rank_{\Zp[[\fH]]}(A^{\vee})$ by loc. cit. again.

By applying the $\Hom$-tensor adjunction with regard to the complete tensor product (\cite[Proposition~5.5.4(c)]{RZ}) we see that we have an isomorphism ${\coind_{\fH}^{\fG}(A)^{\vee} \cong A^{\vee} \, \widehat{\otimes}_{\Zp[[\fH]]} \, \Zp[[\fG]]}$.  Therefore if $A^{\vee}$ is a free $\Zp[[\fH]]$-module we see that $\coind_{\fH}^{\fG}(A)^{\vee}$ is a free $\Zp[[\fG]]$-module.

\end{proof}

Let ${\bullet, \star \in \{+,-\}}$ and ${H \in \mathcal{H}^{\bullet\star}}$. For any $n$, consider the sequence
\begin{align} \begin{tikzcd}[scale cd=0.86] \label{eq:definingSelbulletstar}
\displaystyle 0 \longrightarrow \Selpm(E/\KK_{\infty}^{H_n}) \longrightarrow H^1(G_S(\KK_{\infty}^{H_n}), E[p^{\infty}]) \stackrel{\lambda_{H_n}^{\bullet\star}}{\longrightarrow} J_p^{\bullet\star}(E/\KK_{\infty}^{H_n}) \times \bigoplus_{v \in S, v \nmid p} J_v(E/\KK_{\infty}^{H_n}) \longrightarrow 0, \end{tikzcd} \end{align}
where as in Section~\ref{section:mainintroduction} we define
$$J_p^{\bullet\star}(E/\KK_{\infty}^{H_n}) = \bigoplus_{w \mid \fp} \frac{H^1((\KK_{\infty}^{H_n})_w, E[p^\infty])}{\hat{E}^{\bullet}((\KK_{\infty}^{H_n})_w) \otimes \Q_p/\Z_p} \times \bigoplus_{w \mid \bar{\fp}} \frac{H^1((\KK_{\infty}^{H_n})_w, E[p^\infty])}{\hat{E}^{\star}((\KK_{\infty}^{H_n})_w) \otimes \Q_p/\Z_p},$$
$$J_v(E/\KK_{\infty}^{H_n}) = \bigoplus_{w \mid v} H^1((\KK_{\infty}^{H_n})_w, E[p^\infty]) \quad \text{for}\, \,v \nmid p. $$
In the above definition the sum runs over all primes of $w$ above $v$.
	
\begin{proposition}\label{Seln_surjective_prop} 	
The above sequence~\eqref{eq:definingSelbulletstar} is exact for ${n=0}$, and we have that ${H^2(G_S(\KK_{\infty}^H), E[p^{\infty}])=0}$.
Furthermore, if $X^{\bullet\star}(E/\KK_{\infty})_f$ is finitely generated over $\Lambda(H)$, then for all $n$
\begin{enumerate}
\item $X^{\bullet\star}(E/\KK_{\infty}^{H_n})$ is a torsion $\Lambda_{H,n}$-module.
\item The sequence~\eqref{eq:definingSelbulletstar} is exact.
\item $H^2(G_S(\KK_{\infty}^{H_n}), E[p^{\infty}])=0$.
\end{enumerate}
\end{proposition}
\begin{proof}
	Let ${n \geq 0}$. Assume that $X^{\bullet\star}(E/\KK_{\infty}^{H_n})$ is a torsion $\Lambda_{H,n}$-module. We now show that sequence~\eqref{eq:definingSelbulletstar} is exact and that ${H^2(G_S(\KK_{\infty}^{H_n}), E[p^{\infty}])=0}$. The result for ${n=0}$ will follow because by definition of $\mathcal{H}^{\bullet\star}$ we have that $X^{\bullet\star}(E/\KK_{\infty}^H)$ is a torsion $\Lambda_{H,0}$-module. To prove the exactness of the sequence (i.e. the surjectivity of the second map), we follow the general strategy of the proof of \cite[Prop.~4.4]{LeiSuj} while making many necessary modifications as we go along.

We will need to fix an Iwasawa algebra $\Lambda$ and compare $\Lambda$-coranks of various global and local cohomology groups. When ${n=0}$ and both $\fp$ and $\bar{\fp}$ are totally ramified in $\KK_{\infty}^{H_n}$, then we could simply take our Iwasawa algebra $\Lambda$ to be $\Lambda_{H,0}$. This will simplify the proof. If ${n > 0}$ or at least one of the primes $\fp$ and $\bar{\fp}$ is not totally ramified in $\KK_{\infty}^{H_n}$, then our choice of $\Lambda$ will make the proof much longer.  Let $w$ (resp. $\bar{w}$) be a prime of $\KK_{\infty}^{H_n}$ above $\fp$ (resp. $\bar{\fp}$). Recall from Section~\ref{section:mainintroduction} that $\hat{E}^{\pm}((\KK_{\infty}^{H_n})_w)$ (resp. $\hat{E}^{\pm}((\KK_{\infty}^{H_n})_{\bar{w}})$) is defined as in \eqref{plusdef}, \eqref{minusdef} and \eqref{eq:def2-3} with respect to the $\Zp$-extension $\KK_{\infty}^{H_n}/L^{(\fp)}_{i(\KK_{\infty}^H, \fp)+n}$ (resp. $\KK_{\infty}^{H_n}/L^{(\bar{\fp})}_{i(\KK_{\infty}^H, \bar{\fp})+n}$). Note that any prime of $L^{(\fp)}_{i(\KK_{\infty}^H, \fp)+n}$ below $w$ (resp. $L^{(\bar{\fp})}_{i(\KK_{\infty}^H, \bar{\fp})+n}$ below $\bar{w}$) is totally ramified in $\KK_{\infty}^{H_n}/L^{(\fp)}_{i(\KK_{\infty}^H, \fp)+n}$ (resp. $\KK_{\infty}^{H_n}/L^{(\bar{\fp})}_{i(\KK_{\infty}^H, \bar{\fp})+n}$). Let $\{\KK_{(m, \fp)}^{H_n}\}$ (resp.  $\{\KK_{(m, \bar{\fp})}^{H_n}\}$) be the tower fields of the $\Zp$-extension $\KK_{\infty}^{H_n}/L^{(\fp)}_{i(\KK_{\infty}^H, \fp)+n}$ (resp. $\KK_{\infty}^{H_n}/L^{(\bar{\fp})}_{i(\KK_{\infty}^H, \bar{\fp})+n}$). For some $s \geq 0$ the field $\KK_{(s, \fp)}^{H_n}$ contains $L^{(\bar{\fp})}_{i(\KK_{\infty}^H, \bar{\fp})+n}$. Therefore it follows that for some ${t \geq 0}$ we have ${\KK_{(s, \fp)}^{H_n}=\KK_{(t, \bar{\fp})}^{H_n}}$. Let ${d=[\KK_{(s, \fp)}^{H_n}:K]=[\KK_{(t, \bar{\fp})}^{H_n}:K]}$. We now fix $\Lambda$ to be the Iwasawa algebra ${\Zp[[\Gal(\KK_{\infty}^{H_n}/\KK_{(s, \fp)}^{H_n})]]=\Zp[[\Gal(\KK_{\infty}^{H_n}/\KK_{(t, \bar{\fp})}^{H_n})]]}$. We now study $\Lambda$-ranks of various objects related to $X^{\bullet\star}(E/\KK_{\infty}^{H_n})$. We will use the following fact extensively below (see \cite[Corollary~1.5]{Howson} for a proof): Let ${\tilde{G} \cong \Zp}$ and let $\tilde{N}$ be an open subgroup of $\tilde{G}$. Then $\Zp[[\tilde{G}]]$ is free over $\Zp[[\tilde{N}]]$ of rank $[\tilde{G}:\tilde{N}]$, and the rank of a finitely generated $\Zp[[\tilde{G}]]$-module over $\Zp[[\tilde{N}]]$ equals $[\tilde{G}:\tilde{N}]$ times its rank over $\Zp[[\tilde{G}]]$. In particular, a $\Zp[[\tilde{G}]]$-module is $\Zp[[\tilde{G}]]$-torsion if and only if it is $\Zp[[\tilde{N}]]$-torsion. From Proposition~\ref{pm_points_structure_prop} we have

\begin{align}\label{first_pm_isom}
\Big(\hat{E}^{\pm}((\KK_{\infty}^{H_n})_w)\otimes \Qp/\Zp\Big)^{\vee} &\cong \Zp[[\Gal(\KK_{\infty}^{H_n}/L^{(\fp)}_{i(\KK_{\infty}^H,  \fp)+n})]]^{[(L^{(\fp)}_{i(\KK_{\infty}^H, \fp)+n})_w:\Qp]}\nonumber \\
&\cong \Lambda^{[(L^{(\fp)}_{i(\KK_{\infty}^H, \fp)+n})_w:\Qp][(\KK_{(s, \fp)}^{H_n})_w:(L^{(\fp)}_{i(\KK_{\infty}^H, \fp)+n})_w]} \nonumber \\
&= \Lambda^{[(\KK_{(s, \fp)}^{H_n})_w:\Qp]},
\end{align}

\begin{align}\label{second_pm_isom}
\Big(\hat{E}^{\pm}((\KK_{\infty}^{H_n})_{\bar{w}})\otimes \Qp/\Zp\Big)^{\vee} &\cong \Zp[[\Gal(\KK_{\infty}^{H_n}/L^{(\bar{\fp})}_{i(\KK_{\infty}^H,  \bar{\fp})+n})]]^{[(L^{(\bar{\fp})}_{i(\KK_{\infty}^H, \bar{\fp})+n})_{\bar{w}}:\Qp]}\nonumber \\
&\cong \Lambda^{[(L^{(\bar{\fp})}_{i(\KK_{\infty}^H, \bar{\fp})+n})_{\bar{w}}:\Qp][(\KK_{(t, \bar{\fp})}^{H_n})_{\bar{w}}:(L^{({\bar{\fp}})}_{i(\KK_{\infty}^H, {\bar{\fp}})+n})_{\bar{w}}]} \nonumber \\
&= \Lambda^{[(\KK_{(t, \bar{\fp})}^{H_n})_{\bar{w}}:\Qp]}\nonumber \\
&= \Lambda^{[(\KK_{(s, \fp)}^{H_n})_{\bar{w}}:\Qp]}.
\end{align}

Let ${T=T_p(E)}$ be the $p$-adic Tate module of $E$. For ${i=1,2}$ define
$$H^i_{\text{Iw}}(G_S(\KK_{\infty}^{H_n}), T)=\ilim_m H^i(G_S(\KK_{(m, \fp)}^{H_n}), T).$$
For any ${m \geq 0}$ we define $H^1_{\pm}((\KK_{(m,\fp)}^{H_n})_w, T)$ as the exact annihilator of ${\hat{E}^{\pm}((\KK_{(m,\fp)}^{H_n})_w) \otimes \Qp/\Zp}$ with respect to the Tate pairing
$$H^1((\KK_{(m,\fp)}^{H_n})_w, E[p^{\infty}]) \times H^1((\KK_{(m,\fp)}^{H_n})_w, T) \longrightarrow \Qp/\Zp. $$
For any ${m \geq 0}$ we similarly define $H^1_{\pm}((\KK_{(m,\bar{\fp})}^{H_n})_{\bar{w}}, T)$.

We define
$$H^1_{\text{Iw}}((\KK_{\infty}^{H_n})_w, T)=\ilim H^1((\KK_{(m,\fp)}^{H_n})_w, T), $$
$$H^1_{\text{Iw}}((\KK_{\infty}^{H_n})_{\bar{w}}, T)=\ilim H^1((\KK_{(m,\bar{\fp})}^{H_n})_{\bar{w}}, T). $$
Since ${\KK_{(s, \fp)}^{H_n}=\KK_{(t, \bar{\fp})}^{H_n}}$, we can write
$$H^1_{\text{Iw}}((\KK_{\infty}^{H_n})_{\bar{w}}, T)=\ilim_{m \geq s} H^1((\KK_{(m,\fp)}^{H_n})_{\bar{w}}, T). $$

We define
$$H^1_{\text{Iw},\pm}((\KK_{\infty}^{H_n})_w, T)=\ilim H^1_{\pm}((\KK_{(m,\fp)}^{H_n})_w, T), $$
$$H^1_{\text{Iw},\pm}((\KK_{\infty}^{H_n})_{\bar{w}}, T)=\ilim H^1_{\pm}((\KK_{(m,\bar{\fp})}^{H_n})_{\bar{w}}, T). $$
Since ${\KK_{(s, \fp)}^{H_n}=\KK_{(t, \bar{\fp})}^{H_n}}$, we can write
$$H^1_{\text{Iw},\pm}((\KK_{\infty}^{H_n})_{\bar{w}}, T)=\ilim_{m \geq s} H^1_{\pm}((\KK_{(m,\fp)}^{H_n})_{\bar{w}}, T). $$

Consider the Poitou-Tate exact sequence (\cite[Proposition~A.3.2]{PR})
\begin{align}\label{Poitou-Tate_seq}
H^1_{\text{Iw}}(G_S(\KK_{\infty}^{H_n}), T) \stackrel{\text{loc}^{\bullet\star}}{\longrightarrow} Z \longrightarrow X^{\bullet\star}(E/\KK_{\infty}^{H_n}) \longrightarrow H^2_{\text{Iw}}(G_S(\KK_{\infty}^{H_n}), T) \longrightarrow 0,
\end{align}
where
\[ \resizebox{.97\hsize}{!}{$Z = \bigoplus_{w | \fp} H^1_{\text{Iw}}((\KK_{\infty}^{H_n})_w, T)/H^1_{\text{Iw}, \bullet}((\KK_{\infty}^{H_n})_w, T) \times \bigoplus_{w | \bar{\fp}} H^1_{\text{Iw}}((\KK_{\infty}^{H_n})_{\bar{w}}, T)/H^1_{\text{Iw}, \star}((\KK_{\infty}^{H_n})_{\bar{w}}, T).$} \]
By \eqref{first_pm_isom} and \eqref{second_pm_isom} and Tate local duality we have

$$\rank_{\Lambda}(Z)=[\KK_{(s, \fp)}^{H_n}:\Q]. $$

Taking Lemma~\ref{p-torsion_lemma} into account, from Proposition~\ref{Euler_char_prop} we have
$$\rank_{\Lambda}(H^1_{\text{Iw}}(G_S(\KK_{\infty}^{H_n}), T))-\rank_{\Lambda}(H^2_{\text{Iw}}(G_S(\KK_{\infty}^{H_n}), T))=[\KK_{(s, \fp)}^{H_n}:\Q].$$

Therefore from the exact sequence \eqref{Poitou-Tate_seq} we get that
$$\rank_{\Lambda}(\ker \, \text{loc}^{\bullet\star})=\rank_{\Lambda}(X^{\bullet\star}(E/\KK_{\infty}^{H_n})). $$
Since $X^{\bullet\star}(E/\KK_{\infty}^{H_n})$ is a torsion $\Lambda_{H,n}$-module, it is also a torsion $\Lambda$-module and so from the above we get that $\ker \, \text{loc}^{\bullet\star}$ is $\Lambda$-torsion. We now show that $\ker \, \text{loc}^{}$ is $\Lambda$-torsion-free which will imply that ${\ker \, \text{loc}^{\bullet\star}=0}$. To prove that $\ker \, \text{loc}^{\bullet\star}$ is $\Lambda$-torsion-free, it will suffice to show  $H^1_{\text{Iw}}(G_S(\KK_{\infty}^{H_n}), T)$ is $\Lambda$-torsion-free.

For any $m$ and $k$ we have an exact sequence
$$\begin{tikzcd}[scale cd=0.84] 0 \longrightarrow E(\KK_{(m, \fp)}^{H_n}))[p^{\infty}]/p^k \longrightarrow H^1(G_S(\KK_{(m, \fp)}^{H_n}), E[p^k]) \longrightarrow H^1(G_S(\KK_{(m, \fp)}^{H_n}), E[p^{\infty}])[p^k] \longrightarrow 0. \end{tikzcd} $$

By Lemma~\ref{p-torsion_lemma} the first term in the above sequence is zero. Therefore we have an isomorphism ${H^1_{\text{Iw}}(G_S(\KK_{\infty}^{H_n}), T) \cong \ilim_{m,k} H^1(G_S(\KK_{(m, \fp)}^{H_n}), E[p^{\infty}])[p^k]}$, where the inverse limit over $m$ is with respect to corestriction and over $k$ is with respect to the multiplication by $p$ map. Therefore we need to show that $\ilim_{m,k} H^1(G_S(\KK_{(m, \fp)}^{H_n}), E[p^{\infty}])[p^k]$ is $\Lambda$-torsion-free. Let ${\Gamma=\Gal(\KK_{\infty}^{H_n}/\KK_{(s, \fp)}^{H_n})}$. For any ${m \geq 0}$ we let ${\Gamma_m=\Gamma^{p^m}}$. Now consider the exact sequence
$$\begin{tikzcd}[scale cd=0.84] 0 \longrightarrow H^1(\Gamma_m, E(\KK_{\infty}^{H_n})[p^{\infty}]) \longrightarrow H^1(G_S(\KK_{(s+m, \fp)}^{H_n}), E[p^{\infty}]) \longrightarrow H^1(G_S(\KK_{\infty}^{H_n}), E[p^{\infty}])^{\Gamma_m} \longrightarrow 0. \end{tikzcd} $$
The last map is surjective because ${cd_p(\Gamma_m)=1}$. By Lemma~\ref{p-torsion_lemma} the first term in the above sequence is zero. Therefore we have an isomorphism
${\ilim_{m,k} H^1(G_S(\KK_{(m, \fp)}^{H_n}), E[p^{\infty}])[p^k] \cong \ilim_{m,k} H^1(G_S(\KK_{\infty}^{H_n}), E[p^{\infty}])[p^k]^{\Gamma_m}}$. By \cite[Prop.~5.5.10]{NSW} the latter group is $\Lambda$-torsion-free, whence so is the former. As explained above, this proves that we have ${\ker \, \text{loc}^{\bullet\star}=0}$.

Now consider the fine Selmer group $R_{p^{\infty}}(E/\KK_{\infty}^{H_n})$ defined by the exact sequence
	
	$$\displaystyle 0 \longrightarrow R_{p^{\infty}}(E/\KK_{\infty}^{H_n}) \longrightarrow H^1(G_S(\KK_{\infty}^{H_n}), E[p^{\infty}]) \longrightarrow \bigoplus_{w|v, v \in S} H^1(\KK_{\infty}^{H_n}, E[p^{\infty}]). $$
	As $R_{p^{\infty}}(E/\KK_{\infty}^{H_n})$ is contained in $\Selpm(E/\KK_{\infty}^{H_n})$, the dual $X^{\bullet\star}(E/\KK_{\infty}^{H_n})$ surjects onto $R_{p^{\infty}}(E/\KK_{\infty}^{H_n})^{\vee}$ (the Pontryagin dual of $R_{p^{\infty}}(E/\KK_{\infty}^{H_n})$). So $R_{p^{\infty}}(E/\KK_{\infty}^{H_n})^{\vee}$ is a torsion $\Lambda_{H,n}$-module. By \cite[Theorem~2.2]{Matar_Torsion} this implies that ${H^2(G_S(\KK_{\infty}^{H_n}), E[p^{\infty}])=0}$. Therefore we see from the second Poitou-Tate exact sequence given in \cite[Prop.~A.3.2]{PR} that we have an isomorphism
$$(\coker \lambda_{H_n}^{\bullet\star})^{\vee} \cong \ker \, \text{loc}^{\bullet\star}. $$
As $\ker \, \text{loc}^{\bullet\star}$ is trivial, we get that $\lambda_{H_n}^{\bullet\star}$ is surjective as desired.

Suppose that $X^{\bullet\star}(E/\KK_{\infty})_f$ is finitely generated over $\Lambda(H)$. To complete the proof, from what we just observed, it will suffice to show that $X^{\bullet\star}(E/\KK_{\infty}^{H_n})$ is a torsion $\Lambda_{H,n}$-module. To prove this, we can proceed as in \cite[Proposition~2.5]{CS2}. Consider the following commutative diagram with exact rows
	\begin{equation*}
		\xymatrix {
			X^{\bullet\star}(E/\KK_{\infty})_{H_n}  \ar[d]^{\hat{s}_n^{\bullet\star}} \ar[r] & (X^{\bullet\star}(E/\KK_{\infty})_f)_{H_n}  \ar[d]^{\theta_n^{\bullet\star}} \ar[r] &0\\
			X^{\bullet\star}(E/\KK_{\infty}^{H_n}) \ar[r] & X^{\bullet\star}(E/\KK_{\infty}^{H_n})_f \ar[r] &0}
	\end{equation*}
	The map $\hat{s}_n^{\bullet\star}$ is the dual of the map $s_n^{\bullet\star}$ in Theorem~\ref{SelmerControlThm_prop} and $\theta_n^{\bullet\star}$ is induced by $\hat{s}_n^{\bullet\star}$. By Theorem~\ref{SelmerControlThm_prop}, $\coker \hat{s}_n^{\bullet\star}$ is trivial and hence so is $\coker \theta_n^{\bullet\star}$. Since $X^{\bullet\star}(E/\KK_{\infty})_f$ is finitely generated over $\Lambda(H)$ and $H_n$ has finite index in $H$, the quotient $(X^{\bullet\star}(E/\KK_{\infty})_f)_{H_n}$ is finitely generated over $\Zp$. Thus we get that $X^{\bullet\star}(E/\KK_{\infty}^{H_n})_f$ is finitely generated over $\Zp$. This implies that $X^{\bullet\star}(E/\KK_{\infty}^{H_n})$ is a torsion $\Lambda_{H,n}$-module which completes our proof.
\end{proof}

Note that we already know from Lemma~\ref{WL_lemma} that assertion~(3) of Proposition~\ref{Seln_surjective_prop} also holds true if the Pontryagin dual $Y(E/\KK_\infty)_f$ of the fine Selmer group is finitely generated over $\Lambda(H)$. In fact, it is part of Theorem~\ref{equivalence_theorem}, to be proven in a later section, that if $X^{\bullet\star}(E/\KK_\infty)_f$ is finitely generated over $\Lambda(H)$, then $Y(E/\KK_\infty)_f$ will also be finitely generated over $\Lambda(H)$. Therefore the latter condition turns out to be weaker than the one in the above proposition.
\begin{proposition}\label{Seln_surjective_prop2}
Let ${\bullet, \star \in \{+,-\}}$ and ${H \in \mathcal{H}^{\bullet\star}}$. Then the following assertions hold.
\begin{enumerate}[(a)]
\item We have an exact sequence $$ 0 \longrightarrow \Selpm(E/\KK_{\infty}^{H}) \longrightarrow \Selinf(E/\KK_{\infty}^{H}) \longrightarrow J_p^{\bullet\star}(E/\KK_{\infty}^{H}) \longrightarrow 0$$
\item $J_p^{\bullet\star}(E/\KK_{\infty}^{H})^{\vee}$ is a free $\Lambda$-module of rank two.
\end{enumerate}
\end{proposition}
\begin{proof}
Assertion~$(a)$ follows from Proposition~\ref{Seln_surjective_prop}. Now we prove $(b)$. Write ${L_\infty = \KK_\infty^H}$ for simplicity. Define

$$J_{\fp}^{\bullet}(E/L_{\infty}) = \bigoplus_{w \mid \fp} \frac{H^1(L_{\infty,w}, E[p^\infty])}{\hat{E}^{\bullet}(L_{\infty,w}) \otimes \Q_p/\Z_p}, $$

$$J_{\bar{\fp}}^{\star}(E/L_{\infty})=\bigoplus_{w \mid \bar{\fp}} \frac{H^1(L_{\infty,w}, E[p^\infty])}{\hat{E}^{\star}(L_{\infty,w}) \otimes \Q_p/\Z_p}.$$
Then  ${J_p^{\bullet\star}(E/L_\infty)=J_{\fp}^{\bullet}(E/L_{\infty}) \times J_{\bar{\fp}}^{\star}(E/L_{\infty})}$.
Let $w$ be a prime of $L_{\infty}$ over $\fp$. Let $\Gamma_{\fp}$ be the decomposition group of $w$ inside ${\Gamma:=\Gal(L_{\infty}/K)}$ and ${\Lambda_{\fp}=\Zp[[\Gamma_{\fp}]]}$, ${\Lambda:=\Zp[[\Gamma]]}$. We will now show that the Pontryagin dual of ${J_{\fp,w}^{\bullet}(E/L_{\infty}):=\frac{H^1(L_{\infty,w}, E[p^\infty])}{\hat{E}^{\bullet}(L_{\infty,w}) \otimes \Q_p/\Z_p}}$ is a free $\Lambda_{\fp}$-module of rank one. From Proposition~\ref{pm_points_structure_prop} we have that there exists an open subgroup ${\Delta_{\fp} \subseteq \Gamma_{\fp}}$ such that ${(\hat{E}^{\bullet}(L_{\infty,w}) \otimes \Q_p/\Z_p)^{\vee} \cong \Zp[[\Delta_{\fp}]]^{[\Gamma_{\fp}:\Delta_{\fp}]}}$. From this we see that ${(\hat{E}^{\bullet}(L_{\infty,w}) \otimes \Q_p/\Z_p)^{\vee}}$ has no finite nontrivial $\Zp[[\Delta_{\fp}]]$-submodules and hence also has no finite nontrivial $\Lambda_{\fp}$-submodules. Moreover from \cite[Corollary~1.5]{Howson} we see that ${\rank_{\Lambda_{\fp}}((\hat{E}^{\bullet}(L_{\infty,w}) \otimes \Q_p/\Z_p)^{\vee})=1}$. By Lemma~\ref{p-torsion_lemma} we get that ${E(L_{\infty,w})[p^{\infty}]=0}$ and therefore by \cite[Corollary~2]{Gb_IPR} we have that ${H^1(L_{\infty,w}, E[p^{\infty}])^{\vee}}$ is a free $\Lambda_{\fp}$-module of rank two. By combining the previous facts and applying \cite[Lemma~3.31]{KO} we get that ${J_{\fp,w}^{\bullet}(E/L_{\infty})^{\vee} \cong \Lambda_{\fp}}$.

Let ${\Gamma:=\Gal(L_{\infty}/K)}$ and let $w$ be a prime of $L_{\infty}$ above $\fp$. Then ${J_{\fp}^{\bullet}(E/L_{\infty})=\coind^{\Gamma}_{\Gamma_{\fp}} J_{\fp,w}(E/L_{\infty})}$. Since we have ${J_{\fp,w}^{\bullet}(E/L_{\infty})^{\vee} \cong \Lambda_{\fp}}$, it follows from Lemma~\ref{coind_lemma} that  ${J_{\fp}^{\bullet}(E/L_{\infty})^{\vee} \cong \Lambda}$. An identical proof shows that we also have ${J_{\bar{\fp}}^{\star}(E/L_{\infty})^\vee \cong \Lambda}$. Thus we get $(b)$.
\end{proof}

Recall that in Section~\ref{section:control} we defined the module
$$J_p^{\bullet\star}(E/\KK_{\infty}) = \bigoplus_{w | \fp}\frac{H^1(\KK_{\infty,w}, E[p^{\infty}])}{\hat{E}^{\bullet}(\KK_{\infty,w})\otimes \Qp/\Zp} \times  \bigoplus_{w | \bar{\fp}}\frac{H^1(\KK_{\infty,w}, E[p^{\infty}])}{\hat{E}^{\star}(\KK_{\infty,w})\otimes \Qp/\Zp}$$

and for any ${v \in S \setminus S_p}$ we defined $$J_v(E/\KK_{\infty}) = \bigoplus_{w | v}H^1(\KK_{\infty,w}, E)[p^{\infty}].$$
Consider the exact sequence
\begin{equation}\label{Selpm_Lambda2_seq}\begin{tikzcd}[scale cd=0.89]
\displaystyle 0 \longrightarrow \Selpm(E/\KK_{\infty}) \longrightarrow H^1(G_S(\KK_{\infty}), E[p^{\infty}]) \stackrel{\lambda_{\bullet\star}}{\longrightarrow} J_p^{\bullet\star}(E/\KK_{\infty}) \times \bigoplus_{v \in S, v \nmid p} J_v(E/\KK_{\infty}) \longrightarrow 0. \end{tikzcd}
\end{equation}

\begin{proposition}\label{Sel_Lambda2_surjective_prop}
Let ${\bullet, \star \in \{+,-\}}$. If $X^{\bullet\star}(E/\KK_{\infty})$ is a $\Lambda_2$-torsion module, then we have
\end{proposition}
\begin{enumerate}[(a)]
\item The sequence (\ref{Selpm_Lambda2_seq}) is exact.
\item ${H^2(G_S(\KK_{\infty}), E[p^{\infty}])=0}$.
\end{enumerate}
\begin{proof}
See \cite[Proposition~4.12]{Lei-Lim}.
\end{proof}
Using results from Section~\ref{section:auxiliary} we will actually derive in Lemma~\ref{H2_Kinf_vainishing_lemma} below that Assertion~(b) of Proposition~\ref{Sel_Lambda2_surjective_prop} holds unconditionally in our setting.

\begin{proposition}\label{Sel_Lambda2_surjective_prop2}
Let ${\bullet, \star \in \{+,-\}}$. If $X^{\bullet\star}(E/\KK_{\infty})$ is a $\Lambda_2$-torsion module, then we have
\begin{enumerate}[(a)]
\item We have an exact sequence $$ 0 \longrightarrow \Selpm(E/\KK_{\infty}) \longrightarrow \Selinf(E/\KK_{\infty}) \longrightarrow J_p^{\bullet\star}(E/\KK_{\infty}) \longrightarrow 0$$
\item $J_p^{\bullet\star}(E/\KK_{\infty})^{\vee}$ is a free $\Lambda_2$-module of rank two.
\end{enumerate}
\end{proposition}
\begin{proof}
$(a)$ follows from Proposition~\ref{Sel_Lambda2_surjective_prop}. Now we prove $(b)$. Define

$$J_{\fp}^{\bullet}(E/\KK_{\infty}) = \bigoplus_{w \mid \fp} \frac{H^1(\KK_{\infty,w}, E[p^\infty])}{\hat{E}^{\bullet}(\KK_{\infty,w}) \otimes \Q_p/\Z_p}, $$

$$J_{\bar{\fp}}^{\star}(E/\KK_{\infty})=\bigoplus_{w \mid \bar{\fp}} \frac{H^1(\KK_{\infty,w}, E[p^\infty])}{\hat{E}^{\star}(\KK_{\infty,w}) \otimes \Q_p/\Z_p}.$$
Then  ${J_p^{\bullet\star}(E/{\KK_{\infty}})=J_{\fp}^{\bullet}(E/\KK_{\infty}) \times J_{\bar{\fp}}^{\star}(E/\KK_{\infty})}$.

Let $w$ be a prime of $\KK_{\infty}$ over $\fp$. Let $G_{\fp}$ be the decomposition group of $w$ inside ${G=\Gal(\KK_{\infty}/K)}$ and ${\Lambda_{2,\fp}=\Zp[[G_{\fp}]]}$. Let ${J_{\fp,w}^{\bullet}(E/\KK_{\infty}):=\frac{H^1(\KK_{\infty,w}, E[p^\infty])}{\hat{E}^{\bullet}(\KK_{\infty,w}) \otimes \Q_p/\Z_p}}$. By \cite[Corollary~3.9]{Lei-Lim} we have that ${J_{\fp,w}^{\bullet}(E/\KK_{\infty})^{\vee}\cong \Lambda_{2, \fp}}$.

Let $w$ be a prime of $\KK_{\infty}$ above $\fp$. Then ${J_{\fp}^{\bullet}(E/\KK_{\infty})=\coind^{G}_{G_{\fp}} J_{\fp,w}^{\bullet}(E/\KK_{\infty})}$. Since we have ${J_{\fp,w}^{\bullet}(E/\KK_{\infty})^{\vee} \cong \Lambda_{2,\fp}}$, it follows from Lemma~\ref{coind_lemma} that ${J_{\fp}^{\bullet}(E/\KK_{\infty})^{\vee} \cong \Lambda_2}$. An identical proof shows that we also have ${J_{\bar{\fp}}^{\star}(E/\KK_{\infty})^\vee \cong \Lambda_2}$. Thus we get $(b)$.
\end{proof}

For later use, we collect some more results on the cohomology of the Selmer groups.
\begin{lemma}\label{H2_Kinf_vainishing_lemma}
	${H^2(G_S(\KK_{\infty}), E[p^{\infty}])=0}$.
\end{lemma}
\begin{proof}
	This can be proved along the same lines as \cite[Lemma~2.12]{MHG}, using Proposition~\ref{Hcyc_prop} and Proposition~\ref{Seln_surjective_prop}. Alternatively the desired result follows from Proposition~\ref{Hcyc_prop}, Lemma~\ref{Xinf_torsion_lemma} and Proposition ~\ref{Sel_Lambda2_surjective_prop}.
\end{proof}
\begin{lemma}\label{H_i_vanishing_lemma}
	Let ${\bullet, \star \in \{+,-\}}$ and ${H \in \mathcal{H}^{\bullet\star}}$. \\
	We have ${H^i(H, H^1(G_S(\KK_{\infty}), E[p^{\infty}]))=0}$ for all ${i \geq 1}$. Furthermore, if $X^{\bullet\star}(E/\KK_{\infty})_f$ is finitely generated over $\Lambda(H)$, then for any $n$ we have ${H^i(H_n, H^1(G_S(\KK_{\infty}), E[p^{\infty}]))=0}$ for all ${i \geq 1}$.
\end{lemma}
\begin{proof}
	This can be proved along the same lines as \cite[Lemma~2.13]{MHG}, using Proposition~\ref{Seln_surjective_prop} and Lemma~\ref{H2_Kinf_vainishing_lemma}.
\end{proof}
In order to prove Lemma~\ref{H1_Sel_vanishing_lemma}, which is the final aim of this section, we need one more auxiliary result. Recall that for ${H \in \mathcal{H}^{\bullet\star}}$ and ${n \in \N}$ we abbreviate the Iwasawa algebra ${\Lambda(\Gal(\KK_{\infty}^{H_n}/K_{H,n})) = \Z_p[[\Gal(\KK_{\infty}^{H_n}/K_{H,n})]]}$ to $\Lambda_{H,n}$ (these fields have been defined in Section~\ref{section:mainintroduction}, just above the large field diagram). In the following lemma we use similar notation to the one in the statement of Proposition~\ref{Seln_surjective_prop}.

\begin{lemma}\label{Sel_Hn_surjective_lemma}
	Let $\bullet, \star \in \{+,-\}$ and ${H \in \mathcal{H}^{\bullet\star}}$. For any $n$, consider the sequence
	\begin{align}\begin{tikzcd}[scale cd=0.78] \label{eq:definingSel} 0 \longrightarrow \Selpm(E/\KK_{\infty})^{H_n} \longrightarrow H^1(G_S(\KK_{\infty}), E[p^{\infty}])^{H_n} \longrightarrow J_p^{\bullet\star}(E/\KK_{\infty})^{H_n} \times \bigoplus_{v \in S, v \nmid p} J_v(E/\KK_{\infty})^{H_n} \longrightarrow 0. \end{tikzcd} \end{align}
	This sequence is exact for ${n=0}$. If $X^{\bullet\star}(E/\KK_{\infty})_f$ is finitely generated over $\Lambda(H)$, then sequence~\eqref{eq:definingSel} is exact for all $n$.
\end{lemma}
\begin{proof}
	We proceed as in \cite[Lemma~2.3]{CSS}. Let ${H \in \mathcal{H}^{\bullet\star}}$ and ${n \geq 0}$. For simplicity, let ${F_{\infty}^n=\KK_{\infty}^{H_n}}$. Recall that $S_p$ denotes the set of primes in $S$ above $p$. Consider the map induced by restriction
	$$h_n^{\bullet\star}:J_p^{\bullet\star}(E/F_{\infty}^n) \times \bigoplus_{v \in S \setminus S_p} J_v(E/F_{\infty}^n) \longrightarrow J_p^{\bullet\star}(E/\KK_{\infty})^{H_n} \times \bigoplus_{v \in S \setminus S_p} J_v(E/\KK_{\infty})^{H_n}. $$
	We can write ${h_n^{\bullet\star}=\oplus_{w|p} h_{n,w}^{\bullet\star} \times \oplus_{w \nmid p} h_{n,w}}$ where $w$ runs over all primes of $F_{\infty}^n$ above primes in $S$. We claim that $h_n^{\bullet\star}$ is surjective. Let $w$ be a prime of $F_{\infty}^n$ not above $p$. Then from the inflation-restriction sequence $\coker h_{n,w}$ is contained in $H^2(H_{n,w}, E(\KK_{\infty,w})[p^{\infty}])$ where we have also written $w$ for a fixed prime of $\KK_{\infty}$ above $w$ and $H_{n,w}$ is the decomposition group. Since ${cd_p(H_{n,w}) \le 1}$, we see that ${H^2(H_{n,w}, E(\KK_{\infty,w})[p^{\infty}])=0}$, whence $h_{n,w}$ is surjective.
	
	Now let $w$ be a prime of $F_{\infty}^n$ above $p$. Consider the commutative diagram with exact rows where we have also written $w$ for a fixed prime of $\KK_{\infty}$ above $w$ and $H_{n,w}$ is the decomposition group:
	
	\begin{equation*}
		\begin{tikzcd}[column sep = small, scale cd=0.84]
			0 \arrow[r] & \left(E^{\pm}(\KK_{\infty,w})\otimes \Qp/\Zp\right)^{H_{n,w}} \arrow[r] & H^1(\KK_{\infty,w}, E[p^{\infty}])^{H_{n,w}}  \arrow[r] & \left(\frac{H^1(\KK_{\infty,w}, E[p^{\infty}])}{E^{\pm}(\KK_{\infty,w})\otimes \Qp/\Zp}\right)^{H_{n,w}} \arrow[r]& 0\\
			0 \arrow[r] & E^{\pm}(F^n_{\infty,w})\otimes \Qp/\Zp \ar[r] \arrow[u, "\alpha^{\bullet\star}"] &H^1(F^n_{\infty,w}, E[p^{\infty}]) \arrow[u, "\beta"] \arrow[r] & \frac{H^1(F^n_{\infty,w}, E[p^{\infty}])}{E^{\pm}(F^n_{\infty,w})\otimes \Qp/\Zp} \arrow[u, "h_{n,w}^{\bullet\star}"] \arrow[r] &0 \end{tikzcd}
	\end{equation*}
	
	It follows from Lemma~\ref{freeness_lemma} that ${(E^{\pm}(\KK_{\infty,w})\otimes \Qp/\Zp)_{H_{n,w}}=0}$. Therefore the arrow on the right of the top row is in fact surjective. Since ${cd_p(H_{n,w}) \le 1}$, the same argument above shows that $\beta$ is surjective. The above diagram then shows that $h_{n,w}^{\bullet\star}$ is also surjective. Thus we have shown that $h_n^{\bullet\star}$ is surjective.
	
	Now consider the commutative diagram
	\begin{equation*}
		\xymatrix {
			H^1(G_S(\KK_{\infty}), E[p^{\infty}])^{H_n}  \ar[r]^-{\rho_n^{\bullet\star}} & J_p^{\bullet\star}(E/\KK_{\infty})^{H_n} \times \bigoplus_{v \in S \setminus S_p} J_v(E/\KK_{\infty})^{H_n}\\
			H^1(G_S(F_n), E[p^{\infty}]) \ar[u] \ar[r]^-{\lambda_n^{\bullet\star}} & J_p^{\bullet\star}(E/F_{\infty}^n) \times \bigoplus_{v \in S \setminus S_p} J_v(E/F_{\infty}^n) \ar[u]^{h_n^{\bullet\star}}}
	\end{equation*}
	The exactness of sequence~\eqref{eq:definingSel} is equivalent to the surjectivity of $\rho_n^{\bullet\star}$. Since $h_n^{\bullet\star}$ is surjective, the commutative diagram shows that the surjectivity of $\rho_n^{\bullet\star}$ will follow if $\lambda_n^{\bullet\star}$ is surjective. Therefore, the desired result follows from Proposition~\ref{Seln_surjective_prop}.
\end{proof}

\begin{lemma}\label{H1_Sel_vanishing_lemma}
	Let ${\bullet, \star \in \{+,-\}}$ and ${H \in \mathcal{H}^{\bullet\star}}$. We have ${H^1(H, \Selpm(E/\KK_{\infty}))=0}$. If $X^{\bullet\star}(E/\KK_{\infty})_f$ is finitely generated over $\Lambda(H)$, then ${H^1(H_n, \Selpm(E/\KK_{\infty}))=0}$ for all ${n \geq 0}$.
\end{lemma}
\begin{proof}
	This is proved using the same arguments as in the proof of \cite[Lemma~2.14]{MHG}, by applying Lemma~\ref{Sel_Hn_surjective_lemma} and Lemma~\ref{H_i_vanishing_lemma}.
\end{proof}

\section{Pseudo-null submodules and a comparison theorem}
In this section, we prove the non-existence of nontrivial pseudo-null submodules for both the $\Lambda_2$-module $X^{\bullet\star}(E/\KK_{\infty})$ and the $\Lambda_{H,n}$-modules $X^{\bullet\star}(E/\KK_{\infty}^{H_n})$, ${n \in \N}$ and ${H \in \mathcal{H}^{\bullet\star}}$, provided that they are torsion (see Theorem~\ref{pseudo-null_theorem} below). We also prove a comparison theorem relating results about the Selmer group and signed Selmer group. As an application we obtain a control theorem for $(X^{\bullet\star}(E/\KK_{\infty})_f)_{H_n}$ and $X^{\bullet\star}(E/\KK_{\infty}^{H_n})_f$.

We first prove the following result.

\begin{proposition}\label{pseudo-null_prop}
Let ${\bullet, \star \in \{+,-\}}$. We have
\begin{enumerate}[(a)]
\item If $X^{\bullet\star}(E/\KK_{\infty})$ is a $\Lambda_2$-torsion module, then $X^{\bullet\star}(E/\KK_{\infty})$ has no nontrivial pseudo-null submodules if and only if $X(E/\KK_{\infty})$ has no nontrivial pseudo-null submodules.
\item If $X^{\bullet\star}(E/\KK_{\infty})$ is a $\Lambda_2$-torsion module, then ${\text{pd}_{\Lambda_2}(X^{\bullet\star}(E/\KK_{\infty})) \leq 1}$ if and only if ${\text{pd}_{\Lambda_2}(X(E/\KK_{\infty})) \leq 1}$.
\item Let ${H \in \mathcal{H}^{\bullet\star}}$. Then $X^{\bullet\star}(E/\KK_{\infty}^H)$ has no nontrivial finite $\Lambda$-submodules if and only if $X(E/\KK_{\infty}^H)$ has no nontrivial finite $\Lambda$-submodules.
\item Let ${H \in \mathcal{H}^{\bullet\star}}$. Then ${\text{pd}_{\Lambda}(X^{\bullet\star}(E/\KK_{\infty}^H)) \leq 1}$ if and only if ${\text{pd}_{\Lambda}(X(E/\KK_{\infty}^H) \leq 1}$.
\end{enumerate}
\end{proposition}
\begin{proof}
First we prove $(a)$. Let ${U:=J_p^{\bullet\star}(E/\KK_{\infty})^{\vee}}$. By Proposition~\ref{Sel_Lambda2_surjective_prop2} we have an exact sequence

\begin{equation}\label{Sel_exact_seq_dual}
0 \to U \to X(E/\KK_{\infty})\to X^{\bullet\star}(E/\KK_{\infty}) \to 0
\end{equation}
Furthermore $U$ is a free $\Lambda_2$-module of rank two. Suppose that $X^{\bullet\star}(E/\KK_{\infty})$ has no nontrivial pseudo-null submodules and let $A$ be a pseudo-null submodule of $X(E/\KK_{\infty})$. Then since $U$ is a free $\Lambda_2$-module the exact sequence~\eqref{Sel_exact_seq_dual} shows that $A$ injects into $X^{\bullet\star}(E/\KK_{\infty})$ and therefore ${A=0}$.

Suppose that $X(E/\KK_{\infty})$ has no nontrivial pseudo-null submodules and let $A$ be a pseudo-null submodule of $X^{\bullet\star}(E/\KK_{\infty})$. From the exact sequence~\eqref{Sel_exact_seq_dual} we get an exact sequence
\begin{equation}\label{Ext_exact_seq}
\Hom_{\Lambda_2}(A,X(E/\KK_{\infty})) \to \Hom_{\Lambda_2}(A, X^{\bullet\star}(E/\KK_{\infty})) \to \text{Ext}_{\Lambda_2}^1(A, U)
\end{equation}
Since $X(E/\KK_{\infty})$ has no nontrivial pseudo-null submodules, we have ${\Hom_{\Lambda_2}(A, X(E/\KK_{\infty}))=0}$. Since $A$ is a pseudo-null $\Lambda_2$-module and $U$ is a free $\Lambda_2$ module, it follows that ${\text{Ext}_{\Lambda_2}^1(A, U)=0}$ (see \cite[Example~1.3]{CSS2}). Therefore from the exact sequence~\eqref{Ext_exact_seq} we get that ${\Hom_{\Lambda_2}(A, X^{\bullet\star}(E/\KK_{\infty}))=0}$. It follows that ${A=0}$ and so the proof of $(a)$ is complete. The proof of $(c)$ is identical where we use Proposition~\ref{Seln_surjective_prop2} in place of Proposition~\ref{Sel_Lambda2_surjective_prop2} and $\Lambda$ in place of $\Lambda_2$.

Now we prove $(b)$ using the exact sequence~\eqref{Sel_exact_seq_dual}. Since $U$ is a free $\Lambda_2$-module ${\text{pd}_{\Lambda_2}(U)=0}$. It follows from this, the short exact sequence~\eqref{Sel_exact_seq_dual} and \cite[Tag 065S]{Stacks} that $(b)$ holds. The proof of $(d)$ is identical where we use Proposition~\ref{Seln_surjective_prop2} in place of Proposition~\ref{Sel_Lambda2_surjective_prop2} and $\Lambda$ in place of $\Lambda_2$.
\end{proof}

Now we consider the special case of the cyclotomic $\Z_p$-extension $K_{cyc}$ of $K$. Note that as every prime of $K$ above $p$ is totally ramified in $K_{cyc}/K$, our definition of  $\Selm^{++}(E/K_{cyc})$ coincides with the standard definition as in \cite{Kob}, \cite{IP} and \cite{Hamidi}.
\begin{theorem}\label{finite_submodules_theorem}
$X^{++}(E/K_{cyc})$ is a torsion $\Lambda$-module with no nontrivial finite submodules.
\end{theorem}
\begin{proof}
The module $X^{++}(E/K_{cyc})$ is $\Lambda$-torsion by Proposition~\ref{Hcyc_prop}. By \cite[Theorem~3.14]{KimFiniteSubmodules} $X^{++}(E/K_{cyc})$ has no nontrivial finite submodules.
\end{proof}

Now we are ready to prove the first main result of this section concerning the non-existence of nontrivial pseudo-null submodules. We will prove this for both the $\Lambda_2$-module $X^{\bullet\star}(E/\KK_{\infty})$ and the $\Lambda_{H,n}$-module $X^{\bullet\star}(E/\KK_{\infty}^{H_n})$ assuming it is torsion. The approach we take is an up-down proof where we deduce from Theorem \ref{finite_submodules_theorem} and Proposition~\ref{pseudo-null_prop} that $X^{\bullet\star}(E/\KK_{\infty})$ has no nontrivial pseudo-null $\Lambda_2$-submodules. We then ``go down" to prove from this statement that the $\Lambda_{H,n}$-module $X^{\bullet\star}(E/\KK_{\infty}^{H_n})$ has no nontrivial pseudo-null submodules assuming it is torsion.

\begin{theorem}\label{pseudo-null_theorem}
Let ${\bullet, \star \in \{+,-\}}$. Assume that ${\mathcal{H}^{\bullet\star} \neq \emptyset}$. Then we have
\begin{enumerate}[(a)]
\item $X^{\bullet\star}(E/\KK_{\infty})$ is $\Lambda_2$-torsion with no nontrivial pseudo-null $\Lambda_2$-submodules.
\item For any ${H \in \mathcal{H}^{\bullet\star}}$ and any ${n \geq 0}$, if $X^{\bullet\star}(E/\KK_{\infty}^{H_n})$ is $\Lambda_{H,n}$-torsion, then $X^{\bullet\star}(E/\KK_{\infty}^{H_n})$ has no nontrivial finite $\Lambda_{H,n}$-submodules.
\end{enumerate}
\end{theorem}
\begin{proof}
We use some ideas from \cite{Hamidi} for this proof. First note that by Proposition~\ref{Hcyc_prop} and Lemma~\ref{Xinf_torsion_lemma}, it follows that $X^{++}(E/\KK_{\infty})$ is a $\Lambda_2$-torsion module. Also by Lemma~\ref{Xinf_torsion_lemma} we have that $X^{\bullet\star}(E/\KK_{\infty})$ is a $\Lambda_2$-torsion module. Lemma~\ref{H1_Sel_vanishing_lemma} yields that ${H_1(H_{cyc}, X^{++}(E/\KK_{\infty}))=0}$. It follows from this that
$$\text{depth}_{\Lambda_2}(X^{++}(E/\KK_{\infty}))=\text{depth}_{\Lambda}(X^{++}(E/\KK_{\infty})_{H_{cyc}})+1.$$
By Theorem~\ref{SelmerControlThm_prop} we have ${X^{++}(E/\KK_{\infty})_{H_{cyc}} \cong X^{++}(E/K_{cyc})}$. Therefore
\begin{equation}\label{depth_formula}
\text{depth}_{\Lambda_2}(X^{++}(E/\KK_{\infty}))=\text{depth}_{\Lambda}(X^{++}(E/K_{cyc}))+1.
\end{equation}

In particular, ${\text{depth}_{\Lambda_2}(X^{++}(E/\KK_{\infty})) \geq 1}$. Suppose that ${\text{depth}_{\Lambda_2}(X^{++}(E/\KK_{\infty}))=1}$. Then it follows from equation~\eqref{depth_formula} that ${\text{depth}_{\Lambda}(X^{++}(E/K_{cyc})=0}$. This implies that every element of ${\langle p, T \rangle \subseteq \Zp[[T]]=\Lambda}$ is a zero-divisor for $X^{++}(E/K_{cyc})$. We know that the set of all zero-divisors for $X^{++}(E/K_{cyc})$ is the union of the associated primes of $X^{++}(E/K_{cyc})$ (see \cite[Theorem~6.1]{Matsumura}). Combining this with the fact that every element of ${\langle p, T \rangle}$ is a zero divisor of $X^{++}(E/K_{cyc})$, it follows from the prime avoidance lemma~\cite[Proposition~1.11]{AM} that ${\langle p, T \rangle}$ is an associated primes of $X^{++}(E/K_{cyc})$. Thus we have an injection ${\Zp[[T]]/\langle p, T \rangle \hookrightarrow X^{++}(E/K_{cyc})}$. So $X^{++}(E/K_{cyc})$ has a finite $\Lambda$-submodule isomorphic to $\Fp$ which contradicts Theorem~\ref{finite_submodules_theorem}. Therefore ${\text{depth}_{\Lambda_2}(X^{++}(E/\KK_{\infty})) \geq 2}$. By the Auslander-Buchsbaum formula this implies that ${\text{pd}_{\Lambda_2}(X^{++}(E/\KK_{\infty})) \leq 1}$.

By Proposition~\ref{pseudo-null_prop}$(b)$ we have that ${\text{pd}_{\Lambda_2}(X^{++}(E/\KK_{\infty})) \leq 1}$ if and only if ${\text{pd}_{\Lambda_2}(X(E/\KK_{\infty})) \leq 1}$ if and only if ${\text{pd}_{\Lambda_2}(X^{\bullet\star}(E/\KK_{\infty})) \leq 1}$. Since ${\text{pd}_{\Lambda_2}(X^{++}(E/\KK_{\infty})) \leq 1}$ by the above, we may conclude that ${\text{pd}_{\Lambda_2}(X^{\bullet\star}(E/\KK_{\infty})) \leq 1}$. From \cite[Proposition~3.10]{Venjakob2} this implies that $X^{\bullet\star}(E/\KK_{\infty})$ has no nontrivial pseudo-null $\Lambda_2$-submodules thus proving $(a)$.

Now we proceed to prove part~$(b)$. Let ${H \in \mathcal{H}^{\bullet\star}}$ and let ${n \geq 0}$. Assume that $X^{\bullet\star}(E/K_{\infty}^H)$ is $\Lambda_{H,n}$-torsion. We need to show that $X^{\bullet\star}(E/K_{\infty}^{H_n})$ has no nontrivial finite $\Lambda_{H,n}$-submodules. As in Section~\ref{section:mainintroduction}, let ${G_{H,n}=\Gal(\KK_{\infty}/K_{H,n})}$. From \cite[Theorem~4.3.1]{Weibel} we have
$$\text{pd}_{\Lambda(G_{H,n})}(X^{\bullet\star}(E/\KK_{\infty})) \leq \text{pd}_{\Lambda_2}(X^{\bullet\star}(E/\KK_{\infty})) + \text{pd}_{\Lambda(G_{H,n})}(\Lambda_2). $$

We showed above that ${\text{pd}_{\Lambda_2}(X^{\bullet\star}(E/\KK_{\infty})) \leq 1}$. Also as $G_{H,n}$ has finite index in $G$, the Iwasawa algebra $\Lambda_2$ is a free $\Lambda(G_{H,n})$-module of finite rank. So ${\text{pd}_{\Lambda(G_{H,n})}(\Lambda_2)=0}$. From these two facts and the above inequality we get that ${\text{pd}_{\Lambda(G_{H,n})}(X^{\bullet\star}(E/\KK_{\infty})) \leq 1}$. From \cite[Proposition~3.10]{Venjakob2} this implies that $X^{\bullet\star}(E/\KK_{\infty})$ has no non-trivial pseudo-null $\Lambda(G_{H,n})$-submodules.

Note that since $X^{\bullet\star}(E/\KK_{\infty})$ is $\Lambda_2$-torsion, we have that $X^{\bullet\star}(E/\KK_{\infty})$ is $\Lambda(G_{H,n})$-torsion (see \cite[Corollary~1.5]{Howson}). Let $\gamma_H$ be a topological generator of $H$. Let $U$ be the variable that corresponds to ${\gamma_H-1}$. Since $X^{\bullet\star}(E/\KK_{\infty}^{H_n})$ is $\Lambda_{H,n}$-torsion, we get from Theorem~\ref{SelmerControlThm_prop} that $X^{\bullet\star}(E/\KK_{\infty})/U$ is $\Lambda_{H,n}$-torsion. Therefore $U$ does not divide the characteristic power series of $X^{\bullet\star}(E/\KK_{\infty})$ as a $\Lambda(G_{H,n})$-module (see \cite[Lemma~2.2]{MHG}). Combining this with the fact that $X^{\bullet\star}(E/\KK_{\infty})$ has no nontrivial pseudo-null $\Lambda(G_{H,n})$-submodules we conclude that
$$\text{depth}_{\Lambda(G_{H,n})}(X^{\bullet\star}(E/\KK_{\infty}))=\text{depth}_{\Lambda_{H,n}}(X^{\bullet\star}(E/\KK_{\infty})/U)+1.$$
From the Auslander-Buchsbaum formula we conclude from this that
$$\text{pd}_{\Lambda(G_{H,n})}(X^{\bullet\star}(E/\KK_{\infty}))=\text{pd}_{\Lambda_{H,n}}(X^{\bullet\star}(E/\KK_{\infty})/U).$$
We showed above that ${\text{pd}_{\Lambda(G_{H,n})}(X^{\bullet\star}(E/\KK_{\infty})) \leq 1}$. Combining this with Theorem~\ref{SelmerControlThm_prop} and the above equality, we get ${\text{pd}_{\Lambda_{H,n}}(X^{\bullet\star}(E/\KK_{\infty}^{H_n})) \leq 1}$. By \cite[Proposition~3.10]{Venjakob2}, this implies that $X^{\bullet\star}(E/\KK_{\infty}^{H_n})$ has no nontrivial finite $\Lambda_{H,n}$-submodules. This completes the proof.
\end{proof}
We need the following comparison result, the first part of which first appeared in the work~\cite{Hamidi} of Hamidi.

\begin{proposition}\label{comparison_prop}
Let ${\bullet,\star \in \{+,-\}}$. Then the following assertions hold.
\begin{enumerate}[(a)]
\item Assume that $X^{\bullet\star}(E/\KK_{\infty})$ is a $\Lambda_2$-torsion module. Then we have a $\Lambda_2$-module injection
$$T_{\Lambda_2}(X(E/\K_\infty)) \hookrightarrow X^{\bullet\star}(E/\KK_\infty).$$
\item Let ${H \in \mathcal{H}^{\bullet\star}}$. We have a $\Lambda$-module injection
$$T_{\Lambda}(X(E/\K_\infty^H)) \hookrightarrow X^{\bullet\star}(E/\KK_\infty^H).$$
\end{enumerate}
\end{proposition}
\begin{proof}
Let ${U:=J_p^{\bullet\star}(E/\KK_{\infty})^{\vee}}$. By Proposition~\ref{Sel_Lambda2_surjective_prop2} we have an exact sequence

$$0 \to U \to X(E/\KK_{\infty})\to X^{\bullet\star}(E/\KK_{\infty}) \to 0.$$

Furthermore $U$ is a free $\Lambda_2$-module of rank two. From this we see that $(a)$ follows. The proof of $(b)$ is identical where we use Proposition~\ref{Seln_surjective_prop2} in place of Proposition~\ref{Sel_Lambda2_surjective_prop2}.
\end{proof}

From the above comparison proposition we deduce the following important

\begin{proposition} \label{cor:hamidi}
	Let ${\bullet, \star \in \{+,-\}}$, and let ${L_\infty \in \mathcal{H}^{\bullet\star}}$. We write ${H = \Gal(\K_\infty/L_\infty)}$. Then the following assertions hold. \begin{compactenum}[(a)]
		\item $\mu_G(T_{\Lambda_2}(X(E/\K_\infty))) \le \mu_G(X^{\bullet\star}(E/\K_\infty))$. In particular, if the $\mu$-invariant of $X^{\bullet\star}(E/\K_\infty)$ vanishes for any choice of signs $\bullet$ and $\star$, then the $\mu$-invariant of the torsion submodule of $X(E/\K_\infty)$ vanishes.
		\item If $X^{\bullet\star}(E/\K_\infty)_f$ is finitely generated over $\Lambda(H)$, then $\lambda(T_\Lambda(X(E/F_\infty)))$ is bounded as $F_\infty$ runs through the elements in a suitable Greenberg neighbourhood of $L_\infty$.
		\item If $X^{\bullet\star}(E/\K_\infty)_f$ is finitely generated over $\Lambda(H)$, then $X(E/\KK_\infty)_f$ is finitely generated over $\Lambda(H)$. In other words, if the $\mathfrak{M}_H(G)$-property holds for $X^{\bullet\star}(E/\KK_\infty)$, for any choice of signs $\bullet$ and $\star$, then the $\mathfrak{M}_H(G)$-property holds for the Selmer group of $E$ over $\KK_\infty$.
\item $T_{\Lambda_2}(X(E/\KK_{\infty}))$ has no nontrivial pseudo-null $\Lambda_2$-submodules. Also $T_{\Lambda}({X(E/\KK_{\infty}^H)})$ has no nontrivial pseudo-null $\Lambda$-submodules.
	\end{compactenum}
\end{proposition}
\begin{proof}
By Lemma~\ref{Xinf_torsion_lemma} $X^{\bullet\star}(E/\KK_{\infty})$ is a $\Lambda_2$-torsion module.
	Assertion~$(a)$ follows directly from Proposition~\ref{comparison_prop}. For the second assertion, we choose a neighbourhood ${U \subseteq \mathcal{H}^{\bullet\star}}$ of $L_\infty$ such that $\lambda(X^{\bullet\star}(E/F_\infty))$ is bounded as $F_\infty$ runs over the elements in $U$ (this is possible by Theorem~\ref{main_theorem}, ${(a) \Longrightarrow (h)}$). The assertion then follows from Proposition~\ref{comparison_prop}, which can be applied to any ${F_\infty \in U}$ by construction.
	
	In order to prove assertion~$(c)$ we let ${X = X^{\bullet\star}(E/\KK_\infty)}$ and ${Y = T_{\Lambda_2}(X(E/\KK_\infty))}$ for brevity, and we choose an integer $m$ such that ${X[p^\infty] = X[p^m]}$ and ${Y[p^\infty] = Y[p^m]}$. Then there are canonical isomorphisms ${X_f \cong p^m \cdot X}$ and ${Y_f \cong p^m \cdot Y}$. Therefore Proposition~\ref{comparison_prop} implies that ${Y_f \subseteq X_f}$. The assertion follows since $\Lambda(H)$ is a Noetherian domain.

Assertion~$(d)$ follows from Proposition~\ref{comparison_prop} and Theorem~\ref{pseudo-null_theorem}.
\end{proof}

We now prove the following key lemma by combining Proposition~\ref{cor:hamidi}$(d)$ with Corollary~\ref{cor:HinH}. We note that another proof of this lemma can be given by using the techniques in \cite{CSS}.

\begin{lemma}\label{H1_Sel_vanishing_lemma2}
	Let ${\bullet, \star \in \{+,-\}}$ and ${H \in \mathcal{H}^{\bullet\star}}$. Then
	$$H_1(H, T_{\Lambda_2}(X(E/\KK_{\infty}))=H_1(H, X(E/\KK_{\infty}))=0.$$
	If $Y(E/\KK_{\infty})_f$ is finitely generated over $\Lambda(H)$, then
	$$H_1(H_n, T_{\Lambda_2}(X(E/\KK_{\infty})))=H_1(H_n, X(E/\KK_{\infty}))=0$$
	for all ${n \geq 0}$.
\end{lemma}
\begin{proof}
Assume that $Y(E/\KK_{\infty})_f$ is finitely generated over $\Lambda(H)$ if ${n>0}$. From Proposition~\ref{cor:hamidi}$(d)$ we know that either ${T_{\Lambda_2}(X(E/\KK_{\infty}))=0}$ or ${f_{\infty} \neq 1}$. If ${T_{\Lambda_2}(X(E/\KK_{\infty}))=0}$, then we certainly have ${H_1(H_n, T_{\Lambda_2}(X(E/\KK_{\infty})))=0}$. So assume now that ${f_{\infty} \neq 1}$. Let ${n \geq 0}$ and ${\Upsilon_n=(1+T_1)^{p^na}(1+T_2)^{p^nb}-1}$, where ${H=\overbar{\langle \sigma^a\tau^b \rangle}}$. By Proposition~\ref{cor:hamidi}$(d)$ and \cite[Chapt.~VII, \S 4.4 Theorem~5]{Bourbaki}, there exist irreducible power series ${f_j \in \Zp[[T, U]]}$, integers $m_i$, $n_j$ and an injection
$$\phi: T_{\Lambda_2}(X(E/\KK_{\infty})) \hookrightarrow W, $$
where ${W=\bigoplus_{i=1}^s \Lambda_2/p^{m_i} \oplus \bigoplus_{j=1}^t \Lambda_2/f_j^{n_j}}$ and ${f_{\infty}=p^{\sum_{i=1}^s m_i}\prod_{j=1}^t f_j^{n_j}}$.

From Corollary~\ref{cor:HinH} we have that $\Upsilon_n$ is relatively prime to $f_{\infty}$. From this it is easy to see that ${H_1(H_n, W)=0}$. Since ${cd_p(H)=1}$, the functor $H_1(H, -)$ is left exact. So it follows from the injection $\phi$ that ${H_1(H_n, T_{\Lambda_2}(X(E/\KK_{\infty})))=0}$.

Now consider the exact sequence
$$0 \to T_{\Lambda_2}(X(E/\KK_{\infty})) \to X(E/\KK_{\infty}) \to F_{\Lambda_2}(X(E/\KK_{\infty})) \to 0. $$
Since $F_{\Lambda_2}(X(E/\KK_{\infty}))$ is $\Lambda_2$-torsion-free, we get that ${H_1(H_n, F_{\Lambda_2}(X(E/\KK_{\infty}))) = 0}$. Thus from ${H_1(H_n, T_{\Lambda_2}(X(E/\KK_{\infty})))=0}$ we conclude that ${H_1(H_n, X(E/\KK_{\infty}))=0}$.
\end{proof}

We note that Proposition~\ref{cor:hamidi}$(b)$, to the best of our knowledge, is the first criterion for proving the local boundedness of $\lambda$-invariants for certain \emph{non-torsion} Iwasawa modules. It has been shown in \cite[Theorem~7.7]{non-torsion} that a non-torsion Iwasawa module can force the $\lambda$-invariants to be unbounded. To be more precise, if the $\Lambda$-corank of the Selmer group of a certain $\Z_p$-extension $L_\infty$ differs from the corank of the Selmer groups over arbitrarily close $\Z_p$-extensions $F_\infty$ in a neighbourhood $U$ of $L_\infty$, then the $\lambda$-invariants are likely to be unbounded on $U$. Here we actually have that ${\rg_\Lambda(X(E/F_\infty)) = 2}$ for all ${F_\infty = \KK_\infty^H}$, ${H \in \mathcal{H}^{\bullet\star}}$; see Proposition~\ref{Selrank_prop}(a). Therefore the $\Lambda$-coranks are locally constant around ${L_\infty \in \mathcal{H}^{\bullet\star}}$ since the subset ${\mathcal{H}^{\bullet\star} \subseteq \mathcal{E}}$ is open with regard to Greenberg's topology.

\section{Cohomological interpretation of the $\mathfrak{M}_H(G)$-property}
In this section, we show for both the signed Selmer group and the Selmer group that we can give a cohomological statement that is equivalent to the $\mathfrak{M}_H(G)$-property. We first need the following

\begin{lemma}\label{uniform_m_lemma}
Let ${\bullet, \star \in \{+,-\}}$. Let ${H \in \mathcal{H}^{\bullet\star}}$. If $X^{\bullet\star}(E/\KK_{\infty})_f$ is finitely generated over $\Lambda(H)$, then there exists ${m>0}$ such that ${X^{\bullet\star}(E/\KK_{\infty})[p^{\infty}]=X^{\bullet\star}(E/\KK_{\infty})[p^m]}$ and moreover  ${X^{\bullet\star}(E/\KK_{\infty}^{H_n})[p^{\infty}]=X^{\bullet\star}(E/\KK_{\infty}^{H_n})[p^m]}$ for all $n$.
\end{lemma}
\begin{proof}
As usual, let ${F_{\infty}^n=\KK_{\infty}^{H_n}}$. From Proposition~\ref{Seln_surjective_prop} we get that $X^{\bullet\star}(E/F_{\infty}^n)$ is a torsion $\Lambda_{H,n}$-module for all $n$. From Theorem~\ref{pseudo-null_theorem} we get that the maximal finite $\Lambda_{H,n}$-submodule of $X^{\bullet\star}(E/F_{\infty}^n)$ is trivial for all $n$. Taking Lemma~\ref{H1_Sel_vanishing_lemma} into account, the proof of \cite[Proposition~4.3]{MHG} implies that for any choice of $m$ such that ${X^{\bullet\star}(E/\KK_{\infty})[p^{\infty}]=X^{\bullet\star}(E/\KK_{\infty})[p^m]}$ we actually also have that ${X^{\bullet\star}(E/F^n_\infty)[p^{\infty}]=X^{\bullet\star}(E/F^n_\infty)[p^{m}]}$ for all $n$.
\end{proof}

The maps $\hat{s}_n^{\bullet\star}$ from Theorem~\ref{SelmerControlThm_prop} induce maps \[ \theta_n^{\bullet\star}: (X^{\bullet\star}(E/\KK_{\infty})_f)_{H_n} \longrightarrow X^{\bullet\star}(E/\KK_{\infty}^{H_n})_f. \]
The first main result of this section is

\begin{proposition}\label{MHG_cohomology_group_prop}
Let ${\bullet, \star \in \{+,-\}}$. Let ${H \in \mathcal{H}^{\bullet\star}}$. Consider the following statements:
\begin{enumerate}[(a)]
\item $X^{\bullet\star}(E/\KK_{\infty})_f$ is finitely generated over $\Lambda(H)$.
\item For any $m>0$ such that $p^m$ annihilates $X^{\bullet\star}(E/\KK_{\infty})[p^{\infty}]$ and $X^{\bullet\star}(E/\KK_{\infty}^H)[p^{\infty}]$, the group  $H_1(H,X^{\bullet\star}(E/\KK_{\infty})/p^m)$ is finite.
\item $H_1(H,X^{\bullet\star}(E/\KK_{\infty})_f/p)$ is finite.
\end{enumerate}
Then the following assertions hold:
\begin{enumerate}[(1)]
\item We have the following implications: ${(c) \Longrightarrow (b) \Longleftrightarrow (a)}$.
\item If $X^{\bullet\star}(E/\KK_{\infty})_f$ is finitely generated over $\Lambda(H)$, then there exists ${m>0}$ such that ${X^{\bullet\star}(E/\KK_{\infty})[p^{\infty}]=X^{\bullet\star}(E/\KK_{\infty})[p^m]}$ and ${X^{\bullet\star}(E/\KK_{\infty}^{H_n})[p^{\infty}]=X^{\bullet\star}(E/\KK_{\infty}^{H_n})[p^m]}$ for all $n$. In this case for any such $m$ we have that $H_1(H_n,X^{\bullet\star}(E/\KK_{\infty})/p^m)$ is finite and bounded for all $n$. Furthermore, for all $n$ we have an isomorphism ${\ker \theta_n^{\bullet\star} \cong H_1(H_n,X^{\bullet\star}(E/\KK_{\infty})/p^m)}$.
\end{enumerate}
\end{proposition}
\begin{proof}
First we show ${(a) \Longleftrightarrow (b)}$. Let ${m>0}$ be an integer as in statement $(b)$. Then we have isomorphisms ${X^{\bullet\star}(E/\KK_{\infty})_f \cong p^mX^{\bullet\star}(E/\KK_{\infty})}$ and ${X^{\bullet\star}(E/\KK_{\infty}^H)_f \cong p^mX^{\bullet\star}(E/\KK_{\infty}^H)}$. Consider the commutative diagram with exact rows
\begin{equation*}
\begin{tikzcd}[column sep = small, scale cd=0.75]
			  0 \arrow[r] &  H_1(H, X^{\bullet\star}(E/\KK_{\infty})/p^m) \arrow[r, "\phi^{\bullet\star}"] &(X^{\bullet\star}(E/\KK_{\infty})_f)_H  \arrow[d, "\theta_0^{\bullet\star}"] \arrow[r] & X^{\bullet\star}(E/\KK_{\infty})_H \arrow[d, "\hat{s}_0^{\bullet\star}"] \arrow[r] & (X^{\bullet\star}(E/\KK_{\infty})/p^m)_H \arrow[d] \arrow[r] &0\\
		 & 0 \arrow[r] &X^{\bullet\star}(E/\KK_{\infty}^H)_f \arrow[r] & X^{\bullet\star}(E/\KK_{\infty}^H) \arrow[r] &(X^{\bullet\star}(E/\KK_{\infty}^H)/p^m) \arrow[r] &0
\end{tikzcd}
\end{equation*}
The zero on the left of the top row is because ${H_1(H, X^{\bullet\star}(E/\KK_{\infty}))=0}$ by Lemma~\ref{H1_Sel_vanishing_lemma}. The middle vertical arrow is an isomorphism by Proposition~\ref{SelmerControlThm_prop}. From the commutative diagram we see that ${\ker \theta_0^{\bullet\star}=\phi^{\bullet\star}(H_1(H, X^{\bullet\star}(E/\KK_{\infty})/p^m))}$. Since $\phi^{\bullet\star}$ is an injection and $p^m$ annihilates $H_1(H, X^{\bullet\star}(E/\KK_{\infty})/p^m)$, the equivalence ${(a) \Longleftrightarrow (b)}$ will follow if we show that $(a)$ is equivalent to ${\ker \theta_0^{\bullet\star}}$ being finitely generated over $\Zp$.

If $(a)$ is true then $(X^{\bullet\star}(E/\KK_{\infty})_f)_H$ is finitely generated over $\Zp$ whence so is $\ker \theta_0^{\bullet\star}$. On the other hand, if $\ker \theta_0^{\bullet\star}$ is finitely generated over $\Zp$, then since $X^{\bullet\star}(E/\KK_{\infty}^H)_f$ is finitely generated over $\Zp$ we get that $(X^{\bullet\star}(E/\KK_{\infty})_f)_H$ is a finitely generated $\Zp$-module. It follows that $(a)$ holds. Thus ${(a) \Longleftrightarrow (b)}$.

Now we prove that ${(c) \Longrightarrow (b)}$. Assume that $(c)$ is true. We have an exact sequence
$$0 \to X^{\bullet\star}(E/\KK_{\infty})[p^{\infty}] \to X^{\bullet\star}(E/\KK_{\infty}) \to X^{\bullet\star}(E/\KK_{\infty})_f \to 0.$$

Since $p^m$ annihilates $X^{\bullet\star}(E/\KK_{\infty})[p^{\infty}]$, we have that  ${X^{\bullet\star}(E/\KK_{\infty})[p^{\infty}]/p^m=X^{\bullet\star}(E/\KK_{\infty})[p^{\infty}]}$. Thus from the above exact sequence and the fact that ${X^{\bullet\star}(E/\KK_{\infty})_f[p^{\infty}]=0}$ we get the exact sequence
$$0 \to X^{\bullet\star}(E/\KK_{\infty})[p^{\infty}] \to X^{\bullet\star}(E/\KK_{\infty})/p^m \to X^{\bullet\star}(E/\KK_{\infty})_f/p^m \to 0.$$
This in turn induces an exact sequence
\begin{equation*}
	\begin{tikzcd}[column sep = small, scale cd=0.92]
	0 \to H_1(H, X^{\bullet\star}(E/\KK_{\infty})[p^{\infty}]) \to H_1(H, X^{\bullet\star}(E/\KK_{\infty})/p^m) \to H_1(H, X^{\bullet\star}(E/\KK_{\infty})_f/p^m).
\end{tikzcd}
\end{equation*}	

Since ${cd_p(H)=1}$, the functor $H_1(H, -)$ is left exact. Therefore since ${H_1(H, X^{\bullet\star}(E/\KK_{\infty}))=0}$ by Lemma~\ref{H1_Sel_vanishing_lemma} it follows that ${H_1(H, X^{\bullet\star}(E/\KK_{\infty})[p^{\infty}])=0}$. Thus we see from the above exact sequence that it will suffice to show that $H_1(H, X^{\bullet\star}(E/\KK_{\infty})_f/p^m)$ is finite.

We now proceed to show this by induction on $m$. For ${m=1}$ this is statement~$(c)$. Let ${m>1}$ and assume the statement is true for all ${i <m}$.  We have an exact sequence
$$0 \to p^{m-1}X^{\bullet\star}(E/\KK_{\infty})_f/p^m \to X^{\bullet\star}(E/\KK_{\infty})_f/p^m \to X^{\bullet\star}(E/\KK_{\infty})_f/p^{m-1} \to 0. $$
Consider the surjection ${\phi: X^{\bullet\star}(E/\KK_{\infty})_f/p \twoheadrightarrow p^{m-1}X^{\bullet\star}(E/\KK_{\infty}{})_f/p^m}$ induced by multiplication by $p^{m-1}$. Assume that ${x+pX^{\bullet\star}(E/\KK_{\infty})_f \in \ker \phi}$. Then ${p^{m-1}x=p^m y}$ for some ${y \in X^{\bullet\star}(E/\KK_{\infty})_f}$. So ${p^{m-1}(x-py)=0}$. Since ${X^{\bullet\star}(E/\KK_{\infty})_f[p^{\infty}]=0}$, we may conclude that ${x=py}$. Thus $\ker \phi$ is trivial, i.e. $\phi$ is an isomorphism.
The above sequence induces the following exact sequence
\begin{equation*}
	\begin{tikzcd}[column sep = small, scale cd=0.84]
0 \to H_1(H, p^{m-1}X^{\bullet\star}(E/\KK_{\infty})_f/p^m) \to H_1(H, X^{\bullet\star}(E/\KK_{\infty})_f/p^m) \to H_1(H, X^{\bullet\star}(E/\KK_{\infty})_f/p^{m-1}).
\end{tikzcd}
\end{equation*}
By the induction hypothesis ${H_1(H, X^{\bullet\star}(E/\KK_{\infty})_f/p^{m-1})}$ is finite, and as $\phi$ is an isomorphism we also get that $H_1(H, p^{m-1}X^{\bullet\star}(E/\KK_{\infty})_f/p^m)$ is finite. Thus it follows that ${H_1(H, X^{\bullet\star}(E/\KK_{\infty})_f/p^m)}$ is finite as desired.

Now we prove (2). Assume that $X^{\bullet\star}(E/\KK_{\infty})_f$ is finitely generated over $\Lambda(H)$. From Lemma~\ref{uniform_m_lemma} there exists an integer $m$ satisfying the statement in (2).

An identical proof as in the implication~${(a) \Longrightarrow (b)}$ above shows that for any ${n \geq 0}$ we have that ${H_1(H_n,X^{\bullet\star}(E/\KK_{\infty})/p^m)}$ is finite and furthermore we have an isomorphism ${\ker \theta_n^{\bullet\star} \cong H_1(H_n,X^{\bullet\star}(E/\KK_{\infty})/p^m)}$. Since $X^{\bullet\star}(E/\KK_{\infty})$ is finitely generated over the Noetherian ring $\Lambda_2$, it follows that ${X^{\bullet\star}(E/\KK_{\infty})/p^m}$ is a Noetherian $\Lambda_2$-module. As the $\Lambda_2$-submodules ${H_1(H_n,X^{\bullet\star}(E/\KK_{\infty})/p^m)=(X(E/K_{\infty})/p^m)^{H_n}}$ form an increasing nested chain, the chain must stabilize and so the orders are bounded.
\end{proof}

In the ordinary case we gave a rather involved proof of the analog of the proposition below in \cite[Proposition~4.2]{MHG}. Here we follow a slightly different approach to prove this result using Proposition~\ref{MHG_cohomology_group_prop}. We feel that our proof below is easier to follow than the one in loc. cit.

\begin{proposition}\label{Xf_ControlThm_prop}
Let ${\bullet, \star \in \{+,-\}}$ and ${H \in \mathcal{H}^{\bullet\star}}$. Assume that $X^{\bullet\star}(E/\KK_{\infty})_f$ is finitely generated over $\Lambda(H)$. Then the maps (induced by the duals of the maps $s_n^{\bullet\star}$ in Theorem~\ref{SelmerControlThm_prop})
	$$\theta_n^{\bullet\star}: (X^{\bullet\star}(E/\KK_{\infty})_f)_{H_n} \longrightarrow X^{\bullet\star}(E/\KK_{\infty}^{H_n})_f$$ are surjective and have finite kernels whose orders are bounded independently of $n$.
\end{proposition}
\begin{proof}
We have a commutative diagram

\begin{equation*}
\begin{tikzcd}[column sep = small, scale cd=0.98]
 & X^{\bullet\star}(E/\KK_{\infty})[p^{\infty}]_{H_n} \arrow[d, "\phi^{\bullet\star}_n"] \arrow[r] &  X^{\bullet\star}(E/\KK_{\infty})_{H_n}  \arrow[d, "\hat{s}_n^{\bullet\star}"] \arrow[r] & ( X^{\bullet\star}(E/\KK_{\infty})_f)_{H_n}  \arrow[d, "\theta_n^{\bullet\star}"] \arrow[r] &0\\
0 \arrow[r] &X^{\bullet\star}(E/\KK_{\infty}^{H_n})[p^{\infty}] \arrow[r] & X^{\bullet\star}(E/\KK_{\infty}^{H_n}) \arrow[r] & X^{\bullet\star}(E/\KK_{\infty}^{H_n})_f \arrow[r] &0
\end{tikzcd}
\end{equation*}
The commutative diagram shows that we have a surjection ${\coker \hat{s}_n^{\bullet\star} \twoheadrightarrow \coker \theta_n^{\bullet\star}}$. Since by Theorem~\ref{SelmerControlThm_prop} $\hat{s}_n^{\bullet\star}$ is surjective for all $n$, the same is true for $\theta_n^{\bullet\star}$.

Proposition~\ref{MHG_cohomology_group_prop} shows that $\ker \theta_n^{\bullet\star}$ is finite and bounded with $n$.
\end{proof}

\begin{corollary} \label{X_ptorsion_controlThmCor}
Let ${\bullet, \star \in \{+,-\}}$ and ${H \in \mathcal{H}^{\bullet\star}}$. Assume that $X^{\bullet\star}(E/\KK_{\infty})_f$ is finitely generated over $\Lambda(H)$.	For any ${n \ge 0}$ the dual of the restriction induces a map
$$\phi^{\bullet\star}_n \colon X^{\bullet\star}(E/\KK_{\infty})[p^{\infty}]_{H_n} \longrightarrow X^{\bullet\star}(E/\KK_{\infty}^{H_n})[p^{\infty}].$$
These maps are injective with finite cokernels, and the orders of the cokernels are bounded independently of $n$.
\end{corollary}
\begin{proof}
Consider the commutative diagram from the proof of Proposition~\ref{Xf_ControlThm_prop}.
Since ${cd_p(H_n)=1}$, the functor $H_1(H_n,-)$ is left exact. On the other hand, one has that ${X^{\bullet\star}(E/\KK_\infty)_f \cong p^t  X^{\bullet\star}(E/\KK_\infty)}$ for some large enough $t$. So from Lemma~\ref{H1_Sel_vanishing_lemma} it follows that ${H_1(H_n,  X^{\bullet\star}(E/\KK_\infty)_f)=0}$. Therefore we have a zero on the left of the top row in the commutative diagram.
By Theorem~\ref{SelmerControlThm_prop} the maps $\hat{s_n}^{\bullet\star}$ are isomorphisms. Proposition~\ref{Xf_ControlThm_prop} implies that the maps $\theta_n^{\bullet\star}$ are surjective and have finite kernels of bounded order. The statement of the corollary now follows from the snake lemma.
\end{proof}

The next main result of this section is the analogue of Proposition~\ref{MHG_cohomology_group_prop} for the Selmer group. Here we will content ourselves with only proving the equivalences of the analogues of $(a)$ and $(b)$ in Proposition~\ref{MHG_cohomology_group_prop}.

\begin{proposition}\label{MHG_cohomology_group_prop2}
Let ${\bullet, \star \in \{+,-\}}$. Let ${H \in \mathcal{H}^{\bullet\star}}$. The following statements are equivalent:
\begin{enumerate}[(a)]
\item $X(E/\KK_{\infty})_f$ is finitely generated over $\Lambda(H)$.
\item For any ${m>0}$ such that $p^m$ annihilates $X(E/\KK_{\infty})[p^{\infty}]$ and $X(E/\KK_{\infty}^H)[p^{\infty}]$, the group ${H_1(H,T_{\Lambda_2}(X(E/\KK_{\infty}))/p^m)}$ is finite.
\end{enumerate}
\end{proposition}
\begin{proof}
The proof is almost identical to the proof of Proposition~\ref{MHG_cohomology_group_prop}. Let ${m>0}$ be an integer as in statement $(b)$. Then we have isomorphisms ${X(E/\KK_{\infty})_f \cong p^mT_{\Lambda_2}(X(E/\KK_{\infty}))}$ and ${X(E/\KK_{\infty}^H)_f \cong p^mT_{\Lambda}(X(E/\KK_{\infty}^H))}$. Consider the commutative diagram with exact rows
\begin{equation*}
\begin{tikzcd}[column sep = small, scale cd=0.70]
			  0 \arrow[r] &  H_1(H, T_{\Lambda_2}(X(E/\KK_{\infty}))/p^m) \arrow[r, "\phi"] &(X(E/\KK_{\infty})_f)_H  \arrow[d, "\theta"] \arrow[r] & T_{\Lambda_2}(X(E/\KK_{\infty}))_H \arrow[d, "\hat{s}'_0"] \arrow[r] & (T_{\Lambda_2}(X(E/\KK_{\infty}))/p^m)_H \arrow[d] \arrow[r] &0\\
		 & 0 \arrow[r] &X(E/\KK_{\infty}^H)_f \arrow[r] & T_{\Lambda}(X(E/\KK_{\infty}^H)) \arrow[r] &(T_{\Lambda}(X(E/\KK_{\infty}^H))/p^m) \arrow[r] &0
\end{tikzcd}
\end{equation*}
The zero on the left of the top row is because ${H_1(H, T_{\Lambda_2}(X(E/\KK_{\infty})))=0}$ by Lemma~\ref{H1_Sel_vanishing_lemma2}. The middle vertical arrow is an injection by Theorem~\ref{TorsionControlThm}. From the commutative diagram we see that ${\ker \theta=\phi(H_1(H, T_{\Lambda_2}(X(E/\KK_{\infty}))/p^m))}$. Since $\phi$ is an injection and $p^m$ annihilates $H_1(H, T_{\Lambda_2}(X(E/\KK_{\infty}))/p^m)$, the equivalence ${(a) \Longleftrightarrow (b)}$ will follow if we show that $(a)$ is equivalent to  $\ker \theta$ being finitely generated over $\Zp$.

If $(a)$ is true then $(X^{\bullet\star}(E/\KK_{\infty})_f)_H$ is finitely generated over $\Zp$ whence so is $\ker \theta$. On the other hand, if $\ker \theta$ is finitely generated over $\Zp$, then since $X(E/\KK_{\infty}^H)_f$ is finitely generated over $\Zp$ we get that $(X(E/\KK_{\infty})_f)_H$ is a finitely generated $\Zp$-module. It follows that $(a)$ holds. Thus ${(a) \Longleftrightarrow (b)}$.
\end{proof}

\section{On the variation of Iwasawa invariants of signed Selmer groups in a Greenberg neighborhood} \label{section:fukuda}
In this section, we prove results on the variation of Iwasawa invariants in Greenberg open neighborhoods. Since we need a control theorem which is, to our knowledge, known only for the $+/+$ and $-/-$-Selmer groups, we have to restrict our setting accordingly. We start by proving that control theorem, which generalizes \cite[Theorem 6.8]{IP} proven by Iovita and Pollack. For any ${n \in \N}$ we let
$$ \Phi_n(T) = \sum_{i=0}^{p-1} T^{i p^{n-1}}$$ be the $p^n$-th cyclotomic polynomial, and we set
$$ \omega_n^+ = T \cdot \prod_{1 \le m \le n, \textup{$m$ even}} \Phi_m(1+T), \quad \omega_n^- = T \cdot \prod_{1 \le m \le n, \textup{$m$ odd}} \Phi_m(1+T).$$

Recall that $L_{\infty}/K$ is a $\Zp$-extension, and let ${\Gamma=\Gal(L_{\infty}/K)}$. Let $\gamma$ be a topological generator of $\Gamma$. If $M$ is a $\Gamma$-module and ${n,d \geq 0}$ are integers we define for any ${\bullet \in \{+,-\}}$ the polynomial ${\omega^{\bullet}_{n,d} = \omega^{\bullet}_n(\gamma^{p^d}-1) = \omega^{\bullet}_n((T+1)^{p^d}-1)\in \Z_p[[T]] \cong \Z_p[[\Gamma]]}$, and we let
$$M^{\omega^{\bullet}_{n,d}=0}=\{x \in M \mid \omega^{\bullet}_{n,d} \, x=0\}. $$

\begin{proposition} \label{prop:local_max_control_thm}
	Let ${\bullet \in \{+,-\}}$, let ${L_\infty = \KK_\infty^H}$ for some ${H \in \mathcal{H}^{\bullet\bullet}}$ be a $\Z_p$-extension of $K$, with intermediate fields $L_n$. We assume that ${i(L_{\infty},\fp)=i(L_{\infty}, \bar{\fp})}$ and define $d$ to be this common value.\\ Then for any ${n \in \N}$ the map
	$${ s_n^{\bullet\bullet}: \Selm_{p^{\infty}}^{\bullet\bullet}(E/L_{d+n})^{\omega_{n,d}^{\bullet} = 0} \longrightarrow  \Selm_{p^{\infty}}^{\bullet\bullet}(E/L_{\infty})^{\omega_{n,d}^{\bullet}=0}}$$ is injective with finite cokernel, and the cardinality of the cokernel of $s_n^{\bullet\bullet}$ is bounded as ${n \to \infty}$. We can actually choose a neighborhood ${U = \mathcal{E}(L_\infty,m)}$ of $L_\infty$ such that the upper bound for the cardinalities of the cokernels of the maps $s_n^{\bullet\bullet}$ holds uniformly on $U$.
\end{proposition}
\begin{proof}
	To prove this control theorem we use results of B.D. Kim~\cite{KimParityConj}. First we define ${\Lambda_n^{\bullet}:=\Zp[T]/\langle \omega_n^{\bullet} \rangle}$. Let $w$ be a prime of $L_{\infty}$ above $p$. Recall from Section~\ref{section:mainintroduction} that $\hat{E}^{\bullet}(L_{\infty,w})$ is defined as in \eqref{plusdef}, \eqref{minusdef} and \eqref{eq:def2-3} with respect to the $\Zp$-extension ${L_{\infty}/L^{(v)}_{i(L_{\infty}, v)}}$ where $v$ is the prime of $K$ below $w$. Let $\{L_n\}$ be the tower fields of the $\Zp$-extension ${L_{\infty}/L^{(v)}_{i(L_{\infty}, v)}}$.
	Let $\mathcal{L}_{\infty}$ be the union of the completions at $w$ of the fields $L_n$ and denote the completion of $L^{(v)}_{i(L_{\infty}, v)}$ at the prime below $w$ by $k$. Then $\mathcal{L}_{\infty}/k$ is a totally ramified $\Zp$-extension and $k/\Qp$ is unramified. Moreover $\mathcal{L}_{\infty}/\Qp$ is an abelian extension. This is the setup in \cite{KimParityConj}. Let ${t=[k:\Qp]}$. By \cite[Proposition~3.12]{KimParityConj} and the discussion on page 5 of \cite{KimCofree} we see that the Coleman maps in \cite[Proposition~3.7]{KimParityConj} are surjective. Combining this fact with \cite[Proposition~3.11]{KimParityConj} we get that
	
	\begin{equation}\label{pm_points_isom}
		(\hat{E}^{\bullet}(L_{n,w})\otimes \Qp/\Zp)^{\vee} \cong (\Lambda_n^{\bullet})^t
	\end{equation}\\
	for any ${n \in \N}$.
	
	Let ${n \in \N}$. It is easy to see that $\hat{E}^{\bullet}(L_{n,w})$ is annihilated by ${\omega^{\bullet}_n(\gamma^{p^d}-1)}$. Now consider the map
	
	$$\phi_n: \hat{E}^{\bullet}(L_{n,w})\otimes \Qp/\Zp \to (\hat{E}^{\bullet}(L_{\infty,w})\otimes \Qp/\Zp)^{\omega^{\bullet}_{n,d}=0}. $$\\
	We claim that $\phi_n$ is an isomorphism. An identical proof to \cite[Lemma~8.17]{Kob} shows that the Kummer map induces injections

	\begin{equation}\label{kummer_map3}
		\hat{E}^{\bullet}(L_{n,w})\otimes \Qp/\Zp \hookrightarrow H^1(L_{n,w}, E[p^{\infty}]),
	\end{equation}
	
	\begin{equation}\label{kummer_map4}
		\hat{E}^{\bullet}(L_{\infty,w})\otimes \Qp/\Zp \hookrightarrow H^1(L_{\infty,w}, E[p^{\infty}]).
	\end{equation}\\
	The kernel of the restriction map ${H^1(L_{n,w}, E[p^{\infty}]) \to H^1(L_{\infty,w}, E[p^{\infty}])}$ is ${H^1(\Gal(L_{\infty,w}/L_{n,w}), E(L_{\infty,w})[p^{\infty}])}$ which is trivial by Lemma~\ref{p-torsion_lemma}. Therefore taking into account the injections~\eqref{kummer_map3} and \eqref{kummer_map4} we see that $\phi_n$ is an injection.
	
	Now let ${\Lambda:=\Zp[[\Gal(L_{\infty}/L^{(v)}_{i(L_{\infty}, v)+n})]]}$. Then by Proposition~\ref{pm_points_structure_prop} we have
	
	$$((\hat{E}^{\bullet}(L_{\infty,w})\otimes \Qp/\Zp)^{\omega^{\bullet}_{n,d}=0})^{\vee} \cong \Lambda^t/\langle\omega^{\bullet}_n(\gamma^{p^d}-1)\rangle \cong (\Lambda_n^{\bullet})^t.$$\\
	Therefore by taking Pontryagin duals and using the isomorphism~\eqref{pm_points_isom} we see that $\phi_n^{\vee}$ is surjective where the domain and codomain are free $\Zp$-modules with the same rank. If $\phi_n^{\vee}$ was not injective, then $\ker \phi_n^{\vee}$ would have positive $\Zp$-rank. This would force the $\Zp$-rank of the codomain to be less than the $\Zp$-rank of the domain which would contradict the fact above. Therefore, as claimed, $\phi_n$ is an isomorphism.
	
	As $\phi_n$ is an isomorphism the rest of the proof of our control theorem follows exactly as in the proof of Kobayashi's control theorem \cite[Theorem~9.3]{Kob}. By the proof of this result, an upper bound for the cardinality of the cokernels of $s_n^{\bullet\bullet}$ can be obtained via \cite[Lemma~3.3]{Gb_LNM}. Since ${H \in \mathcal{H}^{\bullet\bullet}}$ we can choose a neighborhood ${U = \mathcal{E}(L_\infty,m)}$ of $L_\infty$ such that the primes $v$ of $K$ where $E$ has bad reduction do not split completely in $L_m/K$. Then the number of primes of $F_\infty$ above such primes $v$ will be constant as $F_\infty$ runs over the elements in $U$, and it follows from the proof of \cite[Lemma~3.3]{Gb_LNM} that the upper bound on the cardinalities of the cokernels is not only uniform in $n$, but also in ${F_\infty \in U}$.
\end{proof}

Now we are ready to prove the main result of this section.
\begin{thm} \label{thm:local-maximality}
Let ${\bullet \in \{+,-\}}$, ${H \in \mathcal{H}^{\bullet\bullet}}$ and ${L_{\infty} = \KK_{\infty}^H}$. We assume that ${i(L_\infty, \fp) = i(L_\infty, \bar{\fp})}$. \\ Then there exists a Greenberg neighborhood ${U = \mathcal{E}(L_{\infty},n)}$ of $L_{\infty}$ such that the following assertions hold for each ${F_{\infty} \in U}$. \begin{compactenum}[(a)]
    \item $\Gal(\KK_{\infty}/F_{\infty}) \in \mathcal{H}^{\bullet\bullet}$,
    \item $\mu(X^{\bullet\bullet}(E/F_{\infty})) \le \mu(X^{\bullet\bullet}(E/L_{\infty}))$,
    \item $\lambda(X^{\bullet\bullet}(E/F_{\infty})) = \lambda(X^{\bullet\bullet}(E/L_{\infty}))$ holds if ${\mu(X^{\bullet\bullet}(E/F_{\infty})) = \mu(X^{\bullet\bullet}(E/L_{\infty}))}$.
  \end{compactenum}
\end{thm}
\begin{proof}
	We know from Proposition~\ref{Zpextensions_prop} that $U$ can be chosen such that the first condition is satisfied.

	Now we proceed along the lines of the proof of \cite[Theorem~4.11]{fukuda-selmer}. Write ${X = X^{\bullet\bullet}(E/L_\infty)}$ and ${\lambda = \lambda(X)}$ for brevity, and consider polynomials
	\[ \nu_{m,n}(T) = \frac{(T+1)^{p^m}-1}{(T+1)^{p^n}-1} \in \Lambda = \Z_p[[T]]\]
	with ${m \ge n}$. Since the polynomials $\nu_{n+1,n}(T)$ are pairwise coprime, we can make $n$ large enough to ensure that $\nu_{m,n}(T)$ is coprime with the characteristic power series of $X$ for each ${m \ge n}$. Below we will typically consider ${m = 2n}$.
	
	Fix a pseudo-isomorphism ${\varphi \colon X \longrightarrow E_X}$, where $E_X$ denotes an elementary $\Lambda$-module. Since $X$ does not contain any non-trivial pseudo-null submodules by  Theorem~\ref{pseudo-null_theorem}, $\varphi$ is actually an injection.
	
	We choose ${s \in \N}$ large enough to ensure that $\nu_{2s,s}(T)$ is coprime with the characteristic power series of $X$. Then
	\[ \rank_s(X) := v_p(|X/\nu_{2s,s}(T)|)\]
	is finite, where we denote by $v_p$ the $p$-adic valuation, normalised such that ${v_p(p) = 1}$. Moreover, it follows from \cite[proof of Theorem~3.10]{local_beh} that
	\[ \rank_s(X) = \rank_s(E_X) \]
	for each such $s$ (here we again use Theorem~\ref{pseudo-null_theorem}).
	
	The characteristic power series of $X$ is associated to the product ${g_X \cdot p^{\mu(X)}}$ of a unique distinguished polynomial $g_X$ of degree ${\lambda = \lambda(X)}$ and a power of $p$. If $s$ is larger than the degree of $g_X$, then
	\[ \rank_s(E_X) = \mu(X) \cdot (p^{2s} - p^s) + \lambda \cdot s \]
	by the proof of \cite[Theorem~3.10]{local_beh}.
	
	Now let $s$ be large enough such that
	\begin {align} \label{eq:s} p^s (p-1) > s \cdot \lambda + v_p(C), \end{align}
	where $C$ denotes an upper bound for the cardinalities of the cokernels of the maps $s_n^{\bullet\bullet}$ from Proposition~\ref{prop:local_max_control_thm}.
	Dualising the maps from Proposition~\ref{prop:local_max_control_thm} we may conclude that
	\begin{eqnarray} v_p(|X/(\omega^{\bullet}_{n,d}, \nu_{2s,s}(T))|) - v_p(C) & \le & \rg_s(X^{\bullet\bullet}(E/L_{d+n})/\omega^{\bullet}_{n,d}) \nonumber \\
		& \le & v_p(|X/(\omega^{\bullet}_{n,d}, \nu_{2s,s}(T))|) \label{eq:sect7}\\
	& \le & \rg_s(X) \nonumber \end{eqnarray}
	for each ${n \in \N}$. 
	In what follows we will assume that $n$ has been chosen large enough to ensure that $\omega^{\bullet}_{n,d}$ annihilates $X/\nu_{2s,s}(T)$. Note that this will be true for any sufficiently large $n$. Indeed, $\omega^{\bullet}_{n,d}$ can be written as a product of at least $n/2$ non-unit factors. By Nakayama's Lemma the cardinality of any finite non-zero $\Lambda$-module decreases on multiplication by a non-unit. If ${|X/\nu_{2s,s}(T)| = p^w}$, then any product of at least $w$ non-unit factors in $\Lambda$ will actually annihilate that quotient.
	
	For these $n$ we have
	\[ v_p(|X/(\omega_{n,d}^\bullet, \nu_{2s,s}(T))|) = \rg_s(X).\]
	We may conclude from the above that there exists some ${m \in \N}$ such that
	\begin{align} \label{eq:inf1} \rg_s(X^{\bullet\bullet}(E/L_{d+k})/\omega_{k,d}^\bullet) = \rg_s(X^{\bullet\bullet}(E/L_{d+m})/\omega_{m,d}^\bullet)\end{align}
	for infinitely many ${k \ge m}$. In what follows we will show that there exist infinitely many $k$ such that not only equation~\eqref{eq:inf1} holds for $k$, but instead also a certain second condition is valid for these $k$.
	
	Let ${N \subseteq \N}$ denote any subset of the infinite set of ${k \ge m}$ for which equation~\eqref{eq:inf1} holds. If $N$ contains at least ${v_p(C) + 1}$ pairwise distinct even (respectively, odd) integers $k$, where $C$ is as above and we are considering even (respectively, odd) integers if ${\bullet = +}$ (respectively, ${\bullet = -}$), then the inequalities~\eqref{eq:sect7} imply that
	\begin{equation} \label{eq:equal_quotients} v_p(|X/(\omega_{i,d}^\bullet, \nu_{2s,s}(T))|) = v_p(|X/(\omega_{j,d}^\bullet, \nu_{2s,s}(T))|)\end{equation}
	for at least two even (respectively, odd) indices $i$ and $j$ from $N$ with ${i < j}$, and it follows from Nakayama's Lemma that
	\begin{align} \label{eq:inf2} \rank_s(X) =  v_p(|X/(\omega_{i,d}^\bullet, \nu_{2s,s}(T))|)\end{align}
	for any such index $i$. Indeed, we have ${\omega_{j,d}^\bullet = \omega_{i,d}^\bullet \cdot h}$ for some non-unit $h$ (recall that $i$ and $j$ have the same parity). Letting ${Z = (\omega_{i,d}^\bullet \cdot X +  \nu_{2s,s}(T) \cdot X)/\nu_{2s,s}(T)X}$, it follows from equation~\eqref{eq:equal_quotients} that
	\[ h \cdot Z = Z. \]
	Since ${Z \subseteq X/\nu_{2s,s}(T)}$ is finite, this is possible only if ${Z = 0}$. Therefore
	\[ \omega_{i,d}^\bullet \cdot X \subseteq \nu_{2s,s}(T) \cdot X, \]
	proving~\eqref{eq:inf2}.
	
	Since there are infinitely many possible pairwise disjoint choices for the set $N$, we may conclude that there exist infinitely many indices ${k \ge m}$ such that both equations~\eqref{eq:inf1} and \eqref{eq:inf2} hold for $k$, i.e.
	\[ \rg_s(X^{\bullet\bullet}(E/L_{d+k})/\omega_{k,d}^\bullet) = \rg_s(X^{\bullet\bullet}(E/L_{d+m})/\omega_{m,d}^\bullet)\]
	and
	\begin{align*} \rank_s(X) =  v_p(|X/(\omega_{k,d}^\bullet, \nu_{2s,s}(T))|)\end{align*}
	hold. In what follows, we will only consider indices $k$ in this infinite set.
		
	Suppose that ${U = \mathcal{E}(L_\infty,d+n)}$ is small enough such that ${n \ge k}$ for at least ${v_p(C) + 1}$ even (respectively, odd) such integers $k$, where again we are considering even (respectively, odd) integers if ${\bullet = +}$ (respectively, ${\bullet = -}$) -- it suffices to increase the value $n$ accordingly. We may also assume that $U$ is small enough to make the conclusion of Proposition~\ref{prop:local_max_control_thm} valid on $U$. Let ${F_\infty \in U}$. Then ${F_{d+k} = L_{d+k}}$ and ${F_{d+m} = L_{d+m}}$ for any pair $(k,m)$ of integers satisfying ${k \le n}$ and ${m \le n}$, and
	\[ \rg_s(X^{\bullet\bullet}(E/F_{d+k})/\omega_{k,d}^\bullet) = \rg_s(X^{\bullet\bullet}(E/F_{d+m})/\omega_{m,d}^\bullet)\]
	if $k$ and $m$ have been chosen as above.
	Since the bound $C$ from Proposition~\ref{prop:local_max_control_thm} holds uniformly on $U$, we may conclude that
	$$ v_p(|X^{\bullet\bullet}(E/F_\infty)/(\omega_{i,d}^\bullet, \nu_{2s,s}(T))|) = v_p(|X^{\bullet\bullet}(E/F_\infty)/(\omega_{j,d}^\bullet, \nu_{2s,s}(T))|)$$
	for at least two different even (respectively, odd) indices ${i < j}$, and it follows from Nakayama's Lemma that
	\begin{align} \label{eq:compare2} \rank_s(X^{\bullet\bullet}(E/F_\infty)) =  v_p(|X^{\bullet\bullet}(E/F_\infty)/(\omega_{i,d}^\bullet, \nu_{2s,s}(T))|)\end{align}
	for such $i$ (cf. also the proof of \cite[Corollary~3.8]{fukuda-selmer}). Taking into account both equations~\eqref{eq:inf2} and \eqref{eq:compare2} (by construction there exists at least one index ${i}$ for which both equations hold) and in view of Proposition~\ref{prop:local_max_control_thm} we may conclude that
	\begin{align} \label{eq:main} \rg_s(X) - v_p(C) \le \rank_s(X^{\bullet\bullet}(E/F_\infty)) \le \rg_s(X) + v_p(C) \end{align}
	for each ${F_\infty \in U}$.

	Since none of the $\Lambda$-modules $X^{\bullet\bullet}(E/F_\infty)$ contains any non-trivial pseudo-null submodules, it follows from the above that
	\begin{align} \label{eq:E_X} \rg_s(X) = \rank_s(E_X) = \mu(X) \cdot (p^{2s} - p^s) + \lambda \cdot s \end{align}
	differs from $\rank_s(E_X^{\bullet\bullet}(E/F_\infty))$ by at most $v_p(C)$, where $E_X^{\bullet\bullet}(E/F_\infty)$ denotes an elementary $\Lambda$-module which is pseudo-isomorphic to $X^{\bullet\bullet}(E/F_\infty)$. More precisely, fix some ${F_\infty \in U}$, and write ${\tilde{X} = X^{\bullet\bullet}(E/F_\infty)}$ and ${E_{\tilde{X}} = E_X^{\bullet\bullet}(E/F_\infty)}$ for brevity. Then $\tilde{X}$ also does not contain any non-trivial pseudo-null submodules. Therefore equation~\eqref{eq:main} can be rewritten as
	\begin{align} \label{eq:main2} \rg_s(E_X) - v_p(C) \le \rg_s(E_{\tilde{X}}) \le \rg_s(E_X) + v_p(C). \end{align}
	
	Since ${\rg_s(E_{\tilde{X}}) \ge \mu(\tilde{X}) \cdot (p^{2s} - p^s)}$, we may conclude that
	\[ \mu(\tilde{X}) \cdot (p^{2s} - p^s) - v_p(C) \le \rank_s(E_X) + v_p(C) = \mu(X) \cdot (p^{2s} - p^s) + \lambda \cdot s + v_p(C). \]
	In view of~\eqref{eq:s} this proves part~(b) of the theorem.
	
	Now fix ${F_\infty \in U}$ such that ${\mu(\tilde{X}) = \mu(X)}$. If ${\lambda(\tilde{X}) \ge p^s (p-1)}$, then
	\[ \rg_s(E_{\tilde{X}}) \ge \mu(\tilde{X}) \cdot (p^{2s} - p^s) + p^s(p-1) \]
	by the proof of \cite[Theorem~3.10]{local_beh} and the proof of \cite[Theorem~4.11]{fukuda-selmer}. By our choice of $s$ (see~\eqref{eq:s}) and in view of equation~\eqref{eq:main2}, this is not possible. It follows that ${\lambda(\tilde{X}) < p^s(p-1)}$. Then the proof of \cite[Theorem~3.10]{local_beh} implies that
	\[ \rg_s(E_{\tilde{X}}) = \mu(\tilde{X}) \cdot (p^{2s} - p^s) + s \cdot \lambda(\tilde{X}). \]
	Now suppose that $s$ has been chosen large enough to ensure that ${s > v_p(C)}$.
	In view of equation~\eqref{eq:main2}, and as ${\mu(\tilde{X}) = \mu(X)}$ by assumption, we may conclude that ${\lambda(\tilde{X}) = \lambda}$.
\end{proof}
\begin{corollary} \label{cor:mu=0}
Let ${\bullet \in \{+,-\}}$, ${H \in \mathcal{H}^{\bullet\bullet}}$ and ${L_{\infty} = \KK_{\infty}^H}$. We assume that $i(L_{\infty},\fp)=i(L_{\infty}, \bar{\fp})$. Suppose that ${\mu(X^{\bullet\bullet}(E/L_{\infty})) = 0}$. \\ Then there exists a Greenberg neighborhood ${U = \mathcal{E}(L_{\infty},n)}$ of $L_{\infty}$ such that
	\[ \mu(X^{\bullet\bullet}(E/F_{\infty})) = 0, \quad \lambda(X^{\bullet\bullet}(E/F_{\infty})) {= } \lambda(X^{\bullet\bullet}(E/L_{\infty})) \]
	for each ${F_\infty \in U}$.
\end{corollary}
\begin{proof}
	Choose $U$ as in Theorem~\ref{thm:local-maximality}. In view of Theorem~\ref{thm:local-maximality}(b) we have
	\[ \mu(X^{\bullet\bullet}(E/F_{\infty})) = 0 = \mu(X^{\bullet\bullet}(E/L_{\infty}))\]
	for each ${F_\infty \in U}$. The second assertion follows from Theorem~\ref{thm:local-maximality}(c).
\end{proof}

Now we derive a useful if-and-only-if condition for the local boundedness of $\lambda$-invariants. An analogous result in the setting of Iwasawa theory of graphs was obtained in \cite[Theorem~6.5]{graphs}.
\begin{thm} \label{thm:strong_local_behavior}
	Let ${\bullet \in \{+,-\}}$, ${H \in \mathcal{H}^{\bullet\bullet}}$ and let ${L_{\infty} = \KK_{\infty}^H}$. We assume that ${i(L_{\infty},\fp)=i(L_{\infty}, \bar{\fp})}$. Then \begin{compactitem}
		\item \textbf{either} there exists a Greenberg neighborhood ${U = \mathcal{E}(L_{\infty},n)}$ of $L_{\infty}$ such that ${\mu(X^{\bullet\bullet}(E/F_\infty)) = \mu(X^{\bullet\bullet}(E/L_{\infty}))}$ and
		      $$\lambda(X^{\bullet\bullet}(E/F_\infty)) = \lambda(X^{\bullet\bullet}(E/L_{\infty}))$$
		      for each ${F_\infty \in U}$,
		\item \textbf{or} for any neighborhood $U$ of $L_\infty$ there exists some ${F_\infty \in U}$ such that $\mu(X^{\bullet\bullet}(E/F_{\infty})) < \mu(X^{\bullet\bullet}(E/L_{\infty}))$, and $\lambda(X^{\bullet\bullet}(E/F_\infty))$ is \emph{unbounded} on some neighborhood of $L_\infty$.
	\end{compactitem}
\end{thm}
\begin{proof}
	Choose a neighborhood $U$ of $L_\infty$ as in Theorem~\ref{thm:local-maximality}. Suppose first that
	\[ \mu(X^{\bullet\bullet}(E/F_\infty)) = \mu(X^{\bullet\bullet}(E/L_{\infty})) \]
	for each ${F_\infty \in U}$. Then we know from Theorem~\ref{thm:local-maximality}(c) that
	\[ \lambda(X^{\bullet\bullet}(E/F_\infty)) = \lambda(X^{\bullet\bullet}(E/L_{\infty})) =:  \lambda \]
	for each ${F_\infty \in U}$.
	
	Now suppose that for \emph{any} arbitrarily small neighborhood $U$ of $L_\infty$, we have that
	\[ \mu(X^{\bullet\bullet}(E/F_{\infty})) < \mu(X^{\bullet\bullet}(E/L_{\infty}))\]
	for some ${F_\infty \in U}$, and fix such a $\Z_p$-extension $F_\infty$ of $K$. As in the proof of Theorem~\ref{thm:local-maximality}, we abbreviate $X^{\bullet\bullet}(E/L_\infty)$, $X^{\bullet\bullet}(E/F_\infty)$, $E_X^{\bullet\bullet}(E/L_\infty)$ and $E_X^{\bullet\bullet}(E/F_\infty)$ to $X$, $\tilde{X}$, $E_X$ and $E_{\tilde{X}}$, respectively.
	
	We choose $s$ as in the proof of Theorem~\ref{thm:local-maximality}, and we assume the neighborhood $U$ of $L_\infty$ to be chosen small enough such that
	\begin{eqnarray*} \rank_s(E_X) - v_p(C) & \le & \rank_s(E_{\tilde{X}}) \\
		& \le & \rank_s(E_X) + v_p(C) \\
		& = & \mu(X) \cdot (p^{2s} - p^s) + \lambda \cdot s + v_p(C). \end{eqnarray*}
	We have that either ${\lambda(\tilde{X}) \ge p^s(p-1)}$, or
	\begin{eqnarray*} \mu(X) \cdot (p^{2s} - p^s) + \lambda \cdot s - v_p(C) & \le & \mu(\tilde{X}) \cdot (p^{2s} - p^s) + \lambda(\tilde{X}) \cdot s \\
		& \le & \mu(X) \cdot (p^{2s} - p^s) + \lambda \cdot s + v_p(C). \end{eqnarray*}
	In the second case we may conclude that
	\[ \lambda(\tilde{X}) \ge \lambda + \left\lfloor \frac{p^{2s} - p^s}{s} \right\rfloor \cdot (\mu(X) - \mu(\tilde{X})) - v_p(C) \ge \left\lfloor \frac{p^{2s} - p^s}{s} \right\rfloor - v_p(C), \]
	where $\left\lfloor \frac{p^{2s} - p^s}{s} \right\rfloor$ means the largest integer which is less than or equal to ${\frac{p^{2s} - p^s}{s}}$.
	
	Now let ${t > s}$ be arbitrary. We can choose a (potentially smaller) neighborhood $U'$ of $L_\infty$ such that
	\[ \rg_t(E_X) - v_p(C) \le \rg_t(E_X^{\bullet\bullet}(E/F_\infty)) \le \rank_t(E_X) + v_p(C)\]
	for each ${F_\infty \in U'}$. By hypothesis, we can again choose $F_\infty$ such that ${\mu(X^{\bullet\bullet}(E/F_{\infty})) < \mu(X)}$. For such $F_\infty$ we have that either ${\lambda(\tilde{X}) \ge p^t(p-1)}$ or ${\lambda(\tilde{X}) \ge \left\lfloor \frac{p^{2t} - p^t}{t} \right\rfloor - v_p(C)}$. Letting $t$ tend to infinity, we may conclude that $\lambda(X^{\bullet\bullet}(E/F_\infty))$ is unbounded on a neighborhood of $L_\infty$.
\end{proof}

\section{Algebraic structure theorems and the $\mathfrak{M}_H(G)$-property} \label{section:Xf}
The aim of this section is to collect algebraic structure results on the Iwasawa module $X^{\bullet\star}(E/\KK_\infty)_f$ and on the $p$-torsion submodules of $X^{\bullet\star}(E/\KK_\infty)$ and $X^{\bullet\star}(E/\KK_\infty^{H_n})$, ${n \in \N}$, with an eye towards the $\mathfrak{M}_H(G)$-property.
	
First we are going to study the $\Lambda(H)$-structure of $X^{\bullet\star}(E/\KK_\infty)_f$.
To be more precise, fix ${\bullet, \star \in \{+,-\}}$ and ${H \in \mathcal{H}^{\bullet\star}}$. Assuming that $X^{\bullet\star}(E/\KK_{\infty})_f$ is finitely generated over $\Lambda(H)$, we will determine its structure as a $\Lambda(H)$-module.
\begin{theorem}\label{Kida_formula}
Let ${\bullet, \star \in \{+,-\}}$ and ${H \in \mathcal{H}^{\bullet\star}}$. Assume that $X^{\bullet\star}(E/\KK_{\infty})_f$ is finitely generated over $\Lambda(H)$. Then for any ${n \ge 0}$ we have $$\rank_{\Zp}(X^{\bullet\star}(E/\KK_{\infty}^{H_n})_f)=p^n\rank_{\Zp}(X^{\bullet\star}(E/\KK_{\infty}^H)_f). $$
\end{theorem}
\begin{proof}
Taking Proposition~\ref{Xf_ControlThm_prop}, Proposition~\ref{Seln_surjective_prop} and Lemma~\ref{H1_Sel_vanishing_lemma} into account, this theorem can be proven as in \cite[Theorem~8.1]{MHG}.
\end{proof}

\begin{theorem}\label{Xinf_structure_theorem}
Let ${\bullet, \star \in \{+,-\}}$ and ${H \in \mathcal{H}^{\bullet\star}}$. Assume that $X^{\bullet\star}(E/\KK_{\infty})_f$ is finitely generated over $\Lambda(H)$. Then we have an injective $\Lambda(H)$-homomorphism
$$X^{\bullet\star}(E/\KK_{\infty})_f \hookrightarrow \Lambda(H)^{\lambda_H^{\bullet\star}}$$ with finite cokernel and a $\Lambda_2$-homomorphism
$$0 \longrightarrow X^{\bullet\star}(E/\KK_{\infty}) \longrightarrow \bigoplus_{i=1}^{s} \Lambda_2/{(f_i^{\bullet\star})}^{n_i} \oplus \bigoplus_{j=1}^{t} \Lambda_2/p^{m_j^{\bullet\star}} \longrightarrow B^{\bullet\star} \longrightarrow 0$$
where ${s \leq \lambda_H^{\bullet\star}}$, $B^{\bullet\star}$ is a pseudo-null $\Lambda_2$-module, ${f_i^{\bullet\star} \in \Lambda_2 \setminus \Lambda(H)}$ are irreducible power series and ${\mu_G(X^{\bullet\star}(E/\KK_{\infty}))=\sum_{j=1}^t m_j^{\bullet\star}}$.
\end{theorem}
\begin{proof}
Taking into account Theorem~\ref{pseudo-null_theorem}, Theorem~\ref{Kida_formula} and Proposition~\ref{Xf_ControlThm_prop}, this theorem can be proven as in \cite[Theorem~8.2]{MHG}.
\end{proof}

Now we are going to relate the $p$-torsion submodules $X^{\bullet\star}(E/\KK_\infty)[p^\infty]$ and $X^{\bullet\star}(E/\KK_\infty^{H_n})[p^\infty]$, ${n \in \N}$. We show that the $\mathfrak{M}_H(G)$-property implies that the $\Lambda_2$-module structure of $X^{\bullet\star}(E/\KK_{\infty})[p^{\infty}]$ is ``similar" to the $\Lambda_{H,n}$-module structure of $X^{\bullet\star}(E/\KK_{\infty}^{H_n})[p^{\infty}]$. First we need the following two lemmas.

\begin{lemma}\label{finiteness_lemma}
Let ${\Lambda:=\Zp[[T]]}$, $M$ a finitely generated $p$-primary $\Lambda$-module and let ${\mathfrak{p}=p\Lambda}$. If $M_{\mathfrak{p}}$ is finite, then ${M_{\mathfrak{p}}=0}$.
\end{lemma}
\begin{proof}
Let ${S=\Lambda \setminus \mathfrak{p}}$ and assume that $S^{-1}M$ is finite. If $M$ is infinite, taking \cite[Lemma~2.1]{MHG} into account we see that there exist integers ${m_i > 0}$ and an injection ${\bigoplus_{i=1}^s \Lambda/p^{m_i} \hookrightarrow M}$ with finite cokernel.  It is easy to see that $S^{-1}(\Lambda/p^{m_i})$ is infinite. This contradicts the assumption that $S^{-1}M$ is finite. Therefore we see that $M$ is finite and so is pseudo-null. Since ${\text{ht}(\mathfrak{p})=1}$, we get ${M_{\mathfrak{p}}=0}$.
\end{proof}

\begin{proposition}\label{p-primary_structure_prop}
Let ${\bullet, \star \in \{+,-\}}$ and ${H \in \mathcal{H}^{\bullet\star}}$. Assume that $X^{\bullet\star}(E/\KK_{\infty})_f$ is finitely generated over $\Lambda(H)$. Then for all $n$, $X^{\bullet\star}(E/\KK_{\infty}^{H_n})$ is a torsion $\Lambda_{H,n}$-module. Furthermore, if $$ X^{\bullet\star}(E/\KK_{\infty})[p^{\infty}] \sim \bigoplus_{i=1}^s \Lambda_2/p^{m_i^{\bullet\star}},$$
then for all $n$ we have pseudo-isomorphisms
$$X^{\bullet\star}(E/\KK_{\infty}^{H_n})[p^{\infty}] \sim \bigoplus_{i=1}^s  \left(\Lambda_{H,n}/p^{m_i^{\bullet\star}}\right)^{p^n}.$$
\end{proposition}
\begin{proof}
Assume that $X^{\bullet\star}(E/\KK_{\infty})_f$ is finitely generated over $\Lambda(H)$. Then by Proposition~\ref{Seln_surjective_prop} we have for any ${n \geq 0}$ that $X^{\bullet\star}(E/\KK_{\infty}^{H_n})$ is a torsion $\Lambda_{H,n}$-module. Now let ${n \geq 0}$ be fixed. In view of Corollary~\ref{X_ptorsion_controlThmCor} it will suffice to prove the statement of the proposition with $X^{\bullet\star}(E/\KK_{\infty}^{H_n})[p^{\infty}]$ replaced by $X^{\bullet\star}(E/\KK_{\infty})[p^{\infty}]_{H_n}$. We now proceed with proving this.

For simplicity, we denote $X^{\bullet\star}(E/\KK_{\infty})[p^{\infty}]$ by $X$. Recall that ${\Lambda_2 \cong \Zp[[T_1,T_2]]}$ and ${G_n = G_{H,n}}$ (see Section~\ref{section:mainintroduction}). Let ${\Upsilon=(1+T_1)^{p^na}(1+T_2)^{p^nb}-1}$. Since ${[G:G_n]=p^n}$, the Iwasawa algebra ${\Lambda_2=\Lambda(G)}$ is a free $\Lambda(G_n)$-module of rank $p^n$. We now replace $G$ by $G_n$. Then ${\Lambda_{H,n} = \Z_p[[G_n/H_n]]}$ is isomorphic to $\Lambda_2/\Upsilon$. For simplicity, we denote $\Lambda_{H,n}$ by $\Lambda$. We have a projection map ${\pi: \Lambda_2 \twoheadrightarrow \Lambda}$. Now let ${\mathfrak{p}=p\Lambda}$, ${\mathfrak{P}=\pi^{-1}(\mathfrak{p})}$, ${\tilde{T}=\Lambda \setminus \mathfrak{p}}$ and ${S=\Lambda_2 \setminus \mathfrak{P}}$. Note that ${\mathfrak{P}=\langle p, \Upsilon \rangle}$ is a prime ideal of $\Lambda_2$ of height two. Now $S^{-1}\Lambda_2$ is a Noetherian, integrally closed domain and as $\mathfrak{P}$ has height two therefore $S^{-1}\Lambda_2$ has dimension two. Let ${E=\bigoplus_{i=1}^s \left(\Lambda_2/p^{m_i}\right)^{p^n}}$ be the direct product of $p^n$ copies of the group in the statement of the proposition with the `$\bullet\star$' sign removed for simplicity. Recall that $\Lambda(G)$ is a free $\Lambda(G_n)$-module of rank $p^n$ and that we have replaced $G$ by $G_n$. Therefore taking \cite[Lemma~2.1]{MHG} into account, we have an exact sequence
$$0 \longrightarrow E \longrightarrow X \longrightarrow C \longrightarrow 0$$
where $C$ is a pseudo-null $\Lambda_2$-module. This exact sequence induces another exact sequence
$$0 \longrightarrow S^{-1}E \longrightarrow S^{-1}X \longrightarrow S^{-1}C \longrightarrow 0.$$

By the definition of a pseudo-null $\Lambda_2$-module, and noting that the $\textup{Ext}$ functor commutes with localisation (see \cite[Proposition~3.3.10]{Weibel}), $S^{-1}C$ is a pseudo-null $S^{-1}\Lambda_2$-module. Since $S^{-1}\Lambda_2$ has Krull dimension two, this means that $S^{-1}C$ is finite. The above exact sequence induces another exact sequence
\begin{equation}\label{seq1}
(S^{-1}C)^{\Upsilon=0} \longrightarrow (S^{-1}E)/\Upsilon  \longrightarrow (S^{-1}X)/\Upsilon \longrightarrow (S^{-1}C)/\Upsilon \longrightarrow 0.
\end{equation}

We have the following isomorphisms
\begin{equation}\label{isom1}
(S^{-1}C)^{\Upsilon=0} \cong S^{-1}(C^{\Upsilon=0}) \cong \tilde{T}^{-1}(C^{\Upsilon=0}),
\end{equation}

\begin{align*}
(S^{-1}E)/\Upsilon \cong S^{-1}\Big(\bigoplus_{i=1}^s (\Lambda_2/p^{m_i})/\Upsilon \Big)^{p^n} &\cong  S^{-1} \Big(\bigoplus_{i=1}^s (\Lambda_2/\Upsilon)/p^{m_i} \Big)^{p^n}\\
&\cong S^{-1}\Big(\bigoplus_{i=1}^s \Lambda/p^{m_i}\Big)^{p^n}\\
&\cong  \tilde{T}^{-1}\Big(\bigoplus_{i=1}^s \Lambda/p^{m_i}\Big)^{p^n}, \stepcounter{equation} \tag{\theequation} \label{isom2}
\end{align*}

\begin{equation}\label{isom3}
(S^{-1}X)/\Upsilon \cong S^{-1}(X/\Upsilon) \cong \tilde{T}^{-1}(X/\Upsilon),
\end{equation}

\begin{equation}\label{isom4}
(S^{-1}C)/\Upsilon \cong S^{-1}(C/\Upsilon) \cong \tilde{T}^{-1}(C/\Upsilon).
\end{equation}

Let ${E'=\bigoplus_{i=1}^s\left(\Lambda/p^{m_i}\right)^{p^n}}$. From the exact sequence~\eqref{seq1} and the isomorphisms~\eqref{isom1} -- \eqref{isom4} we get an exact sequence

\begin{equation}\label{seq2}
\tilde{T}^{-1}(C^{\Upsilon=0}) \longrightarrow \tilde{T}^{-1}E'  \longrightarrow \tilde{T}^{-1}(X/\Upsilon) \longrightarrow \tilde{T}^{-1}(C/\Upsilon) \longrightarrow 0.
\end{equation}

We showed above that $S^{-1}C$ is finite. Therefore from the isomorphisms~\eqref{isom1} and \eqref{isom4} both $\tilde{T}^{-1}(C^{\Upsilon=0})$ and $\tilde{T}^{-1}(C/\Upsilon)$ are finite. Then Lemma~\ref{finiteness_lemma} implies that these groups are zero. It follows that we have an isomorphism ${f_0:\tilde{T}^{-1}E' \isomarrow \tilde{T}^{-1}(X/\Upsilon)}$. We now use the idea in the proof of \cite[Chapter~VII, \S~4.4 Theorem~5]{Bourbaki} to show that $f_0$ can be lifted to a pseudo-isomorphism ${f: E' \to X/\Upsilon}$. This will conclude the proof of the proposition.

By \cite[Proposition~2.10]{Eisenbud} we have a natural isomorphism

$$\Hom_{\tilde{T}^{-1}\Lambda}(\tilde{T}^{-1}E', \tilde{T}^{-1}(X/\Upsilon)) \cong \tilde{T}^{-1}\Hom_{\Lambda}(E', X/\Upsilon). $$

The isomorphism $\psi$ from the right side to the left side is defined as follows: let ${h=g/t \in \tilde{T}^{-1}\Hom_{\Lambda}(E', X/\Upsilon)}$ where ${g \in \Hom_{\Lambda}(E', X/\Upsilon)}$ and ${t \in \tilde{T}}$. Then ${(\psi(g/t))(e/t')=g(e)/(tt')}$ (here $e \in E'$ and ${t' \in \tilde{T}}$).

Therefore there exists a map ${f: E' \to X/\Upsilon}$ and ${t \in \tilde{T}}$ such that ${f=tf_0}$ (this is because the map $\psi$ above is surjective so there exists ${f \in \Hom_{\Lambda}(E', X/\Upsilon)}$ and ${t \in \tilde{T}}$ such that ${\psi(f/t)=f_0}$. This implies that ${f=tf_0}$). We claim that $f$ is a pseudo-isomorphism. Let $\mathfrak{p}'$ be a prime ideal of $\Lambda$ with ${\text{ht}(\mathfrak{p}') \leq 1}$. We need to show that the map ${f_{\mathfrak{p'}}: E'_{\mathfrak{p}'} \to (X/\Upsilon)_{\mathfrak{p'}}}$ is an isomorphism. If ${\mathfrak{p'} \neq \mathfrak{p}}$, then the complement of $\mathfrak{p'}$ in $\Lambda$ contains an element $\alpha$ that annihilates $E'$ and $X/\Upsilon$. Therefore in this case ${E'_{\mathfrak{p'}}=0}$ and ${(X/\Upsilon)_{\mathfrak{p'}}=0}$ whence $f_{\mathfrak{p'}}$ is an isomorphism. When ${\mathfrak{p'}=\mathfrak{p}}$, $t$ is a unit in ${\Lambda_{\mathfrak{p'}}=\tilde{T}^{-1}\Lambda}$. Recall that $f_0$ is an isomorphism. Therefore in this case $f_{\mathfrak{p'}}$ is an isomorphism. It follows from this that $f$ is a pseudo-isomorphism. This implies the desired result.
\end{proof}

\section{$\mu$-invariants and the $\mathfrak{M}_H(G)$-property}\label{section:mu-invariants}
In the current section we start with a theorem which basically proves the equivalence of conditions ${(a)-(c)}$ in Theorem~\ref{main_theorem}. Recall the notation ${G_n = G_{H,n}}$ from Section~\ref{section:mainintroduction}.
\begin{theorem}\label{mu_invariants_theorem}
Let ${\bullet, \star \in \{+,-\}}$ and ${H \in \mathcal{H}^{\bullet\star}}$. We have $$\mu_G(X^{\bullet\star}(E/\KK_{\infty}))=\mu_{G/H}(X^{\bullet\star}(E/\KK_{\infty}^{H}))-\mu_{G/H}(H_0(H, X^{\bullet\star}(E/\KK_{\infty})_f)).$$
	Moreover, the following three statements are equivalent: \begin{compactenum}[(a)]
		\item $X^{\bullet\star}(E/\KK_{\infty})_f$ is finitely generated over $\Lambda(H)$,
		\item for all ${n \geq 0}$, $X^{\bullet\star}(E/\KK_{\infty}^{H_n})$ is a torsion $\Lambda_{H,n}$-module, and we have $$\mu_{G_n}(X^{\bullet\star}(E/\KK_{\infty}))=\mu_{G_n/H_n}(X^{\bullet\star}(E/\KK_{\infty}^{H_n})), $$
		\item ${\mu_G(X^{\bullet\star}(E/\KK_{\infty}))=\mu_{G/H}(X^{\bullet\star}(E/\KK_{\infty}^{H}))}$.
	\end{compactenum}
\end{theorem}
\begin{proof}
  For the first statement and the proof of the implication ${(a) \Longrightarrow (b)}$ we can mimic the proof of \cite[Theorem~5.1]{MHG} using Theorem~\ref{SelmerControlThm_prop} and Lemma~\ref{H1_Sel_vanishing_lemma}. It is clear that ${(b) \Longrightarrow (c)}$. For the implication ${(c) \Longrightarrow (a)}$ we apply (with minor modifications) the proof of \cite[Corollary~5.2]{MHG}.
\end{proof}
We also need an analogue of Theorem~\ref{mu_invariants_theorem} for the Selmer group. Recall that for any finitely generated $R$-module $M$ (where $R$ denotes any integral domain like ${R = \Lambda_2}$) we denote by $T_R(M)$ the torsion $R$-submodule of $M$.
\begin{theorem}\label{mu_invariants_theorem2}
Let ${\bullet, \star \in \{+,-\}}$ and ${H \in \mathcal{H}^{\bullet\star}}$. We have
\begin{align*}
\resizebox{.57\hsize}{!}{$\mu_G({T_{\Lambda_2}(X(E/\KK_{\infty}))})=\mu_{G/H}(T_{\Lambda_{H,0}}(X(E/\KK_{\infty}^{H})))$} & \resizebox{.4\hsize}{!}{$-\mu_{G/H}(T_{\Lambda_{H,0}}([F_{\Lambda_2}(X(E/\KK_{\infty}))]_H))$} \\
&\resizebox{.4\hsize}{!}{$-\mu_{G/H}(H_0(H, T_{\Lambda_2}{(X(E/\KK_{\infty})_f)})).$}
\end{align*}
Consider the following three statements: \begin{compactenum}[(a)]
        \item $X(E/\KK_{\infty})_f$ is finitely generated over $\Lambda(H)$.
        \item For all ${n \geq 0}$ we have $$\resizebox{.94\hsize}{!}{$\mu_{G_n}(T_{\Lambda_2}(X(E/\KK_{\infty})))=\mu_{G_n/H_n}(T_{\Lambda_{H,n}}(X(E/\KK_{\infty}^{H_n})))-\mu_{G_n/H_n}(T_{\Lambda_{H,n}}([F_{\Lambda_2}(X(E/\KK_{\infty}))]_{H_n})).$} $$
        \item For ${n=0}$ we have
        $$\resizebox{.94\hsize}{!}{$\mu_G(T_{\Lambda_2}(X(E/\KK_{\infty})))=\mu_{G/H}(T_{\Lambda_{H,0}}(X(E/\KK_{\infty}^{H})))-\mu_{G/H}(T_{\Lambda_{H,0}}([F_{\Lambda_2}(X(E/\KK_{\infty}))]_H)).$} $$
   \end{compactenum}
Then the following assertions hold
\begin{enumerate}[(1)]
\item We have the implications ${(b) \Longrightarrow (c) \Longrightarrow (a)}$.
\item If $Y(E/\KK_{\infty})_f$ is finitely generated over $\Lambda(H)$, then  $(a)$ implies $(b)$. In other words, by (1), in that case the three statements~$(a)-(c)$ are equivalent.
\end{enumerate}
\end{theorem}
\begin{proof}
Assume that $Y(E/\KK_{\infty})_f$ is finitely generated over $\Lambda(H)$. With this assumption we have that ${H_1(H_n, T_{\Lambda_2}(X(E/\KK_{\infty}))=0}$ for all ${n \ge 0}$ by Lemma~\ref{H1_Sel_vanishing_lemma2}. Using this fact and mimicking the proof of \cite[Theorem~5.1]{MHG} we have that for any ${n \geq 0}$

$$\resizebox{.93\hsize}{!}{$\mu_{G_n}(T_{\Lambda_2}(X(E/\KK_{\infty})))=\mu_{G_n/H_n}((T_{\Lambda_{H,n}}(X(E/\KK_{\infty})))_{H_n})-\mu_{G_n/H_n}(H_0(H_n, T_{\Lambda_2}(X(E/\KK_{\infty})_f))).$} $$

Combining this with the control theorem (Theorem~\ref{TorsionControlThm}) we get
\begin{align}\label{mu_Selmer_relation}
\resizebox{.57\hsize}{!}{$\mu_G({T_{\Lambda_2}(X(E/\KK_{\infty}))})=\mu_{G_n/H_n}(T_{\Lambda_{H,n}}(X(E/\KK_{\infty}^{H_n})))$} & \resizebox{.4\hsize}{!}{$-\mu_{G_n/H_n}(T_{\Lambda_{H,n}}([F_{\Lambda_2}(X(E/\KK_{\infty}))]_{H_n}))$} \nonumber \\
&\resizebox{.4\hsize}{!}{$-\mu_{G_n/H_n}(H_0(H_n, T_{\Lambda_2}{(X(E/\KK_{\infty})_f)})).$}
\end{align}
This implies the second statement. Indeed, assume that $X(E/\KK_{\infty})_f$ is finitely generated over $\Lambda(H)$. Since $H_n$ has finite index in $H$, the quotient $H_0(H_n, T_{\Lambda_2}{(X(E/\KK_{\infty})_f)})$ is finitely generated over $\Zp$ and hence has $\mu_{G_n/H_n}$-invariant zero. Thus from $(\ref{mu_Selmer_relation})$ we get that ${(a) \Longrightarrow (b)}$.

It is clear that ${(b) \Longrightarrow (c)}$. Finally for the implication ${(c) \Longrightarrow (a)}$ we use $\eqref{mu_Selmer_relation}$ for ${n=0}$ and apply (with minor modifications) the proof of \cite[Corollary~5.2]{MHG}.
\end{proof}

\section{Asymptotic growth of Iwasawa invariants} \label{section:monsky}
The last ingredient which is needed for our main results is a growth formula for the Iwasawa invariants of the $\Lambda$-modules $X^{\bullet\star}(E/\KK_\infty^{H_n})$, ${n \in \N}$, provided that all of them are torsion. The proofs are based on work of Cuoco and Monsky.

We fix two $\Z_p$-extensions ${k_{\infty}, F \subseteq \KK_{\infty}}$ of $K$ such that ${k_{\infty} \cap F = K}$. Let ${H \subseteq G}$ be such that ${F = \KK_{\infty}^H}$, and write ${F_n = \KK_{\infty}^{H_n}}$ for each ${n \in \N}$. Let moreover ${k_n \subseteq k_{\infty}}$, ${n \in \N}$, be the unique subfield of degree $p^n$ over $K$. Then $F_n/k_n$ is a $\Z_p$-extension, and we write ${\Lambda_n = \Z_p[[\Gal(F_n/k_n)]]}$. We have the following diagram of fields.
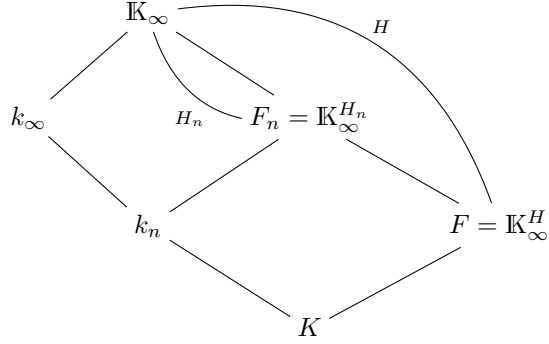
\[ \xymatrix{& \KK_{\infty} \ar@{-}[dr] \ar@{-}@/_1.4pc/[dr]_{H_n} \ar@{-}@/^2.9pc/[ddrr]^{H}& & \\
	k_{\infty} \ar@{-}[ur] \ar@{-}[dr] & & F_n = \KK_{\infty}^{H_n} \ar@{-}[dl] \ar@{-}[dr] & \\
	& k_n \ar@{-}[dr] & & F = \KK_{\infty}^H \ar@{-}[dl] \\
	& & K &} \]

Let $X_n^{\bullet\star}$ be the Pontryagin dual of the signed Selmer group of $E$ over the $\Z_p$-extension $F_n/k_n$, and suppose that $X_n^{\bullet\star}$ is a torsion $\Lambda_n$-module for each ${n \in \N}$. We recall from Proposition~\ref{Seln_surjective_prop} that this will be the case if $X^{\bullet\star}(E/\KK_\infty)_f$ is finitely generated as a $\Lambda(H)$-module. We write ${X^{\bullet\star} = X^{\bullet\star}(E/\KK_{\infty})}$ for brevity. Let $\mu_n^{\bullet\star}$ and $\lambda_n^{\bullet\star}$ denote the Iwasawa invariants of the ${\Lambda_n \cong \Z_p[[T]]}$-module $X_n^{\bullet\star}$. As in Section~\ref{section:mainintroduction} we write the characteristic power series ${f_\infty^{\bullet\star} \in \Z_p[[T_1, T_2]]}$ of $X^{\bullet\star}$ as ${f_\infty^{\bullet\star} = p^{\mu_G(X^{\bullet\star})} \cdot g_\infty^{\bullet\star}}$, and we denote by ${\overline{g_\infty^{\bullet\star}} \in \Z_p[[T_1, T_2]]/p}$ the residue class modulo $p$. Recall that if $\mathcal{H}^{\bullet\star}$ is non-empty, then $X^{\bullet\star}(E/\KK_\infty)$ is $\Lambda_2$-torsion, and thus ${g_\infty^{\bullet\star} \ne 0}$.
\begin{thm} \label{thm:cuoco-monsky}
	Suppose that ${H \in \mathcal{H}^{\bullet\star}}$, let $X^{\bullet\star}$ be as above and assume that ${X_n^{\bullet\star} = X^{\bullet\star}(E/F_n)}$ is a torsion $\Lambda_n$-module for each $n$. Then the following assertions hold. \begin{compactenum}[(a)]
		\item We have
		\[ \mu_n^{\bullet\star} = \mu_G(X^{\bullet\star}) p^n + m_1^{\bullet\star} n + \mathcal{O}(1)\]
		and
		\[ \lambda_n^{\bullet\star} = l^{\bullet\star} p^n + \mathcal{O}(1)\]
		for suitable constants $m_1^{\bullet\star}$ and $l^{\bullet\star}$.
		\item If ${H = \overline{\langle \sigma^a \tau^b \rangle}}$ and ${\Upsilon = (1+T_1)^a (1+T_2)^b - 1}$ (where ${T_1 = \sigma - 1}$ and ${T_2 = \tau - 1}$ as usual), then
		\[ m_1^{\bullet\star} = v_{\overline{\Upsilon}}(\overline{g_\infty^{\bullet\star}}), \]
		i.e. $m_1^{\bullet\star}$ equals the number of factors $\overline{\Upsilon}$ which divide $\overline{g_\infty^{\bullet\star}}$ in $\Z_p[[T_1, T_2]]/p$.
		\item If $X^{\bullet\star}_f$ is finitely generated as a $\Lambda(H)$-module, then ${m_1^{\bullet\star} = 0}$ and ${l^{\bullet\star} = \textup{rank}_{\Lambda(H)}(X^{\bullet\star}_f)}$.
	\end{compactenum}
\end{thm}
\begin{proof}
	For the proof of statement~$(a)$ one can use, with negligible changes, the approach from the proof of \cite[Lemma~6.7]{MHG}, which is based on the work of Cuoco (see \cite{Cuoco1}). Here we need that ${H \in \mathcal{H}^{\bullet\star}}$ in order to be allowed to apply our control theorem (Theorem~\ref{SelmerControlThm_prop}).
	
	Assertion~$(b)$ follows from \cite[Theorem~2.4]{Cuoco3} (see also the proof of \cite[Lemma~6.9]{MHG}). Note that in terms of Cuoco's formulation we have that
	\[ m_1^{\bullet\star} = \sum_{\mathfrak{P}} v_{\mathfrak{P}}(\overline{g_\infty^{\bullet\star}}), \]
	where $\mathfrak{P}$ runs over all prime ideals of $\Z_p[[T_1, T_2]]/p$ of the form ${(\overline{\psi} - \overline{1})}$, where ${\psi \in H \setminus H^p}$. Since ${\gamma := \Upsilon + 1}$ is a topological generator of ${H \cong \Z_p}$, we have that any such element $\psi$ can be written in the form ${\psi = \gamma^u}$ for some unit ${u \in \Z_p^\times}$. This means that we can also write ${\gamma = \psi^v}$, where ${v = u^{-1} \in \Z_p}$. Since ${\psi - 1 = (\Upsilon+1)^u - 1 \in (\Upsilon)}$ and ${\Upsilon = ((\psi-1)+1)^v - 1 \in (\psi - 1)}$, we may conclude that all the resulting prime ideals are in fact equal to the principal ideal generated by the coset of $\Upsilon$. In other words, the above sum has only one summand.
	
	Finally, it follows from Theorem~\ref{mu_invariants_theorem} and the asymptotic formula proven in $(a)$ that ${m_1^{\bullet\star} = 0}$ if $X_f^{\bullet\star}$ is finitely generated over $\Lambda(H)$.
	
	Moreover, since ${H_n = H^{p^n}}$ for each ${n \in \N}$, it follows from a result of Harris (see \cite{Harris}) that
	\[ \textup{rank}_{\Z_p}((X^{\bullet\star}_f)_{H_n}) = \textup{rank}_{\Lambda(H)}(X^{\bullet\star}_f) \cdot p^n + \mathcal{O}(1). \]
	On the other hand, it follows from assertion~(a) and our control theorem that
	\[ \textup{rank}_{\Z_p}((X^{\bullet\star}_f)_{H_n}) = \lambda(X^{\bullet\star}_{H_n}) = \lambda(X_n^{\bullet\star}) = l^{\bullet\star} p^n + \mathcal{O}(1). \]
	Comparing coefficients yields the last assertion and concludes the proof of the theorem.
\end{proof}

We will also need the analogue of Theorem~\ref{thm:cuoco-monsky} for Selmer groups. Let ${X_n=\Selinf(E/F_n)^{\vee}}$  and let ${X=\Sel(E/\KK_{\infty})^{\vee}}$. Let $\mu_n$ and $\lambda_n$ denote the Iwasawa invariants of the ${\Lambda_n \cong \Z_p[[T]]}$-module $T_{\Lambda_n}(X_n)$. As in Section~\ref{section:mainintroduction} (see~\eqref{eq:f_infty}) we write the characteristic power series ${f_\infty \in \Z_p[[T_1, T_2]]}$ of $X$ as ${f_\infty = p^{\mu_G(T_{\Lambda_2}(X))} \cdot g_\infty}$, and we denote by ${\overline{g_\infty} \in \Z_p[[T_1, T_2]]/p}$ the residue class modulo $p$.

\begin{thm} \label{thm2:cuoco-monsky}
	Suppose that ${H \in \mathcal{H}^{\bullet\star}}$, and let $X$ and the $X_n$ be as above. Then the following assertions hold. \begin{compactenum}[(a)]
		\item We have
		\[ \mu_n = \mu_G(T_{\Lambda_2}(X)) p^n + m_1 n + \mathcal{O}(1)\]
		and
		\[ \lambda_n = l p^n + \mathcal{O}(1)\]
		for suitable constants $m_1$ and $l$.
		\item If ${H = \overline{\langle \sigma^a \tau^b \rangle}}$ and ${\Upsilon = (1+T_1)^a (1+T_2)^b - 1}$ (where ${T_1 = \sigma - 1}$ and ${T_2 = \tau - 1}$ as usual), then
		\[ m_1 = v_{\overline{\Upsilon}}(\overline{g_\infty}), \]
		i.e. $m_1$ equals the number of factors $\overline{\Upsilon}$ which divide $\overline{g_\infty}$ in $\Z_p[[T_1, T_2]]/p$.
		\item If $X_f$ is finitely generated as a $\Lambda(H)$-module, then ${m_1 = 0}$.
        \item  If $X_f$ is finitely generated as a $\Lambda(H)$-module, then ${l = \textup{rank}_{\Lambda(H)}(X_f)}$.
	\end{compactenum}
\end{thm}
\begin{proof}
The proof of $(a)$ follows as the proof of Theorem \ref{thm:cuoco-monsky}(a) where we need to use the control theorem for Selmer groups (Theorem~\ref{TorsionControlThm}). In the proof (cf. proof of \cite[Lemma~6.7]{MHG}) we need to show that the $\mu$- and $\lambda$-invariants of the kernel and cokernel of the maps in the control theorem are bounded. This follows from Lemma~\ref{invaraints bound_lemma}. The proof of part~$(b)$ is identical to the proof of Theorem~\ref{thm:cuoco-monsky}$(b)$.\\

For the proof of part $(c)$, assume that  $X_f$ is finitely generated as a $\Lambda(H)$-module. Noting that $\Lambda(G)$ is a free $\Lambda(G_n)$-module of rank $p^n$ and taking into account Lemma~\ref{invaraints bound_lemma}, we get from Theorem~\ref{mu_invariants_theorem2} that
$$\mu_n = \mu_G(T_{\Lambda_2}(X)) p^n + \mathcal{O}(1). $$
Therefore ${m_1=0}$.
The proof of $(d)$ follows as the proof of Theorem~\ref{thm:cuoco-monsky}$(d)$ where we need to use the control theorem for Selmer groups (Theorem~\ref{TorsionControlThm}) while taking Lemma~\ref{invaraints bound_lemma} into account.
\end{proof}

We finally state the analogue of Theorem~\ref{thm:cuoco-monsky} for fine Selmer groups. Let ${Y_n=R_{p^{\infty}}(E/F_n)^{\vee}}$  and let ${Y=R_{p^{\infty}}(E/\KK_{\infty})^{\vee}}$. We assume that $Y_n$ is a torsion module over ${\Lambda_n \cong \Z_p[[T]]}$ for each $n$. In view of Lemma~\ref{WL_lemma} it suffices to assume that $Y(E/\KK_\infty)_f$ is a finitely generated $\Lambda(H)$-module. Let $\tilde{\mu}_n$ and $\tilde{\lambda}_n$ denote the Iwasawa invariants of the $\Lambda_n$-module $Y_n$. As in Section~\ref{section:mainintroduction} we write the characteristic power series ${\tilde{f}_\infty \in \Z_p[[T_1, T_2]]}$ of $Y$ as ${\tilde{f}_\infty = p^{\mu_G(Y)} \cdot \tilde{g}_\infty}$, and we denote by ${\overline{\tilde{g}_\infty} \in \Z_p[[T_1, T_2]]/p}$ the residue class modulo $p$.

\begin{thm} \label{thm3:cuoco-monsky}
	Suppose that ${H \in \mathcal{H}^{\bullet\star}}$. We assume that $Y_n$ is a torsion $\Lambda_n$-module for each $n$. Then the following assertions hold. \begin{compactenum}[(a)]
		\item We have
		\[ \tilde{\mu}_n = \mu_G(Y) p^n + \tilde{m}_1 n + \mathcal{O}(1)\]
		and
		\[ \tilde{\lambda}_n = \tilde{l} p^n + \mathcal{O}(1)\]
		for suitable constants $\tilde{m}_1$ and $\tilde{l}$.
		\item If ${H = \overline{\langle \sigma^a \tau^b \rangle}}$ and ${\Upsilon = (1+T_1)^a (1+T_2)^b - 1}$ (where ${T_1 = \sigma - 1}$ and ${T_2 = \tau - 1}$ as usual), then
		\[ \tilde{m}_1 = v_{\overline{\Upsilon}}(\overline{\tilde{g}_\infty}), \]
		i.e. $\tilde{m}_1$ equals the number of factors $\overline{\Upsilon}$ which divide $\overline{\tilde{g}_\infty}$ in $\Z_p[[T_1, T_2]]/p$.
		\item If $Y_f$ is finitely generated as a $\Lambda(H)$-module, then ${\tilde{m}_1 = 0}$.
        \item  If $Y_f$ is finitely generated as a $\Lambda(H)$-module, then ${\tilde{l} = \textup{rank}_{\Lambda(H)}(Y_f)}$.
	\end{compactenum}
\end{thm}
\begin{proof}
The proof is identical to the proofs of Theorems~\ref{thm:cuoco-monsky} and \ref{thm2:cuoco-monsky} using our exact control theorem Proposition~\ref{fineSelmercontrol_prop1}.
\end{proof}

\part{Proofs of the main results}
We can now prove the main results (see Section~\ref{section:mainintroduction} for the precise statements).
\section{Proof of Theorem~\ref{main_theorem}} \label{section:main-theorem}
In this section, we prove our first main result (Theorem~\ref{main_theorem}).
We first collect a couple of auxiliary results. Recall from Section~\ref{section:mainintroduction} the Greenberg topology on $\mathcal{E}$: For any ${L \in \mathcal{E}}$ and ${n \in \N}$ the set $\mathcal{E}(L,n)$ consists of the $\Z_p$-extensions of $K$ which coincide with $L$ at least up to the $n$-th layer. In order to control the variation of the Iwasawa invariants of signed Selmer groups in Greenberg neighborhoods, we first need the following
\begin{proposition} \label{prop:neighborhood}
Let ${\bullet, \star \in \{+,-\}}$, ${H \in \mathcal{H}^{\bullet\star}}$ and ${L_{\infty} = \KK_{\infty}^H}$. Then there exists an integer $n$ such that ${\Gal(\KK_{\infty}/L'_{\infty}) \in \mathcal{H}^{\bullet\star}}$ for each ${L'_{\infty} \in \mathcal{E}(L_{\infty},n)}$.
\end{proposition}
\begin{proof}
	Since ${H \in \mathcal{H}^{\bullet\star}}$, we know that \begin{enumerate}[(a)]
		\item no prime in $S$ splits completely in $L_{\infty}/K$,
		\item every prime of $K$ above $p$ is totally ramified in $L_{\infty}/K$, and
		\item $X^{\bullet\star}(E/L_{\infty})$ is a torsion $\Lambda$-module.
	\end{enumerate}
Write ${L_{\infty} = \bigcup_n L_n}$, and choose ${n \ge 1}$ large enough such that no prime in $S$ splits completely in $L_{n}/K$ and such that no prime above $p$ is unramified in $L_n/K$. Then any ${L'_{\infty} \in \mathcal{E}(L_{\infty},n)}$ satisfies the first two conditions~$(a)$ and $(b)$. It follows from Lemma~\ref{Xinf_torsion_lemma} and Proposition~\ref{Torsion_prop} that condition~(c) actually excludes only finitely many ${L'_{\infty} \in \mathcal{E}}$. Therefore, by increasing $n$ appropriately, we can ensure that ${\mathcal{E}(L_{\infty},n) \subseteq \mathcal{H}^{\bullet\star}}$.
\end{proof}

Next, we turn to the $\lambda$-invariants. The following proposition establishes a connection between
conditions~$(g)$ and $(h)$ in Theorem~\ref{main_theorem}. Recall that if ${\bullet, \star \in \{+,-\}}$, then ${f_{\infty}^{\bullet\star} = p^{m^{\bullet\star}} \cdot g_{\infty}^{\bullet\star} \in \Lambda(G)}$ denotes the characteristic power series of $X^{\bullet\star}(E/\KK_{\infty})$ and ${p \nmid g_{\infty}^{\bullet\star}}$. We also recall that ${\Lambda(G) \cong \Z_p[[T_1,T_2]]}$, where the two topological generators ${\sigma, \tau \in G}$ correspond to ${T_1+1}$ and ${T_2+1}$.

Using a result of Monsky from \cite{Monsky}, we deduce the following
\begin{proposition}\label{SelmerMonsky_prop}
Let ${\bullet, \star \in \{+,-\}}$. 
Let ${H=\overbar{\langle \sigma^a\tau^b \rangle} \in \mathcal{H}^{\bullet\star}}$.  Then $\lambda(X^{\bullet\star}(E/L))$ is bounded in a neighborhood of $\KK_{\infty}^H$ if and only if the image of $g_{\infty}^{\bullet\star}$ in $\Lambda_2/p$ is not divisible by the coset of ${(1+T_1)^a(1+T_2)^b-1}$.
\end{proposition}
\begin{proof}
	Let ${H' \in \mathcal{H}^{\bullet\star}}$. According to Theorem~\ref{SelmerControlThm_prop}, the maps
	\[ X^{\bullet\star}(E/\KK_{\infty})_{H'} \longrightarrow X^{\bullet\star}(E/\KK_{\infty}^{H'})\]
	are isomorphisms. Therefore, considering both the domain and codomain as $\Lambda(G/H')$-modules, we see that their $\lambda$-invariants are equal. Taking Proposition~\ref{prop:neighborhood} into account, the desired result follows from \cite[Theorem~3.3]{Monsky}.
\end{proof}

\begin{proposition} \label{prop:new}
Let ${\bullet, \star \in \{+,-\}}$. For any ${H = \overline{\langle \sigma^a \tau^b \rangle} \in \mathcal{H}^{\bullet\star}}$, $X^{\bullet\star}(E/\KK_{\infty})_f$ is finitely generated as a $\Lambda(H)$-module if and only if the image of $g_{\infty}^{\bullet\star}$ in $\Lambda_2/p$ is not divisible by the coset of ${(1+T_1)^a (1+T_2)^b - 1}$.
\end{proposition}
\begin{proof}
	Since ${H \in \mathcal{H}^{\bullet\star}}$ we know that ${g_\infty^{\bullet\star} \ne 0}$. Suppose first that $X^{\bullet\star}(E/\KK_{\infty})_f$ is a finitely generated $\Lambda(H)$-module, and let ${L_{\infty} = \KK_{\infty}^H}$. Then it follows from Theorem~\ref{mu_invariants_theorem} that for all $n \geq 0$, $X^{\bullet\star}(E/\KK_{\infty}^{H_n})$ is a torsion $\Lambda_{H,n}$-module, and we have $$\mu_{G_n}(X^{\bullet\star}(E/\KK_{\infty}))=\mu_{G_n/H_n}(X^{\bullet\star}(E/\KK_{\infty}^{H_n})). $$
	This means that we have ${m_1^{\bullet\star} = 0}$ in Theorem~\ref{thm:cuoco-monsky}(a). It now follows from the characterization of $m_1^{\bullet\star}$ in Theorem~\ref{thm:cuoco-monsky}(b) that the image of $g_{\infty}^{\bullet\star}$ in $\Lambda_2/p$ is not divisible by the coset of ${(1+T_1)^a(1+T_2)^b-1}$.
	
	On the other hand, suppose that the image of $g_{\infty}^{\bullet\star}$ in $\Lambda_2/p$ is not divisible by the coset of ${(1+T_1)^a (1+T_2)^b - 1}$. Then we use the completely algebraic proof of \cite[Proposition~7.1]{MHG} in order to deduce that $X^{\bullet\star}(E/\KK_{\infty})$ is finitely generated over $\Lambda(H)$.
\end{proof}

We now prove Theorem~\ref{main_theorem}.
\begin{proof}[Proof of Theorem~\ref{main_theorem}]
Let ${\bullet, \star \in \{+,-\}}$ and ${H \in \mathcal{H}^{\bullet\star}}$. We have that $X^{\bullet\star}(E/\KK_{\infty})$ is $\Lambda_2$-torsion by Lemma~\ref{Xinf_torsion_lemma}. \\
${(c) \Longrightarrow (b) \Longleftrightarrow (a)}$: Proposition~\ref{MHG_cohomology_group_prop}

${(i) \Longrightarrow (c)}$: If $(i)$ holds then we have an exact sequence
$$0 \to X^{\bullet\star}(E/\KK_{\infty})_f \to \Lambda(H)^{\lambda_H^{\bullet\star}} \to C \to 0$$ where $C$ is finite. This sequence induces the exact sequence
$$0 \to C[p] \to X^{\bullet\star}(E/\KK_{\infty})_f/p \xrightarrow{\theta} (\Lambda(H)/p)^{\lambda_H^{\bullet\star}} \to C/p\to 0.$$
Since ${cd_p(H)=1}$, the functor $H_1(H, -)$ is left exact. So as ${H_1(H, (\Lambda(H)/p)^{\lambda_H^{\bullet\star}})=0}$, it follows that ${H_1(H, \img \theta)=0}$. Therefore we see that ${H_1(H, X^{\bullet\star}(E/\KK_{\infty})_f/p) \cong H_1(H, C[p])}$. Since $C$ is finite, $(c)$ holds.

${(a) \Longleftrightarrow (d) \Longleftrightarrow (e)}$: follows from Proposition~\ref{Seln_surjective_prop} and Theorem~\ref{mu_invariants_theorem}, noting that, since $\Lambda(G)$ is a free $\Lambda(G_n)$-module of rank $p^n$, we have ${p^n\mu_G(X^{\bullet\star}(E/\KK_{\infty}))=\mu_{G_n}(X^{\bullet\star}(E/\KK_{\infty}))}$.

${(a) \Longleftrightarrow (g) \Longleftrightarrow (h)}$: follows from Propositions~\ref{SelmerMonsky_prop} and \ref{prop:new}. 

${(f) \Longrightarrow (d)}$: Clear.

${(i) \Longrightarrow (a)}$: Clear.

${(a) \Longrightarrow (f)}$: Proposition~\ref{p-primary_structure_prop}.

${(a) \Longrightarrow (i)}$:  Theorem~\ref{Xinf_structure_theorem}.

${(h) \Longleftrightarrow (j)}$: It is clear that $(j)$ implies $(h)$. Now suppose that $(h)$ holds. Then it follows from Theorem~\ref{thm:strong_local_behavior} that in fact both the $\mu$- and $\lambda$-invariants are constant (in a potentially smaller neighbourhood of $\KK_\infty^H$).
\end{proof}

\begin{proof}[Proof of Proposition~\ref{almost_all_prop}]
Recall from Proposition~\ref{Zpextensions_prop} that ${H \in \mathcal{H}^{\bullet\star}}$ for all but finitely many ${H \subseteq G}$ satisfying ${G/H \cong \Z_p}$. Let ${\Omega_2=\Lambda_2/p\Lambda_2(\cong \Fp[[T, U]])}$. This is a UFD, so by Theorem~\ref{main_theorem}(g), we only need to prove the following statement: If $[(a,b)]$ and $[(c,d)]$ are two distinct elements of $\mathbb{P}^1(\Zp)$, then ${\sigma^a\tau^b-1}$ and ${\sigma^c\tau^d-1}$ are relatively prime elements of $\Omega_2$. This can be proven in the same way as \cite[Lemma~2.4]{MHG}.
\end{proof}

\section{Proof of the first part of Theorem~\ref{main_theorem2} and Theorems \ref{2dimWingberg_theorem1}, \ref{2dimWingberg_theorem2} and ~\ref{equivalence_theorem}}\label{section:main-theorem2a}
As in Section~\ref{section:main-theorem}, we deduce from a result of Monsky from \cite{Monsky} the following
\begin{proposition}\label{FineSelmerMonsky_prop}
Let ${\bullet, \star \in \{+,-\}}$. 
Let ${H=\overbar{\langle \sigma^a\tau^b \rangle} \in \mathcal{H}^{\bullet\star}}$.  Then $\lambda(Y(E/L))$ is bounded in a neighborhood of $\KK_{\infty}^H$ if and only if the image of $\tilde{g}_{\infty}$ in $\Lambda_2/p$ is not divisible by the coset of ${(1+T_1)^a(1+T_2)^b-1}$.
\end{proposition}
\begin{proof}
	Let ${H' \in \mathcal{H}^{\bullet\star}}$. According to Proposition \ref{fineSelmercontrol_prop1}, the maps
	\[ Y(E/\KK_{\infty})_{H'} \longrightarrow Y(E/\KK_{\infty}^{H'})\]
	are isomorphisms. Therefore, considering both the domain and codomain as $\Lambda(G/H')$-modules, we see that their $\lambda$-invariants are equal. Taking Proposition~\ref{prop:neighborhood} into account, the desired result follows from \cite[Theorem~3.3]{Monsky}.
\end{proof}

We can now prove the first part of Theorem~\ref{main_theorem2}.
\begin{proof}[Proof of implications ${(a) \Longleftrightarrow (b) \Longleftrightarrow (c)}$ of Theorem~\ref{main_theorem2}] $\,$ \\
$Y(E/\KK_{\infty})$ is a torsion $\Lambda_2$-module by Lemma~\ref{Yinf_torsion_lemma}.

${(a) \Longrightarrow (b)}$: Theorem~\ref{thm3:cuoco-monsky}.

${(b) \Longrightarrow (a)}$: This follows from the completely algebraic proof of \cite[Proposition~7.1]{MHG}.

${(b) \Longleftrightarrow (c)}$: We get the required equivalence from Proposition~\ref{FineSelmerMonsky_prop}.
\end{proof}

Now we also need the following result for $\lambda$-invariants of Selmer groups which we deduce from Monsky's paper \cite{Monsky}.
\begin{proposition}\label{SelmerMonsky_prop2}
Let ${\bullet, \star \in \{+,-\}}$. 
Let ${H=\overbar{\langle \sigma^a\tau^b \rangle} \in \mathcal{H}^{\bullet\star}}$.  Then $\lambda(T_{\Lambda}(X(E/L)))$ is bounded in a neighborhood of $\KK_{\infty}^H$ if and only if the image of $g_{\infty}$ in $\Lambda_2/p$ is not divisible by the coset of ${(1+T_1)^a(1+T_2)^b-1}$.
\end{proposition}
\begin{proof}
	Let ${H' \in \mathcal{H}^{\bullet\star}}$. Taking Lemma~\ref{invaraints bound_lemma} into account, we see that according to Theorem~\ref{TorsionControlThm} the maps
	\[ T_{\Lambda_2}(X(E/\KK_{\infty}))_{H'} \longrightarrow T_{\Lambda}(X(E/\KK_{\infty}^{H'}))\]
	are injections whose cokernels have bounded $\lambda$-invariants. Therefore, considering both the domain and codomain as $\Lambda(G/H')$-modules, we see that the difference of their $\lambda$-invariants is bounded. Taking Proposition~\ref{prop:neighborhood} into account, the desired result follows from \cite[Theorem~3.3]{Monsky}.
\end{proof}

\begin{proposition}\label{Monskyequiv_prop}
Let ${\bullet, \star \in \{+,-\}}$ and ${H=\overbar{\langle \sigma^a\tau^b \rangle}\in \mathcal{H}^{\bullet\star}}$. The following statements are equivalent:
\begin{enumerate}[(a)]
\item $T_{\Lambda_2}(X(E/\KK_{\infty}))_f$ is finitely generated over $\Lambda(H)$.
\item The image of $g_{\infty}$ in $\Lambda_2/p$ is not divisible by the coset of ${(1+T_1)^a(1+T_2)^b-1}$.
\item $\lambda_{G/H}(T_{\Lambda}(X(E/L_\infty)))$ is bounded on a neighbourhood $U$ of $\KK_\infty^H$.
\end{enumerate}
\end{proposition}
\begin{proof}
${(a) \Longrightarrow (b)}$: Theorem~\ref{thm2:cuoco-monsky}.

${(b) \Longrightarrow (a)}$: This follows from the completely algebraic proof of \cite[Proposition~7.1]{MHG}.

${(b) \Longleftrightarrow (c)}$: We get the required equivalence from Proposition \ref{SelmerMonsky_prop2}.
\end{proof}

Let $\Gamma$ be algebraically and topologically isomorphic to $\Z_p$. We identify the completed group ring $\Z_p[[\Gamma]]$ with $\Lambda$. For any finitely generated $\Lambda$-module $Z$ we denote by $\dot{Z}$ the $\Lambda$-module $Z$ with the inverse $\Gamma$-action, i.e. ${\gamma \cdot z = \gamma^{-1}z}$ for any ${z \in Z}$ and ${\gamma \in \Gamma}$. Using this notation we have the following result which can be derived easily from \cite[Theorem~1.1]{Matar_Torsion}.

\begin{proposition} \label{fineSelmerisom_prop}
	Let ${\bullet, \star \in \{+,-\}}$ and let ${L_\infty = \KK_\infty^H}$ be a $\Zp$-extension of $K$ such that ${H \in \mathcal{H}^{\bullet\star}}$. Then $Y(E/L_\infty)$ is a $\Lambda$-torsion module, and there exists a pseudo-isomorphism
	\[ T_{\Lambda}({\dot{X}}(E/L_\infty)) \sim Y(E/L_\infty). \]
	Furthermore, if $X^{\bullet\star}(E/\KK_\infty)_f$ is finitely generated over $\Lambda(H)$, then $Y(E/\KK_\infty^{H_n})$ is a torsion $\Lambda$-module and
	\[ T_{\Lambda}({\dot{X}}(E/\KK_\infty^{H_n})) \sim Y(E/\KK_\infty^{H_n}) \]
	for all ${n \geq 0}$.
\end{proposition}
\begin{proof}
	We have ${E(\KK_\infty^{H_n})[p^\infty] = 0}$ for all $n$ by Lemma~\ref{p-torsion_lemma} and ${H^2(G_S(\KK_\infty^{H_n}), E[p^\infty]) = 0}$ by Proposition~\ref{Seln_surjective_prop} (in order to apply the proposition in the ${n > 0}$ case we have to assume that $X^{\bullet\star}(E/\KK_\infty)_f$ is finitely generated over $\Lambda(H)$). Therefore we get the desired results from \cite[Theorem~1.1]{Matar_Torsion}.
\end{proof}
\begin{remark}\label{inverting_action_remark}
Let ${\Gamma \cong \Z_p}$, let ${\Lambda = \Z_p[[T]] \cong \Z_p[[\Gamma]]}$ and suppose that $Z$ is a finitely generated and torsion $\Lambda$-module. Then
\begin{compactenum}[(a)]
	\item ${\mu(\dot{Z}) = \mu(Z)}$, and
	\item ${\lambda(\dot{Z})=\lambda(Z)}$.
\end{compactenum}
\end{remark}
\begin{proof}
Let $f$ be the characteristic polynomial of $Z$. Then it is easy to show that the characteristic polynomial of $\dot{Z}$ is ${u(1+T)^{\deg(f)} \cdot f((1+T)^{-1}-1)}$ for some unit ${u \in \Lambda}$. The assertions follow from this observation.
\end{proof}

As a first application of Proposition~\ref{fineSelmerisom_prop} we can now prove Theorem~\ref{equivalence_theorem} from Section~\ref{section:mainintroduction}, which provides implications between the $\mathfrak{M}_H(G)$-properties of signed Selmer groups, Selmer groups and fine Selmer groups.
\begin{proof}[Proof of Theorem~\ref{equivalence_theorem}]
${(a) \Longrightarrow (b)}$: Proposition~\ref{cor:hamidi}(c).

${(b) \Longleftrightarrow (c)}$: Let ${\bullet, \star \in \{+,-\}}$ and ${H=\overbar{\langle \sigma^a\tau^b \rangle}\in \mathcal{H}^{\bullet\star}}$. From Proposition~\ref{fineSelmerisom_prop} and Remark~\ref{inverting_action_remark} we have that ${\lambda(Y(E/\KK_{\infty}^H))=\lambda(T_{\Lambda}(X(E/\KK_{\infty}^H))}$. Taking Proposition~\ref{prop:neighborhood} into account, the equivalence of $(b)$ and $(c)$ therefore follows from Theorem~\ref{main_theorem2} and Proposition~\ref{Monskyequiv_prop}.
\end{proof}

\begin{corollary}\label{mu_invariants_corollary2}
		Let ${\bullet, \star \in \{+,-\}}$ and ${H=\overbar{\langle \sigma^a\tau^b \rangle}\in \mathcal{H}^{\bullet\star}}$. The three assertions~$(a)$, $(b)$ and $(c)$ from Theorem~\ref{mu_invariants_theorem2} are equivalent.
	\end{corollary}
	\begin{proof}
		We already know from Theorem~\ref{mu_invariants_theorem2} that ${(b) \Longrightarrow (c) \Longrightarrow (a)}$ and that $(a)$ implies $(b)$ provided that $Y(E/\KK_\infty)_f$ is finitely generated over $\Lambda(H)$. But the latter hypothesis is automatically satisfied under condition~$(a)$ by Theorem~\ref{equivalence_theorem}.
\end{proof}

A result analogous to Proposition~\ref{fineSelmerisom_prop} under a slightly different set of hypotheses has been obtained by Wingberg through a different approach in~\cite{Wingberg}. We find it remarkable that we can derive a partial result (Theorems \ref{2dimWingberg_theorem1} and \ref{2dimWingberg_theorem2}) in that direction on the $\Z_p^2$-level by using the asymptotic formulas from Section~\ref{section:monsky}.

\begin{proof}[Proof of Theorems~\ref{2dimWingberg_theorem1} and \ref{2dimWingberg_theorem2}]
	We will use the notation from Theorems~\ref{thm2:cuoco-monsky} and \ref{thm3:cuoco-monsky}. If $Y(E/\KK_\infty)_f$ is finitely generated as a $\Lambda(H)$-module \textit{for some} ${\bullet, \star \in \{+,-\}}$ and some ${H \in \mathcal{H}^{\bullet\star}}$, then Lemma~\ref{WL_lemma} implies that $Y_n$ is a torsion $\Lambda_n$-module for all $n$. Therefore the asymptotic formulas from Theorem~\ref{thm2:cuoco-monsky}$(a)$ and Theorem~\ref{thm3:cuoco-monsky}$(a)$ hold true. In view of Proposition~\ref{fineSelmerisom_prop} and Remark~\ref{inverting_action_remark} we have ${\mu_n = \tilde{\mu}_n}$ and ${\lambda_n = \tilde{\lambda}_n}$ for all $n$. Therefore the constants in the asymptotic formulas have to coincide. This will prove Theorem~\ref{2dimWingberg_theorem1} if we can show that there exist ${\bullet, \star \in \{+,-\}}$ and ${H \in \mathcal{H}^{\bullet\star}}$ such that $Y(E/\KK_\infty)_f$ is finitely generated as a $\Lambda(H)$-module. We choose ${\bullet=\star=+}$. Then the existence of such an $H$ follows from Propositions~\ref{Hcyc_prop}, \ref{almost_all_prop} and Theorem~\ref{equivalence_theorem}. If we now let ${H \in \mathcal{H}^{\bullet\star}}$ be as in Theorem~\ref{2dimWingberg_theorem2}, the above observations imply the first part of Theorem~\ref{2dimWingberg_theorem2} since ${m_1 = v_{\bar{\Upsilon}}(\overline{g_\infty})}$ and ${\tilde{m_1} = v_{\bar{\Upsilon}}(\overline{\tilde{g}_\infty})}$ by Theorem~\ref{thm2:cuoco-monsky}$(b)$ and Theorem~\ref{thm3:cuoco-monsky}$(b)$.
	
	We can conclude via Proposition~\ref{Monskyequiv_prop} that $X(E/\KK_\infty)_f$ is finitely generated over $\Lambda(H)$ if and only if ${v_{\bar{\Upsilon}}(\overline{g_\infty}) = 0}$. On the other hand, it follows from the already proven part of Theorem~\ref{main_theorem2} (to be more precise, we use the equivalence ${(a) \Longleftrightarrow (b)}$ in Theorem~\ref{main_theorem2}, which has been shown at the beginning of the current section) that $Y(E/\KK_\infty)_f$ is finitely generated over $\Lambda(H)$ if and only if ${v_{\bar{\Upsilon}}(\overline{\tilde{g}_\infty}) = 0}$.
	
The second part of Theorem~\ref{2dimWingberg_theorem2} follows from Theorem~\ref{thm2:cuoco-monsky}$(d)$ and Theorem~\ref{thm3:cuoco-monsky}$(d)$, since ${l = \tilde{l}}$ in the asymptotic formulas.
\end{proof}
\begin{remark}
	Note that the above theorem is quite close to a two-dimensional analogue of Proposition~\ref{fineSelmerisom_prop}. If $f_\infty$ and $\tilde{f}_\infty$ denote the characteristic power series of $T_{\Lambda_2}(X(E/\KK_\infty))$ and $Y(E/\KK_\infty)$, then we have shown that the valuations of $f_\infty$ and $\tilde{f}_\infty$ with regard to $p$ and the prime elements $\Upsilon$ attached to some ${H \in \mathcal{H}^{\bullet\star}}$ coincide. In other words, if there are no further prime divisors of $f_\infty$ or $\tilde{f}_\infty$ in the UFD $\Lambda_2$, then these two characteristic power series are associated.
\end{remark}

For the proof of the remaining assertions from Theorem~\ref{main_theorem2}, parts of Theorem~\ref{main_theorem3} are used. Therefore these proofs are postponed to the next section.

\section{Proof of Theorem~\ref{main_theorem3} and the second part of Theorem~\ref{main_theorem2}} \label{section:main-theorem2}

We have the following important lemma.
\begin{lemma}\label{lemma:fukuda-for-selmer}
	Let ${\bullet, \star \in \{+,-\}}$ and ${H \in \mathcal{H}^{\bullet\star}}$. Then there exists a Greenberg open neighbourhood $U$ of $\KK_\infty^H$ such that the following assertions hold.
	\begin{compactenum}[(a)]
		\item $\mu(T_\Lambda(X(E/L_\infty))) \le \mu(T_\Lambda(X(E/\KK_\infty^H)))$ for each ${L_\infty \in U}$, and
		\item there exists a constant $C$ such that ${\lambda(T_\Lambda(X(E/L_\infty))) \le C}$ for each ${L_\infty \in U}$ which satisfies ${\mu(T_\Lambda(X(E/L_\infty))) = \mu(T_\Lambda(X(E/\KK_\infty^H)))}$.
	\end{compactenum}
\end{lemma}
\begin{proof}
	As ${H \in \mathcal{H}^{\bullet\star}}$ no prime of $S$ is totally split in $\KK_\infty^H/K$. Moreover, by Lemma \ref{p-torsion_lemma} we have ${E(\KK_{\infty,w}^H)[p^\infty] = 0}$ for each prime $w$ above $p$. Therefore the hypotheses from \cite[Theorems~5.11 and 5.15]{fukuda-selmer} are satisfied. Therefore we can choose a neighbourhood $U$ of $\KK_\infty^H$ such that \begin{compactenum}[(i)]
		\item $\mu(Y(E/L_\infty)) \le \mu(Y(E/\KK_\infty^H))$ for each ${L_\infty \in U}$, and
		\item $\lambda(Y(E/L_\infty)) \le \lambda(Y(E/\KK_\infty^H))$ for each ${L_\infty \in U}$ which satisfies ${\mu(Y(E/L_\infty)) = \mu(Y(E/\KK_\infty^H))}$.
	\end{compactenum}
	The assertion follows from Proposition~\ref{fineSelmerisom_prop} since reversing the $\Gamma$-action on a finitely generated torsion $\Lambda$-module $Z$ does not change its $\mu$-invariant. Moreover, if the $\lambda$-invariants of a given set of such $\Lambda$-modules $Z$ are bounded, then the $\lambda$-invariants of the modules $\dot{Z}$ are still bounded (see Remark \ref{inverting_action_remark}).
\end{proof}

\begin{lemma}\label{mu_inequality_lemma}
Let ${\bullet, \star \in \{+,-\}}$ and ${H=\overbar{\langle \sigma^a\tau^b \rangle}\in \mathcal{H}^{\bullet\star}}$. We have ${\mu_{G/H}(T_{\Lambda_2}(X(E/\KK_\infty))_H) \geq \mu_G(T_{\Lambda_2}(X(E/\KK_\infty)))}$.
\end{lemma}
\begin{proof}
If ${\mu_G(T_{\Lambda_2}(X(E/\KK_\infty)))=0}$, then the statement is true. Now assume that ${\mu_G(T_{\Lambda_2}(X(E/\KK_\infty))) \neq 0}$. 
By Proposition~\ref{cor:hamidi}(d) and \cite[Chapt.~VII, \S 4.4 Theorem~5]{Bourbaki}, there exist a pseudo-null $\Lambda_2$-module $A$, irreducible power series ${f_j \in \Zp[[T, U]]}$, integers $m_i$, $n_j$ and an exact sequence

\begin{equation}\label{torsion_exactseq1}
0 \to T_{\Lambda_2}(X(E/\KK_{\infty})) \to W \to A \to 0
\end{equation}
where ${W=\bigoplus_{i=1}^s \Lambda_2/p^{m_i} \oplus \bigoplus_{j=1}^t \Lambda_2/f_j^{n_j}}$ and ${f_{\infty}=p^{\sum_{i=1}^s m_i}\prod_{j=1}^t f_j^{n_j}}$.

Let ${\Upsilon=(1+T_1)^a(1+T_2)^b-1}$. From Corollary~\ref{cor:HinH} we have that $\Upsilon$ is relatively prime to $f_{\infty}$. Therefore from this fact we see that ${W^{\Upsilon=0}=0}$. It follows from this and the exact sequence~(\ref{torsion_exactseq1}) that we have an exact sequence
\begin{equation}\label{torsion_exactseq2}
0 \to A^{\Upsilon=0} \to T_{\Lambda_2}(X(E/\KK_{\infty}))/\Upsilon \to W/\Upsilon \to A/\Upsilon \to 0.
\end{equation}

From Corollary~\ref{cor:HinH} we have that $T_{\Lambda_2}(X(E/\KK_{\infty}))/\Upsilon$ is $\Lambda_{H,0}$-torsion. Also as $A$ is a pseudo-null $\Lambda_2$-module, $A$ has Krull dimension at most one. Therefore both ${A^{\Upsilon=0}}$ and $A/\Upsilon$ are $\Lambda_{H,0}$-torsion. So we see from these facts that all the modules appearing in the exact sequence~(\ref{torsion_exactseq2}) are $\Lambda_{H,0}$-torsion. Therefore the alternating sum of the $\mu_{G/H}$-invariants is zero. By Lemma~\ref{mu_pseudo-null_lemma} we have ${\mu_{G/H}(A^{\Upsilon=0})=\mu_{G/H}(A/\Upsilon)}$. So we see that
$$\mu_{G/H}(T_{\Lambda_2}(X(E/\KK_{\infty}))/\Upsilon)=\mu_{G/H}(W/\Upsilon). $$
It is easy to see that ${\mu_{G/H}(W/\Upsilon) \geq \mu_G(W)}$. This implies the intended result.
\end{proof}

\begin{proof}[Proof of Theorem~\ref{main_theorem3}]

First we prove statement (1):
${(a) \Longleftrightarrow (b)}$: Proposition~\ref{MHG_cohomology_group_prop2}.

${(a) \Longleftrightarrow (c) \Longleftrightarrow (d)}$: 
These equivalences follow from Corollary~\ref{mu_invariants_corollary2}, noting that, since $\Lambda(G)$ is a free $\Lambda(G_n)$-module of rank $p^n$, we have ${p^n\mu_G(T_{\Lambda_2}(E/\KK_{\infty}))=\mu_{G_n}(T_{\Lambda_2}(X(E/\KK_{\infty}))}$.

${(a) \Longleftrightarrow (e) \Longleftrightarrow (f)}$: Proposition \ref{Monskyequiv_prop}.

Thus we have now shown that all the statements~$(a)-(f)$ are equivalent.

\vspace{2mm}
Now we prove assertion~$(2)$ of Theorem~\ref{main_theorem3}.

${(g) \Longrightarrow (h)}$: From Theorem~\ref{TorsionControlThm} we know that
\begin{align} \label{eq:musum} \resizebox{.94\hsize}{!}{$\mu_{G/H}(T_{\Lambda_2}(X(E/\KK_\infty))_H) = \mu_{G/H}(T_{\Lambda}(X(E/\KK_\infty^H))) - \mu_{G/H}(T_{\Lambda}([F_{\Lambda_2}(X(E/\KK_{\infty}))]_H))$}. \end{align}

Since ${\mu_{G/H}(T_{\Lambda}([F_{\Lambda_2}(X(E/\KK_{\infty}))]_H)) \ge 0}$, it follows from equation~\eqref{eq:musum} that for each ${H \in \mathcal{H}^{\bullet\star}}$ we have an inequality
\[ \mu_{G/H}(T_{\Lambda}(X(E/\KK_\infty^H))) \ge \mu_{G/H}(T_{\Lambda_2}(X(E/\KK_\infty))_H). \]
Since by Lemma~\ref{mu_inequality_lemma} the right-hand side of this inequality is larger than or equal to $\mu_G(T_{\Lambda_2}(X(E/\KK_\infty)))$ we can conclude that
\begin{align} \label{eq:muinv} \mu_{G/H}(T_{\Lambda}(X(E/\KK_\infty^H))) \ge \mu_G(T_{\Lambda_2}(X(E/\KK_\infty)))
\end{align}
for each ${H \in \mathcal{H}^{\bullet\star}}$.

In other words, if condition~$(g)$ holds true, then the value of $\mu_{G/H}(T_{\Lambda}(X(E/\KK_\infty^H)))$ is as small as possible. If we choose a neighbourhood $U$ of $\KK_\infty^H$ such that $U$ is contained in $\mathcal{H}^{\bullet\star}$ (this is possible in view of Proposition~\ref{Zpextensions_prop}), then this means that
\[ \mu(T_\Lambda(X(E/L_\infty))) \ge \mu(T_\Lambda(X(E/\KK_\infty^H)))\]
for each ${L_\infty \in U}$. On the other hand, we can use Lemma~\ref{lemma:fukuda-for-selmer} in order to make $U$ small enough such that also the reverse inequalities
\[ \mu(T_\Lambda(X(E/L_\infty))) \le \mu(T_\Lambda(X(E/\KK_\infty^H)))\]
hold true for each ${L_\infty \in U}$. This implies that condition~$(h)$ holds true.

${(h) \Longrightarrow (e)}$: Lemma~\ref{lemma:fukuda-for-selmer}$(b)$.

(3) Proposition~\ref{mu_vanishing_prop}.

(4) Lemma~\ref{invaraints bound_lemma}.
\end{proof}

Now we can prove the remaining assertions from Theorem~\ref{main_theorem2}.
\begin{proof}[Second part of the proof of Theorem~\ref{main_theorem2}] $\,$ \\
	$(a) \Longrightarrow (e)$: If $Y(E/\KK_\infty)_f$ is finitely generated over $\Lambda(H)$, then Lemma~\ref{WL_lemma} implies that $Y(E/\KK_\infty^{H_n})$ is a torsion $\Lambda_{H,n}$-module for all n. Moreover, it follows from Theorem~\ref{equivalence_theorem} that $X(E/\KK_\infty)_f$ is finitely generated over $\Lambda(H)$. Therefore Theorem~\ref{main_theorem3}${(a) \Longrightarrow (d)}$ implies that
	$$
	\resizebox{.97\hsize}{!}{$p^n\mu_G(T_{\Lambda_2}(X(E/\KK_\infty))) = \mu_{G_n/H_n}(T_{\Lambda_{H,n}}(X(E/\KK_\infty^{H_n}))) - \mu_{G_n/H_n}(T_{\Lambda_{H,n}}([F_{\Lambda_2}(X(E/\KK_{\infty}))]_{H_n}))$}$$
	for all ${n \ge 0}$. Assertion~$(e)$ follows since
	\[ \mu_G(T_{\Lambda_2}(X(E/\KK_\infty))) = \mu_G(Y(E/\KK_\infty))\]
	by Theorem~\ref{2dimWingberg_theorem1} and
	\[ \mu_{G_n/H_n}(T_{\Lambda_{H,n}}(X(E/\KK_\infty^{H_n}))) = \mu_{G_n/H_n}(Y(E/\KK_\infty^{H_n}))\]
	by Proposition~\ref{fineSelmerisom_prop}.
	
	$(e) \Longrightarrow (d)$: Clear.
	
	$(d) \Longrightarrow (a)$: If $(d)$ holds, then Theorem~\ref{2dimWingberg_theorem1} and Proposition~\ref{fineSelmerisom_prop} imply that condition~$(c)$ in Theorem~\ref{main_theorem3} is satisfied. Therefore $X(E/\KK_\infty)_f$ is finitely generated as a $\Lambda(H)$-module, and condition~$(a)$ is satisfied by Theorem~\ref{equivalence_theorem}.
	
	$(f) \Longrightarrow (g)$: In view of Theorem~\ref{2dimWingberg_theorem1} and Proposition~\ref{fineSelmerisom_prop}, conditions~$(f)$ and $(g)$ correspond to conditions~$(g)$ and $(h)$ in Theorem~\ref{main_theorem3}, respectively. Therefore the desired implication follows from Theorem~\ref{main_theorem3}.
	
	$(g) \Longrightarrow (c)$: Again using Theorem~\ref{2dimWingberg_theorem1} and Proposition~\ref{fineSelmerisom_prop}, we can use the implication ${(h) \Longrightarrow (e)}$ from Theorem~\ref{main_theorem3}.
\end{proof}

Let ${\bullet, \star \in \{+,-\}}$ and ${H \in \mathcal{H}^{\bullet\star}}$. Proposition~\ref{cor:hamidi}$(c)$ states that if $X^{\bullet\star}(E/\K_\infty)_f$ is finitely generated over $\Lambda(H)$, then $X(E/\KK_\infty)_f$ is finitely generated over $\Lambda(H)$. We end this section with a necessary and sufficient criterion for the reverse implication. Let ${U:=J_p^{\bullet\star}(E/\KK_{\infty})^{\vee}}$. By Lemma \ref{Xinf_torsion_lemma} and Proposition~\ref{Sel_Lambda2_surjective_prop2} we have an exact sequence

$$0 \to U \xrightarrow{\phi^{\bullet\star}} X(E/\KK_{\infty}) \xrightarrow{\psi^{\bullet\star}} X^{\bullet\star}(E/\KK_{\infty}) \to 0.$$

Furthermore $U$ is a free $\Lambda_2$-module of rank two. Recall that we defined $F_{\Lambda_2}(X(E/\KK_\infty))$ to be $X(E/\KK_\infty)/T_{\Lambda_2}(X(E/\KK_\infty))$. Now consider the projection map
$$\pi: X(E/\KK_{\infty}) \to F_{\Lambda_2}(X(E/\KK_\infty)).$$
Since $U$ is a free $\Lambda_2$-module the map ${\pi \circ \phi^{\bullet\star}}$ is an injection. By Proposition~\ref{Selrank_prop} we have that $$\rank_{\Lambda_2}(F_{\Lambda_2}(X(E/\KK_{\infty})))=\rank_{\Lambda_2}(X(E/\KK_{\infty}))=2.$$
Therefore we see that
$$Z^{\bullet\star}(E/\KK_{\infty}):=F_{\Lambda_2}(X(E/\KK_{\infty}))/(\pi\circ \phi^{\bullet\star})(U)$$
is a torsion $\Lambda_2$-module.
Let ${\tilde{f}_{\infty}^{\bullet\star} \in \Lambda_2}$ be the characteristic power series of $Z^{\bullet\star}(E/\KK_{\infty})$. Since $Z^{\bullet\star}(E/\KK_\infty)$ is $\Lambda_2$-torsion we know that ${\tilde{f}_{\infty}^{\bullet\star} \ne 0}$. We write ${\tilde{f}_{\infty}^{\bullet\star}=p^{m}\tilde{g}_{\infty}^{\bullet\star}}$ where ${m=\mu_G(Z^{\bullet\star}(E/\KK_{\infty}))}$, so that ${p \nmid \tilde{g}_{\infty}^{\bullet\star}}$. 

\begin{proposition} \label{prop:MH(G)impliesMH(G)+-}
	Let ${\bullet, \star \in \{+,-\}}$ and ${H=\overbar{\langle \sigma^a\tau^b \rangle}\in \mathcal{H}^{\bullet\star}}$. Assume that $X(E/\K_\infty)_f$ is finitely generated over $\Lambda(H)$. Then $X^{\bullet\star}(E/\KK_\infty)_f$ is finitely generated over $\Lambda(H)$ if and only if the image of $\tilde{g}_{\infty}^{\bullet\star}$ in $\Lambda_2/p$ is not divisible by the coset of ${(1+T_1)^a(1+T_2)^b-1}$.
\end{proposition}
\begin{proof}
	Recall that ${f_{\infty}^{\bullet\star} \in \Lambda_2}$ is the characteristic power series of $X^{\bullet\star}(E/\KK_{\infty})$. As ${H \in \mathcal{H}^{\bullet\star}}$ we know that $X^{\bullet\star}(E/\KK_\infty)$ is a torsion $\Lambda_2$-module. Therefore ${f_\infty^{\bullet\star} \ne 0}$, and we write ${f^{\bullet\star}_{\infty}=p^{m^{\bullet\star}}g^{\bullet\star}_{\infty}}$ where ${m^{\bullet\star}=\mu_G(X^{\bullet\star}(E/\KK_{\infty}))}$, so that ${p \nmid g^{\bullet\star}_{\infty}}$. Also recall that ${f_{\infty} \in \Lambda_2}$ is the characteristic power series of $T_{\Lambda_2}(X(E/\KK_{\infty}))$ and that we write ${f_{\infty}=p^{k}g_{\infty}}$ where ${k=\mu_G(T_{\Lambda_2}(X(E/\KK_{\infty})))}$, so that ${p \nmid g_{\infty}}$.

	It is easy to see that the map $\psi^{\bullet\star}$ induces an isomorphism
	$$Z^{\bullet\star}(E/\KK_{\infty}) \cong X^{\bullet\star}(E/\KK_{\infty})/T_{\Lambda_2}(X(E/\KK_{\infty})).$$
	Therefore we have an exact sequence
	$$0 \to T_{\Lambda_2}(X(E/\KK_{\infty})) \to X^{\bullet\star}(E/\KK_{\infty}) \to Z^{\bullet\star}(E/\KK_{\infty}) \to 0.$$
	
	As in the proof of \cite[Proposition~15.22]{Washington} the above exact sequence shows that  ${g_{\infty}^{\bullet\star}=g_{\infty}\tilde{g}_{\infty}^{\bullet\star}}$. The result then follows from this observation and Theorems~\ref{main_theorem} and \ref{main_theorem3}.
\end{proof}

\section{Fine Selmer groups, Conjectures~A and B and the $\mathfrak{M}_H(G)$-property} \label{section:fineSelmer}
In this section, we prove Theorem~\ref{freestructure_theorem} and Propositions~\ref{propositionA}, \ref{propositionB} and \ref{propositionC}.

\begin{proof}[Proof of Proposition~\ref{propositionA}]
Recall from Proposition~\ref{Hcyc_prop} that ${H_{cyc} \in \mathcal{H}^{++} \cap \mathcal{H}^{--}}$; in what follows we will work with ${(\bullet\star) = (++)}$ and ${H = H_{cyc}}$. Suppose that Conjecture~\ref{conjectureA} holds true. Then Proposition~\ref{fineSelmerisom_prop} implies that ${\mu_{G/H}(T_{\Lambda}(X(E/\KK_\infty^H))) = 0}$. It follows from inequality~\eqref{eq:muinv} in the proof of Theorem~\ref{main_theorem3} above that also ${\mu_G(T_{\Lambda_2}(X(E/\KK_\infty))) = 0}$. This means that condition~$(g)$ from Theorem~\ref{main_theorem3} is satisfied for ${H = H_{cyc}}$. This proves that $X(E/\KK_\infty))_f$ is finitely generated over $\Lambda(H)$. From Theorem~\ref{equivalence_theorem} we also have that $Y(E/\KK_\infty))_f$ is finitely generated over $\Lambda(H)$. We can prove that $Y(E/\KK_\infty))_f$ is finitely generated over $\Lambda(H)$ in a more direct way by noting that since $Y(E/K_{cyc})$ is a finitely generated $\Zp$-module, we can derive from Proposition~\ref{fineSelmercontrol_prop1} that $Y(E/\KK_{\infty})_H$ is finitely generated over $\Zp$. Hence $Y(E/\KK_{\infty})$ is finitely generated over $\Lambda(H)$. This certainly implies that $Y(E/\KK_{\infty})_f$ is also finitely generated over $\Lambda(H)$.

Now we show that ${\mu_G(T_{\Lambda_2}(X(E/\KK_\infty)))=0}$. To show this, it will suffice by Theorem~\ref{2dimWingberg_theorem1} to show that ${\mu_G(Y(E/\KK_{\infty}))=0}$. Assume that ${\mu_G(Y(E/\KK_{\infty}))>0}$. Then taking into account \cite[Lemma~2.1]{MHG}, the structure theorem implies that for some ${m >0}$ we have an injection
$$\Lambda_2/p^m \hookrightarrow Y(E/\KK_{\infty}). $$
Since $Y(E/\KK_{\infty})$ is finitely generated over the Noetherian ring $\Lambda(H)$, it follows that ${\Lambda_2/p^m}$ is finitely generated over $\Lambda(H)$. This is clearly not true and hence we get a contradiction. Hence we do indeed have that ${\mu_G(Y(E/\KK_{\infty}))=0}$.
\end{proof}

\begin{proof}[Proof of Proposition~\ref{propositionB}]
Assume that $Y(E/\KK_{\infty})$ is a pseudo-null $\Lambda_2$-module and let ${H \in \mathcal{H}^{\bullet\star}}$. Then ${\tilde{f}_{\infty}=1}$ (the characteristic power series of $Y(E/\KK_{\infty})$). Therefore by Theorem~\ref{main_theorem2} $Y(E/\KK_{\infty})_f$ is finitely generated over $\Lambda(H)$. Then by Theorem~\ref{equivalence_theorem} we also get that $X(E/\KK_{\infty})_f$ is finitely generated over $\Lambda(H)$.

The following is an alternative longer argument to show that $X(E/\KK_{\infty})_f$ is finitely generated over $\Lambda(H)$. For any ${H' \in \mathcal{H}^{\bullet\star}}$ let ${\psi_{H'}: R_{p^{\infty}}(E/\KK_{\infty}^{H'}) \to R_{p^{\infty}}(E/\KK_{\infty})^{H'}}$ be the isomorphism of Proposition~\ref{fineSelmercontrol_prop1}. Dualising the map $\psi_{H'}$ we get an isomorphism ${\widehat{\psi_{H'}}: Y(E/\KK_{\infty})_{H'} \isomarrow Y(E/\KK_{\infty}^{H'})}$.

Since $Y(E/\KK_{\infty})$ is a pseudo-null $\Lambda_2$-module, the quotient $Y(E/\KK_{\infty})_f$ is also pseudo-null and hence has Krull dimension at most one.  Moreover, multiplication by $p$ is injective on $Y(E/\KK_{\infty})_f$ by definition. Therefore $Y(E/\KK_{\infty})_f/p$ is finite by \cite[Corollary~11.9]{AM}, i.e. in this case $Y(E/\KK_{\infty})_f$ is a finitely generated free $\Z_p$-module. Let ${H' \in \mathcal{H}^{\bullet\star}}$ and consider the commutative diagram with exact rows

\begin{equation*}
\xymatrix {
Y(E/\KK_{\infty})_{H'}  \ar[d]^{\psi_{H'}} \ar[r] & (Y(E/\KK_{\infty})_f)_{H'}  \ar[d]^{\theta_{H'}} \ar[r] &0\\
Y(E/\KK_{\infty}^{H'}) \ar[r] & Y(E/\KK_{\infty}^{H'})_f \ar[r] &0}
\end{equation*}

Since $\psi_{H'}$ is surjective, the map $\theta_{H'}$ is also surjective. It follows that

\begin{eqnarray*}
	\lambda(Y(E/\KK_{\infty}^{H'}))=\rank_{\Zp}(Y(E/\KK_{\infty}^{H'})_f) & \leq & \rank_{\Zp}((Y(E/\KK_{\infty})_f)_{H'})\\
	& \leq & \rank_{\Zp}(Y(E/\KK_{\infty})_f).
\end{eqnarray*}

So we see that $\lambda(Y(E/L_{\infty}))$ is bounded in a neighborhood of $\KK_{\infty}^H$. Taking Proposition~\ref{prop:neighborhood} into account and applying Proposition~\ref{fineSelmerisom_prop} we then see that $\lambda(T_{\Lambda}(X(E/L_\infty)))$ is bounded on a neighbourhood $U$ of $\KK_\infty^H$. Therefore from Theorem~\ref{main_theorem3} we have that $X(E/\KK_\infty)$ satisfies the $\mathfrak{M}_H(G)$-property.

Finally, since $Y(E/\KK_{\infty})$ is a pseudo-null $\Lambda_2$-module, we have ${\mu_G(Y(E/\KK_{\infty}))=0}$. Then Theorem \ref{2dimWingberg_theorem1} shows that ${\mu_G(T_{\Lambda_2}(X(E/\KK_\infty)))=0}$ as desired.
\end{proof}

\begin{proof}[Proof of Proposition~\ref{propositionC}]
Assume that $Y(E/K_{cyc})$ is finite. Then by Proposition~\ref{fineSelmerisom_prop} $T_{\Lambda}(X(E/K_{cyc}))$ is finite. Let ${H=H_{cyc}}$. It then follows from this and Theorem~\ref{TorsionControlThm} that $T_{\Lambda_2}(X(E/\KK_{\infty}))_H$ is finite and hence $T_{\Lambda_2}(X(E/\KK_{\infty}))$ is a finitely generated torsion $\Lambda(H)$-module. By \cite[Proposition~5.4]{Venjakob1} this implies that $T_{\Lambda_2}(X(E/\KK_{\infty}))$  is a pseudo-null $\Lambda_2$-module. Since by Proposition~\ref{cor:hamidi} $T_{\Lambda_2}(X(E/\KK_{\infty}))$ has no nontrivial pseudo-null submodules, we conclude that ${T_{\Lambda_2}(X(E/\KK_{\infty}))=0}$ as desired.
\end{proof}

\section{On the freeness of the Selmer group} \label{section:freeSelmer}
Now we prove Theorem~\ref{freestructure_theorem}. Let ${y_K \in E(K)}$ be the Heegner point defined before the statement of Theorem~\ref{freestructure_theorem}. First we need the following

\begin{theorem}\label{MN_theorem}
If ${y_K \notin pE(K)}$, then ${E(K)\otimes \Zp=\Zp(y_K\otimes 1) \cong \Zp}$ and ${\Sha(E/K)[p^{\infty}]=0}$.
\end{theorem}
\begin{proof}
Since $E$ has good supersingular reduction at $p$ from \cite[Prop.~12(c)]{Serre1} we have that $E[p]$ is an irreducible $\Gal(\Q(E[p])/\Q)$-module. Therefore the result follows from \cite[Theorem~6.7]{MN1} (see also \cite{MN2}).
\end{proof}

\begin{proof}[Proof of Theorem~\ref{freestructure_theorem}]
First we show that $(2)$ and $(3)$ follow from $(1)$. Assume that ${X(E/\KK_{\infty}) \cong (\Lambda_2)^2}$. Let ${\bullet, \star \in \{+,-\}}$ and let ${H \in \mathcal{H}^{\bullet\star}}$. By Proposition~\ref{SelmerControlThm_prop2} we see that ${X(E/\KK_{\infty}^H) \cong \Lambda^2}$. Therefore by Proposition~\ref{fineSelmerisom_prop} we have that $Y(E/\KK_{\infty}^H)$ is finite. Taking into account Proposition~\ref{Hcyc_prop} we get in particular that $Y(E/K_{cyc})$ is finite. This proves $(2)$. Let ${H=H_{cyc}}$. By Proposition~\ref{fineSelmercontrol_prop1} we have an isomorphism ${Y(E/\KK_{\infty})_H \cong Y(E/K_{cyc})}$. Since $Y(E/K_{cyc})$ is finite, the quotient $Y(E/\KK_{\infty})_H$ is finite and hence $Y(E/\KK_{\infty})$ is a finitely generated torsion $\Lambda(H)$-module. Therefore by \cite[Proposition~5.4]{Venjakob1} $Y(E/\KK_{\infty})$ is a pseudo-null $\Lambda_2$-module. This proves $(3)$.

We now prove (1). The proof of part~(1) of the theorem will be very similar to the proof of \cite[Theorem~3.2]{Matar_Tower}.  We begin the proof by studying the kernel and cokernel of the map ${s: \Selinf(E/K) \to \Selinf(E/\KK_{\infty})^G}$.

Define $S$ to be the set of primes of $K$ dividing $Np$ and $S_{\infty}$ to be the primes of $\KK_{\infty}$ above those in $S$.

Now consider the following commutative diagram
\begin{equation*}
	\begin{tikzcd}[column sep = small]
0 \arrow[r] & \Selinf(E/\KK_{\infty})^{G} \arrow[r] & H^1(G_S(\KK_{\infty}), E[p^{\infty}])^{G} \arrow[r]
& \left(\underset{v \in S_{\infty}}{\bigoplus} H^1(\KK_{\infty,v}, E)[p^{\infty}]\right)^{G} \\
0 \arrow[r] & \Selm_{p^{\infty}}(E/K) \arrow[u, "s"] \arrow[r] & H^1(G_S(K), E[p^{\infty}]) \arrow[u, "h"] \arrow[r, "\psi"]
& \underset{v \in S} {\bigoplus} H^1(K_v, E)[p^{\infty}] \arrow[u, "g"]
\end{tikzcd}
\end{equation*}

Applying the snake lemma to the above diagram we get

$$0 \to \ker s \to \ker h \to \ker g \cap \img \psi \to \coker s \to \coker h. $$\\
We have ${\ker h=H^1(G, E(\KK_{\infty})[p^{\infty}])}$ and we have an injection ${\coker h \hookrightarrow H^2(G, E(\KK_{\infty})[p^{\infty}])}$. By Lem\-ma~\ref{p-torsion_lemma} we have ${E(\KK_{\infty})[p^{\infty}]=\{0\}}$ and so the map $h$ is an isomorphism. Therefore from the above exact sequence we get that $s$ is an injection and that ${\coker s = \ker g \cap \img \psi}$.

We now analyze $\ker g$. Let $v$ be a prime of $K$ that does not divide $p$ and consider the map ${g_v: H^1(K_v, E)[p^{\infty}] \to \left(\oplus_{w | v} H^1(\KK_{\infty, w}, E)[p^{\infty}]\right)^G}$. By assumption~$(b)$ in the theorem we see that as in the proof of \cite[Proposition~10.3]{MHG} we have ${\ker g_v=0}$.

Now let ${v \in S_{\infty}}$ be a prime above $p$. Since $E$ has good supersingular reduction at $p$ and as every prime of $K$ above $p$ ramifies in $\KK_{\infty}/K$ it follows from \cite[Proposition~4.8]{CG} that ${H^1(\KK_{{\infty},v}, E)[p^{\infty}]=0}$.

The two observations above imply that ${\ker g = H^1(K_{\fp}, E)[p^{\infty}] \times H^1(K_{\bar{\fp}}, E)[p^{\infty}]}$. Therefore ${\coker s = \img \psi \cap \left(H^1(K_{\fp}, E)[p^{\infty}] \times H^1(K_{\bar{\fp}}, E)[p^{\infty}]\right)}$. Let $\psi'$ be the map ${\psi': H^1(G_S(K), E[p^{\infty}]) \to H^1(K_{\fp}, E)[p^{\infty}] \times H^1(K_{\bar{\fp}}, E)[p^{\infty}]}$ such that ${\coker s = \img \psi'}$.

We will show below that ${\coker s = \img \psi' \cong \Qp/\Zp}$. Let us explain how this implies the desired result. Since ${\Selinf(E/K) \cong \Qp/\Zp}$ by Theorem~\ref{MN_theorem} and ${\ker s = 0}$, we get an exact sequence
$$0 \to \Qp/\Zp \stackrel{s}{\to} \Selinf(E/\KK_{\infty})^G \to \Qp/\Zp \to 0. $$

Then taking the dual of the above sequence and noting that $\Zp$ is a projective $\Zp$-module we get ${X(E/\KK_{\infty})_G=\Zp \times \Zp}$. This implies by Nakayama's Lemma that $X(E/\KK_{\infty})$ is generated by two elements as a $\Lambda_2$-module and so ${X(E/\KK_\infty) \cong (\Lambda_2)^2/I}$ for some $\Lambda_2$-submodule $I$ of $(\Lambda_2)^2$. If ${I \neq 0}$ then ${\rank_{\Lambda_2}(X(E/\KK_{\infty})) \le 1}$. This contradicts Proposition~\ref{Selrank_prop}. Therefore ${I=0}$ and we get our desired result.

For any ${n > 0}$ the usual $p^n$-Selmer group is defined as
$$\displaystyle 0 \longrightarrow \Selm_{p^n}(E/K) \longrightarrow H^1(G_S(K), E[p^n])\longrightarrow \prod_{v\in S} H^1(K_v, E)[p^n]. $$

We now define a subgroup ${\Selm^{\dagger}_{p^n}(E/K) \subseteq \Selm_{p^n}(E/K)}$ as

$$\displaystyle 0 \longrightarrow \Selm^{\dagger}_{p^n}(E/K) \longrightarrow \Selm_{p^n}(E/K) \longrightarrow \prod_{v\in S \setminus\{\fp, \bar{\fp}\}} H^1(K_v, E[p^n]). $$

We study $\img \psi'$ by using the Cassels-Poitou-Tate exact sequence (see \cite[Theorem~1.7.3]{Rubin}):

$$H^1(G_S(K_n), E[p^{\infty}]) \xrightarrow{\psi'} \underset{i=1,2}{\bigoplus} H^1(K_{\fp_i}, E)[p^{\infty}] \xrightarrow{\uptheta} S_p^{\dagger}(E/K)^{\vee}, $$
where ${S_p^{\dagger}(E/K)^{\vee} = \ilim_n \Selm^{\dagger}_{p^n}(E/K)}$, ${\fp_1 = \fp}$ and ${\fp_2 = \bar{\fp}}$.

Define ${S_p(E/K) = \ilim_n \Selm_{p^n}(E/K)}$. We want to show that ${S_p^{\dagger}(E/K)^{\vee}=S_p(E/K)}$. From the definitions we see that in order to show this we need to prove that ${E(K_v)^*:=\ilim E(K_v)/p^n}$ (the $p$-adic completion of $E(K_v)$) is trivial for any ${v \in S \setminus \{\fp, \bar{\fp} \}}$.

Let ${v \in S \setminus \{\fp, \bar{\fp} \}}$. By Mattuck's Theorem we have that ${E(K_v) \cong \mathbb{Z}_l \times T}$, where $l$ is the residue characteristic of $K_v$ and $T$ is a finite group. Since ${l \ne p}$ we have ${E(K_v)^*=E(K_v)[p^{\infty}]}$. Assume that $E$ has additive reduction at $v$. Let $E_0(K_v)$ be the group of points of $E(K_v)$ of nonsingular reduction. Since ${p \geq 5}$, by \cite[pg.~448]{Silverman1} the cardinality of $E(K_v)/E_0(K_v)$ is prime to $p$. Therefore ${E(K_v)[p^{\infty}]=E_0(K_v)[p^{\infty}]}$. Let $\hat{E}$ be the formal group of $E/K_v$ and $\mathfrak{m}_v$ the maximal ideal of $O_v$ (the ring of integers of $K_v$). Since ${v \nmid p}$, $\hat{E}(\mathfrak{m}_v)$ has no elements of order $p$ and so ${E(K_v)[p^{\infty}]=E_0(K_v)[p^{\infty}]}$ injects into the set $\tilde{E}_{ns}(k_v)$ of nonsingular points of the reduction of $E$ modulo the residue field $k_v$. But as $E$ has additive reduction at $v$ we have ${\tilde{E}_{ns}(k_v)\cong k_v}$. Since $k_v$ has no points of order $p$ (because ${v \nmid p}$), we see that ${E(K_v)[p^{\infty}]=0}$.
	
Now assume that $E$ has multiplicative reduction at $v$. By $(b)$ ${p \nmid c_v}$ and by $(c)$ we have ${p \nmid \#k_v^{\times}}$ where $k_v$ is the residue field at $v$. As above, since ${p \nmid c_v}$, we have that ${E(K_v)[p^{\infty}]=E_0(K_v)[p^{\infty}]}$ injects into $\tilde{E}_{ns}(k_v)$. As $E$ has multiplicative reduction at $v$, ${\tilde{E}_{ns}(k_v)\cong k_v^{\times}}$. Since ${p \nmid \#k_v^{\times}}$, we have that ${E(K_v)[p^{\infty}]=0}$. Thus we see that for any ${v \in S \setminus \{\fp, \bar{\fp} \}}$ we have ${E(K_v)^*=0}$ and thus ${S_p^{\dagger}(E/K)^{\vee}=S_p(E/K)}$.

We want to show that ${\ker \uptheta = \img \psi' \cong \Qp/\Zp}$ or equivalently ${\coker \hat{\uptheta} \cong \Zp}$ where $\hat{\uptheta}$ is the dual of $\uptheta$. Since ${S_p^{\dagger}(E/K)^{\vee}=S_p(E/K)}$ we have the map
$$\hat{\uptheta}: S_p(E/K) \to E(K_{\fp})^* \times E(K_{\bar{\fp}})^*. $$
By Mattuck's Theorem ${E(K_{\fp}) \cong \Zp \times T}$ where $T$ is a finite group. By Lemma~\ref{p-torsion_lemma} the order of $T$ is not divisible by $p$. Therefore ${E(K_{\fp})^* \cong \Zp}$. Similarly ${E(K_{\bar{\fp}})^* \cong \Zp}$. Also Theorem~\ref{MN_theorem} implies that ${S_p(E/K)=E(K)\otimes \Zp=\Zp}$.

Without loss of generality, we assume that $y_K$ is not divisible by $p$ in $E(K_{\fp})$ (it can be shown that this also implies that $y_K$ is not divisible by $p$ in $E(K_{\bar{\fp}})$, but we won't need this). This implies that the restriction map from ${S_p(E/K) = \Zp}$ to ${E(K_{\fp})\otimes \Zp=\Zp}$ is an isomorphism. So now we have a map ${\hat{\uptheta}: \Zp \to \Zp \times \Zp}$ such that if $\pi_i$ is the projection of the target group onto its $i$-th factor then ${\pi_1 \circ \hat{\uptheta}}$ is an isomorphism. We want to show that ${\coker \hat{\uptheta} \cong \Zp}$. Since $\hat{\uptheta}$ is not the zero map, to show this we only need to show that ${\text{Tors}_{\Zp}(\coker \hat{\uptheta}) = \{0\}}$.

Let ${(a,b) \in \Zp \times \Zp}$ be such that ${(ra,rb) \in \img\hat{\uptheta}}$ for some ${r \in \Zp \backslash \{0\}}$. We must show that ${(a,b) \in \img\hat{\uptheta}}$. Let ${x \in \Zp}$ be such that ${\hat{\uptheta}(x) = (ra,rb)}$. Since ${\pi_1 \circ \hat{\uptheta}}$ is surjective there exists ${y \in \Zp}$ such that ${\hat{\uptheta}(y)=(a,c)}$ for some ${c \in \Zp}$. Then we have ${\pi_1(\hat{\uptheta}(ry))=r\pi_1(\hat{\uptheta}(y))=ra=\pi_1(\hat{\uptheta}(x))}$ which implies that ${ry=x}$ since ${\pi_1 \circ \hat{\uptheta}}$ is injective. Therefore ${(ra,rc)=(ra,rb)}$ so ${b=c}$ showing that ${(a,b) \in \img \hat{\uptheta}}$ as desired. This completes the proof.
\end{proof}

\part{Applications}
In the final part of our paper, we focus on applications of our main theorems. We start with results concerning Mazur's Conjecture~\ref{MazurConjecture}.

\section{Bounding Ranks}\label{section:bounding_ranks}
In this section, we prove Theorem~\ref{rankbound_theorem}. Throughout this section, let $K_{ac}$ be the anticyclotomic $\Zp$-extension of $K$. First we need
\begin{lemma}\label{rankbound_lemma}
If $L_{\infty}/K$ is a $\Zp$-extension with $X^{\bullet \star}(E/L_{\infty})$ a torsion $\Lambda$-module for each of the four choices of the pair ${\bullet, \star}$, then the rank of $E$ is bounded in $L_{\infty}/K$.
\end{lemma}
\begin{proof}
Let $v$ be a prime of $L_{\infty}$ above $p$. For any ${n \in N}$ we define $\hat{E}^{\pm}_g(L_{n,v})$ as in (1) and (2) of the paper with respect to the $\Zp$-extension $L_{\infty}/K$. Note that this definition differs from the definition of $\hat{E}^{\pm}(L_{n,v})$ listed in Section~\ref{section:mainintroduction}, since we ignore the possible index shifts by the constants $i(L_\infty, \fp)$ and $i(L_\infty, \bar{\fp})$. The definition is the same that Kobayashi~\cite{Kob} used; note that he only worked over the cyclotomic $\Zp$-extension.

Following Kobayashi, we define for any ${n \geq 0}$
$$\displaystyle 0 \longrightarrow \Selpmsingle(E/L_n) \longrightarrow \Selinf(E/L_n) \longrightarrow \prod_{\fp | p} \frac{H^1(L_{n,\fp}, E[p^{\infty}])}{\hat{E}^{\pm}_g(L_{n,\fp})\otimes \Q_p/\Z_p}. $$

For any ${n \in \N}$ consider the map
$$\psi_n \colon \Selm^+_{p^{\infty}}(E/L_n) \oplus \Selm^-_{p^{\infty}}(E/L_n) \longrightarrow \Selinf(E/L_n)$$
given by ${\psi_n(x,y)=x-y}$. Then \cite[Proposition~10.1]{Kob} shows that $\coker \psi_n$ is finite. Using the injection ${E(L_n) \otimes \Qp/\Zp \hookrightarrow \Selinf(E/L_n)}$, it is easy to see that in order to show that $\rank(E(L_n))$ is bounded with $n$ it suffices to show that both $\corank_{\Zp}(\Selm^+_{p^{\infty}}(E/L_n))$ and $\corank_{\Zp}(\Selm^-_{p^{\infty}}(E/L_n))$ are bounded with $n$.

We define $\text{sgn}: \Z \longrightarrow \{+,-\}$ by
\begin{equation*}
\text{sgn}(n)=
    \begin{cases}
        + & \text{if } n \equiv 0 \mod 2\\
        - & \text{if } n \not\equiv 0 \mod 2
    \end{cases}
\end{equation*}

We also define ${s,t: \{+,-\} \longrightarrow \{+,-\}}$ as follows. For any ${\bullet \in \{+,-\}}$ we have

$$s(\bullet)=\bullet \times \text{sgn}(i(L_{\infty}, \fp)),$$
$$t(\bullet)=\bullet \times \text{sgn}(i(L_{\infty}, \bar{\fp})).$$

For any ${n \in \N}$ let ${\Gamma_n=\Gal(L_{\infty}/L_n)}$. For any ${\bullet \in \{+,-\}}$ it follows from the definitions that the restriction map induces a map
$$\Selm^{\bullet}_{p^{\infty}}(E/L_n) \longrightarrow \Selm^{s(\bullet) t(\bullet)}_{p^{\infty}}(E/L_{\infty})^{\Gamma_n}.$$
Assuming that $X^{\bullet \star}(E/L_{\infty})$ is a torsion $\Lambda$-module for each of the four choices of the pair ${\bullet, \star}$, we get that $\rank_{\Zp}(X^{\bullet \star}(E/L_{\infty})_{\Gamma_n})$ is bounded with $n$ for all four choices. From Lemma~\ref{p-torsion_lemma} we have that the maps in each of the cases above are injections. Therefore it follows that the rank of $E$ is bounded in $L_{\infty}/K$.
\end{proof}
\begin{remark}
If ${i(L_\infty, \fp) \equiv i(L_\infty, \overline{\fp}) \!\mod 2}$, then the proof of Lemma~\ref{rankbound_lemma} shows that $X^{++}(E/L_{\infty})$ and $X^{--}(E/L_{\infty})$ being $\Lambda$-torsion are sufficient to imply that the rank of $E$ is bounded in $L_{\infty}/K$. \\
In the complementary case, i.e. if ${i(L_\infty, \fp) \not\equiv i(L_\infty, \overline{\fp}) \!\mod 2}$, then $X^{+-}(E/L_{\infty})$ and $X^{-+}(E/L_{\infty})$ both being $\Lambda$-torsion also implies the result.
\end{remark}

The main tool needed to prove Theorem~\ref{rankbound_theorem} is
\begin{theorem}\label{rankbound_theorem2}
Let ${\bullet, \star \in \{+,-\}}$. Let $s$ be the number of $\Zp$-extensions of $L_{\infty}/K$, where $X^{\bullet\star}(E/L_{\infty})$ is not a torsion $\Lambda$-module. Then
$$s \le \min\{\lambda^{\bullet\star}_{H^{\bullet\star}} \, | \, H^{\bullet\star} \in \Omega^{\bullet\star} \}. $$
\end{theorem}
\begin{proof}
The proof is identical to the proof of \cite[Theorem~1.5]{MHG} where we use Proposition~\ref{Xf_ControlThm_prop} to replace \cite[Proposition~4.2]{MHG} in the proof.
\end{proof}

\begin{proof}[Proof of Theorem~\ref{rankbound_theorem}]
The first statement follows immediately from Lemma~\ref{rankbound_lemma} and Theorem~\ref{rankbound_theorem2}. Now assume that every prime dividing $N$ splits in $K/\Q$, ${p \nmid h_K}$ and ${\Gal(\Q(E[p])/\Q)=GL_2(\Fp)}$. The second statement follows from Lemma~\ref{rankbound_lemma} and Theorem~\ref{rankbound_theorem2} if one observes the following
\begin{enumerate}[(a)]
\item By \cite[Theorem 1.4]{LV} both $X^{++}(E/K_{ac})$ and $X^{--}(E/K_{ac})$ have $\Lambda$-rank equal to one.
\item By \cite[Theorem 1.4]{Vatsal} together with the main result of \cite{BD}, it follows that ${\rank(E(K_{ac,n}))=p^n +O(1)}$, so the rank of $E$ is unbounded in $K_{ac}/K$.
\end{enumerate}
\end{proof}

In the next section, this theorem will be applied in order to produce examples which (partly) verify Mazur's Conjecture~\ref{MazurConjecture}.

\section{Further results and examples related to Conjecture~\ref{MazurConjecture}} \label{section:appendix}

The aim of this section is to analyse Conjecture~\ref{MazurConjecture} in some concrete cases. Recall that $E$ is an elliptic curve with good supersingular reduction at the prime ${p \ge 5}$, and that $K$ denotes an imaginary quadratic number field in which $p$ splits. We split this final section into two subsections, one devoted to the case where the root number is 1 and the other where the root number is -1. As mentioned in Section~\ref{section:mainintroduction}, we can easily produce examples verifying Conjecture~\ref{MazurConjecture} when the root number is 1 using Theorem~\ref{conj_rankzero_theorem} in Subsection~\ref{subsection:appendix1} below. However, unlike the ordinary case which we studied in \cite{MHG}, we are unable to produce any examples verifying Conjecture~\ref{MazurConjecture} when the root number is -1. In this case, the best we can do in Subsection~\ref{subsection:appendix2} below, is to produce examples to Theorem~\ref{rankbound_theorem} where ${t \leq 3}$.

\subsection{Root number 1}\label{subsection:appendix1}
In this section, attached to any $\Zp$-extension $L_{\infty}/K$ we prove a control theorem for signed Selmer groups only from the layer $K$ to $L_{\infty}$. When ${L_{\infty}=\KK_{\infty}^H}$ for some ${H \in \mathcal{H}^{\bullet\star}}$ and both $\fp$ and $\bar{\fp}$ are totally ramified in $L_{\infty}/K$ our result below will follow from Proposition~\ref{prop:local_max_control_thm}. However we will not be able to use this proposition for the general case.

\begin{proposition}\label{freeness_local_points_prop}
Let ${v \in \{\fp, \bar{\fp}\}}$ and let $w$ be any prime of $L^{(v)}_{\infty}$ above $v$. Then ${(E((L^{(v)}_{\infty})_w)\otimes \Qp/\Zp)^{\vee} \cong \Zp[[\Gal((L^{(v)}_{\infty})_w/K_v)]]}$.
\end{proposition}
\begin{proof}
Let $\mathcal{L}_{\infty}$ be the union of the completions at $w$ of all number fields inside $L^{(v)}_{\infty}$. By \cite[Lemma~3.1]{Minardi} $\mathcal{L}_{\infty}/K_v$ is a $\Zp$-extension. Let $\{F_n\}$ be the tower fields of the $\Zp$-extension $\mathcal{L}_{\infty}/K_v$ so that ${[F_n:K_v]=p^n}$. For any $n$ we have by Tate duality for abelian varieties over local fields \cite{Tate} an isomorphism

$$(E(F_n) \otimes \Qp/\Zp)^{\vee} \cong \ilim_m H^1(F_n, E)[p^m]$$\\
where the inverse limit is taken with respect to multiplication by $p$. It is shown in \cite[\S~3]{Tate} that the restriction map is dual to the corestriction map in this isomorphism. Therefore we have
$$\Big(E(\mathcal{L}_{\infty})\otimes \Qp/\Zp\Big)^{\vee}=\Big(\dlim_n E(F_n)\otimes \Qp/\Zp\Big)^{\vee}\cong \ilim_{n,m} H^1(F_n, E)[p^m]$$
where in the above the inverse limit over $n$ is taken with respect to the corestriction maps. Let ${\Gamma=\Gal(\mathcal{L}_{\infty}/K_v)}$ and ${\Lambda=\Zp[[\Gamma]]}$. We let $\Gamma_n$ denote $\Gamma^{p^n}$. By the above we must show that ${\ilim_{n,m} H^1(F_n, E)[p^m] \cong \Lambda}$. First we show that ${\corank_{\Lambda}(H^1(\mathcal{L}_{\infty}, E)[p^{\infty}])=1}$.

Let ${n \geq 0}$. Since $E$ has good reduction at $p$ and $\mathcal{L}_{\infty}/K$ is unramified, we get from \cite[Prop.~4.3]{Mazur} and its proof that ${H^1(\Gamma_n, E(\mathcal{L}_{\infty}))=H^2(\Gamma_n, E(\mathcal{L}_{\infty}))=0}$. Therefore it follows that $(H^1(\mathcal{L}_{\infty}, E)[p^{\infty}])^{\Gamma_n}\cong H^1(F_n, E)[p^{\infty}]$. By Tate duality $(\ilim_s E(F_n)/p^s)^{\vee} \cong H^1(F_n, E)[p^{\infty}]$. By Mattuck's theorem ${E(F_n) \cong \mathbb{Z}_p^{p^n} \times T}$ where $T$ is a finite group. It follows from this that ${\rank_{\Zp}(\ilim_s E(F_n)/p^s)=p^n}$ and therefore ${\corank_{\Zp}((H^1(\mathcal{L}_{\infty}, E)[p^{\infty}])^{\Gamma_n})=\corank_{\Zp}(H^1(F_n, E)[p^{\infty}])=p^n}$. This implies that ${\corank_{\Lambda}(H^1(\mathcal{L}_{\infty}, E)[p^{\infty}])=1}$.

Since ${(H^1(\mathcal{L}_{\infty}, E)[p^{\infty}])^{\Gamma_n} \cong H^1(F_n, E)[p^{\infty}]}$, we have that
\[ \ilim_{n,m} H^1(F_n, E)[p^m] \cong \ilim_{n,m} (H^1(\mathcal{L}_{\infty}, E)[p^m])^{\Gamma_n}\]
where the second inverse limit is taken over $n$ with respect to the norm maps. Since ${\corank_{\Lambda}(H^1(\mathcal{L}_{\infty}, E)[p^{\infty}])=1}$, we get from \cite[Proposition~5.5.10]{NSW} that ${\ilim_{n,m} (H^1(\mathcal{L}_{\infty}, E)[p^m])^{\Gamma_n} \cong \Lambda}$ whence ${\ilim_{n,m} H^1(F_n, E)[p^m]\cong \Lambda}$. This completes the proof.
\end{proof}

Let ${\bullet, \star \in \{+,-\}}$. Clearly we have ${\Selpm(E/K)=\Selinf(E/K)}$ since ${\hat{E}^\pm_v(K)= E(K_v)}$ for both ${v = \fp}$ and ${v = \bar{\fp}}$. The desired control theorem is

\begin{proposition}\label{level_zero_control_thm}
Let ${\bullet, \star \in \{+,-\}}$, $L_{\infty}/K$ be a $\Zp$-extension with ${\Gamma=\Gal(L_{\infty}/K)}$. Consider the map induced by restriction ${s: \Selinf(E/K) \to \Selpm(E/L_{\infty})^{\Gamma}}$. Then $s$ is an injection with finite cokernel.
\end{proposition}
\begin{proof}
Let $w$ be a prime of $\KK_{\infty}$ above $\fp$. Let ${\Gamma_w=\Gal(L_{\infty,w}/K_{\fp})}$. Consider the map

$$\phi_{\fp}: E(K_{\fp})\otimes\Qp/\Zp \to (\hat{E}^{\bullet}(L_{\infty,w})\otimes \Qp/\Zp)^{\Gamma_w}. $$\\
We claim that $\phi_{\fp}$ is an isomorphism. This is the key result needed to prove the proposition. As in the proof of Proposition~\ref{prop:local_max_control_thm} the map $\phi_{\fp}$ is an injection. Let ${H_w=\Gal(\KK_{\infty,w}/L_{\infty,w})}$. As above, we have an injection

$$\theta_{\fp}: \hat{E}^{\bullet}(L_{\infty,w})\otimes \Qp/\Zp \hookrightarrow (\hat{E}^{\bullet}(\KK_{\infty,w})\otimes \Qp/\Zp)^{H_w}. $$\\
Let ${\Lambda_w=\Zp[[\Gamma_w]]}$. By dualizing the above map and taking into account Lemma~\ref{freeness_lemma} we get a surjection

$$\theta_{\fp}^{\vee}: \Lambda_w \twoheadrightarrow (\hat{E}^{\bullet}(L_{\infty,w})\otimes \Qp/\Zp)^{\vee}. $$\\
If $\theta_{\fp}^{\vee}$ was not injective, then $\ker \theta_{\fp}^{\vee}$ would have $\Lambda_w$-rank one. This would then imply that ${(\hat{E}^{\bullet}(L_{\infty,w})\otimes \Qp/\Zp)^{\vee}}$ is a $\Lambda_w$-torsion module. If ${L_{\infty}=L^{(\fp)}_{\infty}}$, Proposition~\ref{freeness_local_points_prop} shows that this is not the case. Otherwise Proposition~\ref{pm_points_structure_prop} shows that for some open subgroup ${(\Gamma_w)^{p^n} \subseteq \Gamma_w}$,  ${(\hat{E}^{\bullet}(L_{\infty,w})\otimes \Qp/\Zp)^{\vee}}$ is not a torsion $\Zp[[(\Gamma_w)^{p^n}]]$-module and hence is not a torsion $\Lambda_w$-module (see \cite[Corollary~1.5]{Howson}). Note that the statement of Proposition~\ref{pm_points_structure_prop} required that ${L_{\infty}=\KK_{\infty}^H}$ for some ${H \in \mathcal{H}^{\ast\diamond}}$ where ${\ast,\diamond \in \{+,-\}}$, however the proof of the proposition did not use conditions~(a) and (c) in Definition~\ref{def:mathcal{H}}. Moreover the proof shows that the statement of Proposition~\ref{pm_points_structure_prop} is also valid for $L^{(\bar{\fp})}_{\infty}$. Therefore we see that
$$(\hat{E}^{\bullet}(L_{\infty,w})\otimes \Qp/\Zp)^{\vee} \cong \Lambda_w. $$\\
Now dualizing the map $\phi_{\fp}$ and using the above isomorphism we get a surjection

$$\phi_{\fp}^{\vee}: \Zp \twoheadrightarrow (E(K_{\fp})\otimes\Qp/\Zp)^{\vee}. $$\\
By Mattuck's theorem the codomain of this map is isomorphic to $\Zp$. So it follows that $\phi_{\fp}^{\vee}$ is an isomorphism and hence the same is true for $\phi_{\fp}$. In the same way we can show the analogous result for the prime $\bar{\fp}$. Now consider the following maps

$$g: H^1(K, E[p^{\infty}]) \to H^1(L_{\infty}, E[p^{\infty}])^{\Gamma}. $$\\
If ${v \in S \setminus S_p}$, fix a prime $L_{\infty}$ above $v$ which we also denote by $v$. We have a map

$$h_v: H^1(K_v, E[p^{\infty}]) \to H^1(L_{\infty,v}, E[p^{\infty}]). $$\\
Given the result that $\phi_{\fp}$ and $\phi_{\bar{\fp}}$ are isomorphisms, we see from the proof of Theorem~\ref{SelmerControlThm_prop} using the snake lemma that the proof of this proposition will be complete if we can show that $g$ is an isomorphism and that $\ker h_v$ is finite for any ${v \in S \setminus S_p}$.

Let ${v \in S \setminus S_p}$. If $v$ splits completely in $L_{\infty}/K$, then ${\ker h_v=0}$. Otherwise, we can see that $\ker h_v$ is finite from \cite[Lemma~3.3]{Gb_LNM}. As for the map $g$, we have ${\ker g = H^1(\Gamma, E(L_{\infty})[p^{\infty}])}$ and ${\coker g \subseteq H^2(\Gamma, E(K_{\infty})[p^{\infty}])}$. Thus ${\ker g=\coker g=0}$ from Lemma~\ref{p-torsion_lemma}. This completes the proof of the proposition.
\end{proof}

Using the above the proposition we can now prove

\begin{theorem}\label{conj_rankzero_theorem}
If ${\rank(E(K))=0}$ and ${\Sha(E/K)[p^{\infty}]=0}$, then Conjecture~\ref{MazurConjecture} is true.
\end{theorem}
\begin{proof}
By the discussion before \cite[Theorem~9.3]{MHG} we have that the root number of $L(E_K, s)$ is 1. Let $L_{\infty}/K$ be a $\Zp$-extension with ${\Gamma=\Gal(L_{\infty}/K)}$. Also let ${\bullet, \star \in \{+,-\}}$. By Proposition~\ref{level_zero_control_thm} the map ${s: \Selinf(E/K) \to \Selpm(E/L_{\infty})^{\Gamma}}$ has finite cokernel. Therefore, as $\Selinf(E/K)$ is finite, it follows that $\Selpm(E/L_{\infty})^{\Gamma}$ is also finite. This implies by the structure theorem for $X^{\bullet\star}(E/L_{\infty})$ as a $\Lambda$-module that $X^{\bullet\star}(E/L_{\infty})$ is a torsion $\Lambda$-module. The desired result follows from Lemma~\ref{rankbound_lemma}.
\end{proof}

\subsection{Root number -1} \label{subsection:appendix2}

Throughout this subsection let $K_{ac}$ be the anticyclotomic $\Zp$-extension of $K$ with tower fields $K_n$, where we assume that ${K \neq \Q(\sqrt{-1}), \Q(\sqrt{-3})}$. In this subsection we assume that ${p \nmid h_K}$. This will ensure that both $\fp$ and $\bar{\fp}$ are totally ramified in $K_{ac}/K$. By abuse of notation, we will also refer to the prime of $K_{ac}$ or $K_n$ above $\fp$ (resp. $\bar{\fp}$) by $\fp$ (resp. $\bar{\fp}$). In Section~\ref{section:mainintroduction} we defined the fine Selmer group of $E/L$ for any $\Zp$-extension $L_{\infty}/K$ and denoted it by $R_{p^{\infty}}(E/L_{\infty})$. We use the same notation in this section. For any $\Zp$-extension $L_{\infty}/K$, let $X(E/L_{\infty})$ be the Pontryagain dual of $\Selinf(E/L_{\infty})$. Let $N$ be the conductor of $E$. We will also need to define Selmer groups with various local conditions.

\begin{definition} \label{def:modifiedSelmer}
For ${v \in \{\fp, \bar{\fp}\}}$ and ${\mathscr{L}_v \in \{\pm, 0 , 1, \emptyset\}}$ we define
$$H^1_{\mathscr{L}_v}(K_{n,v}, E[p^{\infty}])=
\begin{cases}
E^{\pm}(K_{n,v})\otimes \Qp/\Zp & \text{if} \, \mathscr{L}_v=\pm,\\
\{0\} & \text{if} \, \mathscr{L}_v=0,\\
E(\Qp)\otimes \Qp/\Zp & \text{if} \, \mathscr{L}_v=1,\\
H^1(K_{n,v}, E[p^{\infty}]) & \text{if} \, \mathscr{L}_v=\emptyset.
\end{cases}$$

Then for ${\mathscr{L}=\{\mathscr{L}_{\fp}, \mathscr{L}_{\bar{\fp}}\}}$ we define the modified Selmer group $\Selm^{\mathscr{L}}_{p^{\infty}}(E/K_n)$ as the kernel of the map
$$H^1(G_S(K_n), E[p^\infty]) \to \prod_{v \in \{\fp, \bar{\fp}\}} \frac{H^1(K_{n,v}, E[p^{\infty}])}{H^1_{\mathscr{L}_v}(K_{n,v}, E[p^{\infty}])} \times \prod_{v|w, w \in S \setminus \{\fp, \bar{\fp}\}} H^1(K_{n,v}, E)[p^{\infty}].$$
We define ${\Selm^{\mathscr{L}}_{p^{\infty}}(E/K_{ac})=\dlim \Selm^{\mathscr{L}}_{p^{\infty}}(E/K_n)}$. We denote the Pontryagin dual of $\Selm^{\mathscr{L}}_{p^{\infty}}(E/K_{ac})$ by $X^{\mathscr{L}}(E/K_{ac})$.
\end{definition}

\begin{remark}\label{fineSelmer_remark}
If ${F=K_n}$ or $K_{ac}$, then ${\Selm^{00}(E/F) = R_{p^{\infty}}(E/F)}$ is the fine Selmer group, as defined in Section~\ref{section:definitions}. This follows from the fact that if $L$ is a finite extension of $\Ql$ where ${l \neq p}$, then by Mattuck's theorem ${E(L) \cong \mathbb{Z}_l^d \times T}$ where ${d=[L:\Ql]}$ and $T$ is a finite group. This implies that ${E(L) \otimes \Qp/\Zp=0}$.
\end{remark}

For the rest of this section we assume the Heegner hypothesis: Every prime dividing $N$ splits in $K/\Q$. If $K[1]$ is the Hilbert class field of $K$, then a choice of an ideal $\cN$ of $\cO_K$ such that ${\cO_K/\cN \cong \cy{N}}$ and a modular parametrization ${X_0(N) \to E}$ allow us to define a Heegner point ${y_1 \in E(K[1])}$. Let $y_K$ be the trace of this point down to $K$. For any prime $v$ of $\Q$, we let $c_v$ be the Tamagawa number of $E$ at $v$.

\begin{proposition}\label{mixedSelmertorsion_prop}
Assume that ${\Gal(\Q(E[p])/\Q)=GL_2(\Fp)}$. Then both $X^{+-}(E/K_{ac})$ and $X^{-+}(E/K_{ac})$ are $\Lambda$-torsion.
\end{proposition}
\begin{proof}
This is proven in \cite[Theorem~A.12]{GHKL} using the key results of Castella-Wan \cite{CastellaWan}.
\end{proof}

\begin{lemma}\label{Heegnerpointdiv_lemma}
${y_K \notin pE(K_{\fp})}$ if and only if ${y_K \notin pE(K_{\bar{\fp}})}$.
\end{lemma}
\begin{proof}
Let $\tau$ be the complex conjugation automorphism and let $\epsilon$ be the root number of $E/\Q$. By \cite[Prop.~5.3]{Gross} we have that ${\tau(y_K)=\epsilon y_K + Q}$ for some torsion point $Q$ in $E(K)$. By Lemma~\ref{p-torsion_lemma} ${E(K)[p]=0}$. Therefore we have ${\tau(y_K)+ pE(K) = \epsilon y_K + pE(K)}$. The desired result follows from this.
\end{proof}

We need the following control theorem.

\begin{proposition}\label{fineSelmercontrol_prop2}
Assume that ${p \nmid \prod_{v | N} c_v}$. Let ${s: R_{p^{\infty}}(E/K) \to R_{p^{\infty}}(E/K_{ac})^{\Gamma}}$ be the map induced by restriction. Then $s$ is an isomorphism.
\end{proposition}
\begin{proof}

We have a commutative diagram

\begin{equation*}
\xymatrix {
0 \ar[r] & R_{p^{\infty}}(E/K_{ac})^{\Gamma} \ar[r] & H^1(G_S(K_{ac}), E[p^{\infty}])^{\Gamma}  \ar[r] & \bigoplus_{v \in S} M_v(E/K_{ac})^{\Gamma}\\
0 \ar[r] & R_{p^{\infty}}(E/K) \ar[r] \ar[u]^{s} &H^1(G_S(K), E[p^{\infty}]) \ar[u]^{g} \ar[r] & \bigoplus_{v \in S} M_v(E/K) \ar[u]^{h}}
\end{equation*}
where $M_v(E/K)$ and $M_v(E/K_{ac})$ are defined in Section~\ref{section:mainintroduction}, before Conjecture~\ref{conjectureA}. By the snake lemma, to prove that $s$ is an isomorphism, we only need to show that $g$ is an isomorphism and that $h$ is an injection.

We have
\[ \ker g=H^1(\Gamma, E(K_{ac})[p^{\infty}])\]
and an injection
\[ \coker g \hookrightarrow H^2(\Gamma, E(K_{ac})[p^{\infty}]).\]
Lemma~\ref{p-torsion_lemma} shows that ${\ker g=\coker g=0}$. For every ${v \in S}$ we fix a prime $w$ of $K_{ac}$ above $v$. Taking into account Shapiro's lemma, to show that $h$ is an injection, we need to show that for any ${v \in S}$ the restriction map ${h_v: H^1(K_v, E[p^{\infty}]) \to H^1(K_{ac,w}, E[p^{\infty}])}$ is an injection. First assume that $v$ divides $p$. Then ${\ker h_v=H^1(\Gamma_w,  E(K_{ac,w})[p^{\infty}])}$ where $\Gamma_w$ is the decomposition group. This is trivial by Lemma~\ref{p-torsion_lemma}. Now assume that $v$ does not divide $p$. If $F$ is a finite extension of $K_v$, then by Mattuck's theorem we have ${E(F) \cong \mathbb{Z}_l^d \times T}$ where ${l \neq p}$ is the residue characteristic of $K_v$ and ${d=[F:\Ql]}$. Therefore ${E(F) \otimes \Qp/\Zp=0}$. It follows from this that ${E(K_v) \otimes \Qp/\Zp=E(K_{ac,w}) \otimes \Qp/\Zp=0}$. So ${H^1(K_v, E[p^{\infty}])=H^1(K_v, E)[p^{\infty}]}$ and ${H^1(K_{ac,w}, E[p^{\infty}])=H^1(K_{ac,w}, E)[p^{\infty}]}$. Therefore since ${p \nmid c_v}$ we get from \cite[remark following Lemma~3.3]{Gb_LNM} that ${\ker h_v=0}$.
\end{proof}

\begin{theorem}\label{fineSelmertrivial_theorem}
Assume the following:
\begin{enumerate}[(a)]
\item $y_K \notin pE(K_v)$ for some prime $v$ above $p$
\item $p \nmid \prod_{v | N} c_v$
\end{enumerate}
Then $R_{p^{\infty}}(E/K_{ac})=0$.
\end{theorem}
\begin{proof}
By Proposition~\ref{fineSelmercontrol_prop2} it will suffice to show that ${R_{p^{\infty}}(E/K)=0}$. We have an exact sequence
$$0 \to R_{p^{\infty}}(E/K) \to \Selinf(E/K) \to \prod_{v \in S} H^1(K_v, E[p^{\infty}]). $$
Now consider the localization map $\psi_{\fp}: \Selinf(E/K) \to H^1(K_{\fp}, E[p^{\infty}])$. From the exact sequence above we see that to show ${R_{p^{\infty}}(E/K)=0}$ it will suffice to prove that ${\ker \psi_{\fp}=0}$.
By Theorem \ref{MN_theorem} we have ${\Selinf(E/K)=E(K) \otimes \Qp/\Zp}$. Therefore it will suffice to show that the map
$$\psi'_{\fp}: E(K) \otimes \Qp/\Zp \to E(K_{\fp}) \otimes \Qp/\Zp$$
is an injection.
Note that ${E(K) \otimes \Qp/\Zp=\dlim E(K)/p^n}$ and ${E(K_{\fp}) \otimes \Qp/\Zp = \dlim E(K_{\fp})/p^n}$.

By Theorem \ref{MN_theorem} we have that ${E(K)[p^{\infty}]=0}$ and ${\rank(E(K))=1}$. Furthermore if $P$ is a generator of the free part of $E(K)$, then ${y_K=tP}$ for some integer $t$ not divisible by $p$. Let ${\phi_n: E(K)/p^n \to \dlim E(K)/p^m}$ and ${\phi_{\fp,n}: E(K_{\fp})/p^n \to \dlim E(K_{\fp})/p^m}$ be the maps into the direct limits.

Now let ${x \in \dlim E(K)/p^n}$ and suppose that ${\psi_{\fp}'(x)=0}$. By the above ${x=\phi_n(sP + p^nE(K))}$ for some integers $n$ and $s$. Since ${\psi_{\fp}'(x)=0}$, we can conclude that ${\phi_{\fp,n}(sP + p^nE(K_{\fp}))=0}$.

Now we claim that the transition maps in the direct limit ${\dlim E(K_{\fp})/p^n}$ are injective. To see this, let ${Q \in E(K_{\fp})/p^m}$ and suppose that for ${n > m}$ we have ${p^{n-m}Q \in p^nE(K_{\fp})}$. So ${p^{n-m}Q=p^nQ'}$ for some ${Q' \in E(K_{\fp})}$. This implies that ${p^{n-m}(Q-p^mQ')=0}$. Lemma~\ref{p-torsion_lemma} then implies that ${Q=p^mQ'}$. Therefore the transition maps are in fact injective.

Since ${\phi_{\fp,n}(sP + p^nE(K_{\fp}))=0}$ and the transition maps in the direct limit ${\dlim E(K_{\fp})/p^m}$ are injective, we have that ${sP+p^nE(K_{\fp})=0}$, i.e. ${sP=p^nQ}$ for some ${Q \in E(K_{\fp})}$. This furthermore implies that ${sy_K=p^ntQ}$. Let $m$ be the largest power of $p$ dividing $s$. If ${m < n}$, then ${p^m(s'y_K-p^{n-m}tQ)=0}$, where ${s'=s/p^m}$. This implies by Lemma~\ref{p-torsion_lemma} that ${s'y_K=p^{n-m}tQ}$ and as $s'$ is prime to $p$, it follows that ${y_K \in pE(K_{\fp})}$. Taking Lemma~\ref{Heegnerpointdiv_lemma} into account, this contradicts $(a)$. So we see that ${m \geq n}$ whence ${x=0}$ and so $\psi_{\fp}'$ is injective.
\end{proof}

Recall the notation introduced in Definition~\ref{def:modifiedSelmer}. From the previous theorem we deduce

\begin{theorem}\label{Selmer11_theorem}
Assume the following:
\begin{enumerate}[(a)]
\item $y_K \notin pE(K_v)$ for some prime $v$ above $p$
\item $p \nmid \prod_{v | N} c_v$
\end{enumerate}
Then $X^{11}(E/K_{ac})$ is a torsion $\Lambda$-module with ${\mu=0}$ and ${\lambda=1}$.
\end{theorem}
\begin{proof}
Note that ${\Selm^{11}_{p^{\infty}}(E/K)=\Selinf(E/K)}$. By \cite[Lemma~4.3]{Matar_Tower} ${\Selinf(E/K) \cong \Selm^{11}_{p^{\infty}}(E/K_{ac})^{\Gamma}}$. Combining this with Theorem~\ref{MN_theorem} implies that ${\Selm^{11}_{p^{\infty}}(E/K_{ac})^{\Gamma} \cong \Qp/\Zp}$. Now consider the commutative diagram with exact rows

\begin{equation*}
\resizebox{\displaywidth}{!}{
\xymatrix{
& 0 \ar[d] \ar[r] & \Selinf(E/K_{ac}) \ar[d] \ar[r] & \Selinf(E/K_{ac}) \ar[d] \ar[r]  & 0 \\
0 \ar[r] & \prod \limits_{i=1,2} E(\Qp)\otimes \Qp/\Zp \ar[r] & \prod \limits_{v \in \{\fp, \bar{\fp}\}} H^1(K_{ac,v}, E[p^{\infty}]) \ar[r] & \prod \limits_{v \in \{\fp, \bar{\fp}\}} \frac{H^1(K_{ac,v}, E[p^{\infty}])}{E(\Qp)\otimes \Qp/\Zp} \ar[r] & 0
}}
\end{equation*}

Taking Remark~\ref{fineSelmer_remark} into account, it follows from applying the snake lemma to the above diagram that we have an exact sequence

$$0 \longrightarrow R_{p^{\infty}}(E/K_{ac}) \longrightarrow \Selm^{11}_{p^{\infty}}(E/K_{ac}) \longrightarrow \prod_{i=1,2} E(\Qp)\otimes \Qp/\Zp. $$
From Theorem~\ref{fineSelmertrivial_theorem} this implies that we have an injection
$$\Selm^{11}_{p^{\infty}}(E/K_{ac}) \hookrightarrow \prod_{i=1,2} E(\Qp)\otimes \Qp/\Zp. $$
By Mattuck's theorem ${E(\Qp) \cong \Zp \times T}$ where $T$ is a finite group. Therefore ${E(\Qp) \otimes \Qp/\Zp \cong \Qp/\Zp}$. From this, the injection above and the fact that ${\Selm_{p^{\infty}}^{11}(E/K_{ac})^{\Gamma} \cong \Qp/\Zp}$ we see that we must have ${\Selm_{p^{\infty}}^{11}(E/K_{ac})\cong \Qp/\Zp}$. Therefore $X^{11}(E/K_{ac})$ is $\Lambda$-torsion with ${\mu=0}$ and ${\lambda=1}$.
\end{proof}

On the other hand, we have
\begin{theorem}\label{Selmerpmfree_theorem}
Assume the following:
\begin{enumerate}[(a)]
\item $y_K \notin pE(K)$
\item $p \nmid \prod_{v | N} c_v$
\end{enumerate}
Then both $X^{++}(E/K_{ac})$ and $X^{--}(E/K_{ac})$ are free $\Lambda$-modules of rank one.
\end{theorem}
\begin{proof}
See \cite[Theorem~3.1]{Matar_Tower}.
\end{proof}

\begin{theorem}\label{Selmeremptypmfree_theorem}
Assume the following:
\begin{enumerate}[(a)]
\item $y_K \notin pE(K_v)$ for some prime $v$ above $p$
\item $p \nmid \prod_{v | N} c_v$
\item $\Gal(\Q(E[p])/\Q)=GL_2(\Fp)$
\end{enumerate}
Then for any ${\bullet \in \{+.-\}}$ we have that both $X^{\bullet\emptyset}(E/K_{ac})$ and $X^{\emptyset\bullet}(E/K_{ac})$ are free $\Lambda$-modules of rank one.
\end{theorem}
\begin{proof}
Assuming $(a)$, $(b)$ and $(c)$ we will show that $X^{\emptyset\bullet}(E/K_{ac})$ is a free $\Lambda$-module of rank one. The proof uses the same technique used to prove \cite[Theorem~3.1]{Matar_Tower}. The proof for $X^{\bullet\emptyset}(E/K_{ac})$ will be identical. First we show that ${\Selm^{\emptyset1}_{p^{\infty}}(E/K) \cong \Qp/\Zp}$.

We have an exact sequence

$$0 \longrightarrow \Selinf(E/K) \to \Selm^{\emptyset1}_{p^{\infty}}(E/K) \xrightarrow{\psi} H^1(K_{\fp}, E)[p^{\infty}]. $$
By Theorem~\ref{MN_theorem} ${\Selinf(E/K) \cong \Qp/\Zp}$. Therefore to show that ${\Selm^{\emptyset1}_{p^{\infty}}(E/K) \cong \Qp/\Zp}$, we need to prove that ${\img \psi=0}$. By \cite[Theorem~1.7.3]{Rubin} we have an exact sequence
$$\Selm^{\emptyset1}_{p^{\infty}}(E/K) \xrightarrow{\psi} H^1(K_{\fp}, E)[p^{\infty}] \xrightarrow{\theta} S_p(E/K)^{\vee}$$
where ${S_p(E/K) =\ilim_n \Selm_{p^n}(E/K)}$. Here $\Selm_{p^n}(E/K)$ is the classical $p^n$-Selmer group defined by
$$0 \to \Selm_{p^n}(E/K) \to H^1(G_S(K), E[p^n]) \to \prod_{v \in S} H^1(K_v, E)[p^n].$$
We need to show that ${\ker \theta=\img \psi=0}$ or equivalently that the dual map to $\theta$ is surjective. The dual map is
$$\hat{\theta}: S_p(E/K) \to E(K_{\fp})^*$$
where $E(K_{\fp})$ is the $p$-adic completion of $E(K_{\fp})$ (see \cite[Theorem~1.7.3(ii)]{Rubin}).

By Mattuck's theorem ${E(K_{\fp}) \cong \Zp \times T}$ where $T$ is a finite group. By Lemma~\ref{p-torsion_lemma} the order of $T$ is not divisible by $p$. Therefore ${E(K_{\fp})^* \cong \Zp}$. By Theorem~\ref{MN_theorem} ${S_p(E/K)=E(K)^*=E(K)\otimes \Zp=\Zp(y_K\otimes 1) \cong \Zp}$. By condition $(a)$ and Lemma~\ref{Heegnerpointdiv_lemma} ${y_K \notin pE(K_{\fp})}$. Combining all the previous facts together shows that $\hat{\theta}$ is surjective. Therefore ${\Selm^{\emptyset1}_{p^{\infty}}(E/K) \cong \Qp/\Zp}$ as claimed.

Now the restriction map induces a map ${s: \Selm^{\emptyset1}_{p^{\infty}}(E/K) \to \Selm^{\emptyset\bullet}_{p^{\infty}}(E/K_{ac})^{\Gamma}}$. We claim that this map is an isomorphism. To see this, we refer the reader to the proof of \cite[Theorem~6.8]{IP} (which is based on the proof of \cite[Theorem~9.3]{Kob}). The proof reveals that the map $s$ is an injection and that the map
$$H^1(K_{\bar{\fp}}, E[p^{\infty}])/E(K_{\bar{\fp}})\otimes \Qp/\Zp \to (H^1(K_{ac, \bar{\fp}}, E[p^{\infty}])/E^{\bullet}(K_{ac, \bar{\fp}})\otimes \Qp/\Zp)^{\Gamma}$$
is also an injection.

We have a map
$$\phi_{\fp}: H^1(K_{\fp}, E[p^{\infty}])/H^1_{\emptyset}(K_{\fp}, E[p^{\infty}]) \to H^1(K_{ac, \fp}, E[p^{\infty}])/H^1_{\emptyset}(K_{ac, \fp}, E[p^{\infty}]).$$

For every prime $v$ of $K$ dividing $N$ we fix a prime $w$ of $K_{ac}$ above $v$. For every such prime $v$ we have a map
$$\phi_v: H^1(K_v, E)[p^{\infty}] \to H^1(K_{ac,w}, E)[p^{\infty}].$$

From the proof of loc. cit., we see that to prove that $s$ is an isomorphism we need to show that $\phi_{\fp}$ is an injection and that for any prime $v$ of $K$ dividing $N$ we have that $\phi_v$ is an injection. There is nothing to prove regarding the map $\phi_{\fp}$ because both the domain and codomain of this map are trivial. As for $\ker \phi_v$, this is trivial by condition~$(b)$ (see \cite[remark following Lemma~3.3]{Gb_LNM}).

Therefore the map $s$ is in fact an isomorphism. Since ${\Selm^{\emptyset1}_{p^{\infty}}(E/K) \cong \Qp/\Zp}$, it follows from Nakayama's lemma that $X^{\emptyset\bullet}(E/K_{ac})$ is a cyclic $\Lambda$-module i.e. $X^{\emptyset\bullet}(E/K_{ac}) \cong \Lambda/I$ for some ideal $I$ of $\Lambda$. The proof of \cite[Theorem~A.12]{GHKL} shows that ${\rank_{\Lambda}(X^{\emptyset\bullet}(E/K_{ac}))=1}$ so ${I=0}$. Therefore $X^{\emptyset\bullet}(E/K_{ac})$ is a free $\Lambda$-module of rank one. This completes the proof.
\end{proof}

\begin{theorem}\label{mixedSelmertorsion_theorem}
Assume the following:
\begin{enumerate}[(a)]
\item $y_K \notin pE(K_v)$ for some prime $v$ above $p$
\item $p \nmid \prod_{v | N} c_v$
\item $\Gal(\Q(E[p])/\Q)=GL_2(\Fp)$
\end{enumerate}
Then both $X^{+-}(E/K_{ac})$ and $X^{-+}(E/K_{ac})$ are $\Lambda$-torsion with $\mu=0$ and $\lambda=1$.
\end{theorem}
\begin{proof}
We will prove the result for $X^{+-}(E/K_{ac})$. The proof for $X^{-+}(E/K_{ac})$ will be identical.

We have inclusion maps
$$\iota^+: \Selm^{+1}_{p^{\infty}}(E/K_{ac}) \hookrightarrow \Selm^{+-}_{p^{\infty}}(E/K_{ac}), $$
$$\iota^-: \Selm^{1-}_{p^{\infty}}(E/K_{ac}) \hookrightarrow \Selm^{+-}_{p^{\infty}}(E/K_{ac}). $$

Our strategy will be to show that both $\iota^+$ and $\iota^-$ are surjective. Once this is shown we will have ${\Selm^{+-}_{p^{\infty}}(E/K_{ac})=\Selm^{+1}_{p^{\infty}}(E/K_{ac})=\Selm^{1-}_{p^{\infty}}(E/K_{ac})}$. It follows from this that ${\Selm^{+-}_{p^{\infty}}(E/K_{ac})=\Selm^{11}_{p^{\infty}}(E/K_{ac})}$ and so the statement of the theorem will follow from Theorem~\ref{Selmer11_theorem}.

To simplify notation in the proof below we will denote $\iota^+$ by $\iota$. We now proceed to show that $\iota$ is surjective. Equivalently we show that the dual map
$$\hat{\iota}: X^{+-}(E/K_{ac}) \twoheadrightarrow X^{+1}(E/K_{ac})$$
is injective. We also have an inclusion map
$$\kappa: \Selm^{+1}_{p^{\infty}}(E/K_{ac}) \hookrightarrow \Selm^{++}_{p^{\infty}}(E/K_{ac}). $$
Now consider the diagonal map
$$\theta: \Selm^{+1}_{p^{\infty}}(E/K_{ac}) \hookrightarrow \Selm^{++}_{p^{\infty}}(E/K_{ac}) \oplus \Selm^{+-}_{p^{\infty}}(E/K_{ac})$$
defined as ${\theta=(\kappa, \iota)}$. Let $\pi_1$ and $\pi_2$ be the projection maps of the codomain of $\theta$ onto its first and second components. Then we have ${\iota=\pi_2 \circ \theta}$. Taking dual maps, we need to show that ${\hat{\iota} =\hat{\theta} \circ \hat{\pi}_2}$ is injective. Note that
\begin{equation}\label{directsum_dual}
(\Selm^{++}_{p^{\infty}}(E/K_{ac}) \oplus \Selm^{+-}_{p^{\infty}}(E/K_{ac}))^{\vee}=X^{++}(E/K_{ac}) \oplus X^{+-}(E/K_{ac})
\end{equation}
and that the maps
$$\hat{\pi}_1: X^{++}(E/K_{ac}) \hookrightarrow X^{++}(E/K_{ac}) \oplus X^{+-}(E/K_{ac})$$
and
$$\hat{\pi}_2: X^{+-}(E/K_{ac}) \hookrightarrow X^{++}(E/K_{ac}) \oplus X^{+-}(E/K_{ac})$$
are the natural injections into the direct sum. Now consider the sequence

\begin{equation}\label{pmSelmer_seq1}
0 \to \Selm^{+1}_{p^{\infty}}(E/K_{ac}) \xrightarrow{\theta} \Selm^{++}_{p^{\infty}}(E/K_{ac}) \oplus \Selm^{+-}_{p^{\infty}}(E/K_{ac}) \xrightarrow{\phi} \Selm^{+\emptyset}_{p^{\infty}}(E/K_{ac}).
\end{equation}
Considering $\Selm^{++}_{p^{\infty}}(E/K_{ac})$ and $\Selm^{+-}_{p^{\infty}}(E/K_{ac})$ as subgroups of $\Selm^{+\emptyset}_{p^\infty}(E/K_{ac})$, the map $\phi$ above is ${\phi(x,y)=x-y}$. From \cite[Lemma~2.6.5]{CW} which uses the results of Iovita-Pollack \cite{IP} we see that ${\ker \phi = \img \theta}$ and therefore the sequence~\eqref{pmSelmer_seq1} is exact. Dualizing this exact sequence and taking \eqref{directsum_dual} into account we get an exact sequence
\begin{equation}\label{pmSelmer_seq2}
X^{+\emptyset}(E/K_{ac}) \xrightarrow{\hat{\phi}} X^{++}(E/K_{ac}) \oplus X^{+-}(E/K_{ac}) \xrightarrow{\hat{\theta}} X^{+1}(E/K_{ac}) \to 0.
\end{equation}
From Proposition~\ref{mixedSelmertorsion_prop} $X^{+-}(E/K_{ac})$ is $\Lambda$-torsion. From Theorems~\ref{Selmerpmfree_theorem} and \ref{Selmeremptypmfree_theorem} we have ${X^{++}(E/K_{ac}) \cong X^{+\emptyset}(E/K_{ac}) \cong \Lambda}$. Also as shown in the proof of \cite[Theorem~A.12]{GHKL} we have that $X^{+1}(E/K_{ac})$ is $\Lambda$-torsion. Now we claim that the map $\hat{\phi}$ is an injection. To see this, suppose that ${\ker \hat{\phi} \neq 0}$. Since ${X^{+\emptyset}(E/K_{ac}) \cong \Lambda}$, we then must have that ${\rank_{\Lambda}(\ker \hat{\phi})=1}$ and hence ${\rank_{\Lambda}(\img \hat{\phi})=0}$. This is clearly impossible since ${\rank_{\Lambda}(X^{++}(E/K_{ac}) \oplus X^{+-}(E/K_{ac}))=1}$ and ${\rank_{\Lambda}(X^{+1}(E/K_{ac}))=0}$. Therefore, as claimed, $\hat{\phi}$ is an injection.

Recall that we need to show that ${\hat{\theta} \circ \hat{\pi}_2}$ is an injection. Since $\hat{\pi}_2$ is an injection, it follows from the exact sequence~\eqref{pmSelmer_seq2} that ${\hat{\theta} \circ \hat{\pi}_2}$ is an injection if and only if ${\img \hat{\phi} \cap \img \hat{\pi}_2 = 0}$. Suppose that ${\img \hat{\phi} \cap \img \hat{\pi}_2 \neq 0}$. Then since ${\img \hat{\phi} \cong \Lambda}$, we conclude that ${\rank_{\Lambda}(\img \hat{\phi} \cap \img \hat{\pi}_2)=1}$.

We have a map ${\rho: \img \hat{\phi} \to \img \hat{\phi} \cap \img \hat{\pi}_2}$ which is the projection onto the second component in the direct sum. The kernel of $\rho$ is ${\img \hat{\phi} \cap \img \hat{\pi}_1}$. Therefore we have an exact sequence
$$0\to \img \hat{\phi} \cap \img \hat{\pi}_1 \to \img \hat{\phi} \to \img \hat{\phi} \cap \img \hat{\pi}_2 \to 0.$$
Since the second and third terms of the exact sequence have $\Lambda$-rank one, it follows that ${\img \hat{\phi} \cap \img \hat{\pi}_1}$ is $\Lambda$-torsion. We have an exact sequence
$$0 \to \img \hat{\phi} \cap \img \hat{\pi}_1 \to \img \hat{\pi}_1 \to \hat{\theta}(\img \hat{\pi}_1) \to 0. $$
Since $X^{++}(E/K_{ac}) \cong \Lambda$ and $\hat{\pi}_1$ is an injection, we have that ${\rank_{\Lambda}(\img \hat{\pi}_1)=1}$. Since ${\img \hat{\phi} \cap \img \hat{\pi}_1}$ is $\Lambda$-torsion, the above exact sequence implies that ${\rank_{\Lambda}(\hat{\theta}(\img \hat{\pi}_1))=1}$. This contradicts the fact that $X^{+1}(E/K_{ac})$ is $\Lambda$-torsion. Therefore we see that ${\hat{\theta} \circ \hat{\pi}_2}$ is an injection whence ${\iota^+=\iota}$ is surjective.

To prove that $\iota^-$ is surjective we proceed as above but rather than \eqref{pmSelmer_seq1} we use the exact sequence
\begin{equation}
0 \to \Selm^{1-}_{p^{\infty}}(E/K_{ac}) \xrightarrow{\theta} \Selm^{--}_{p^{\infty}}(E/K_{ac}) \oplus \Selm^{+-}_{p^{\infty}}(E/K_{ac}) \xrightarrow{\phi} \Selm^{\emptyset-}_{p^{\infty}}(E/K_{ac}).
\end{equation}
Note that it is also shown in the proof of \cite[Theorem~A.12]{GHKL} that $X^{1-}(E/K_{ac})$ is $\Lambda$-torsion. This completes the proof.
\end{proof}

Now let ${L^{\pm}_p(E), L^{\pm}_p(E^{(K)}) \in \Zp[[T]]}$ be the plus/minus $p$-adic $L$-functions attached to $E/\Q$ and $E^{(K)}/\Q$ of Pollack~\cite{Pollack}. Combining the previous results, we now show

\begin{proposition}\label{Rankbound_prop}
Assume the following
\begin{enumerate}[(a)]
\item Every prime dividing $pN$ splits in $K/\Q$
\item $p \nmid h_K$
\item $\Gal(\Q(E[p])/\Q)=GL_2(\Fp)$
\item $y_K \notin pE(K_v)$ for some prime $v$ above $p$
\item $p \nmid \prod_{v | N} c_v$
\item $\lambda(L^{\bullet}_p(E))+\lambda(L^{\bullet}_p(E^{(K)}))=1$ for each ${\bullet \in \{+,-\}}$
\item $\mu(L^{\bullet}_p(E))=\mu(L^{\bullet}_p(E^{(K)}))=0$ for each ${\bullet \in \{+,-\}}$
\end{enumerate}
Under the above assumptions the number of $\Zp$-extensions of $K$, where the rank of $E$ does not stay bounded is at most 3.
\end{proposition}
\begin{proof}
For this proof we fix $S$ to be the set of primes of $K$ dividing $Np$.
Let $\Q_{cyc}$ be the cyclotomic $\Zp$-extension of $\Q$. Let $\Selpmsingle(E/K_{cyc})$ be the plus and minus Selmer groups over $\Q_{cyc}$ as in Kobayashi's paper \cite{Kob}. Let $X^{\pm}(E/\Q_{cyc})$ denote the Pontryagin dual of $\Selpmsingle(E/\Q_{cyc})$. By \cite[Theorem~1.3]{Kob} we have for each ${\bullet \in \{+,-\}}$ that
\begin{equation}\label{lambda_relation}
\lambda(X^{\bullet}(E/\Q_{cyc})) \leq \lambda(L^{\bullet}_p(E)),
\end{equation}
\begin{equation}\label{mu_relation}
\mu(X^{\bullet}(E/\Q_{cyc})) \leq \mu(L^{\bullet}_p(E)),
\end{equation}
and similarly for $E^{(K)}$. Since $p$ is odd, we can identify $\Gal(\Q_{cyc}/\Q)$ and $\Gal(K_{cyc}/K)$. Let ${\Lambda_{cyc}=\Zp[[\Gal(K_{cyc}/K)]]}$. As in Proposition~\ref{Hcyc_prop}, we have for each ${\bullet \in \{+,-\}}$ an isomorphism of $\Lambda_{cyc}$-modules

$$X^{\bullet\bullet}(E/K_{cyc})\cong X^{\bullet}(E/\Q_{cyc}) \oplus X^{\bullet}(E^{(K)}/\Q_{cyc}). $$
Therefore from $(f)$, $(g)$, \eqref{lambda_relation} and \eqref{mu_relation} we have for each ${\bullet \in \{+,-\}}$ that
\begin{equation}\label{mu_rel1}
\mu(X^{\bullet\bullet}(E/K_{cyc}))=0,
\end{equation}
\begin{equation}\label{lambda_rel1}
\lambda(X^{\bullet\bullet}(E/K_{cyc})) \leq 1.
\end{equation}

From Theorem~\ref{mixedSelmertorsion_theorem} we have that $X^{+-}(E/K_{ac})$ and $X^{-+}(E/K_{ac})$ are $\Lambda$-torsion and
\begin{equation}\label{mu_rel2}
\mu(X^{+-}(E/K_{ac}))=\mu(X^{-+}(E/K_{ac}))=0,
\end{equation}
\begin{equation}\label{lambda_rel2}
\lambda(X^{+-}(E/K_{ac}))=\lambda(X^{-+}(E/K_{ac}))=1.
\end{equation}

By Proposition~\ref{Hcyc_prop} ${H_{cyc} \in \mathcal{H}^{++} \cap \mathcal{H}^{--}}$. Therefore from \eqref{mu_rel1} and Theorem~\ref{SelmerControlThm_prop} we have that for ${\bullet \in \{+,-\}}$ that $X^{\bullet \bullet}(E/\K_{\infty})_{H_{cyc}}$ is finitely generated over $\Zp$. So $X^{\bullet \bullet}(E/\K_{\infty})$ is finitely generated over $\Lambda(H_{cyc})$ whence also $X^{\bullet \bullet}(E/\K_{\infty})_f$ is finitely generated over $\Lambda(H_{cyc})$.

Since all the primes dividing $N$ split in $K/\Q$, it follows from \cite[Theorem~2]{Brink} that there are a finite number of primes of $K_{ac}$ above any prime of $K$ dividing $N$. Therefore it follows that no prime in $S$ splits completely in $K_{ac}/K$. Also we know that every prime of $K$ above $p$ ramifies in $K_{ac}/K$ (in fact, since we assumed that ${p \nmid h_K}$ such primes are totally ramified). Let ${H_{ac}=\Gal(K_{\infty}/K_{ac})}$. Then from the previous observations and Proposition~\ref{mixedSelmertorsion_prop} we get that ${H_{ac} \in \mathcal{H}^{+-}, \mathcal{H}^{-+}}$. Also just as above, we get from \eqref{mu_rel2} that both $X^{+-}(E/\K_{\infty})_f$ and $X^{-+}(E/\K_{\infty})_f$ are finitely generated over $\Lambda(H_{ac})$.

From the above and \eqref{lambda_rel1} and \eqref{lambda_rel2} we see that the desired result follows from Theorem~\ref{rankbound_theorem}.
\end{proof}
The fact that all the primes dividing $N$ are assumed to split in $K/\Q$ in the above proposition implies that the root number of $L(E_K, s)$ is $-1$. So in this case Conjecture \ref{MazurConjecture} predicts that the rank of $E$ stays bounded along all $\Zp$-extensions of $K$ except the anticyclotomic one. Therefore in order to give examples in this case satisfying Conjecture~\ref{MazurConjecture}, we need to lower the bound in the statement of the proposition by two. It is unclear to the authors how to achieve this. Nevertheless, we give examples below satisfying Proposition~\ref{Rankbound_prop}. Conditions~$(f)$ and $(g)$ were checked in the LMFDB database \cite{LMFDB} and the other conditions were verified by computations in SAGE \cite{Sage}.

\begin{itemize}
\item Let $E=21a1$ (Cremona labeling \cite{Cremona}), ${K=\Q(\sqrt{-5})}$ and ${p=23}$.
\item Let $E=34a1$, ${K=\Q(\sqrt{-15})}$ and ${p=23}$.
\item Let $E=44a1$, ${K=\Q(\sqrt{-7})}$ and ${p=29}$.
\end{itemize}

\section{Finding examples where the $\mathfrak{M}_H(G)$-property fails} \label{section:counterexample}
\subsection{Introduction}
Our aim in this section is to seek a generalization of the main results (Theorems~\ref{main_theorem} and \ref{main_theorem3}) to $\Zp$-extensions $\KK_{\infty}^H/K$ where we allow primes in $S$ to split completely. Working in this more general setting will allow us to construct examples where the $\mathfrak{M}_H(G)$-property for signed Selmer groups fails. In order to achieve this generalization we need to define new Selmer groups, the so-called \emph{Greenberg Selmer groups} (see also \cite{Gb_IPR, Gb_motives}). We will also define doubly signed Greenberg Selmer groups as in \cite{PW}, and we compare the Greenberg and doubly-signed Greenberg Selmer groups with the classical Selmer and doubly-signed Selmer groups (see Proposition~\ref{Selmercomp_prop} below).

We define ${S_p:=\{\fp, \bar{\fp}\}}$. Let $F_{\infty}/K$ be any infinite subextension of $\KK_{\infty}/K$. For any finite extension $L/K$ inside $F_{\infty}$ and any prime $w$ of $L$ let $I_{L_w}$ be the inertia subgroup of $G_{L_w}$. We have the restriction map ${\res_{L_w}: H^1(L_w, E[p^{\infty}]) \to H^1(I_{L_w}, E[p^{\infty}])}$. For any ${v \in S \setminus S_p}$ we define
$$J_v^s(E/F_{\infty}) = \dlim \bigoplus_{w | v} \img \res_{L_w}$$
where the direct limit runs over all finite extensions $L/K$ inside $F_{\infty}$. For any prime ${v \in S_p}$ we define
$$J_v(E/F_{\infty}) = \dlim \bigoplus_{w | v} H^1(L_w, E)[p^{\infty}]$$
where the direct limit runs over all finite extensions $L/K$ inside $F_{\infty}$. We define the Greenberg $p^{\infty}$-Selmer group of $E/F_{\infty}$ as
$$\displaystyle 0 \longrightarrow \Selinfg(E/F_{\infty}) \longrightarrow H^1(G_S(F_{\infty}), E[p^{\infty}]) \longrightarrow \bigoplus_{v \in S_p} J_v(E/F_{\infty}) \times \bigoplus_{v \in S \setminus S_p} J_v^s(E/F_{\infty}). $$
We denote the Pontryagin dual of $\Selinfg(E/F_{\infty})$ by $X^{Gr}(E/F_{\infty})$.

Now let $F_{\infty}/K$ be a $\Zp$-extension and let ${\bullet, \star \in \{+,-\}}$. We define as in \eqref{J_v_def} in Section~\ref{section:mainintroduction}
$$J_p^{\bullet\star}(E/F_{\infty}) = \bigoplus_{w | \fp} \frac{H^1(F_{\infty,w}, E[p^{\infty}])}{\hat{E}^{\bullet}(F_{\infty,w})\otimes \Qp/\Zp} \times \bigoplus_{w | \bar{\fp}} \frac{H^1(F_{\infty,w}, E[p^{\infty}])}{\hat{E}^{\star}(F_{\infty,w})\otimes \Qp/\Zp}. $$

We define the signed Greenberg Selmer group of $E/F_{\infty}$ as

$$\displaystyle 0 \longrightarrow \Selgpm(E/F_{\infty}) \longrightarrow \Selinfg(E/F_{\infty}) \longrightarrow J_p^{\bullet\star}(E/F_\infty). $$\\
We denote the Pontryagin dual of $\Selgpm(E/F_{\infty})$ by $X^{Gr, \bullet\star}(E/F_{\infty})$.

For ${\bullet, \star \in \{+,-\}}$ we define, as we did right after Proposition~\ref{invariants_prop},
$$J_p^{\bullet\star}(E/\KK_{\infty}) = \bigoplus_{w | \fp}\frac{H^1(\KK_{\infty,w}, E[p^{\infty}])}{\hat{E}^{\bullet}(\KK_{\infty,w})\otimes \Qp/\Zp} \times  \bigoplus_{w | \bar{\fp}}\frac{H^1(\KK_{\infty,w}, E[p^{\infty}])}{\hat{E}^{\star}(\KK_{\infty,w})\otimes \Qp/\Zp}, $$

and we define the signed Greenberg Selmer group of $E/\KK_{\infty}$ as

$$\displaystyle 0 \longrightarrow \Selgpm(E/\KK_{\infty}) \longrightarrow \Selinfg(E/\KK_{\infty}) \longrightarrow J_p^{\bullet\star}(E/\KK_{\infty}). $$\\
We denote the Pontryagin dual of $\Selgpm(E/\KK_{\infty})$ by $X^{Gr, \bullet\star}(E/\KK_{\infty})$.

\begin{remark}\label{Seldef_remark}
Let $F_{\infty}/K$ be an infinite subextension of $\KK_{\infty}/K$. For any ${v \in S \setminus S_p}$ we define
$$J_v^t(E/F_{\infty}) = \prod_{w | v} H^1(I_{F_{\infty,w}}, E[p^{\infty}]), $$
where $I_{F_{\infty,w}}$ is the inertia group. The proof of \cite[Lemma~2.1]{KM_Cong} shows that we get the same definition for $\Selinfg(E/F_{\infty})$ if we replace $J_v^s(E/F_{\infty})$ by $J_v^t(E/F_{\infty})$. It follows that this is also true for $\Selgpm(E/F_{\infty})$ for any ${\bullet, \star \in \{+,-\}}$.
\end{remark}

\begin{remark}
Let $F_{\infty}/K$ be an infinite subextension of $\KK_{\infty}/K$ and let $w$ be a prime of $F_{\infty}$ not above $p$. If $F_{\infty,w}/K_w$ contains a $\Zp$-subextension, then the profinite degree of ${\Gal(\overbar{F_{\infty,w}}/F_{\infty,w})/I_{F_{\infty,w}}}$ is prime to $p$. Therefore from Remark~\ref{Seldef_remark} we see that if $F_{\infty}/K$ is a $\Zp$-extension and no prime in $S$ splits completely $F_{\infty}/K$, then ${\Selinfg(E/F_{\infty})=\Selinf(E/F_{\infty})}$ and for any ${\bullet, \star \in \{+,-\}}$ we have ${\Selgpm(E/F_{\infty})=\Selpm(E/F_{\infty})}$. Moreover, we always have ${\Selinfg(E/\KK_{\infty})=\Selinf(E/\KK_{\infty})}$ and ${\Selgpm(E/\KK_{\infty})=\Selpm(E/\KK_{\infty})}$.
\end{remark}

We will work with these Greenberg Selmer groups and Greenberg signed Selmer groups in this section. Our idea to work with the Greenberg Selmer groups and Greenberg signed Selmer groups was inspired by the paper~\cite{PW} of Pollack--Weston. The advantage of working with these Selmer groups is that they behave well with $\Zp$-extensions where primes in $S$ split completely. On the contrary, the Selmer groups defined in Section~\ref{section:mainintroduction} are not well-behaved if primes in $S$ split completely.  For example, Theorem~\ref{SelmerControlThm_prop} shows that for any ${\bullet, \star \in \{+,-\}}$ and any ${H \in \mathcal{H}^{\bullet\star}}$ the map
$$s^{\bullet\star}: \Selpm(E/\KK_{\infty}^H) \longrightarrow \Selpm(E/\KK_{\infty})^H$$
is an isomorphism. It is a key assumption for this result that primes in $S$ do not split completely in $\KK_{\infty}^H/K$. For if a prime ${v \in S}$ splits completely in $\KK_{\infty}^H/K$, then $\coker s^{\bullet\star}$ can be infinite. We shall see that if we work with Greenberg signed Selmer groups instead, then the map $s^{\bullet\star}$ is an isomorphism even if primes in $S$ split completely in $\KK_{\infty}^H/K$. Another nice feature of Greenberg Selmer groups and Greenberg signed Selmer groups is that if $F_{\infty}/K$ is a $\Zp$-extension, then primes in $S$ that split completely in $F_{\infty}/K$ can contribute to the $\mu$-invariants of $X^{Gr, \bullet\star}(E/F_{\infty})$ and $X^{Gr}(E/F_{\infty})$. In the last subsection we will utilize this relationship with $\mu$-invariants to give examples where the $\mathfrak{M}_H(G)$-property for signed Selmer groups fails.

In this section we would like to prove similar equivalences to those in Theorems~\ref{main_theorem} and \ref{main_theorem3} for Greenberg signed Selmer groups and Greenberg Selmer groups where we remove the condition that no prime in $S$ splits completely in $\KK_{\infty}^H/K$. We feel that all the equivalences in Theorems~\ref{main_theorem} and \ref{main_theorem3} can be shown in this setup. However we only focus on proving the equivalence of the $\mathfrak{M}_H(G)$-property with the equality of $\mu$-invariants (Theorem~\ref{main_theorem}, ${(a) \Longleftrightarrow (d)}$ and Theorem~\ref{main_theorem3}, ${(a) \Longleftrightarrow (c)}$).

\begin{definition} \label{def:mathcal{H_s}}
For ${\bullet, \star \in \{+,-\}}$ let ${\mathcal{H}_s^{\bullet\star}}$ be the set of all subgroups ${H = \overbar{\langle \sigma^a\tau^b \rangle}}$ of ${G = \Gal(\KK_{\infty}/K)}$ which are topologically generated by $\sigma^a\tau^b$ for some ${[(a,b)] \in \mathbb{P}^1(\Zp)}$ such that
\begin{enumerate}[(a)]
\item Every prime of $K$ above $p$ ramifies in $\KK_{\infty}^H/K$,
\item $X^{Gr, \bullet\star}(E/\KK_{\infty}^H)$ is a torsion $\Lambda_{G/H}$-module, where ${\Lambda_{G/H} = \Z_p[[G/H]]}$.
\end{enumerate}
By abuse of notation, we say that a $\Z_p$-extension ${L_{\infty} \in \mathcal{E}}$ of $K$ is contained in $\mathcal{H}_s^{\bullet\star}$ iff ${\Gal(\KK_{\infty}/L_{\infty}) \in \mathcal{H}_s^{\bullet\star}}$.
\end{definition}
In the above definition $\mathcal{H}_s^{\bullet\star}$ is a potentially larger set than $\mathcal{H}^{\bullet\star}$ since we allow primes in $S$ to split completely (it is indeed larger if ${S \ne \emptyset}$). Our main aim in this section is to prove the following theorems.

\begin{theorem}\label{splitmain_theorem}
With notation as above, let ${\bullet, \star \in \{+,-\}}$. We assume that $\mathcal{H}_s^{\bullet\star}$ is non-empty and we choose some ${H=\overbar{\langle \sigma^a\tau^b \rangle}\in \mathcal{H}_s^{\bullet\star}}$.

Then $X^{Gr, \bullet\star}(E/\KK_{\infty})$ is $\Lambda_2$-torsion and the following are equivalent:
\begin{enumerate}[(a)]
\item $X^{Gr, \bullet\star}(E/\KK_{\infty})_f$ is finitely generated over $\Lambda(H)$.
\item ${\mu_G(X^{Gr, \bullet\star}(E/\KK_{\infty}))=\mu_{G/H}(X^{Gr, \bullet\star}(E/\KK_{\infty}^H))}$.
\end{enumerate}
\end{theorem}

\begin{theorem}\label{splitmain_theorem3}
Let ${\bullet, \star \in \{+,-\}}$ and ${H=\overbar{\langle \sigma^a\tau^b \rangle}\in \mathcal{H}_s^{\bullet\star}}$. The following statements are equivalent:
\begin{enumerate}[(a)]
\item $X^{Gr}(E/\KK_{\infty})_f$ is finitely generated over $\Lambda(H)$.
\item \resizebox{.92\hsize}{!}{${\mu_G(T_{\Lambda_2}(X^{Gr}(E/\KK_\infty))) = \mu_{G/H}(T_{\Lambda}(X^{Gr}(E/\KK_\infty^H))) - \mu_{G/H}(T_{\Lambda}([F_{\Lambda_2}(X^{Gr}(E/\KK_{\infty}))]_H))}$}.
\end{enumerate}
\end{theorem}

\subsection{Proofs of Theorems~\ref{splitmain_theorem} and \ref{splitmain_theorem3}}
First, we prove our control theorem for Greenberg signed Selmer groups.

\begin{theorem}\label{splitSelmerControlThm_prop}
Let ${\bullet, \star \in \{+,-\}}$ and ${H \in \mathcal{H}_s^{\bullet\star}}$. The restriction map induces a map
$$s^{\bullet\star}: \Selpm(E/\KK_{\infty}^H) \longrightarrow \Selpm(E/\KK_{\infty})^H$$
which is an isomorphism.
\end{theorem}
\begin{proof}
As we observed in Remark~\ref{rem:welldefined} the proof of Lemma~\ref{containment_lemma} shows that the restriction map does in fact induce the map $s^{\bullet\star}$. Now we need to show that $s^{\bullet\star}$ is an isomorphism. Let ${F_{\infty}=\KK_{\infty}^H}$ and let ${S_p=\{\fp, \bar{\fp}\}}$.
Now consider the commutative diagram
\begin{equation*}
\begin{tikzcd}[column sep = small, scale cd=0.83]
0 \arrow[r] & \Selgpm(E/\KK_{\infty})^H \arrow[r] & H^1(G_S(\KK_{\infty}), E[p^{\infty}])^H  \arrow[r] & J_p^{\bullet\star}(E/\KK_{\infty})^H \times \bigoplus_{v \in S \setminus S_p} J_v^s(E/\KK_{\infty})^H\\
0 \arrow[r] & \Selpm(E/F_{\infty}) \arrow[r] \arrow[u, "s^{\bullet\star}"] & H^1(G_S(F_{\infty}), E[p^{\infty}]) \arrow[u, "g"] \arrow[r] & J_p^{\bullet\star}(E/F_{\infty}) \times \bigoplus_{v \in S \setminus S_p} J_v^s(E/F_{\infty}) \arrow[u, "h^{\bullet\star}"]
\end{tikzcd}
\end{equation*}
To prove that $s$ is an isomorphism, using the snake lemma applied to the diagram above it will suffice to show that $g$ is an isomorphism and that $h^{\bullet\star}$ is an injection. We have ${\ker g=H^1(H, E(\KK_{\infty})[p^{\infty}])}$ and an injection ${\coker g \hookrightarrow H^2(H, E(\KK_{\infty})[p^{\infty}])}$. So by Lemma~\ref{p-torsion_lemma} we get that $g$ is an isomorphism. We can write ${h^{\bullet\star}=\gamma^{\bullet\star} \times h'}$ where $\gamma^{\bullet\star}$ (resp. $h'$) is the restriction of $h$ to $J_p^{\bullet\star}(E/F_{\infty})$ (resp. ${\bigoplus_{v \in S \setminus S_p} J_v^s(E/F_{\infty})}$).

Proposition~\ref{invariants_prop} holds true for $H$ as the proof does not depend on whether any prime in $S$ splits completely in $\KK_{\infty}^H/K$. Therefore the proof of Proposition~\ref{localinjection_prop} shows that $\gamma^{\bullet\star}$ is an injection. So we see that the proof will be complete if we can show if $h'$ is an injection. We can write ${h'=\bigoplus_{v \in S\setminus S_p} h'_v}$. Let ${v \in S \setminus S_p}$ and let $\Delta_{F_{\infty}}$ (resp. $\Delta_{\KK_{\infty}}$) be the set of all finite extensions $L/K$ inside $F_{\infty}$ (resp. $\KK_{\infty}$). We can write $h'_v$ as
$$h'_v: \dlim_{L \in \Delta_{F_{\infty}}} \bigoplus_{v | w} \img \res_{L_w} \to \left(\dlim_{L' \in \Delta_{\KK_{\infty}}} \bigoplus_{v | w} \img \res_{L'_w}\right)^H, $$
where $\res_{L_w}$ is defined as at the beginning of this section.

For any prime $w$ of $F_{\infty}$ (resp. $\KK_{\infty}$) above $v$ we let $I_{F_{\infty,w}}$ (resp. $I_{\KK_{\infty,w}}$) be the inertia group. The proof of \cite[Lemma~2.1]{KM_Cong} shows that we have injections

$$\dlim_{L \in \Delta_{F_{\infty}}} \bigoplus_{w | v} \img \res_{L_w} \hookrightarrow \prod_{w|v} H^1(I_{F_{\infty,w}}, E[p^{\infty}]), $$

$$\dlim_{L' \in \Delta_{\KK_{\infty}}} \bigoplus_{w | v} \img \res_{L'_w} \hookrightarrow \prod_{w|v} H^1(I_{\KK_{\infty,w}}, E[p^{\infty}]), $$
where the direct products range over all primes $w$ of $F_{\infty}$ (resp. $\KK_{\infty}$) above $v$. Thus to show that $h'_v$ is an injection it will suffice to show that the map
$$\psi_v: \prod_{w|v} H^1(I_{F_{\infty,w}}, E[p^{\infty}]) \to \prod_{w|v} H^1(I_{\KK_{\infty,w}}, E[p^{\infty}])$$
is an injection. For any prime $w$ of $\KK_{\infty}$ above $v$ we have that $\KK_{\infty,w}/F_{\infty,w}$ is unramified. Therefore ${I_{F_{\infty,w}}=I_{\KK_{\infty,w}}}$ and so $\psi_v$ is an injection. Thus $h'_v$ is an injection which in turn shows that $h'$ is an injection. This completes the proof.
\end{proof}

\begin{proposition}\label{splitSelmerControlThm_prop2}
Let ${\bullet, \star \in \{+,-\}}$ and ${H \in \mathcal{H}_s^{\bullet\star}}$. Then the natural map (induced by restriction)
$$s: \Selinfg(E/\KK_{\infty}^H) \to \Selinfg(E/\KK_{\infty})^H$$
is an isomorphism.
\end{proposition}
\begin{proof}
For simplicity, let ${F_{\infty}=\KK_{\infty}^H}$. For any prime $v$ of $F_{\infty}$ above $p$ we have ${J_p(E/F_{\infty})=0}$. This follows from \cite[Proposition~4.8]{CG} since $E$ has good supersingular reduction at $p$ and every prime of $K$ above $p$ ramifies in $F_{\infty}/K$. Similarly we have ${J_p(E/\KK_{\infty})=0}$. Taking this into account we have a commutative diagram with vertical maps induced by restriction
\begin{equation*}
\begin{tikzcd}[column sep = small, scale cd=0.85]
0 \arrow[r] & \Selinfg(E/\KK_{\infty})^H \arrow[r] &  H^1(G_S(\KK_{\infty}), E[p^{\infty}])^H \arrow[r] & \bigoplus_{v \in S \setminus S_p} J_v^s(E/\KK_{\infty})^H\\
0 \arrow[r] & \Selinfg(E/F_{\infty}) \ar[r] \arrow[u, "s"] &H^1(G_S(F_{\infty}), E[p^{\infty}]) \arrow[u, "g"] \arrow[r, "\theta"] & \bigoplus_{v \in S \setminus S_p} J_v^s(E/F_{\infty})\arrow[u, "h"]
\end{tikzcd}
\end{equation*}
As in the proof of Proposition~\ref{SelmerControlThm_prop2} the map $g$ is an isomorphism and so by the snake lemma to show that $s$ is an isomorphism it will suffice to show that $h$ is injective. The injectivity of the map $h$ is shown in the proof of Theorem~\ref{splitSelmerControlThm_prop}.
\end{proof}

The main control theorem for Greenberg Selmer groups that we need is the following result. For any ${\bullet, \star \in \{+,-\}}$ and ${H \in \mathcal{H}_s^{\bullet\star}}$, let $\hat{s}$ denote the dual of the map $s$ in Proposition~\ref{splitSelmerControlThm_prop2}.

\begin{theorem}\label{splitTorsionControlThm}
Each map $\hat{s}$ induces an injection
$$\hat{s}': (T_{\Lambda_2}(X^{Gr}(E/\KK_{\infty})))_H \hookrightarrow T_{\Lambda}(X^{Gr}(E/\KK_{\infty}^H))$$
such that there exists an $\Lambda$-isomorphism
\[ \phi_H: T_{\Lambda}([F_{\Lambda_2}(X^{Gr}(E/\KK_{\infty}))]_H)  \isomarrow \coker \hat{s}'. \]
\end{theorem}
\begin{proof}
The proof is identical to that of Theorem~\ref{TorsionControlThm} using Proposition~\ref{splitSelmerControlThm_prop2}.
\end{proof}

Let ${\bullet, \star \in \{+,-\}}$ and ${H \in \mathcal{H}_s^{\bullet\star}}$. Consider the sequence
\begin{align} \begin{tikzcd}[scale cd=0.86] \label{eq:definingsplitSelbulletstar}
\displaystyle 0 \longrightarrow \Selgpm(E/\KK_{\infty}^H) \longrightarrow H^1(G_S(\KK_{\infty}^H), E[p^{\infty}]) \stackrel{\lambda_H^{\bullet\star}}{\longrightarrow} J_p^{\bullet\star}(E/\KK_{\infty}^H) \times \bigoplus_{v \in S, v \nmid p} J_v^s(E/\KK_{\infty}^H) \longrightarrow 0. \end{tikzcd} \end{align}
	
\begin{proposition}\label{splitSeln_surjective_prop} 	
The above sequence~\eqref{eq:definingsplitSelbulletstar} is exact and we have that ${H^2(G_S(\KK_{\infty}^H), E[p^{\infty}])=0}$.
\end{proposition}
\begin{proof}
The proof is very similar to the proof of Proposition~\ref{Seln_surjective_prop}. We just need to make one adjustment. We use the setup and notation in the proof of Proposition~\ref{Seln_surjective_prop}. Let $w$ be a prime of $\KK_{\infty}^H$ above a prime of ${S \setminus \{\fp, \bar{\fp}\}}$.
For any ${m \geq 0}$ we let $I_{(\KK_{(m,\fp)}^H)_w}$ be the inertia group of ${(\KK_{(m,\fp)}^H)_w}$. Then ${H^1(\Gal(\overbar{\KK_{(m,\fp)}^H)_w}/(\KK_{(m,\fp)}^H)_w)/I_{(\KK_{(m,\fp)}^H)_w}, E[p^{\infty}]^{I_{(\KK_{(m,\fp)}^H)_w}})}$ is the kernel of the restriction map
$$\res_{(\KK_{(m,\fp)}^H)_w}: H^1((\KK_{(m,\fp)}^H)_w, E[p^{\infty}]) \to H^1(I_{(\KK_{(m,\fp)}^H)_w}, E[p^{\infty}]). $$
We define $H^1_{nr}((\KK_{(m,\fp)}^H)_w, T)$ as the exact annihilator of ${H^1(\Gal(\overbar{\KK_{(m,\fp)}^H)_w}/(\KK_{(m,\fp)}^H)_w)/I_{(\KK_{(m,\fp)}^{H_n})_w}, E[p^{\infty}]^{I_{(\KK_{(m,\fp)}^H)_w}})}$ with respect to the Tate pairing
$$H^1((\KK_{(m,\fp)}^H)_w, E[p^{\infty}]) \times H^1((\KK_{(m,\fp)}^H)_w, T) \longrightarrow \Qp/\Zp $$
(recall that ${T = T_p(E)}$ denotes the $p$-adic Tate module of $E$).
We define
$$H^1_{\text{Iw},nr}((\KK_{\infty}^H)_w, T)=\ilim H^1_{nr}((\KK_{(m,\fp)}^H)_w, T). $$
Let $S_{\text{tame}}$ be the set of primes of $\KK_{\infty}^H$ above a prime ${v \in S}$ where ${v \nmid p}$.
Then as in the proof of Proposition~\ref{Seln_surjective_prop} we have the Poitou--Tate exact sequence
\begin{align}\label{Poitou-Tate_seq2}
H^1_{\text{Iw}}(G_S(\KK_{\infty}^H), T) \stackrel{\text{loc}^{\bullet\star}}{\longrightarrow} Z \longrightarrow X^{Gr, \bullet\star}(E/\KK_{\infty}^H) \longrightarrow H^2_{\text{Iw}}(G_S(\KK_{\infty}^H), T) \longrightarrow 0
\end{align}
where
$$Z=A_{\fp} \times A_{\bar{\fp}} \times A_{\text{tame}}, $$
$$A_{\fp}:=\bigoplus_{w | \fp} H^1_{\text{Iw}}((\KK_{\infty}^H)_w, T)/H^1_{\text{Iw}, \bullet}((\KK_{\infty}^H)_w, T), $$
$$A_{\bar{\fp}}:=\bigoplus_{w | \bar{\fp}} H^1_{\text{Iw}}((\KK_{\infty}^H)_{\bar{w}}, T)/H^1_{\text{Iw}, \star}((\KK_{\infty}^H)_{\bar{w}}, T), $$
$$A_{\text{tame}}:= \bigoplus_{w \in S_{\text{tame}}} H^1_{\text{Iw}}((\KK_{\infty}^H)_w, T)/H^1_{\text{Iw}, nr}((\KK_{\infty}^H)_w, T). $$

As in the proof of Proposition~\ref{Seln_surjective_prop} we will get our desired result if we can show that ${\rank_{\Lambda}(Z)=[\KK_{(s, \fp)}^H:\Q]}$. In the proof of Proposition~\ref{Seln_surjective_prop} we showed that ${\rank_{\Lambda}(A_{\fp} \times A_{\bar{\fp}})=[\KK_{(s, \fp)}^H:\Q]}$, so we need to show that $A_{\text{tame}}$ is $\Lambda$-torsion.
We can write
$$A_{\text{tame}}=\bigoplus_{v \in S, v \nmid p} A_v$$
where
$$A_v=\bigoplus_{w \mid v} H^1_{\text{Iw}}((\KK_{\infty}^H)_w, T)/H^1_{\text{Iw}, nr}((\KK_{\infty}^H)_w, T). $$
Let ${v \in S}$ where ${v \nmid p}$. We need to show that $A_v$ is $\Lambda$-torsion. We distinguish two cases. The first case is when $v$ does not split completely in $\KK_{\infty}^H/K$. In this case choose a prime $w$ of $\KK_{\infty}^H$ above $v$. We have by Tate local duality that ${H^1_{\text{Iw}}((\KK_{\infty}^H)_w, T)/H^1_{\text{Iw}, nr}((\KK_{\infty}^H)_w, T)}$ is isomorphic to the Pontryagin dual of ${H^1(\Gal(\overbar{\KK^H_{\infty,w}}/\KK^H_{\infty,w})/I_{\KK^H_{\infty,w}}, E[p^{\infty}]^{I_{\KK^H_{\infty,w}}})}$, where $I_{\KK^H_{\infty,w}}$ is the inertia group. Since $v$ does not split completely in $\KK_{\infty}^H/K$, the quotient group ${\Gal(\overbar{\KK^H_{\infty,w}}/\KK^H_{\infty,w})/I_{\KK^H_{\infty,w}}}$ has profinite degree prime to $p$. Thus we see in this case that ${A_v=0}$.

Now consider the case when $v$ splits completely in $\KK_{\infty}^H/K$. Let $w$ be a prime of $\KK_{\infty}^H$ above $v$. Then as above ${H^1_{\text{Iw}}((\KK_{\infty}^H)_w, T)/H^1_{\text{Iw}, nr}((\KK_{\infty}^H)_w, T)}$ is isomorphic to the Pontryagin dual of ${H^1(\Gal(\overbar{\KK^H_{\infty,w}}/\KK^H_{\infty,w})/I_{\KK^H_{\infty,w}}, E[p^{\infty}]^{I_{\KK^H_{\infty,w}}})}$. If $\sigma$ is a topological generator of the $p$-primary subgroup of ${\Gal(\overbar{\KK^H_{\infty,w}}/\KK^H_{\infty,w})/I_{\KK^H_{\infty,w}}}$, then
$$H^1(\Gal(\overbar{\KK^H_{\infty,w}}/\KK^H_{\infty,w})/I_{\KK^H_{\infty,w}}, E[p^{\infty}]^{I_{\KK^H_{\infty,w}}})\cong E[p^{\infty}]^{I_{\KK^H_{\infty,w}}}/(\sigma-1)E[p^{\infty}]^{I_{\KK^H_{\infty,w}}}. $$
Since $\KK^H_{\infty,w}$ is a finite extension of $\Ql$ where $l$ is the prime of $\Q$ below $v$, the kernel of ${\sigma-1}$ acting on ${E[p^{\infty}]^{I_{\KK^H_{\infty,w}}}}$ is finite. It follows from this that ${H^1(\Gal(\overbar{\KK^H_{\infty,w}}/\KK^H_{\infty,w})/I_{\KK^H_{\infty,w}}, E[p^{\infty}]^{I_{\KK^H_{\infty,w}}})}$ is finite.

Now let ${\Gamma:=\Gal(\KK_{\infty}^H/K)}$ and choose a prime $w$ of $\KK_{\infty}^H$ above $v$. Then ${A_v=\coind_{\{1\}}^{\Gamma} H^1_{\text{Iw}}((\KK_{\infty}^H)_w, T)/H^1_{\text{Iw}, nr}((\KK_{\infty}^H)_w, T)}$. Since we have shown that ${H^1_{\text{Iw}}((\KK_{\infty}^H)_w, T)/H^1_{\text{Iw}, nr}((\KK_{\infty}^H)_w, T)}$ is finite, it follows from Lemma~\ref{coind_lemma} that ${\rank_{\Zp[[\Gamma]]}(A_v)=0}$. Therefore by \cite[Corollary~1.5]{Howson} we have ${\rank_{\Lambda}(A_v)=0}$ as desired. This completes the proof.
\end{proof}

\begin{lemma}\label{splitSel_Hn_surjective_lemma}
	Let ${\bullet, \star \in \{+,-\}}$ and ${H \in \mathcal{H}_s^{\bullet\star}}$. We have an exact sequence
	\begin{align}\begin{tikzcd}[scale cd=0.78] 0 \longrightarrow \Selgpm(E/\KK_{\infty})^H\longrightarrow H^1(G_S(\KK_{\infty}), E[p^{\infty}])^H \longrightarrow J_p^{\bullet\star}(E/\KK_{\infty})^H \times \bigoplus_{v \in S, v \nmid p} J_v(E/\KK_{\infty})^H \longrightarrow 0. \end{tikzcd} \end{align}
\end{lemma}
\begin{proof}
We prove this lemma using the same strategy as in the proof of Lemma~\ref{Sel_Hn_surjective_lemma}. Let ${F_{\infty}=\KK_{\infty}^H}$ and ${S_p=\{\fp, \bar{\fp}\}}$. Taking into account Proposition~\ref{splitSeln_surjective_prop}, the proof of Lemma~
\ref{Sel_Hn_surjective_lemma} shows that we only need to prove that the restriction map
$$h:\bigoplus_{v \in S \setminus S_p} J_v^s(E/F_{\infty}) \longrightarrow \bigoplus_{v \in S \setminus S_p} J_v^s(E/\KK_{\infty})^H$$
is surjective. We can write ${h=\bigoplus_{v \in S\setminus S_p} h_v}$. Let ${v \in S \setminus S_p}$ and let $\Delta_{F_{\infty}}$ (resp. $\Delta_{\KK_{\infty}}$) be the set of all finite extensions $L/K$ inside $F_{\infty}$ (resp. $\KK_{\infty}$). We can write $h_v$ as
$$h_v: \dlim_{L \in \Delta_{F_{\infty}}} \bigoplus_{v | w} \img \res_{L_w} \to \left(\dlim_{L' \in \Delta_{\KK_{\infty}}} \bigoplus_{v | w} \img \res_{L'_w}\right)^H. $$
We now show that $h_v$ is surjective. Consider the commutative diagram
\begin{equation*}
\xymatrix {
\left(\dlim_{L' \in \Delta_{\KK_{\infty}}} \bigoplus_{w | v} H^1(L'_w, E[p^{\infty}])\right)^H  \ar[r]^-{\theta_v} & \left(\dlim_{L' \in \Delta_{\KK_{\infty}}} \bigoplus_{w | v} \img \res_{L'_w}\right)^H\\
\dlim_{L \in \Delta_{F_{\infty}}} \bigoplus_{w | v} H^1(L_w, E[p^{\infty}]) \ar[u]^-{g_v} \ar[r] & \dlim_{L \in \Delta_{F_{\infty}}} \bigoplus_{w | v} \img \res_{L_w}\ar[u]^-{h_v}}
\end{equation*}

We claim that $g_v$ is surjective and $\theta_v$ is an isomorphism. From this it will follow that $h_v$ is surjective. First we show that $g_v$ is surjective. We consider two cases. The first case is when $v$ does not split completely in $F_{\infty}/K$. In this case let ${w_1, w_2, \ldots, w_n}$ be the primes of $F_{\infty}$ above $v$. For each prime $w_i$ we also denote by $w_i$ a fixed prime of $\KK_{\infty}$ above this prime. Taking into account Shapiro's Lemma and the inflation-restriction sequence we have an injection
$$\coker g_v \hookrightarrow \prod_{i=1}^n H^2(H_{w_i}, E(\KK_{w_i})[p^{\infty}])$$
where ${H_{w_i}=\Gal(K_{\infty,w_i}/F_{\infty,w_i})}$. Since $v$ does not split completely in $F_{\infty}/K$ and as $K_v$ does not have a $\Z_p^2$-extension we have ${H_{w_i}=0}$ for all $i$. Thus $g_v$ is surjective.

Now assume that $v$ splits completely in ${F_{\infty}/K}$. In this case our analysis of $\coker g_v$ is a bit more tricky. Let $\KK_n$ be the fixed field of  $G^{p^n}$, so that ${\Gal(\KK_n/K)\cong (\Z/p^n\Z)^2}$. Let $F_n$ be the $n$-th tower field of the $\Zp$-extension $F_{\infty}/K$ so that ${\Gal(F_n/K) \cong \Z/p^n\Z}$. Then we have that ${F_n \subseteq \KK_n}$ for all ${n \geq 0}$. The set $\{\KK_n\}$ (resp. $\{F_n\}$) is a cofinal subset of $\Delta_{\KK_{\infty}}$ (resp. $\Delta_{F_{\infty}}$). Therefore we have isomorphisms

$$\dlim_{L \in \Delta_{F_{\infty}}} \bigoplus_{w | v} H^1(L_w, E[p^{\infty}]) \cong \dlim_{n \in \N} \bigoplus_{w | v} H^1(F_{n, w}, E[p^{\infty}]), $$

$$\left(\dlim_{L' \in \Delta_{\KK_{\infty}}} \bigoplus_{w | v} H^1(L'_w, E[p^{\infty}])\right)^H \cong \left(\dlim_{n \in \N} \bigoplus_{w | v} H^1(\KK_{n,w}, E[p^{\infty}])\right)^H. $$

Furthermore we have
$$\left(\dlim_{n \in \N} \bigoplus_{w | v} H^1(\KK_{n,w}, E[p^{\infty}])\right)^H \cong \dlim_{n \in \N} \left(\bigoplus_{w | v} H^1(\KK_{n,w}, E[p^{\infty}])\right)^H. $$

Therefore by abuse of notation we can write our map $g_v$ as

$$g_v: \dlim_{n \in \N} \bigoplus_{w | v} H^1(F_{n,w}, E[p^{\infty}]) \to \dlim_{n \in \N} \left(\bigoplus_{w | v} H^1(\KK_{n,w}, E[p^{\infty}])\right)^H. $$

Recall that $v$ splits completely in $F_{\infty}/K$. Therefore for each ${n \in \N}$ there exist $p^n$ primes ${w_1, w_2, \ldots, w_{p^n}}$ of $F_n$ above $K$. For each ${n \in \N}$ we choose for each ${1 \leq i \leq p^n}$ a fixed prime of $\KK_n$ above $w_i$ which we also denote by the same letter. For any ${n \in \N}$ and any ${1 \leq i \leq p^n}$ we let ${B_{n, w_i}=\Gal(\KK_{n,w_i}/F_{n,w_i})}$. Then taking into account Shapiro's Lemma and the inflation-restriction sequence, we have an injection
$$\coker g_v \hookrightarrow \dlim_{n \in \N} \; \prod_{i=1}^{p^n} H^2(B_{n,w_i}, E(\KK_{n,w_i})[p^{\infty}]).$$
Let ${v_0=v, v_1, v_2, \ldots}$ be a sequence of primes ${v_i \in F_i}$ with ${v_{i+1} \mid v_i}$. As above, for each prime $v_i$ we also denote by $v_i$ a fixed prime of $\KK_{\infty}$ above this prime. We choose a prime $w$ of $\KK_{\infty}$ lying above all the primes $v_i$. From the fact that $v$ splits completely in $F_{\infty}/K$ it is easy to see that ${\ilim_{n \in \N} B_{n, v_i} \cong \Zp}$. Note that ${\Gal(\KK_{\infty,w}/K_v)\cong \ilim_{n \in \N} B_{n, v_i}}$. Since ${cd_p(\Zp)=1}$, we can conclude that
$$\dlim_{n \in \N} H^2(B_{n, v_i}, E(\KK_{n, v_i})[p^{\infty}]) \cong H^2(\Gal(\KK_{\infty,w}/K_v), E(\KK_{\infty, w})[p^{\infty}])=0. $$
It follows from this that ${\dlim_{n \in \N} \prod_{i=1}^{p^n} H^2(B_{n,w_i}, E(\KK_{n,w_i})[p^{\infty}])=0}$ and therefore $g_v$ is surjective.

Now we prove that $\theta_v$ is an isomorphism. Consider the surjective map that induces $\theta_v$:
$$\psi_v: \dlim_{L' \in \Delta_{\KK_{\infty}}} \bigoplus_{w | v} H^1(L'_w, E[p^{\infty}]) \twoheadrightarrow \dlim_{L' \in \Delta_{\KK_{\infty}}} \bigoplus_{w | v} \img \res_{L'_w}. $$
For any prime $w$ of $\KK_{\infty}$ above $v$ we let $I_{\KK_{\infty,w}}$ be the inertia group. The proof of \cite[Lemma~2.1]{KM_Cong} shows that we have injections

$$\dlim_{L' \in \Delta_{\KK_{\infty}}} \bigoplus_{w | v} H^1(L'_w, E[p^{\infty}]) \hookrightarrow \prod_{w | v} H^1(\KK_{\infty,w}, E[p^{\infty}]), $$
$$\dlim_{L' \in \Delta_{\KK_{\infty}}} \bigoplus_{w | v} \img \res_{L'_w}  \hookrightarrow \prod_{w | v} H^1(I_{\KK_\infty,w}, E[p^{\infty}]), $$
where the products run over all primes of $w$ of $\KK_{\infty}$ above $v$. From this we see that in order to prove that $\psi_v$ is injective it suffices to show that for any prime $w$ of $\KK_{\infty}$ above $v$ that the map
$$\res_w:  H^1(\KK_{\infty,w}, E[p^{\infty}]) \to H^1(I_{\KK_\infty,w}, E[p^{\infty}])$$
is an injection. For any such prime $w$ we have that
$$\ker \res_w=H^1(\Gal(\overbar{\KK_{\infty,w}}/\KK_{\infty,w})/I_{\KK_{\infty,w}}, E[p^{\infty}]^{I_{\KK_{\infty,w}}}). $$
The group ${\Gal(\overbar{\KK_{\infty,w}}/\KK_{\infty,w})/I_{\KK_{\infty,w}}}$ has profinite degree prime to $p$. Therefore we see that ${\ker \res_w=0}$. Thus $\psi_v$ is injective and hence is an isomorphism. It follows that $\theta_v$ is also an isomorphism.  This completes the proof.
\end{proof}

\begin{lemma}\label{splitH1_Sel_vanishing_lemma}
Let ${\bullet, \star \in \{+,-\}}$ and ${H \in \mathcal{H}_s^{\bullet\star}}$. Then ${H_1(H, X^{Gr, \bullet\star}(E/\KK_{\infty}))=0}$.
\end{lemma}
\begin{proof}
This can be proven in the same way as Lemma~\ref{H1_Sel_vanishing_lemma}, using Lemma~\ref{splitSel_Hn_surjective_lemma}.
\end{proof}

\begin{lemma}\label{splitH1_Sel_vanishing_lemma2}
Let ${\bullet, \star \in \{+,-\}}$ and ${H \in \mathcal{H}_s^{\bullet\star}}$. Then
\[ H_1(H, T_{\Lambda_2}(X^{Gr}(E/\KK_{\infty})))=0. \]
\end{lemma}
\begin{proof}
Noting that ${X^{Gr}(E/\KK_{\infty}))=X(E/\KK_{\infty}))}$ and ${X^{Gr, \bullet\star}(E/\KK_{\infty}))=X^{\bullet\star}(E/\KK_{\infty}))}$, we have by Proposition~\ref{comparison_prop} a $\Lambda_2$-module injection
$$T_{\Lambda_2}(X^{Gr}(E/\K_\infty)) \hookrightarrow X^{Gr, \bullet\star}(E/\KK_\infty). $$
Therefore the result follows from Lemma~\ref{splitH1_Sel_vanishing_lemma}.
\end{proof}

\begin{proof}[Proof of Theorem \ref{splitmain_theorem}]
The proof is identical to the proof of Theorem~\ref{mu_invariants_theorem}, using Theorem~\ref{splitSelmerControlThm_prop} and Lemma~\ref{splitH1_Sel_vanishing_lemma}.
\end{proof}

\begin{proof}[Proof of Theorem \ref{splitmain_theorem3}]
The proof is identical to the proof of Theorem~\ref{mu_invariants_theorem2}, using Theorem~\ref{splitTorsionControlThm} and Lemma~\ref{splitH1_Sel_vanishing_lemma2}.
\end{proof}

\subsection{Results on $\mu$-invariants}
In this subsection we collect some results on $\mu$-invariants that we need for the examples in the last subsection.

\begin{definition}
Let ${\bullet, \star \in \{+,-\}}$ and ${H \in \mathcal{H}_s^{\bullet\star}}$. For simplicity we let ${F_{\infty}=\KK_{\infty}^H}$ and we let $F_n$ be the $n$-th tower field of $F_{\infty}$. For any ${v \in S \setminus S_p}$ we define
\[ \Omega_{H,v} = \dlim_n \bigoplus_{w \mid v} H^1(\Gal(\overbar{F_{n,w}}/F_{n,w})/I_{F_{n,w}}, E[p^{\infty}]^{I_{F_{n,w}}}), \]
where the sum runs over all primes $w$ of $F_n$ above $v$. Here ${I_{F_{n,w}} \subseteq \Gal(\overbar{F_{n,w}}/F_{n,w})}$ denotes the inertia subgroup.
\end{definition}
For any prime $v$ of $K$ let $c_v$ be the Tamagawa number at $v$.

\begin{proposition}\label{Omega_prop}
Let ${\bullet, \star \in \{+,-\}}$, ${H \in \mathcal{H}_s^{\bullet\star}}$ and ${v \in S \setminus S_p}$. Let $l$ be the prime of $\Q$ below $v$. Then  $\Omega_{H,v}$ is a cofinitely generated torsion $\Lambda$-module. Furthermore, the following statements are true.
\begin{enumerate}[(a)]
\item  ${\mu((\Omega_{H,v})^{\vee})=0}$ if any one of the following conditions are true:
\begin{enumerate}[(i)]
\item $v$ does not split completely in $\KK_{\infty}^H/K$.
\item $E$ has good reduction at $l$.
\item $E$ has additive reduction at $l$.
\end{enumerate}
\item If $v$ splits completely in $\KK_{\infty}^H/K$, then ${\mu((\Omega_{H,v})^{\vee})=v_p(c_v)}$ if either
\begin{enumerate}[(i)]
\item $E$ has split multiplicative reduction at $l$, or
\item $l$ is inert in $K/\Q$ and $E$ has non-split multiplicative reduction at $l$.
\end{enumerate}
\end{enumerate}
\end{proposition}
\begin{proof}
Let $w$ be a prime of ${F_{\infty}:=\KK_{\infty}^H}$ above $v$. For simplicity we will denote ${\Gal(\overbar{F_{n,w}}/F_{n,w})/I_{F_{n,w}}}$ by $\mathcal{G}_n$ and $I_{F_{n,w}}$ by $I_n$. Let $\Lambda$ be the Iwasawa algebra of $F_{\infty}/K$. We have that
${H^1(\mathcal{G}_n, E[p^{\infty}]^{I_n})=E[p^{\infty}]^{I_n}/(\sigma_n-1)E[p^{\infty}]^{I_n}}$ where $\sigma_n$ is a topological generator of $\mathcal{G}_n$. The kernel of ${\sigma_n-1}$ acting on $E[p^{\infty}]^{I_n}$ is $E(F_{n,w})[p^{\infty}]$, which is finite. Let $(E[p^{\infty}]^{I_n})_{div}$ be the maximal divisible subgroup of $E[p^{\infty}]^{I_n}$. Since $(E[p^{\infty}]^{I_n})_{div}$ has finite $\Zp$-corank and the kernel of ${\sigma_n-1}$ acting on $(E[p^{\infty}]^{I_n})_{div}$ is finite, it follows that ${(\sigma_n-1)(E[p^{\infty}]^{I_n})_{div}=(E[p^{\infty}]^{I_n})_{div}}$. Therefore we have a surjection ${E[p^{\infty}]^{I_n}/(E[p^{\infty}]^{I_n})_{div} \twoheadrightarrow H^1(\mathcal{G}_n, E[p^{\infty}]^{I_n})}$. Let $I_{\infty}$ be the inertia subgroup of $F_{\infty,w}$. Since $F_{\infty,w}/K_v$ is unramified, we have ${I_n=I_{\infty}}$ for all $n$. Thus we see that ${H^1(\mathcal{G}_n, E[p^{\infty}]^{I_n})}$ is bounded by ${E[p^{\infty}]^{I_{\infty}}/(E[p^{\infty}]^{I_{\infty}})_{div}}$. From this we see that if $v$ does not split completely in $\KK_{\infty}^H/K$, then $\Omega_{H,v}$ is a direct limit of finite groups of bounded order, hence is finite. Therefore in this case $\Omega_{H,v}$ is a cofinitely generated $\Lambda$-torsion module with ${\mu(\Omega_{H,v}^{\vee})=0}$.

If $v$ splits completely in $\KK_{\infty}^H/K$, then we have ${\Omega_{H,v}=\coind_{\{1\}}^{\Gamma} H^1(\mathcal{G}_0, E[p^{\infty}]^{I_0})}$. Since as shown above ${H^1(\mathcal{G}_0, E[p^{\infty}]^{I_0})}$ is finite, it follows from Lemma~\ref{coind_lemma} that $\Omega_{H,v}^{\vee}$ is a finitely generated $\Lambda$-torsion module.

If $E$ has good reduction at $l$, then for all $n$ we have that ${E[p^{\infty}]^{I_n}=E[p^{\infty}]}$ (see \cite[Chapt.~VII, Propostion~4.1]{Silverman1}) is divisible. Since the kernel of ${\sigma_n-1}$ acting on $E[p^{\infty}]^{I_n}$ is finite, we conclude as above that ${H^1(\mathcal{G}_n, E[p^{\infty}]^{I_n})=0}$. Thus ${\Omega_{H,v}=0}$ in this case. If $E$ has additive reduction at $l$ (and hence additive reduction at $v$), then from the proof of \cite[Theorem~{1.2}]{Kida1} we see that ${E[p^{\infty}]^{I_n}=0}$ for all $n$. Therefore we also get that ${\Omega_{H,v}=0}$ in this case. The previous arguments prove statement~$(a)$.

Now we consider the case when $E$ has multiplicative reduction at $l$ and $v$ splits completely in $\KK_{\infty}^H/K$. Since $v$ splits completely in $F_{\infty}/K$, we have that ${F_{\infty,w}=K_v}$. We denote $I_0$ by $I$, $\mathcal{G}_0$ by $\mathcal{G}$ and $\sigma_0$ by $\sigma$. From \cite[Theorem~4.11]{Silverman1} and \cite[Chapt.~V,  Lemma~5.2 and Theorem~5.3]{Silverman2} one can deduce that there is a (unique) unramified extension $L/\Ql$ with ${[L:\Ql] \leq 2}$ such that $E$ is isomorphic to a Tate curve $E_q$ over $L$. If $E$ has split multiplicative reduction at $l$ we can take ${L=\Ql}$. Therefore if $E$ has split multiplicative reduction at $l$ or $E$ has non-split multiplicative reduction at $l$ and $l$ is inert in $K/\Q$, then $E$ is isomorphic to $E_q$ over $K_v$. We assume that we are in this case. Then as shown in \cite[\S~A.1.2]{Serre2} we have an exact sequence of $\Gal(\bar{K_v}/K_v)$-modules
\begin{equation}
0 \to \mu_{p^{\infty}} \to E[p^{\infty}] \to \Qp/\Zp \to 0
\end{equation}
where $\mu_{p^{\infty}}$ is the group of $p$-power roots of unity. Moreover, we have that $\mu_{p^{\infty}}$ is $I$-invariant. From the criterion of N\'{e}on-Ogg-Shafarevich we know that ${E[p^{\infty}]^I \neq E[p^{\infty}]}$. Therefore the above exact sequence induces the exact sequence
\begin{equation}\label{p-torsion_seq1}
0 \to \mu_{p^{\infty}} \to E[p^{\infty}]^I \to \Z/p^t\Z \to 0
\end{equation}
for some ${t \geq 0}$. Note that since $\mu_{p^{\infty}}$ is a divisible abelian group, we have an isomorphism ${E[p^{\infty}] \cong \mu_{p^{\infty}} \times \Z/p^t\Z}$ of abelian groups. Now \cite[Theorem~{1.2}]{Kida1} shows that ${t=v_p(c_v)}$. The exact sequence~\eqref{p-torsion_seq1} induces an exact sequence
\begin{equation}\label{p-torsion_seq2}
\mu_{p^{\infty}}/(\sigma-1)\mu_{p^{\infty}} \to E[p^{\infty}]^I/(\sigma-1)E[p^{\infty}]^I \to \Z/p^t\Z/(\sigma-1)\Z/p^t\Z \to 0.
\end{equation}
Since $\mu_{p^{\infty}}$ is divisible and as its $\mathcal{G}$-invariants are finite, the first term in the exact sequence~\eqref{p-torsion_seq2} is zero. Clearly the last term is $\Z/p^t\Z$. Thus we have ${H^1(\mathcal{G}, E[p^{\infty}]^I) \cong \Z/p^t\Z}$. Since $v$ splits completely in $F_{\infty}/K$ it follows from \cite[Corollary~A.4]{Kidwell} that ${(\Omega_{H,v})^{\vee} \cong \Lambda/p^t}$. This proves statement~$(b)$.
\end{proof}

Let ${\bullet, \star \in \{+,-\}}$ and ${H \in \mathcal{H}_s^{\bullet\star}}$. Recall from the beginning of Section~\ref{section:mainintroduction} that we  defined for any ${v \in S \setminus S_p}$:
$$J_v(E/\KK_{\infty}^H)=\dlim \bigoplus_{w|v}H^1(F_w,E[p^{\infty}]), $$
where the direct limit runs over finite extensions $F$ of $K$ contained in $\KK_{\infty}^H$ (note that ${H^1(F_w,E[p^{\infty}]) = H^1(F_w,E)[p^\infty]}$ for each $w$ by Mattuck's theorem, since ${w \mid v}$ does not divide $p$). If $v$ does not split completely in ${\KK_{\infty}^H/K}$, then $$J_v(E/\KK_{\infty}^H)= \bigoplus_{w|v}H^1((\KK_{\infty}^H)_w,E[p^{\infty}])$$
where the product ranges over all primes $w$ of $\KK_{\infty}^H$ above $v$. Therefore for ${H \in \mathcal{H}^{\bullet\star}}$ our definition above of $J_v(E/\KK_{\infty}^H)$ coincides with the definition used in the previous sections. 

Now we need the following two results that prove the surjectivity of the local-to-global maps defining the signed Selmer group and Selmer group. Let ${\bullet, \star \in \{+,-\}}$ and ${H \in \mathcal{H}_s^{\bullet\star}}$. Consider the sequence
\begin{align} \begin{tikzcd}[scale cd=0.86] \label{eq:definingSelbulletstarHs}
\displaystyle 0 \longrightarrow \Selpm(E/\KK_{\infty}^H) \longrightarrow H^1(G_S(\KK_{\infty}^H), E[p^{\infty}]) \stackrel{\lambda_H^{\bullet\star}}{\longrightarrow} J_p^{\bullet\star}(E/\KK_{\infty}^H) \times \bigoplus_{v \in S \setminus S_p} J_v(E/\KK_{\infty}^H) \longrightarrow 0. \end{tikzcd} \end{align}

\begin{proposition}\label{signedSeln_surjective_propHs}
The above sequence~\eqref{eq:definingSelbulletstarHs} is exact.
\end{proposition}
\begin{proof}
Proposition~\ref{Seln_surjective_prop} proves the exactness of the sequence for any ${H \in \mathcal{H}^{\bullet\star}}$. Since the proof does not use the fact that no prime in $S$ splits completely in ${\KK_{\infty}^H/K}$, this proof actually holds for any ${H \in \mathcal{H}_s^{\bullet\star}}$.
\end{proof}

Let ${\bullet, \star \in \{+,-\}}$ and ${H \in \mathcal{H}_s^{\bullet\star}}$. Consider the sequence
\begin{align} \begin{tikzcd}[scale cd=0.86] \label{eq:definingSelHs}
\displaystyle 0 \longrightarrow \Selinf(E/\KK_{\infty}^H) \longrightarrow H^1(G_S(\KK_{\infty}^H), E[p^{\infty}]) \stackrel{\lambda_H}{\longrightarrow} \bigoplus_{v \in S \setminus S_p} J_v(E/\KK_{\infty}^H) \longrightarrow 0. \end{tikzcd} \end{align}

\begin{proposition}\label{Seln_surjective_propHs}
The above sequence~\eqref{eq:definingSelHs} is exact.
\end{proposition}
\begin{proof}
For any prime $v$ of ${\KK_{\infty}^H}$ above $p$ we have ${H^1(\KK^H_{\infty,v}, E)[p^{\infty}]=0}$. This follows from \cite[Proposition~4.8]{CG} since $E$ has good supersingular reduction at $p$ and every prime of $K$ above $p$ ramifies in ${\KK_{\infty}^H/K}$. Therefore we see that ${\Selinf(E/\KK_{\infty}^H)}$ is the kernel of the map $\lambda_H$. Now consider the localization map

$$\lambda_H^{\bullet\star}: H^1(G_S(\KK_{\infty}^H), E[p^{\infty}]) \to J_p^{\bullet\star}(E/\KK_{\infty}^H) \times \bigoplus_{v \in S \setminus S_p} J_v(E/\KK_{\infty}^H)$$

and the projection map

$$\pi_H: J_p^{\bullet\star}(E/\KK_{\infty}^H) \times \bigoplus_{v \in S, v \nmid p} J_v(E/\KK_{\infty}^H) \to \bigoplus_{v \in S \setminus S_p} J_v(E/\KK_{\infty}^H).$$

By Proposition~\ref{signedSeln_surjective_propHs} the map $\lambda_H^{\bullet\star}$ is surjective. Since $\pi_H$ is also surjective, we can conclude that ${\lambda_H=\pi_H \circ \lambda_H^{\bullet\star}}$ is surjective.
\end{proof}

\begin{proposition}\label{Selmercomp_prop}
Let ${\bullet, \star \in \{+,-\}}$ and ${H \in \mathcal{H}_s^{\bullet\star}}$. Then we have exact sequences
$$0 \to \Selpm(E/\KK_{\infty}^H) \to \Selgpm(E/\KK_{\infty}^H) \to \bigoplus_{v \in S \setminus S_p} \Omega_{H,v} \to 0, $$

$$0 \to \Selinf(E/\KK_{\infty}^H) \to \Selinfg(E/\KK_{\infty}^H) \to \bigoplus_{v \in S \setminus S_p} \Omega_{H,v} \to 0. $$
\end{proposition}
\begin{proof}
Consider the maps
$$\alpha: H^1(G_S(\KK_{\infty}^H), E[p^{\infty}]) \to J_p^{\bullet\star}(E/\KK_{\infty}^H) \times \bigoplus_{v \in S \setminus S_p} J_v(E/\KK_{\infty}^H), $$
$$\beta: J_p^{\bullet\star}(E/\KK_{\infty}^H) \times \bigoplus_{v \in S \setminus S_p} J_v(E/\KK_{\infty}^H) \to J_p^{\bullet\star}(E/\KK_{\infty}^H) \times \bigoplus_{v \in S \setminus S_p} J_v^s(E/\KK_{\infty}^H). $$
Then we have an exact sequence
$$0 \to \ker(\alpha) \to \ker (\beta \circ \alpha) \to \ker (\beta) \to \coker(\alpha).$$
We have that ${\ker(\alpha)=\Selpm(E/\KK_{\infty}^H)}$, ${\ker(\beta \circ \alpha)=\Selgpm(E/\KK_{\infty}^H)}$ and ${\ker \beta = \bigoplus_{v \in S \setminus S_p} \Omega_{H,v}}$. By Proposition~\ref{signedSeln_surjective_propHs} ${\coker(\alpha)=0}$. Therefore we get the first exact sequence. Similarly we get the second exact sequence using Proposition~\ref{Seln_surjective_propHs}.
\end{proof}

We will utilize the following two results for examples in the last subsection.

\begin{corollary}\label{muinequality_corollary}
Let ${\bullet, \star \in \{+,-\}}$ and ${H \in \mathcal{H}_s^{\bullet\star}}$. We have
$$\mu(X^{Gr, \bullet\star}(E/\KK_{\infty}^H)) \geq {\sum_{v \in S  \setminus S_p}} \mu((\Omega_{H,v})^{\vee})$$
and
$$\mu(T_{\Lambda}(X^{Gr}(E/\KK_{\infty}^H))) \geq {\sum_{v \in S  \setminus S_p}} \mu((\Omega_{H,v})^{\vee}). $$
\end{corollary}

\begin{proposition}\label{mu_vanishing_prop2}
Suppose that we have ${\mu(X^{Gr, \bullet\star}(E/\KK_{\infty}^H))=0}$ for some ${\bullet, \star \in \{+,-\}}$ and ${H \in \mathcal{H}_s^{\bullet\star}}$. Then $X^{Gr, \bullet\star}(E/\KK_{\infty})$ is $\Lambda_2$-torsion and we have
$$\mu(T_{\Lambda_2}(X^{Gr}(E/\KK_{\infty})))=\mu(X^{Gr, \bullet\star}(E/\KK_{\infty}))=0. $$
\end{proposition}
\begin{proof}
By Proposition~\ref{splitSelmerControlThm_prop} we have an isomorphism
$$(X^{Gr, \bullet\star}(E/\KK_{\infty}))_H \isomarrow X^{Gr, \bullet\star}(E/\KK_{\infty}^H). $$
Since ${\mu(X^{Gr, \bullet\star}(E/\KK_{\infty}^H))=0}$, it follows that ${X^{Gr, \bullet\star}(E/\KK_{\infty}^H)_H}$ is finitely generated over $\Zp$. Therefore the same arguments as in the proof of Lemma~\ref{Xinf_torsion_lemma} show that $X^{Gr, \bullet\star}(E/\KK_{\infty})$ is a finitely generated torsion $\Lambda_2$-module.

Suppose that ${\mu(X^{Gr, \bullet\star}(E/\KK_{\infty})) \neq 0}$. Then taking into account \cite[Lemma~2.1]{MHG}, the structure theorem implies that for some ${m >0}$ we have an injection
$$\Lambda_2/p^m \hookrightarrow X^{Gr, \bullet\star}(E/\KK_{\infty}). $$
Since $X^{Gr, \bullet\star}(E/\KK_{\infty})$ is finitely generated over the Noetherian ring $\Lambda(H)$, it follows that ${\Lambda_2/p^m}$ is finitely generated over $\Lambda(H)$. This is clearly not true and hence we get a contradiction. Therefore ${\mu(X^{Gr, \bullet\star}(E/\KK_{\infty}))=0}$. Noting that ${X^{Gr}(E/\KK_{\infty})=X(E/\KK_{\infty})}$ and ${X^{Gr, \bullet\star}(E/\KK_{\infty})=X^{\bullet\star}(E/\KK_{\infty})}$, we have by Proposition~\ref{comparison_prop} a $\Lambda_2$-module injection
$$T_{\Lambda_2}(X^{Gr}(E/\K_\infty)) \hookrightarrow X^{Gr, \bullet\star}(E/\KK_\infty). $$
Therefore we also have ${\mu(T_{\Lambda_2}(X^{Gr}(E/\KK_{\infty})))=0}$.
\end{proof}

\subsection{Examples}\label{section:examples}
In this subsection we give examples where the $\mathfrak{M}_H(G)$-property for signed Selmer groups fails. Recall from Section~\ref{section:mainintroduction} the notation of the two special $\Z_p$-extensions $L^{(\fp)}_{\infty}$ and $L^{(\bar{\fp})}_{\infty}$ of $K$.
\begin{lemma}\label{extensionsplit_lemma}
No prime of $K$ splits completely in $L^{(\fp)}_{\infty}$ or $L^{(\bar{\fp})}_{\infty}$.
\end{lemma}
\begin{proof}
We prove this for $L^{(\fp)}_{\infty}$; the proof for $L^{(\bar{\fp})}_{\infty}$ will be similar. For any ${n \geq 0}$, let $K(\fp^n)$ (resp. $K(p^n)$) be the ray class field modulo $\fp^n$ (resp. $p^n$). Let ${K(\fp^{\infty})=\bigcup_{n=0}^{\infty} K(\fp^n)}$ and ${K(p^{\infty})=\bigcup_{n=0}^{\infty} K(p^n)}$. Then $L^{(\fp)}$ is the unique $\Zp$-extension inside $K(\fp^{\infty})$. By \cite[Main Lemma~1]{Brink} ${\Gal(K(p^{\infty})/K) \cong \Zp \times \Zp \times T}$ where $T$ is a finite group. Since ${K(\fp^{\infty}) \subseteq K(p^{\infty})}$, we see that ${\Gal(K(\fp^{\infty})/K) \cong \Zp \times S}$ where $S$ is a finite group. Let $I_{\fp}$ be the group of fractional ideals of $K$ that are relatively prime to $\fp$ and for any ${n \geq 0}$ let $P_{\fp^n}$ be the group of principal ideals $(\alpha)$ with integral ${\alpha \equiv 1 \pmod{\fp^n}}$. By class field theory, the Artin symbol is a surjective homomorphism
$$\Big(\frac{K(\fp^n)/K}{}\Big): I_{\fp^n} \twoheadrightarrow \Gal(K(\fp^n)/K)$$
with kernel $P_{\fp^n}$.

Now assume that $\fq$ is a prime of $K$ that splits completely in ${L^{(\fp)}_{\infty}/K}$. Then ${\fq \in I_{\fp}}$. Let $m$ be the order of the group $S$. Then it follows that for any ${n \geq 0}$ we have ${\Big(\frac{K(\fp^n)/K}{\fq^m}\Big)=1}$. Therefore ${\fq^m \in P_{\fp^n}}$ for all ${n \geq 0}$. In particular, ${\fq^m=(\alpha)}$ for some ${\alpha \in \mathcal{O}_K}$. Let $\mu$ be the finite group of roots of unity in $K$. Now let ${n \geq 0}$. Since ${\fq^m \in P_{\fp^n}}$, there exists ${\zeta_{n} \in \mu}$ with ${\zeta_{n} \alpha \equiv 1 \pmod{\fp^n}}$. Now let $t$ be the maximum value of ${\ord_{\fp}(\zeta \alpha -1)}$ as $\zeta$ ranges over $\mu$. Then for any ${n > t}$ we have ${\fq^m \not\in P_{\fp^{n}}}$. This contradicts the above.
\end{proof}

Now let ${L^{\pm}_p(E), L^{\pm}_p(E^{(K)}) \in \Zp[[T]]}$ be the plus/minus $p$-adic $L$-functions of Pollack \cite{Pollack} attached to $E/\Q$ and ${E^{(K)}/\Q}$. Our examples will make use of the following

\begin{proposition}\label{muzero_prop}
Let ${\bullet \in \{+,-\}}$. Assume that ${\mu(L^{\bullet}_p(E))=\mu(L^{\bullet}_p(E^{(K)}))=0}$ and let ${t:=\lambda(L^{\bullet}_p(E))+\lambda(L^{\bullet}_p(E^{(K)}))}$. Then
\begin{enumerate}[(a)]
\item ${\mu(T_{\Lambda_2}(X^{Gr}(E/\KK_{\infty})))=\mu(X^{Gr, \bullet\bullet}(E/\KK_{\infty}))=0}$.
\item The number of $\Zp$-extensions of $L_{\infty}/K$, where $X^{\bullet\bullet}(E/L_{\infty})$ is not a torsion $\Lambda$-module is at most $t$.
\end{enumerate}
\end{proposition}
\begin{proof}
Let $\Q_{cyc}$ be the cyclotomic $\Zp$-extension of $\Q$. Let ${\Selpmsingle(E/K_{cyc})}$ be the plus and minus Selmer groups over $\Q_{cyc}$ as in Kobayashi's paper~\cite{Kob}. Let ${X^{\pm}(E/\Q_{cyc})}$ denote the Pontryagin dual of $\Selpmsingle(E/\Q_{cyc})$. By \cite[Theorem~1.3]{Kob} we have that
\begin{equation}\label{Grlambda_relation}
\lambda(X^{\bullet}(E/\Q_{cyc})) \leq \lambda(L^{\bullet}_p(E)),
\end{equation}
\begin{equation}\label{Grmu_relation}
\mu(X^{\bullet}(E/\Q_{cyc})) \leq \mu(L^{\bullet}_p(E)),
\end{equation}
and similarly for $E^{(K)}$. Since $p$ is odd, we can identify $\Gal(\Q_{cyc}/\Q)$ and $\Gal(K_{cyc}/K)$. Let ${\Lambda_{cyc}=\Zp[[\Gal(K_{cyc}/K)]]}$. As in Proposition~\ref{Hcyc_prop}, we have an isomorphism of $\Lambda_{cyc}$-modules

$$X^{\bullet\bullet}(E/K_{cyc})\cong X^{\bullet}(E/\Q_{cyc}) \oplus X^{\bullet}(E^{(K)}/\Q_{cyc}). $$

Therefore from \eqref{Grlambda_relation} and \eqref{Grmu_relation} we have that
\begin{equation}\label{Grmu_rel}
\mu(X^{\bullet\bullet}(E/K_{cyc}))=0,
\end{equation}
\begin{equation}\label{Grlambda_rel}
\lambda(X^{\bullet\bullet}(E/K_{cyc})) \leq t.
\end{equation}
From Proposition~\ref{Hcyc_prop} we have that ${H_{cyc} \in \mathcal{H}_s^{\bullet\bullet}}$. Therefore from \eqref{Grmu_rel} and Proposition~\ref{mu_vanishing_prop2} we get $(a)$. Since ${\mu(X^{\bullet\bullet}(E/K_{cyc}))=0}$, we have that $X^{\bullet\bullet}(E/K_{cyc})$ is a finitely generated $\Zp$-module. By Theorem~\ref{SelmerControlThm_prop} we have an isomorphism ${X^{\bullet\bullet}(E/\KK_{\infty})_{H_{cyc}}\cong X^{\bullet\bullet}(E/K_{cyc})}$. Therefore by Nakayama's Lemma $X^{\bullet\bullet}(E/\KK_{\infty})$ is finitely generated over $\Lambda(H)$. It follows that $X^{\bullet\bullet}(E/\KK_{\infty})_f$ is also finitely generated over $\Lambda(H)$. Therefore $(b)$ follows from \eqref{Grlambda_rel} and Theorem~\ref{rankbound_theorem2}.
\end{proof}

Let $N$ be the conductor of $E$. For the next result we assume that no prime dividing $N$ ramifies in $K/\Q$. Write ${N=N^+N^-}$ where all the prime divisors of $N^+$ (respectively $N^-$) are split (respectively inert) in $K/\Q$. Let $K_{ac}/K$ be the anticyclotomic $\Zp$-extension of $K$ and ${H_{ac}:=\Gal(\KK_{\infty}/K_{ac})}$. We now need the following important theorem due to Pollack and Weston.
\begin{theorem}\label{PW_theorem}
Let ${\bullet \in \{+,-\}}$. Assume the following
\begin{enumerate}[(a)]
\item $N^-$ is a squarefree product of an odd number of primes,
\item both primes of $K$ above $p$ are totally ramified in $K_{ac}/K$,
\item ${\Gal(\Q(E[p])/\Q)=GL_2(\Fp)}$,
\item if ${q \mid N^-}$ and ${q \equiv \pm 1 \pmod{p}}$, then $q$ ramifies in ${\Q(E[p])/\Q}$.
\end{enumerate}
Then $X^{Gr, \bullet\bullet}(E/K_{ac})$ is a torsion $\Lambda$-module.
\end{theorem}
\begin{proof}
By Proposition~\ref{Omega_prop} we have that ${\bigoplus_{v \in S \setminus S_p} \Omega_{H_{ac},v}}$ is a cotorsion $\Lambda$-module. Therefore from Proposition~\ref{Selmercomp_prop} we have that $X^{Gr, \bullet\bullet}(E/K_{ac})$ is a torsion $\Lambda$-module if and only if $X^{\bullet\bullet}(E/K_{ac})$ is a torsion $\Lambda$-module. This latter group is shown to be $\Lambda$-torsion in \cite[Theorem~4.1]{PW}; however, we should note that the results of loc. cit. implicitly assume that every prime $q$ dividing $N^+$ ramifies in $\Q(E[p])/\Q$. This latter assumption of $N^+$-minimality can be removed by the level-lowering trick in \cite[Corollary~2.3]{KPW}.
\end{proof}

\begin{lemma}\label{Kida_lemma}
Let $q$ be a rational prime dividing $N$ that is unramified in $K/\Q$. The following are equivalent:
	\begin{enumerate}[(a)]
		\item The Kodaira type of $E$ at $q$ is $I_n$ with ${p \mid n}$.
		\item $q$ does not ramify in ${\Q(E[p])/\Q}$.
		\item ${p \mid c_v}$ for any prime $v$ of $K$ above $q$.
	\end{enumerate}
\end{lemma}
\begin{proof}
	This is \cite[Lemma~10.5]{MHG}.
\end{proof}
Let $h_K$ be the class number of $K$.

\begin{theorem}\label{example_theorem1}
Let ${\bullet \in \{+,-\}}$. Assume the following
\begin{enumerate}[(a)]
\item $N^-$ is a squarefree product of an odd number of primes.
\item ${p \nmid h_K}$.
\item ${\Gal(\Q(E[p])/\Q)=GL_2(\Fp)}$.
\item If ${q \mid N^-}$ and ${q \equiv \pm 1 \pmod{p}}$, then ${p \nmid c_v}$ for the prime $v$ of $K$ above $q$.
\item For some prime $q$ dividing $N^-$, we have that $E$ has multiplicative reduction at $q$ and ${p \mid c_v}$ for the prime $v$ of $K$ above $q$.
\item ${\mu(L^{\bullet}_p(E))=\mu(L^{\bullet}_p(E^{(K)}))=0}$.
\end{enumerate}
Then ${H_{ac} \subseteq \mathcal{H}_s^{\bullet\star}}$ and furthermore we have
\begin{enumerate}
\item $X^{Gr, \bullet\bullet}(E/\KK_{\infty})_f$ is not finitely generated over $\Lambda(H_{ac})$.
\item If $X^{Gr}(E/\KK_{\infty})_f$ is finitely generated over $\Lambda(H_{ac})$, then ${\mu_{G/H_{ac}}(T_{\Lambda}([F_{\Lambda_2}(X^{Gr}(E/\KK_{\infty}))]_{H_{ac}})) >0}$.
\end{enumerate}
\end{theorem}
\begin{proof}
By Lemma~\ref{Kida_lemma} we have that $(d)$ of this theorem is equivalent to Theorem~\ref{PW_theorem}$(d)$. Also $(b)$ of this theorem implies Theorem~\ref{PW_theorem}$(b)$. Therefore by Theorem~\ref{PW_theorem} $X^{Gr, \bullet\bullet}(E/K_{ac})$ is a torsion $\Lambda$-module. Since both $\fp$ and $\bar{\fp}$ ramify in $K_{ac}/K$, it follows that ${H_{ac} \subseteq \mathcal{H}_s^{\bullet\star}}$. By $(f)$ and Proposition~\ref{muzero_prop} we have that
\begin{equation}\label{muzero_equality}
\mu(T_{\Lambda_2}(X^{Gr}(E/\KK_{\infty})))=\mu(X^{Gr, \bullet\star}(E/\KK_{\infty}))=0.
\end{equation}
From $(e)$, Proposition~\ref{Omega_prop} and Corollary~\ref{muinequality_corollary} we have that
\begin{equation}\label{mupositive_equality1}
\mu(X^{Gr, \bullet\bullet}(E/K_{ac}))>0,
\end{equation}

\begin{equation}\label{mupositive_equality2}
\mu(T_{\Lambda}(X^{Gr}(E/K_{ac})))>0.
\end{equation}
Then (1) follows from \eqref{muzero_equality}, \eqref{mupositive_equality1} and  Theorem~\ref{splitmain_theorem}. (2) follows from \eqref{muzero_equality}, \eqref{mupositive_equality2} and  Theorem~\ref{splitmain_theorem3}.
\end{proof}

\begin{theorem}\label{example_theorem2}
Let ${\bullet \in \{+,-\}}$. Assume the following
\begin{enumerate}[(a)]
\item $N^-$ is a squarefree product of an even number of primes.
\item ${p \nmid h_K}$.
\item ${\Gal(\Q(E[p])/\Q)=GL_2(\Fp)}$.
\item For some prime $q$ dividing $N^+$, we have that $E$ has split multiplicative reduction at $q$ and ${p \mid c_v}$ for both primes $v$ of $K$ above $q$.
\item ${\mu(L^{\bullet}_p(E))=\mu(L^{\bullet}_p(E^{(K)}))=0}$.
\item ${\lambda(L^{\bullet}_p(E))+\lambda(L^{\bullet}_p(E^{(K)})) \leq 2}$.
\end{enumerate}
Then there exist distinct groups ${H_1, H_2 \subseteq \mathcal{H}_s^{\bullet\star}}$ such that for ${i=1,2}$ we have
\begin{enumerate}
\item $X^{Gr, \bullet\bullet}(E/\KK_{\infty})_f$ is not finitely generated over $\Lambda(H_i)$.
\item If $X^{Gr}(E/\KK_{\infty})_f$ is finitely generated over $\Lambda(H_i)$, then ${\mu_{G/H_i}(T_{\Lambda}([F_{\Lambda_2}(X^{Gr}(E/\KK_{\infty}))]_{H_i})) >0}$.
\end{enumerate}
\end{theorem}
\begin{proof}
By $(a)$, $(b)$, $(c)$ and \cite[Theorem~1.4]{LV} we have that ${\rank_{\Lambda}(X^{\bullet\bullet}(E/K_{ac}))=1}$. By $(e)$ and Proposition~\ref{muzero_prop} we have that
\begin{equation}\label{muzero_equality2}
\mu(T_{\Lambda_2}(X^{Gr}(E/\KK_{\infty})))=\mu(X^{Gr, \bullet\star}(E/\KK_{\infty}))=0.
\end{equation}
Furthermore, we have that the number of $\Zp$-extensions of ${L_{\infty}/K}$ where $X^{\bullet\star}(E/L_{\infty})$ is not a torsion $\Lambda$-module is at most 2. Let $q$ be a prime as in $(d)$. Then $q$ splits in ${K/\Q}$: ${q\cO_k=\fq\bar{\fq}}$. Since ${q \neq p}$, the prime $\fq$ (resp. $\bar{\fq}$) splits completely in a $\Zp$-extension $L_{1,\infty}/K$ (resp. $L_{2,\infty}/K$). We claim that ${L_{1,\infty} \neq L_{2,\infty}}$. Let $\tau$ be the complex conjugation automorphism of $\bar{\Q}$. Since ${\tau(\fp)=\bar{\fp}}$ it follows that ${\tau(L_{1,\infty})=L_{2,\infty}}$. So if ${L_{1,\infty}=L_{2,\infty}}$ then $L_{1,\infty}/\Q$ would be a Galois extension. However the only $\Zp$-extensions of $K$ that are Galois over $\Q$ are $K_{cyc}$ and $K_{ac}$. No primes of $K$ split completely $K_{cyc}/K$ and as $q$ splits in $K/\Q$ neither $\fp$ nor $\bar{\fp}$ split completely in $K_{ac}/K$ (see \cite[Corollary~1]{Brink}). Therefore as claimed ${L_{1,\infty} \neq L_{2,\infty}}$.

Since ${\tau(L_{1,\infty})=L_{2,\infty}}$, it is easy to see that $X^{\bullet\star}(E/L_{1,\infty})$ is not a torsion $\Lambda$-module if and only if $X^{\bullet\star}(E/L_{2,\infty})$ is not a torsion $\Lambda$-module (here $\Lambda$ refers to the corresponding Iwasawa algebra). Furthermore, we have that the number of $\Zp$-extensions of $L_{\infty}/K$ where $X^{\bullet\star}(E/L_{\infty})$ is not a torsion $\Lambda$-module is at most two and ${\rank_{\Lambda}(X^{\bullet\bullet}(E/K_{ac}))=1}$. Therefore we see from this that both $X^{\bullet\star}(E/L_{1,\infty})$ and $X^{\bullet\star}(E/L_{2,\infty})$ are $\Lambda$-torsion. By Lemma~\ref{extensionsplit_lemma} we have that ${L_{1,\infty}, L_{2,\infty} \not\in \{L^{(\fp)}_{\infty}, L^{(\bar{\fp})}_{\infty}\}}$. Thus all primes of $K$ above $p$ ramify in $L_{1,\infty}/K$ and $L_{2,\infty}/K$. Now let ${H_i=\Gal(\KK_{\infty}/L_{i,\infty})}$. Then ${H_1, H_2 \subseteq \mathcal{H}_s^{\bullet\star}}$ and ${H_1 \neq H_2}$.

From $(d)$, Proposition~\ref{Omega_prop} and Corollary~\ref{muinequality_corollary} we have that for ${i=1,2}$
\begin{equation}\label{mupositive_equality3}
\mu(X^{Gr, \bullet\bullet}(E/\KK_{\infty}^{H_i}))>0
\end{equation}
and
\begin{equation}\label{mupositive_equality4}
\mu(T_{\Lambda}(X^{Gr}(E/\KK_{\infty}^{H_i})))>0.
\end{equation}
Then (1) follows from \eqref{muzero_equality2}, \eqref{mupositive_equality3} and  Theorem~\ref{splitmain_theorem}. Moreover, (2) follows from \eqref{muzero_equality2}, \eqref{mupositive_equality4} and  Theorem~\ref{splitmain_theorem3}.
\end{proof}

Now we provide four concrete numerical examples: three of these examples illustrate Theorem~\ref{example_theorem1} and one illustrates Theorem~\ref{example_theorem2}. \\
\textbf{Example 1 for Theorem~\ref{example_theorem1}}: Let ${E=522g1}$ (Cremona labelling~\cite{Cremona}), ${K=\Q(\sqrt{-35})}$ and ${p=11}$. The conditions~${(a)-(e)}$ of Theorem~\ref{example_theorem1} are verified by SAGE~\cite{Sage}. The LMFDB database~\cite{LMFDB} shows that ${\mu(L^{\bullet}_p(E))=0}$ for both ${\bullet \in \{+,-\}}$. The quadratic twist $E^{(K)}$ is not in the LMFDB database. Computations done by Robert Pollack verify that ${\mu(L^{\bullet}_p(E^{(K)}))=0}$ for both ${\bullet \in \{+,-\}}$ as well\footnote{We are grateful to Robert Pollack for his help with the computations.}. Thus condition $(f)$ is true for both ${\bullet=+}$ and ${\bullet=-}$.

Note that the computations done by Robert Pollack also show that ${\lambda(L^{\bullet}_p(E^{(K)}))=0}$. By \cite[Theorem~3.14]{KimFiniteSubmodules} $X^{\bullet}(E^{(K)}/\Q_{cyc})$ does not contain any non-trivial finite $\Lambda$-submodules. It follows that ${X^{\bullet}(E^{(K)}/\Q_{cyc})=0}$. On the other hand, $X^{Gr, \bullet}(E/\Q_{cyc})$ and $X^{Gr, \bullet\bullet}(E/\KK_{\infty})$ are non-trivial. In fact, the analytic $\lambda$-invariant of $X^{Gr, \bullet}(E/\Q_{cyc})$ is equal to 1 by the LMFDB database, and $X^{Gr, \bullet\bullet}(E/\KK_{\infty})$ cannot be trivial as $X^{Gr, \bullet\bullet}(E/\KK_{\infty})_f$ is not finitely generated over $\Lambda(H_{ac})$ by Theorem~\ref{example_theorem1}.

\textbf{Example 2 for Theorem~\ref{example_theorem1}}: Let ${E=522g1}$, ${K=\Q(\sqrt{-83})}$ and ${p=11}$. The conditions~${(a)-(e)}$ of Theorem~\ref{example_theorem1} are verified by SAGE. The LMFDB database shows that ${\mu(L^{\bullet}_p(E))=0}$ for both ${\bullet \in \{+,-\}}$. The quadratic twist $E^{(K)}$ is not in the LMFDB database. Computations done by Robert Pollack verify that ${\mu(L^{\bullet}_p(E^{(K)}))=0}$ for both ${\bullet \in \{+,-\}}$ as well. Thus condition $(f)$ is true for both ${\bullet=+}$ and ${\bullet=-}$.

\textbf{Example 3 for Theorem~\ref{example_theorem1}}: Let ${E=542a1}$, ${K=\Q(\sqrt{-3})}$ and ${p=7}$. The conditions~${(a)-(e)}$ of Theorem~\ref{example_theorem1} are verified by SAGE. The LMFDB database shows that ${\mu(L^{\bullet}_p(E))=\mu(L^{\bullet}_p(E^{(K)}))=0}$ for both ${\bullet \in \{+,-\}}$. Thus condition~$(f)$ is true for both ${\bullet=+}$ and ${\bullet=-}$.

\textbf{Example for Theorem~\ref{example_theorem2}}: Let ${E=615b1}$, ${K=\Q(\sqrt{-17})}$ and ${p=7}$. The conditions~${(a)-(d)}$ of Theorem~\ref{example_theorem2} are verified by SAGE. The LMFDB database shows that ${\mu(L^{\bullet}_p(E))=0}$ for both ${\bullet \in \{+,-\}}$ and ${\lambda(L^{\bullet}_p(E))=1}$ for both ${\bullet \in \{+,-\}}$. The quadratic twist $E^{(K)}$ is not in the LMFDB database. Computations done by Robert Pollack verify that ${\mu(L^{\bullet}_p(E))=\lambda(L^{\bullet}_p(E))=0}$ for both ${\bullet \in \{+,-\}}$. Thus conditions~$(e)$ and $(f)$ are true for both ${\bullet=+}$ and ${\bullet=-}$.

\end{document}